\documentclass[10pt]{article} 
\usepackage[letterpaper]{geometry}

\usepackage[english]{babel}

\usepackage{setspace}
\usepackage{amsmath}
\usepackage{amsthm}
\usepackage{amssymb}
\usepackage{amsfonts}
\usepackage{mathrsfs}
\usepackage{thmtools}
\usepackage{thm-restate}
\usepackage{graphicx}
\usepackage[colorlinks=true, allcolors=Nord8]{hyperref}
\usepackage{tikz-cd} 
\usepackage{quiver}
\usepackage{enumitem}
\usepackage{xcolor}
\usepackage{etoolbox}
\usepackage[colorinlistoftodos]{todonotes}
\usepackage{float}
\usepackage{theoremref}
\definecolor{Nord0}{HTML}{2e3440}

\definecolor{Nord4}{HTML}{d8dee9}

\definecolor{Nord7}{HTML}{8fbcbb}
\definecolor{Nord8}{HTML}{88c0d0}
\definecolor{Nord9}{HTML}{81a1c1}
\definecolor{Nord10}{HTML}{5e81ac}
\definecolor{Nord11}{HTML}{bf616a}
\definecolor{Nord12}{HTML}{d08770}
\definecolor{Nord13}{HTML}{ebcb8b}
\definecolor{Nord14}{HTML}{a3be8c} 
\definecolor{Nord15}{HTML}{b48ead} 

\pagecolor{Nord0} 
\color{Nord4} 

\theoremstyle{definition}
\newtheorem{theorem}{Theorem}[section]
\newtheorem{proposition}[theorem]{Proposition}
\newtheorem{corollary}[theorem]{Corollary}
\newtheorem{lemma}[theorem]{Lemma}

\newtheorem{definition}[theorem]{Definition}

\newtheorem{notation}[theorem]{Notation}
\newtheorem{assumption}[theorem]{Assumption}
\newtheorem{convention}[theorem]{Convention}

\newtheorem{remark}[theorem]{Remark}
\newtheorem{claim}{Claim}[theorem]

\newcommand{\mc}{\mathcal}

\setuptodonotes{inline, color=Nord13} 

\newcommand{\PP}{\mathbb{P}}

\renewcommand{\a}{\alpha}
\renewcommand{\b}{\beta}

\newcommand{\e}{\varepsilon}

\newcommand{\UU}{\mc{U}}
\newcommand{\VV}{\mc{V}}
\newcommand{\WW}{\mc{W}}

\renewcommand{\hat}{\widehat}

\newcommand{\total}[1]{\overline{X_{#1}}}
\newcommand{\hatt}[1]{\widehat{X_{#1}}}
\newcommand{\hatto}[1]{\widehat{X^\circ_{#1}}}
\newcommand{\Geod}[0]{\mathrm{Geod}}
\newcommand{\diam}[0]{\mathrm{diam}}
\newcommand{\XX}[0]{\mathcal{X}}
\newcommand{\YY}[0]{\mathcal{Y}}

\newcommand{\OO}[0]{\mathcal{O}}
\newcommand{\internal}[1]{X_{#1}^\circ}

\newcommand{\tCone}[0]{\fitTilde{\Cone}}
\newcommand{\pCone}[0]{\mathrm{pCone}}
\newcommand{\Cone}[0]{\mathrm{Cone}}

\usepackage{tikz}
\usetikzlibrary{math}
\usepgfmodule{oo}
\usepackage{scalerel}
\usepackage{stackengine,wasysym}
\usepackage{mathtools} 
\newcommand{\fitTilde}[1]{\ThisStyle{%
  \setbox0=\hbox{$\SavedStyle#1$}%
  \stackengine{-.1\LMpt}{$\SavedStyle#1$}{%
    \stretchto{\scaleto{\SavedStyle\mkern.2mu\AC}{.5150\wd0}}{.6\ht0}%
  }{O}{c}{F}{T}{S}%
}}

\usepackage{thm-restate}

\newcommand{\llangle}[1][]{\savebox{\@brx}{\(\m@th{#1\langle}\)}%
  \mathopen{\copy\@brx\mkern2mu\kern-0.9\wd\@brx\usebox{\@brx}}}
\newcommand{\rrangle}[1][]{\savebox{\@brx}{\(\m@th{#1\rangle}\)}%
  \mathclose{\copy\@brx\mkern2mu\kern-0.9\wd\@brx\usebox{\@brx}}}
\makeatother

\hypersetup{allcolors = blue}
\pagecolor{white}
\color{black}

\title{Combining relatively hyperbolic groups over a complex of groups}
\author{Darius Alizadeh}

\begin{document}
\maketitle

\begin{abstract}
Given a complex of groups $G(\YY) = (G_\sigma, \psi_a, g_{a,b})$ where all $G_\sigma$ are relatively hyperbolic, the $\psi_a$ are inclusions of full relatively quasiconvex subgroups, and the universal cover $X$ is CAT$(0)$ and $\delta$--hyperbolic, we show $\pi_1(G(\YY))$ is relatively hyperbolic. The proof extends the work of Dahmani and Martin by constructing a model for the Bowditch boundary of $\pi_1(G(\YY))$. We prove the model is a compact metrizable space on which $G$ acts as a geometrically finite convergence group, and a theorem of Yaman then implies the result. More generally, this model shows how any suitable action of a relatively hyperbolic group on a simply connected cell complex encodes a decomposition of the Bowditch boundary into the boundary of the cell complex and the boundaries of cell stabilizers. We hope this decomposition will be helpful in answering topological questions about Bowditch boundaries.

\end{abstract}


\tableofcontents
\newpage

\section{Introduction}
\label{sec:introduction}

This work is another chapter in the story of combination theorems which began with \cite{BestvinaFeighn}, where Bestvina and Feighn give conditions on a graph of hyperbolic spaces which guarantee the resulting space is itself hyperbolic. Their conditions deal directly with the metric on the resulting space, and this metric approach has been elaborated on in \cite{alibegovic}, \cite{Mj_2008} with applications to limit groups and relatively hyperbolic groups. An alternative, dynamical approach was opened by a theorem of Bowditch in \cite{Bowditch1998}, which equates geometrically finite convergence groups with relatively hyperbolic groups. Dahmani \cite{Dahmani_2003} applied this theorem to prove a combination theorem for graphs of relatively hyperbolic groups and more. He combined Bowditch boundaries over the Bass--Serre tree to build a compact metrizable space on which the fundamental group acts as a geometrically finite convergence group, then applied Bowditch's theorem to conclude the fundamental group was relatively hyperbolic. Martin \cite{Martin} pushed this dynamical strategy into arbitrary dimensions by combining boundaries of hyperbolic groups over the development of a complex of groups, and we extend Martin's work to relatively hyperbolic groups. 

Recall from \cite{BH}[I.7] that for $\kappa \in \mathbb{R}$, an $M_\kappa$--complex is a cell complex is where each $n$--cell is modeled on a simplex in the model space $M^n_\kappa$, which is the unique simply--connected Riemannian manifold of dimension $n$ with constant curvature $\kappa$. For example, a cube complex is an $M_0$--complex. See \thref{defn of M_kappa} for details. Here is our main theorem.

\begin{restatable}{theorem}{MainTheorem}\label{main theorem complex of groups}
    Let $G(\YY) = (G_\sigma,\psi_a,g_{a,b})$ be a nonpositively curved developable complex of groups over a scwol $\YY$, where each $G_\sigma$ is a relatively hyperbolic group and each $\psi_a$ is the inclusion of a full relatively quasiconvex subgroup. Let $\XX$ be the universal cover of $G(\YY)$, and let $X$ be the geometric realization of $\XX$ equipped with an $M_\kappa$ structure. Suppose $X$ is $\delta$--hyperbolic and the action of $G = \pi_1(G(\YY))$ on $X$ is acylindrical. Then $G$ is relatively hyperbolic. The maximal parabolic subgroups of $G$ are virtually parabolic subgroups of vertex stabilizers, and the stabilizer of each simplex in $X$ is a full relatively quasiconvex subgroup of $G$. 
\end{restatable}

The relevant terms are explained in Section \ref{section:Background}. Just as a group action on a tree induces a graph of groups structure in Bass--Serre theory, a cocompact action on a simply connected cell complex induces a complex of groups over the quotient. These cocompact actions are usually easier to describe and arise more naturally, e.g., groups acting on $\textrm{CAT}(0)$--cube complexes. Here is a rephrasing of our main theorem from this point of view.

\begin{theorem}\thlabel{main theorem}
    Let $X$ be a $\delta$--hyperbolic, $\textrm{CAT}(0)$ $M_\kappa$--complex with $\kappa \leq 0$. Let $G$ be a group acting on $X$ cocompactly and acylindrically. Write $G_\sigma$ for the point-wise stabilizer of a simplex $\sigma$ of $X$. Suppose that for every pair of simplices $\sigma' \subset \sigma$ in $X$, $G_{\sigma'}$ and $G_{\sigma}$ are relatively hyperbolic and $G_{\sigma}$ is a full RQC subgroup of $G_{\sigma'}$. Then $G$ is relatively hyperbolic. The maximal parabolic subgroups of $G$ are virtually the maximal parabolic subgroups of vertex stabilizers, and each $G_\sigma$ is a full RQC subgroup of $G$. 
\end{theorem}

Like \cite{Dahmani_2003} and \cite{Martin}, we construct a model of the boundary for $G$ by gluing the Bowditch boundaries of simplex stabilizers together and show $G$ acts on our model as a geometrically finite convergence group. Unlike Martin however, we must deal with parabolic points in these boundaries. Much of the work in defining this model and its topology comes from \cite{Martin}. It is rather complex, and we hope to clarify many of the details involved. Further, while the conclusion of \thref{main theorem} is that a certain group is relatively hyperbolic, the real fruit of this labor is this model. Given a group acting on a tree, Bass--Serre theory explains how the tree encodes the combinatorial structure of the group. Given a relatively hyperbolic group acting on a cell complex, this model shows how the cell complex encodes the boundary of the group. We describe an illustrative example. 

Let $\Sigma$ be the surface of genus 2, $c$ be a geodesic which separates $\Sigma$ into two tori with boundary, and let $G = \pi_1(\Sigma)$. Since $G$ acts geometrically on the universal cover $\mathbb{H}^2$, $G$ is hyperbolic and $\partial G = S^1$. Let $\gamma \in G$ be the element corresponding to $c$. Then each conjugate of $\gamma$ acts on $\mathbb{H}^2$ as a hyperbolic isometry fixing two points in $S^1$, namely the endpoints of some lift of $c$. It turns out $G$ is hyperbolic relative to $\langle \gamma \rangle$, and using a theorem of Tran \cite{TRAN_2013}, the Bowditch boundary of $\partial_{\langle \gamma \rangle}G$ is given by contracting the boundary of each parabolic subgroup to a point. So each lift of $c$ to $\mathbb{H}^2$ gives a pair of points in the boundary $S^1$, and $\partial_{\langle \gamma \rangle}G$ is given by contracting each pair to a point. Each contraction takes the boundary circle and pinches it together at two points, so the resembles a tree of circles, see Figure \ref{fig:example tree}. If we fix a basepoint $x_0 \in \mathbb{H}^2$, then some geodesics from $x_0$ to the boundary $S^1$ will cross infinitely many lifts of $c$. These points are not in one of the circles of the tree of circles, but correspond to points in the boundary of the underlying tree.

The two tori with boundary and $c$ make $\Sigma$ into a graph of spaces, which induces a graph of groups structure for $G$:
    \[\textcolor{red}{(F_2, [a_1,b_1])} \longleftarrow \langle \gamma \rangle \longrightarrow \textcolor{blue}{(F_2, [a_2,b_2])}.\]
The edge group is simply $\gamma$, which is hyperbolic relative to itself and has Bowditch boundary a single point. Each vertex group is a free group, and $\gamma$ maps to the commutator. A free group on two generators is hyperbolic relative to this commutator and its Bowditch boundary is $S^1$, which can be seen by taking a hyperbolic structure on the cusped torus. Let $X$ be the Bass--Serre tree of this splitting. Then $G$ acts on $X$ as in \thref{main theorem}, and $X$ is exactly the underlying tree of the tree of circles in the previous paragraph. 

To construct our model, we take a copy of the Bowditch boundary for each cell of $X$ and glue the boundaries using the inclusion of cell stabilizers and limit sets. In this example, this means a circle for each vertex, a point for each edge, and gluing these circles along those points. The resulting space is almost the tree of circles above, but is not compact -- a sequence of points in circles which escape to infinity will not converge. Such a sequence represents a sequence of vertices in $X$, so we compactify by adding $\partial X$. The proof of \thref{main theorem} shows how to construct such a model in general.

\begin{figure}
    \begin{center}
    \begin{tikzpicture}

     
    \draw[fill=red] (-.25,.97) circle[radius=2pt];
    \draw[fill=blue] (.25,-.97) circle[radius=2pt];
    \draw[shorten >=2pt,shorten <=2pt](-.25,.97) -- (.25,-.97);
    \begin{scope}[shift={(-.25,.97)}]
    \foreach \a in {0,20,40,...,220}{
    \draw[fill = blue] (\a: 1.05cm) circle[radius=1.5pt];
    \draw[shorten <=2pt, shorten >=1.5pt] (0:0cm) -- (\a:1.05cm);
    
    \begin{scope}[shift={(\a:1.05cm)}]
        \foreach \b in {270,310,350,390}{
        \draw[fill = red] (\b+1.3*\a:.15cm) circle[radius = .75pt];
        \draw[shorten <=1.5pt, shorten >=.75pt] (0,0) -- (\b+1.3*\a: .15cm); 
        }
    \end{scope}
    }
    \end{scope}

    \begin{scope}[shift={(.25,-.97)}]
    \foreach \a in {170,190,210,...,390}{
    \draw[fill = red] (\a: 1.05cm) circle[radius=1.5pt];
    \draw[shorten <=2pt, shorten >=1.5pt] (0:0cm) -- (\a:1.05cm);
    
    \begin{scope}[shift={(\a:1.05cm)}]
        \foreach \b in {270,310,350,390}{
        \draw[fill = blue] (\b+1.1*\a:.15cm) circle[radius = .75pt];
        \draw[shorten <=1.5pt, shorten >=.75pt] (0,0) -- (\b+1.1*\a: .15cm); 
        }
    \end{scope}
    }
    \end{scope}
    \node at (-4,0) {\includegraphics[height = 5cm]{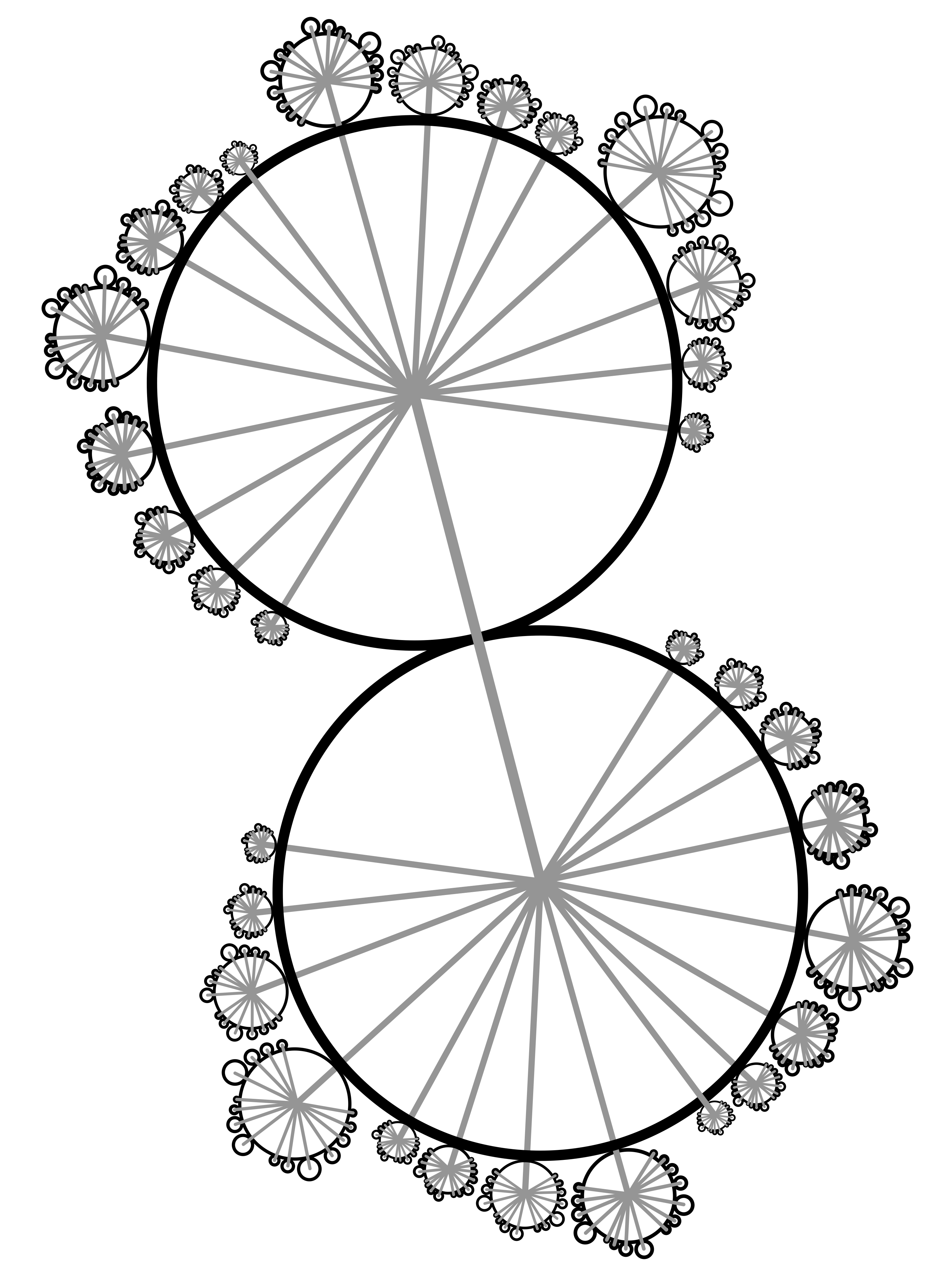}};
    \node at (-9,0) {\includegraphics[height = 5cm]{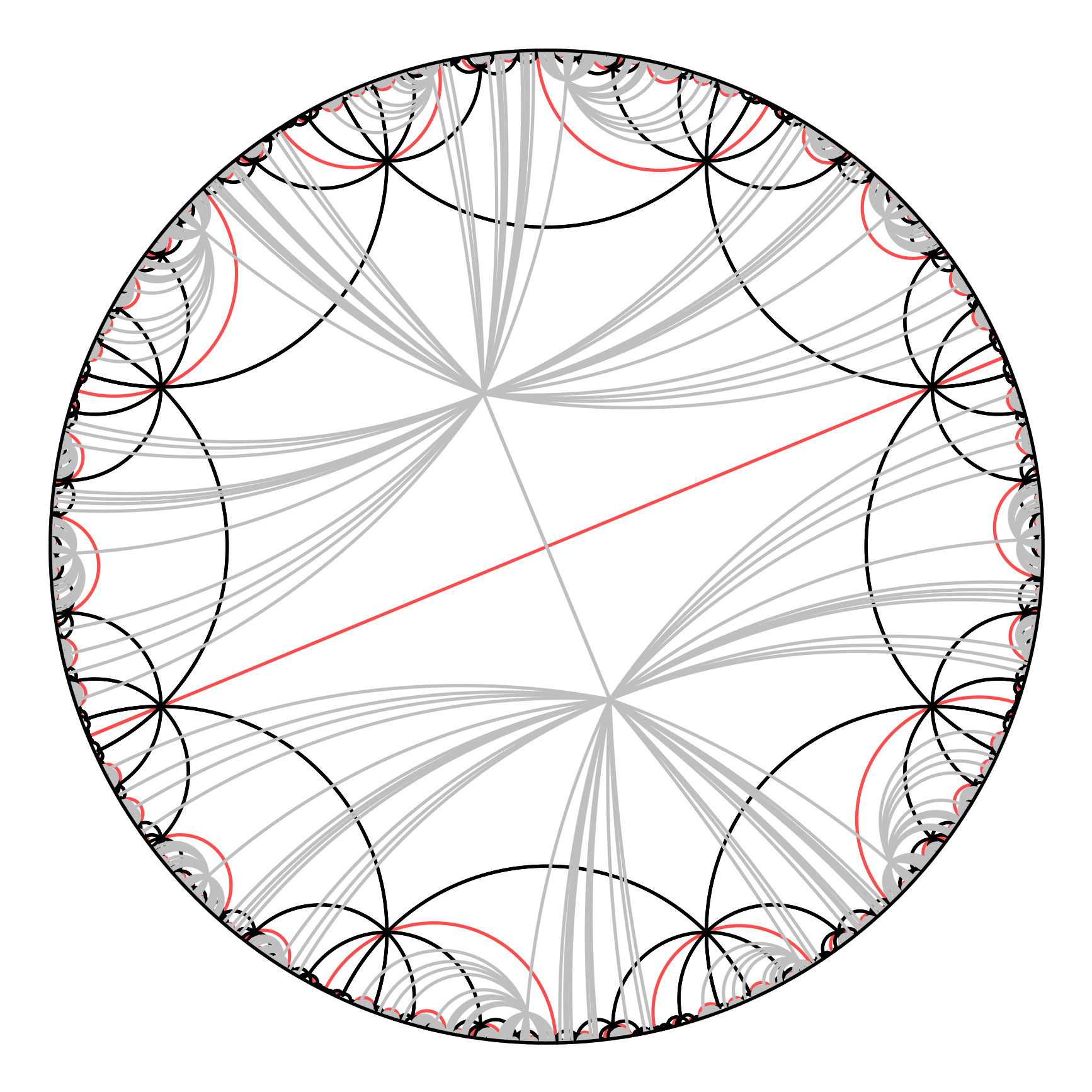}};
    \end{tikzpicture}\end{center}
    
    \caption{On the left is $\mathbb{H}^2$ tiled by octagons, describing an action of $G$. The red curves represent lifts of $c$. In the center is the tree of circles we get by contracting the endpoints of each lift of $c$. These images come from \cite{benzvi2022hyperbolicboundariesvshyperbolic}. On the right is part of the Bass--Serre tree of the splitting for $G$. Red and blue vertices are stabilized by conjugates of $[a_1,b_1]$ and $[a_2,b_2]$ respectively.}
    \label{fig:example tree}
\end{figure}
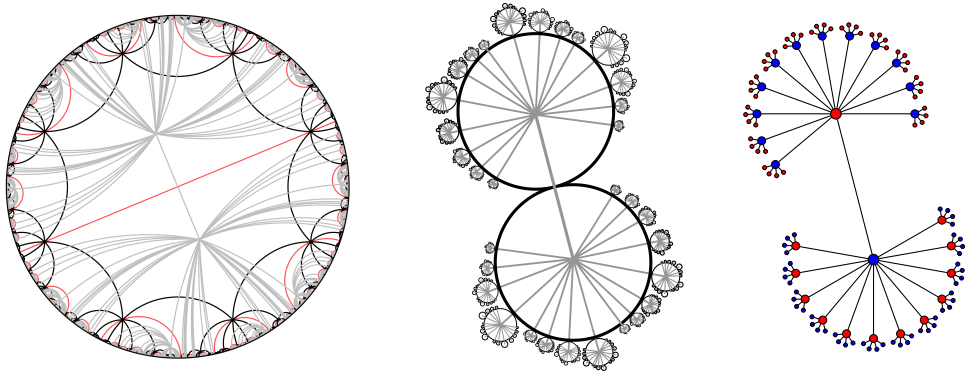

We describe the layout. Section \ref{section:Background} provides the necessary background on relatively hyperbolic groups and $M_{\kappa}$--complexes. Section \ref{section:Constructing overlineZ} defines our model $\overline{Z}$ as a set and proves some basic properties about it. Section \ref{section:Geometric Tools} proves many lemmas which we use throughout the later sections. In Sections \ref{section:The Topology} and \ref{section: Properties of the Topology}, we define a neighborhood basis for points in $\overline{Z}$ and show that it indeed forms the basis for a compact metrizable topology. In section \ref{section:Dynamics}, we show $G$ acts as a convergence group on $\overline{Z}$, understand its limit set, and apply Yaman's \thref{Geometrically Finite Convergence implies RelHyp} to conclude.

\section{Background}\label{section:Background}

\subsection{Relatively Hyperbolic Groups}
For this subsection, let $G$ be a group acting on a compact metrizable space $M$.
\begin{definition}[Convergence Groups]\thlabel{defn:Convergence group}
    Given a sequence $(g_n)_n$ in $G$ and points $\xi_+,\xi_- \in M$, we call the triple $(g_n,\xi_+,\xi_-)$ an \emph{attracting repelling triple} (ART) if for any compact set $K \subset M\setminus \{\xi_-\}$, the sets $g_nK$ converge to $\xi_+$ uniformly. In this situation, $(g_n)_n$ is called a \emph{convergence sequence} and the points $\xi_+$ and $\xi_-$ are the \emph{attracting} and \emph{repelling} points of the sequence $(g_n)_n$. We say $G$ acts as a \emph{convergence group} on $M$ if every infinite sequence in $G$ has a subsequence which is a convergence sequence.
\end{definition}

 If $(g_n,\xi_+,\xi_-)$ is an ART, it is possible that $\xi_+ = \xi_-$. It's clear that being a convergence group passes to subgroups. Acting as a convergence group is equivalent to acting properly discontinuously on the space of distinct triples of $M$ \cite{Bowditch1999}. 
 
\begin{definition}[Conical and Parabolic Limit Points]
    Let $G$ be a convergence group on $M$. A point $\xi \in M$ is a \emph{conical limit point} if there exists a sequence $(g_n)_n$ in $G$ and points $\xi_+ \neq \xi_-$ in $G$ such that $g_n\xi \longrightarrow \xi_-$ and $g_n\xi'\longrightarrow \xi_+$ for all $\xi' \neq \xi$ in $M$.
    
    A subgroup $P < G$ is called \emph{parabolic} if it is infinite, fixes a point $\xi \in M$, and contains no loxodromics (elements which fix exactly $2$ points of $M$). This fixed point $\xi$ is unique and is called a \emph{parabolic point}. Further, $\xi$ is a \emph{bounded parabolic point} if $\mathrm{Stab}_G(\xi)$ acts properly discontinuously and cocompactly on $M \setminus \{\xi\}$. 
\end{definition}

\begin{definition}[Geometrically Finite]\thlabel{Geometrically Finite}
    Let $G$ act as a convergence group on $M$. The action of $G$ is \emph{geometrically finite} if every point of $M$ is either a bounded parabolic point or conical limit point.
\end{definition}

In \cite{Hruska_2010}, Hruska lays out six equivalent definitions of relative hyperbolicity. We will only be concerned with two of them. We note that some of these definitions allow finite parabolic subgroups, which we do not allow. Recall from \cite[III.H]{BH} that if $\Sigma$ is a proper $\delta$--hyperbolic metric space, then its boundary $\partial \Sigma$ is defined by equivalence classes of asymptotic rays. Then $\overline{\Sigma}:=\Sigma \cup \partial \Sigma $ can be topologized as a compact metrizable space, and if $G$ acts by isometries on $\Sigma$, this extends to an action on $\overline{\Sigma}$.

\begin{definition}[Relatively Hyperbolic]\thlabel{defn:Relatively Hyperbolic}
    Let $G$ be a group acting properly discontinuously by isometries on a proper hyperbolic metric space $\Sigma$ so that $G$ acts on $\partial \Sigma$ as a geometrically finite convergence group. Let $\mathbb{P}$ be a collection of one representative from each conjugacy class of maximal parabolic subgroups, and assume each element of $\mathbb{P}$ is finitely generated. Then we say $G$ is \emph{hyperbolic relative to} $\mathbb{P}$, or $(G,\mathbb{P})$ is relatively hyperbolic. The collection $\mathbb{P}$ is the \emph{peripheral structure}. The boundary $\partial \Sigma$ is the \emph{Bowditch boundary} of $(G,\mathbb{P})$, denoted $\partial_\mathbb{P}G$.
\end{definition}

The peripheral structure is fundamental to how we view $G$ as a relatively hyperbolic group --- different peripheral structures on the same group $G$ lead to different Bowditch boundaries. However, when context makes the peripheral structure clear, we will be glib and say that $G$ is relatively hyperbolic with Bowditch boundary $\partial G$. The following remarkable theorem shows that the only geometrically finite convergence actions on compact metrizable spaces are relatively hyperbolic groups acting on their Bowditch boundaries. This also implies the Bowditch boundary depends only on the group $G$ and its peripheral structure. 

\begin{theorem} \thlabel{Geometrically Finite Convergence implies RelHyp}\cite{Yaman_2004} \cite{Bowditch1998}
    Let $G$ be a geometrically finite convergence group on a perfect compact metrizable space $M$. Let $\mathbb{P}$ be a collection of representatives of conjugacy classes of maximal parabolic subgroups, and assume that each element of $P$ is finitely generated. Then $(G,\mathbb{P})$ is relatively hyperbolic and $M$ is equivariantly homeomorphic to $\partial_{\mathbb{P}}G$. 
\end{theorem}

Our goal is to show a particular $G$ is relatively hyperbolic. We will do it by constructing a candidate for $M$ and applying \thref{Geometrically Finite Convergence implies RelHyp}. Tukia has shown the set $\mathbb{P}$ is always finite:

\begin{proposition}\cite{Tukia1998}\thlabel{finitely many parabolics}
    If $G$ is a geometrically finite convergence group on a compact metrizable space $M$, then there are finitely many conjugacy classes of maximal parabolic subgroups.
\end{proposition}

The boundary of a hyperbolic group naturally compactifies the Cayley graph of the group, as described in \cite[III.H]{BH}. In \cite{GrovesManning2006}, Groves and Manning provide a space analogous to the Cayley graph for relatively hyperbolic groups by attaching combinatorial horoballs to the Cayley graph. Called the cusped space, this graph is compactified by the Bowditch boundary.

\begin{definition}\cite[Definition 3.1]{GrovesManning2006} Let $\Gamma$ be any graph. The \emph{combinatorial horoball based on $\Gamma$} has vertices $\Gamma^{(0)} \times (\ \mathbb{N} \sqcup \{ 0 \})$ and three kinds of edges:
\begin{enumerate}
    \item If $e$ is an edge of $\Gamma$ joining $v$ to $w$, then there is a corresponding edge $\overline{e}$ joining $(v,0)$ to $(w,0)$.
    \item If $k > 0$ and $0 < d_\Gamma(v,w) <2^k$ then there is a single edge joining $(v,k)$ to $(w,k)$.
    \item If $k \geq 0$ and $v \in \Gamma^{(0)}$, then there is an edge joining $(v,k)$ to $(v,k+1)$.
\end{enumerate}
\end{definition}

\begin{definition}\cite[Definition 3.12]{GrovesManning2006}
    Let $G$ be a finitely generated group with $\mathbb{P} = \{P_1, \ldots P_n\}$ a finite collection of finitely generated subgroups, and let $S$ be a finite generating set for $G$ so that $P_i \cap S$ generates $P_i$ for each $i$. Let $\Gamma$ be the Cayley graph of $G$ with respect to $S$. The \emph{cusped space} for $(G, \mathbb{P},S)$ is defined by attaching a combinatorial horoball to each coset of $P_i$, viewed as a subset of $\Gamma$. 
\end{definition}

\begin{theorem}\cite[Theorem 3.25]{GrovesManning2006} Let $G$ be a finitely generated group with $\mathbb{P}$ a finite collection of finitely generated subgroups, and let $X$ be a cusped space for $(G,\mathbb{P)}$. Then $(G,\mathbb{P})$ is relatively hyperbolic in a sense equivalent to \thref{defn:Relatively Hyperbolic} if and only if $X$ is $\delta$--hyperbolic, and in this case $\partial X = \partial_{\mathbb{P}}G$. 
\end{theorem}

\cite[Section 4]{Hruska_2010} explains how to adapt the cusped space when the peripheral subgroups are not finitely generated, as well as showing the equivalence of the various definitions of relative hyperbolicity. The Bowditch boundary only contains information about infinite subgroups through their limit sets, and for us the cusped space is a convenient way to manage finite subgroups. Except for a few lemmas like the next one, we do not deal in the details of cusped spaces, so we direct the reader to \cite{GrovesManning2006} for further details.

\begin{lemma}\thlabel{Parabolics are almost cocompact}
    Let $(G,\mathbb{P})$ be relatively hyperbolic with cusped space $C$. Let $P$ be a finite index subgroup of an element of $\PP$ fixing $\xi \in \partial C = \partial_\PP G$, and let $\mathbb{F} = \{F_1,\ldots F_m\}$ be a collection of finite subgroups of $G$. Then there exists a compact subset $K \subset \overline{C} = C \cup \partial C$ so that 
    \begin{enumerate}
        \item $P(\partial C \cap K) = \partial C \setminus \{\xi\}$, that is, $K \cap \partial C$ is a coarse fundamental domain for $P$ acting on $\partial G \setminus \{\xi\}$, and
        \item for any coset $gF_i$, there is some $p \in P$ so that $pgF_i \cap K \neq \varnothing$.
    \end{enumerate}
\end{lemma}
\begin{proof}
    Because $G$ is relatively hyperbolic, $\xi$ is a bounded parabolic point. Because $P$ is finite index in the stabilizer of $\xi$, we can choose a compact set $K_1 \subset \partial C \setminus \{\xi\}$ so that $PK_1 =\partial C \setminus \{\xi\}$.

    Consider the set $B = \{gF_i \, | \, d_C(gF_i,P) \text{ is realized by } d_C(gF_i,1)\}$ and let $\overline{B}$ be the closure of $\bigcup_{gF_i \in B} gF_i$ in $\overline{C}$. Notice $\overline{B}$ is compact because $\overline{C}$ is compact. Any point $\eta \in \overline{B}$ is the limit of a sequence of cosets $(g_nF_{i_n})_n$, and by Arzela --Ascoli the geodesics $[1,g_nF_{i_n}]$ subconverge to a geodesic $[1,\eta)$. Because the geodesics $[1,g_nF_{i_n}]$ all travel immediately away from the horoball containing $P$ for longer and longer distances, their limit cannot fellow travel the vertical geodesic $[1,\xi)$ forever. Thus $\xi \notin \overline{B}$ and $\overline{B}$ is compact in $\overline{C} \setminus \{\xi\}$. 

    Let $K = K_1 \cup \overline{B}$ and we show $K$ satisfies the conclusion. Clearly $K$ is compact since $K_1, \overline{B}$ are and clearly $(1)$ is satisfied because $K_1 \subset K \cap \partial C$. For any coset $gF_i$, suppose $d_C(gF_i,P)$ is realized by $d_C(gF_i,p)$ for some $p \in P$. Then $d_C(p^{-1}gF_i,P)$ is realized by $d_C(p^{-1}gF_i,1)$, and $p^{-1}gF_i \in B$. So $p^{-1}gF_i \subset \overline{B}$, and $(2)$ is satisfied.
\end{proof}

\subsection{Quasiconvexity}

Quasiconvex subgroups play an important role in hyperbolic groups. Depending on context, relatively quasiconvex and full relatively quasiconvex subgroups play the analogous role in relatively hyperbolic groups. 

\begin{definition}\label{defn:LimitSet}[Limit Set, \cite{Bowditch1998}\cite{Tukia1994}]\label{limit set}
If $G$ is a convergence group on a compact metrizable space $M$ and $H$ is an infinite subgroup, the \emph{limit set} $\Lambda H$ has three equivalent characterizations:
\begin{enumerate}
    \item the unique minimal nonempty closed $H$--invariant subset of $M$,
    \item the set of points in $M$ at which $H$ does not act properly discontinuously,
    \item the set of attractive points of convergence sequences in $H$. 
\end{enumerate}
If $H$ is finite, $\Lambda H$ is empty.
\end{definition}

\begin{proposition}\cite{Bowditch1998}[Prop $3.1, 3.2$]\thlabel{limit points are limit points}
    Let $G$ be a convergence group on a compact metrizable space $M$. Then conical limit points are in $\Lambda G$ and conical limit points are not parabolic points. 
\end{proposition}

\cite{Hruska_2010} lays out many definitions of relative quasiconvexity as well, including the first of the following definitions.

\begin{definition}\thlabel{RQC}
    Let $G$ act on a compact metrizable space $M$ as a geometrically finite convergence group, let $\mathbb{P}$ be the set of maximal parabolic subgroups, and let $H$ be a subgroup.
    \begin{enumerate}
        \item $H$ is \emph{relatively quasi-convex} in $G$, or \emph{RQC}, if the action of $H$ on $\Lambda H$ is geometrically finite. 
        \item $H$ is \emph{fully relatively quasi-convex} in $G$, or \emph{fully RQC}, if it is RQC and, for any infinite sequence of elements $(g_n)_n$ each in a distinct $H$ coset, we have $\bigcap_n g_n\Lambda H = \varnothing$, 
        \item $H$ is \emph{full relatively quasi-convex} in $G$, or \emph{full RQC}, if it is RQC and, for every maximal parabolic subgroup $P$ of $G$, $H\cap P$ is either finite or finite index in $P$. 
    \end{enumerate}
\end{definition}

The reader should be disturbed by these names. Fortunately, fully RQC and full RQC are equivalent conditions, as we will see in \thref{Finite Height}. The fully implies full direction is due to Dahmani \cite[Lemma 1.7]{Dahmani_2003}, and he also shows that a subgroup of a hyperbolic group is metrically quasiconvex as a subset of the Cayley graph if and only if it is fully RQC in the above sense. 

The following proposition shows that an RQC subgroup inherits a peripheral structure from the larger group which makes it relatively hyperbolic. 

\begin{proposition}\thlabel{RQC means induced structure is the same}\cite[Theorem 9.1]{Hruska_2010}
    Let $H$ be an RQC subgroup of a relatively hyperbolic group $(G,\PP)$ and let 
    \[\overline{\mathbb{O}} = \{H \cap P^g \: | \; P \in \mathbb{P}, \, g \in G, \, H \cap P^g \text{ infinite} \}\]
    Then $\overline{\mathbb{O}}$ consists of finitely many $H$ conjugacy classes. If $\mathbb{O}$ is a set of representatives of $H$ conjugacy classes from $\overline{\mathbb{O}}$, $(H,\mathbb{O})$ is relatively hyperbolic and $ \partial_{\mathbb{O}}H$ is $H$--equivariantly homeomorphic to $\Lambda H \subset \partial _{\mathbb{P}}G$.
\end{proposition}

\begin{remark}
    
    \begin{enumerate}
    \item If $G$ is relatively hyperbolic and $\Sigma$ is as in \thref{defn:Relatively Hyperbolic}, then $\Lambda H$ can be seen as the points of $\partial \Sigma$ which are limit points of some $H$--orbit in $\Sigma$. Changing which base point defines the orbit translates the orbit a bounded amount, leaving the limit set unchanged.
    
    \item If $H_1 < H_2 < G$ are groups with $G$ relatively hyperbolic, $H_2$ RQC in $G$, and $H_1$ RQC in $H_2$, then $H_1$ is RQC in $G$. The limit set of $H_1$ in $\partial G$ is the image of the inclusion $\partial H_1 \longrightarrow \partial H_2  \longrightarrow \partial G$. 

    \item Parabolic subgroups are RQC since their limit set is a single point. Further, maximal parabolic subgroups are full RQC; If $P,P'$ are maximal parabolic subgroups of a relatively hyperbolic group and $P \cap P'$ is infinite, we must have $P = P'$. This really shows maximal parabolic subgroups are almost malnormal, for if $P' = P^g$ and $P \cap P^g$ is infinite, then it must have a limit set, and it must be the unique fixed point of $P$. Hence $g$ fixes this point and $g \in P$ by the maximality of $P$.
    \end{enumerate}
\end{remark}

\begin{proposition}\thlabel{lipschitz} \cite[Lemma 3.1]{Agol_2009}
    Let $H$ be a RQC subgroup of a relatively hyperbolic group $(G,\mathbb{P})$ and let $CH$ and $CG$ be the corresponding cusped spaces. The inclusion $i:H \hookrightarrow G$ induces an $H$--equivariant Lipschitz map $\hat{i}:CH \rightarrow CG$ with quasiconvex image.
\end{proposition}
\begin{proof}[Proof Sketch]
    As above, the peripheral structure of $H$ can be expressed as $\mathbb{D} = \{D_i = H \cap P_i^{c_i} \; | \; P_i \in \mathbb{P}, c_1,\ldots, c_\ell \in G, H \cap P_i^{c_i} \text{ infinite}\}$, where some of the $P_i$ may be identical. To extend $i$ to $\hat{i}$, consider a point of a horoball in $H$ is given by $(hD_i,hd,n)$, where $h \in H, d \in D_i$, and $n \geq 1$. Define
    \[\hat{i}(hD_i,hd,n) = (hc_iP_i, hdc_i,n).\]
    Keeping track of the generating sets for $H,G$ used to build $CH, CG$ shows $\hat{i}$ is Lipschitz. 
    
    In \cite[Definition 3.11]{GrovesManning2006}, a subgroup of $(G,\mathbb{P})$ is defined as $C$\emph{--relatively quasiconvex} exactly when the image of $\hat{i}$ is $C$--quasiconvex as a subset of the cusped space for $G$. \cite[Theorem A.10]{Manning_2009} shows $H$ is $C$--relatively quasiconvex if and only if $H$ is RQC in the sense of \thref{RQC}. By assumption, $H$ is RQC in the sense above, so \cite{Manning_2009} implies this image is metrically quasiconvex as in the statement. 
\end{proof}

In \cite{Yang2010}, Yang studies intersections of RQC subgroups and gives us the following useful properties.

\begin{proposition} \cite[Thm $1.1$]{Yang2010}\thlabel{General Limit Set Property}
    Let $H_1,H_2$ be RQC subgroups of a relatively hyperbolic group $G$. Then 
    \[\Lambda H_1 \cap \Lambda H_2 = \Lambda (H_1\cap H_2) \sqcup E\]
    where the exceptional set $E$ consists of parabolic points $\xi \in\Lambda H_1 \cap \Lambda H_2$ so that $\mathrm{Stab}_{H_1}(\xi) \cap \mathrm{Stab}_{H_2}(\xi)$ is finite.
\end{proposition}

\begin{proposition}\cite[Prop $1.3$]{Yang2010}
    If $H_1, H_2$ are RQC subgroups of a relatively hyperbolic group, then $H_1 \cap H_2$ is RQC.
\end{proposition}

\begin{corollary}[Limit Set Property]\thlabel{Limit Set Property}
    If $H_1, H_2$ are full RQC subgroups of a relatively hyperbolic group $G$, then $H_1 \cap H_2$ is full RQC and
    \[\Lambda H_1 \cap \Lambda H_2 = \Lambda (H_1\cap H_2).\]
\end{corollary}
\begin{proof}
    From the previous two propositions we know $H_1 \cap H_2$ is RQC and that the intersection of limit sets is the limit set of the intersection together with some exceptional points. If $\xi \in \Lambda H_1 \cap \Lambda H_2$ is the fixed point of a maximal parabolic subgroup $P$, then we must have $P\cap H_i$ finite index in $P$ for each $i$ since the $H_i$ are full. Since the intersection of finite index subgroups is again finite index, we know $(P\cap H_1) \cap (P\cap H_2) = P\cap (H_1\cap H_2)$ is finite index in $P$, and therefore infinite. So there are no exceptional points and we have the equality above. Similarly, if $P\cap (H_1 \cap H_2)$ is infinite, then $P \cap H_i$ is infinite for each $i$, and because each $H_i$ is full, $P\cap(H_1 \cap H_2) = (P\cap H_1) \cap (P\cap H_2)$ is an intersection of finite index subgroups of $P$, hence is finite index in $P$. This shows $H_1 \cap H_2$ is full RQC.
\end{proof}

To show that full RQC and fully RQC are equivalent, we need the following definitions.

\begin{definition}
    Let $G$ be a group and $H$ a subgroup. The \emph{height} of $H$ in $G$ is the minimal number $n$ so that for any collection of distinct cosets $\{g_1 H, g_2H, \ldots g_{n+1}H\}$, the intersection $\bigcap_i H^{g_i}$ is finite. If $G$ is relatively hyperbolic, the \emph{relative height} of $H$ in $G$ is the minimal number $n$ so that for any collection of distinct cosets $\{g_1 H, g_2H, \ldots g_{n+1}H\}$, the intersection $\bigcap_i H^{g_i}$ is either finite or parabolic. 
\end{definition}

In \cite{HruskaWise}, Hruska and Wise use the term ``height" for what we are calling ``relative height" and prove the following. 

\begin{proposition} \cite[Corollary $8.6$]{HruskaWise}
    Let $H$ be a RQC subgroup of a relatively hyperbolic group $G$. Then $H$ has finite relative height in $G$.
\end{proposition}

If $H$ is full RQC, we can upgrade this to truly finite height.

\begin{proposition}
    If $H$ is a full RQC subgroup of a relatively hyperbolic group $(G,\mathbb{P})$, then $H$ has finite height. 
\end{proposition}
\begin{proof}

As in \thref{RQC means induced structure is the same}, $G$ induces a peripheral structure $\mathbb{O} = \{P_1,\ldots P_k\}$ on $H$. Explicitly, $P_i = H \cap P'_i$, where $P'_i<G$ is a maximal parabolic subgroup with $H\cap P_i'$ infinite, and $\mathbb{O}$ contains one representative of each $H$--conjugacy class of such subgroups. Let 
\[m = \max_i [P_i':P_i] =\max \{[P : H \cap P] \; | \; H \cap P \text{ infinite, } P \text{ maximal parabolic in } G \}.\]

From the previous proposition, $H$ has finite relative height, say $n$. We claim $H$ has height at most $N = n(m+1)(k+1)$. 

For a contradiction, suppose $\{g_1H,\ldots , g_{N+1}H\}$ are distinct $H$--cosets in $G$ with $\bigcap_i H^{g_i}$ infinite. After conjugating by $g_1^{-1}$, we may assume $g_1 = 1$ so that this intersection is contained in $H$. Because $N+1$ is larger than the relative height of $H$, we must have $\bigcap_i H^{g_i} < H\cap P$ for some maximal parabolic subgroup $P < G$. For each $i$, we have $\bigcap_i H^{g_i} \subset H^{g_i} \cap P$, hence $H \cap P^{g_i^{-1}}$ is infinite, hence $H \cap P^{g_i^{-1}}$ is a maximal parabolic subgroup of $H$. Apply the pigeonhole principle to $W := \{H \cap P^{g_i^{-1}} \; | \; i =1,\ldots N+1\}$ where the pigeons are elements of $W$ and the holes are $H$--conjugacy classes of maximal parabolic subgroups. Because $N+1 > (m+1)(k+1)$ and there are $k$ different $H$--conjugacy classes of maximal parabolic subgroups, some $H$--conjugacy class has at least $m+1$ representatives in $W$. After reindexing, $H\cap P^{g_1^{-1}}, \ldots, P^{g_{m+1}^{-1}}$ are conjugate in $H$, so there are $h_i \in H$ so that $H\cap P^{g_1^{-1}} = H \cap P^{h_ig_i^{-1}}$. Then 
\[ H \cap P^{g_1^{-1}} = H \cap P^{h_i g_i^{-1}} \subset P^{g_1^{-1}}\cap P^{h_ig_i^{-1}},\]
and the sets on the left are infinite. Because $P$ is almost malnormal, this implies $P^{g_1^{-1}} = P^{h_ig_1^{-1}}$, equivalently $P = P^{g_1h_ig_i^{-1}}$, and $g_1h_ig_i^{-1} \in P$ for $i = 1,\ldots , m+1$. By the definition of $m$, $D:= H^{g_1}\cap P$ has index at most $m$ in $P$, so we may write $P$ as a disjoint union of cosets: 
\[P = D{r_1}\sqcup D{r_2} \sqcup \cdots \sqcup D r_m.\]
Viewing these cosets as pigeonholes and the elements $g_1h_ig_i^{-1}$ as pigeons, there is some coset which contains two of these elements. This means there is some $1 \leq i,j, \leq m +1$ and $r \in P, p_i,p_j \in D$ so that $p_i r= g_1h_ig_i^{-1}$ and $p_j r = g_1h_jg_j^{-1}$. Rearranging this, we have $p_i^{-1}g_1h_ig_i^{-1} = r = p_j^{-1}g_1h_jg_j^{-1}$. Since $p_i,p_j \in D:= H^{g_1}\cap P$, we can write $p_i = g_1 q_ig_1^{-1}, p_j = g_1 q_jg_1^{-1}$ for some $q_i,q_j \in H$. Substituting this into the previous equality, we have $g_1q_i^{-1}h_ig_i^{-1} =g_1q_j^{-1}h_jg_j^{-1}$. Canceling the $g_1$ on each side and rearranging, we have
\[g_i = g_j h_j^{-1}q_jq_ih_i.\]
But $h_j^{-1}q_jq_ih_i$ is a product of elements of $H$, so this implies $g_iH = g_jH$, contradicting the assumption that $g_i$ were distinct coset representatives. Thus $N$ is a bound on the height of $H$.

\end{proof}

We remark that there is a simpler proof that \emph{fully} RQC subgroups have finite height, but finite height is necessary for proving the full implies fully direction in the following corollary.

\begin{corollary}[Finite Height]\thlabel{Finite Height}
    Let $H$ be a subgroup of a relatively hyperbolic group $G$. Then $H$ is fully RQC if and only if $H$ is full RQC, and either condition implies $H$ has finite height in $G$.     
\end{corollary}
\begin{proof}
    The fully implies full direction is \cite[Lemma 1.7]{Dahmani_2003}, and we repeat the proof here for completeness. Suppose $H$ is fully RQC in $G$ and $P$ is a maximal parabolic subgroup fixing $\xi \in \partial G$. Suppose $H\cap P$ is infinite so that $\xi \in \Lambda( H \cap P)$. For a contradiction, suppose $(p_n)_n$ is an infinite sequence of distinct coset representatives for $H \cap P$. Each $p_n$ represents a distinct $H$ coset too, since $p_n H = p_mH$, implies $p_np_m^{-1} \in H \cap P$, hence $p_n(H\cap P) = p_m(H\cap P)$ and $n = m$. Each $p_n$ also fixes $\xi$, so
    
    \[\xi \in \bigcap_n p_n\Lambda(H\cap P)\subset \bigcap_n p_n\Lambda H.\]
    But the intersection on the right is empty because $H$ is fully RQC. This contradiction shows $H$ is also full RQC.
    
    For the other direction, suppose $H$ is full RQC. Then $H$ has finite height in $G$ by the previous proposition. If $(g_n)_n$ is an infinite sequence of elements in distinct $H$ cosets, then $\bigcap_n H^{g_n}$ is finite because $H$ has finite height. In fact, if $H$ has height $N$, then $\bigcap_{n=1}^{N+1}H^{g_n}$ is finite. By the Limit Set Property \ref{Limit Set Property} and because the limit set of a finite group is empty, we have
    \[\bigcap^\infty_{n=1} g_n \Lambda H \subset  \bigcap_{n=1}^{N+1} g_n\Lambda H = \bigcap_{n=1}^{N+1} \Lambda H^{g_n} = \Lambda \Big(\bigcap_{n=1}^{N+1} H^{g_n}\Big) = \varnothing.\]
\end{proof}

\begin{proposition}(Convergence Property)\thlabel{Convergence Property}
    Let $(G,\mathbb{P})$ be a relatively hyperbolic group, $H$ a RQC subgroup, and $(g_nH)_n$ a sequence of distinct cosets. Let $X$ be the cusped space for $(G,\mathbb{P})$. Let $Y$ be the image of the Lipschitz map provided by Proposition \ref{lipschitz}. After taking a subsequence, there is some $\xi \in \partial G$ so that $g_n(Y \sqcup \Lambda H) \longrightarrow \xi$ uniformly. 
\end{proposition}
\begin{proof}
    Recall $Y$ consists of $H$ together with some points of positive depth and is quasiconvex for some constant $K$. Choose $x_n \in g_nY$ so that $d(1,x_n) = d(1,g_nY)$. The cosets $g_nH$ are all distinct, hence disjoint, and $X$ is proper, so $d(1,x_n) \longrightarrow \infty$. Applying Arzela--Ascoli, we can choose a subsequence so that the geodesics $[1,x_n]$ converge to some ray $r$ from $1$ to a point $\xi \in \partial X = \partial G$. The set $Y \sqcup \Lambda H$ is the closure of $Y$ in the compact metrizable space $X \sqcup \partial X$, so to show the claim it suffices to show $g_nY \longrightarrow \xi$ uniformly. To show this, it suffices to show that for any $M$, there exists $N$ so that if $n  \geq N$ and $y \in Y$, the geodesic ray $[1,g_ny]$ stays within $\delta$ of $r$ for time $M$. 

    We can assume $d(1,x_n) \geq K+\delta$ for all $n$. We claim that for all $n$ and $y\in g_nY$, $(1,y)_{x_n} \leq K+\delta$. Given $y \in g_nY$, consider the geodesic triangle made by $1,x_n$ and $y$. Let $z \in [1,x_n]$ and $z' \in [x_n,y]$ be the corresponding tripod points. Then $d(z,z') \leq \delta$ because $X$ is $\delta$--hyperbolic, and there is some $z'' \in g_nY$ so that $d(z',z'') \leq K$ because $g_nY$ is $K$--quasiconvex. Because $x_n$ achieves the distance $d(1,g_nY)$, we must have $d(z,g_nY) = d(z,x_n) = (1,y)_{x_n}$, but using $z',z''$ and the triangle inequality we have $d(z,g_nY) \leq K + \delta$. Hence $(1,y)_{x_n} \leq K+\delta$ as claimed. 
    
    Fix $M$ and choose $N$ so that $n \geq N$ implies $[1,x_n]$ stays within $\delta$ of $r$ for time at least $M+2\delta+K$, which is possible because the $x_n$ converge to $\xi$. For $n \geq N$ and $y' \in g_nY$, let $x \in [1,x_n], y \in [1,y'],z \in r$ be the points at distance $M+1\delta$ from $1$. Considering the comparison tripod for the triangle on $1,y',x_n$ and noticing that $d(1,x_n) \geq M +\delta + K$ and $(1,y')_{x_n} \leq K +\delta$, we see that $x,y$ are both closer to $1$ than the tripod points on either leg, hence $d(x,y) \leq \delta$. Further, by the assumption on $n$, we have $d(x,z) \leq \delta$. Therefore $d(y,z) \leq 2\delta$ and 
    \[(y,z)_1 = M+2\delta -\tfrac{1}{2} d(y,z) \geq M.\]

    Since $y \in [1,y'], z \in r$, this implies $[1,y']$ and $r$ stay within $\delta$ of eachother for time at least $M$, and we are finished.
\end{proof}

Dahmani proves an analogous result to the Convergence Property \ref{Convergence Property} in \cite[Prop $1.8$]{Dahmani_2003} using only the Bowditch boundary, but we need the version above to manage finite groups which do not interact with the boundary. 

\begin{lemma}\thlabel{NoSubsequence3}
    Suppose $H_1,H_2$ are full RQC subgroups of a relatively hyperbolic group $(G,\mathbb{P})$ and $(a_n)_n$ a sequence in $G$. Let $X$ be the cusped space for $(G,\mathbb{P})$ and let $Q_1,\,Q_2$ be the images of cusped spaces for $H_1,\,H_2$ in $X$ from \thref{lipschitz}. Then there is a sequence $(k_n)_n$ in $H_1$ and a subsequence of $(a_n)_n$ (still denoted $(a_n)_n$) so that the translates $k_na_n Q_2$ are either constant or converge to a point in $\partial X \setminus \Lambda Q_1 = \partial G \setminus \Lambda H_1$. 
\end{lemma}

\begin{proof}

    By \thref{lipschitz}, $Q_1,Q_2$ are $K$--quasiconvex. Let $\pi:X \longrightarrow Q_1$ be the projection of $X$ onto $Q_1$. Slightly abusing language, we will call an image of a horoball from the cusped space for $H_i$ a horoball of $Q_i$.  

    Suppose the peripheral structure on $H_1$ is given by $\mathbb{D} = \{H_1 \cap P_i^{c_i} \; | \; P_i \in \mathbb{P}, c_1,\ldots, c_\ell \in G, H_1 \cap P_i^{c_i} \text{ infinite}\}$, where some of the $P_i$ may be identical. A point of $Q_1$ with positive depth is in a horoball, so it has the form
    \[(hc_iP_i,hdc_i,n)\]
    where $h \in H_1$ and $d \in H_1\cap c_iP_ic_i^{-1}$. We can translate such a point by $(hd)^{-1} \in H_1$ to the point $(c_iP_i,c_i,n)$. Points of depth $0$ are simply elements of $H_1$, which can be translated to the identity by an element of $H_1$. Thus $Q_1$ is the $H_1$ orbit of the identity together with finitely many $H_1$ orbits of vertical paths starting at the $(c_iP_i,c_i,1)$. 

    Because $H_1$ is full RQC, $H_1 \cap P_i^{c_i}$ is finite index in $P_i^{c_i}$ for every $i$, so there is a constant $M$ so that if $A'$ is a horoball of $Q_1$ contained in a horoball $A$ of $X$, then $A$ is contained in the $M$ neighborhood of $A'$. 
    
    We claim that if $x \in X$ has depth $0$, then $\pi(x)$ has depth at most $M$. If $\pi(x)$ has positive depth, let $y$ be the first point of $[x,\pi(x)]$ in the horoball of $X$ containing $\pi(x)$. Then $y$ has depth $0$ because it is the first point in this horoball, but also $d(y,Q_1) \leq M$, so $[y,\pi(x)]$ can travel at most $M$ vertically. Thus $\pi(x)$ has depth at most $M$. 

    For each $n$, choose a point $x_n \in a_nQ_2$ achieving the distance $d(a_nQ_2,Q_1)$. If $a_nQ_2 \cap Q_1 \neq \varnothing$, we can assume $x_n$ has depth $0$. If $a_nQ_2 \cap Q_1 = \varnothing$, then $\pi(x_n)$ has depth at most $M$ by the previous paragraph. We choose $k_n \in H_1$ so that $k_n\pi(x_n)$ is either the identity or some $(c_iP_i,c_i,n)$, where $n \leq M$. The definition of $\pi$ is only coarse so $\pi$ may not be equivariant, but this implies $\pi(k_nx_n)$ is some uniformly bounded distance from $k_n\pi(x_n)$. Since $X$ is locally finite, this implies there are finitely many choices for $\pi(k_nx_n)$, and after a subsequence we can assume $\pi(k_nx_n)$ is constant, say $z \in Q_1$.   

    After a further subsequence, we can assume that $k_nx_n$ is either constant or the geodesics $[z,k_nx_n]$ converge to a geodesic $[z,\eta)$ for some $\eta \in \partial X$. If $k_nx_n$ is constant, then $k_n x_n \in k_na_nQ_2$, and the translates of $Q_2$ are disjoint, so $k_na_nQ_2$ is constant and we are finished. In the second case, the distances $d(k_nx_n,z) = d(k_nx_n,Q_1)$ must be unbounded and the geodesics $[z,k_nx_n]$ travel further and further from $Q_1$, hence $\eta \notin \Lambda Q_1$. Applying the Convergence Property \ref{Convergence Property}, we can assume the $k_na_nQ_2$ converge to some point in $\partial X$, and since $k_nx_n \longrightarrow \eta$, this point must be $\eta \notin \Lambda Q_1$, and we're finished.
\end{proof}

\begin{proposition}\thlabel{kg1 remark}
    Let $(G,\mathbb{P})$ be a relatively hyperbolic group with cusped space $C$. There exists a cell complex $K(G)$ satisfying the following.
    \begin{enumerate}
        \item $K(G)$ is a $K(G,1)$. 
        \item The universal cover $\fitTilde{K(G)}$ has $1$--skeleton equal to $C$, possibly with duplicated edges.
        \item If $H$ is a full RQC subgroup of $G$ and $\psi: C_H \longrightarrow C$ is a Lipschitz map as in \thref{lipschitz}, then there is a map $\varphi:K(H) \longrightarrow K(G)$ realizing the inclusion $H \longrightarrow G$ on fundamental groups and which lifts to a map $\fitTilde{\varphi}:K(H) \longrightarrow K(G)$ restricting to $\psi$ on $1$--skeletons.
    \end{enumerate}

    $G$ admits a $K(G,1)$ whose universal cover has $1$--skeleton equal to $C$, with duplicated edges if $G$ has $2$--torsion. If $H$ is a full RQC subgroup of $G$ and $C_H$
\end{proposition}
\begin{proof}
    Let $C$ be a cusped space for $(G,\mathbb{P})$. \cite{GrovesManning2006} describes how to add $2$--cells to $C$ using a relative presentation for $G$ to get a simply connected space. However, if $G$ is not torsion free, the action on this space may not be free; An element of order $2$ may invert an edge, and if an element fixes the boundary of an $2$--cell, it will fix the center of the cell. Therefore if $e$ is an edge of $C$ which is inverted by some element of $G$, we replace $e$ with two edges so that an element which inverts $e$ swaps these two edges instead. For each $2$--cell added to make $C$ simply connected with a boundary of length $n$, we add $n$ $2$--cells along the same boundary so that an element which fixes this boundary permutes these added $2$--cells instead. After these modifications, we have a simply connected $2$--complex on which $G$ acts freely, say $C'$. Then $\pi_1(C'/G) = G$, and it is a standard construction to add higher dimensional cells to $C'/G$ so that it is aspherical, and these added cells lift to a higher dimensional cell structure on $C'$. This constructs $K(G)$.

    With $H, C_H, \psi$ from (3), $\psi$ extends to a map $\fitTilde{K(H)} \longrightarrow \fitTilde{K(G)}$ which projects to a map realizing the inclusion on fundamental groups. 
\end{proof}

\subsection{$M_\kappa$--Complexes}
\begin{definition}\cite[I.2]{BH}
    For $\kappa < 0$, the model space $M_\kappa^n$ is simply $\mathbb{H}^n$ with the metric scaled so that the curvature is $\kappa$. For $\kappa = 0$, $M_\kappa^n = \mathbb{R}^n$, and for $\kappa > 0$, $M_\kappa = S^n$ with the metric scaled so that the curvature is $\kappa$. 
\end{definition}

\begin{definition}[$M_\kappa$--complex]\cite[I.7]{BH}\thlabel{defn of M_kappa}
    Let $\kappa \leq 0$. A simplicial complex $X$ is called an $M_\kappa$--\emph{complex} if it satisfies the following conditions.
    \begin{itemize}
        \item Each simplex of $X$ is modeled after a geodesic simplex in $M^n_\kappa$, which is the convex hull of finitely many points in general position.
        \item If $\sigma,\sigma'$ are two simplices meeting in a face $\tau$, then the identity map from $\tau \subset \sigma$ to $\tau \subset \sigma'$ is an isometry. 
    \end{itemize}
\end{definition}

The path metric, also called the simplicial metric, on an $M_\kappa$--complex is natural and pleasant by the following theorem.

\begin{theorem}\cite[Theorem 1.1]{Bridson1991}
    If $X$ is an $M_\kappa$--complex with $\kappa \leq 0$ and finitely many isometry types of simplices, then the simplicial path metric is complete and geodesic.
\end{theorem}

$M_\kappa$--complexes are very general and Bridson's thesis \cite{Bridson1991} developed much of the machinery we will use. For the remainder of this section, we fix a $\mathrm{CAT}(0)$ $M_\kappa$--complex $X$ with finitely many isometry types of simplices.

\begin{definition}
    Given two subsets $K,K'$ of $X$, let $\Geod(K,K')$ denote the set of points lying on a geodesic segment from a point of $K$ to a point of $K'$.
\end{definition}

\begin{definition}[Simplicial Neighborhood]
        Let $K$ be a subcomplex of $X$. 
    \begin{enumerate}
        \item The \emph{open simplicial neighborhood of} $K$, denoted $N(K)$, is the union of open simplices of $X$ whose closure meets $K$. If $K=\sigma$ is a single simplex, we write $st(\sigma) = N(\sigma)$.
        \item The \emph{closed simplicial neighborhood of} $K$, denoted $\overline{N}(K)$, is subcomplex spanned by closed simplices that meet $K$. Equivalently, it is the closure of $N(K)$. 
        \item The \emph{simplicial link of} $K$, denoted $Lk(K)$, is $N(K) \setminus K$. 
    \end{enumerate}
\end{definition}

For example, if $X$ is a graph and $K=v$ is a vertex, then $N(K)$ is $v$ together with the interior of each edge with $v$ as an endpoint, $Lk(v)$ is the interior of each of these edges, and $\overline{N}(K)$ is $N(K)$ together with the other vertices at the end of those edges.

\begin{definition}[Path of Simplices]
A \emph{path of simplices} is a sequence of open simplices $\sigma_1,\ldots ,\sigma_n$ so that either $\overline{\sigma_i} \subset \overline{\sigma_{i+1}}$ or $\overline{\sigma_{i+1}} \subset \overline{\sigma_i}$ for each $i = 1,\ldots , n-1$. Equivalently, it is a finite path in the 1--skeleton of the first barycentric subdivision of $X$. The integer $n$ is the \emph{length} of the path of simplices.    
\end{definition}

\begin{lemma}\thlabel{geod meets finitely many}\cite[Lemma 3.5]{Martin} 
For finite subcomplexes $K,K' \subset X$, $\Geod(K,K')$ meets finitely many open simplices. 
\end{lemma}

\begin{lemma}\cite[Theorem 1.11]{Bridson1991}\thlabel{Bounded Length To Bounded Number Of Simplices}
    For every $n$, there exists a constant $k$ so that every geodesic segment of length at most $n$ meets at most $k$ open simplices. 
\end{lemma}

\begin{lemma}\cite[Theorem 1.11]{Bridson1991}[Containment] \thlabel{Containment} 
For every $n$ there exists a constant $k$ such that for every finite subcomplex $K$ of $X$ containing at most $n$ simplices, any geodesic path in the open simplicial neighborhood of $K$ meets at most $k$ simplices. 
\end{lemma}

See Figure \ref{fig:ShortPaths} for an illustration of the set up for the following lemma.

\begin{lemma}\cite[Lemma 3.7]{Martin}[Short Paths of Simplices]\thlabel{Short Paths of Simplices}
    There is a function $F:\mathbb{N} \longrightarrow \mathbb{N}$ so that the following holds: Let $K$ be a convex subcomplex of $X$ and $K'$ a connected subcomplex of $X$ both containing at most $n$ simplices. Let $x,y \in K$ and $x',y' \in K'$, and assume there exists a path between $x'$ and $y'$ in $K'$ that does not meet $K$. Let $\tau,\tau'$ be two simplices of $Lk(K)$ so that $[x,x']$ (resp. $[y,y'])$ meets $\tau$ (resp. $\tau'$). Then there exists a path of simplices in $Lk(K)$ of length at most $F(n)$ between $\tau$ and $\tau'$. 
\end{lemma}

\begin{figure}
        \centering
        \begin{tikzpicture}
        \draw[very thick] (2,5) -- (9,5);
        \node at (4,5.5) {$x'$};
        \node at (7,5.5) {$y'$};
        \draw[very thick] (0,0) -- (8,0);
        \node at (1,-0.5) {$x$};
        \node at (6,-0.5) {$y$};

        \node at (-.5,0) {$K$};
        \node at (1.5,5) {$K'$};
        \node at (3.5,0.5) {$\cdots$};
        \node at (1,1.25) {$\tau$};
        \node at (7.25,.75) {$\tau'$};

        \filldraw[fill=lightgray, draw=black] (.5,0) -- (.25,1) -- (1.75,0.75) -- cycle;
        \filldraw[fill=lightgray, draw=black] (.5,0) -- (1.75,0.75) -- (2.5,0) -- cycle;

        \filldraw[fill=lightgray, draw=black] (7,0) -- (5.75,1.25) -- (5.5,0) -- cycle;
        \filldraw[fill=lightgray, draw=black] (5.75,1.25) -- (5.5,0) -- (5,.75) -- cycle;

        \draw (1,0) -- (4,5);
        \draw (6,0) -- (7,5);
        \end{tikzpicture}
        
        \caption{The situation of \thref{Short Paths of Simplices}.}
        \label{fig:ShortPaths}
\end{figure}

\begin{definition}[Acylindrical]
    Let $G$ be a group acting on an $M_\kappa$--complex $X$. The action is called \emph{acylindrical} if there is a constant $A$ so that for any set $K \subset X$ with $\diam(K) \geq A$, $K$ has finite pointwise stabilizer. 
\end{definition}

\subsection{Complexes of Groups}

Given a group acting cocompactly on a tree, Bass--Serre theory explains how to use vertex and edge stabilizers with the quotient graph to get a graph of groups structure for the original group. Conversely, given a graph of groups, we can build a tree on which the fundamental group acts with quotient our original graph of groups. In \cite{BH}, Bridson and Haefliger extend this correspondence to groups acting on any simply connected category. Here we introduce the definitions and notation  we need from their theory. 

\begin{definition}[Scwol]
    Let $\mathcal{Y}$ be a category with objects $V(\mathcal{Y})$ and morphisms (or arrows) $E(\mathcal{Y})$. For $a \in E(\mathcal{Y})$ we write $i(a), t(a)$ for the source and target of $a$. Then $\mathcal{Y}$ is a \emph{small category without loops}, or a \emph{scwol}, if $V(\mathcal{Y}), E(\mathcal{Y})$ are both sets and for any $a \in E(\mathcal{Y})$, we have 
    \[i(a) = t(a) \Longrightarrow a \text{ is the identity morphism of } i(a) = t(a)\]
    Further, $\mathcal{Y}$ is \emph{simple} if there is at most one morphism between any two objects of $\mathcal{Y}$.
\end{definition}

\begin{definition}
    Let $Y$ be a cell complex. The \emph{scwolification} of $Y$, denoted $\YY$, is the scwol with objects corresponding to cell of $Y$ and arrows corresponding to reverse inclusion.
\end{definition}

\begin{definition}
    Let $\YY$ be a scwol. The \emph{geometric realization of} $\YY$, denoted $|\YY|$ is a the flag simplicial complex with vertices $V(\YY)$ and edges $E(\YY)$. An $n$--simplex of $|\YY|$ corresponds to an $n$-tuple of composable edges in $\YY$. 
\end{definition}

The only scwols we will be interested in are the scwolification of cell complexes. If $Y$ is a cell complex and $\YY$ is its scwolification, then $|\YY|$ is the barycentric subdivision of $Y$. See \ref{fig:scwolifcation} for an example where $Y$ is a single triangle. 

Note that if a group $G$ acts on \emph{simplicial} complex $X$ and fixes some simplex setwise but not pointwise, then the quotient will not be simplicial, for example $\mathbb{Z}/3\mathbb{Z}$ acting on a triangle by rotation. To remedy this, one can consider the action on the barycentric subdivision of $X$, say $X_b$. If $\sigma$ is an $n$--cell of $X_b$, then $\sigma$ corresponds to a chain $\sigma_1 \subset \sigma_2 \subset \cdots \subset \sigma_n$ where each $\sigma_i$ is a cell of $X$. If an element of $G$ fixes $\sigma$, it must fix each $\sigma_i$, so $\mathrm{Stab}_G(\sigma)$ fixes $\sigma$ pointwise. 

\begin{convention}\thlabel{stabilizer convention}
    Whenever is $G$ is a group acting on a simplicial complex $X$, we assume the quotient is simplicial and stabilizers of simplices fix simplices pointwise. As above, this can always be achieved by barycentrically subdividing if necessary.  
\end{convention}

\begin{figure}
        \centering
        \begin{tikzpicture}

        \filldraw [fill=lightgray,thick] (-7,0) -- (-5,1.732*2) -- (-3,0) -- (-7,0);

        \node at (-2,0)[circle,fill,inner sep=1pt]{};
        \node at (2,0)[circle,fill,inner sep=1pt]{};
        \node at (0,1.732 * 2)[circle,fill,inner sep=1pt]{};
        \node at (-1,1.732)[circle,fill,inner sep = 1pt]{};
        \node at (1,1.732)[circle,fill,inner sep = 1pt]{};
        \node at (0,0)[circle,fill,inner sep = 1pt]{};
        \node at (0,2*1.732*.3333)[circle,fill,inner sep = 1pt]{};

        \draw[shorten >=7pt,shorten <=4pt, ->](0,2*1.732*.3333) -- (-2,0);
        \draw[shorten >=7pt,shorten <=4pt, ->](0,2*1.732*.3333) -- (0,2*1.732);
        \draw[shorten >=7pt,shorten <=4pt, ->](0,2*1.732*.3333) -- (2,0);
        \draw[shorten >=2.5pt,shorten <=4pt, ->](0,2*1.732*.3333) -- (-1,1.732);
        \draw[shorten >=2.5pt,shorten <=4pt, ->](0,2*1.732*.3333) -- (1,1.732);
        \draw[shorten >=2.5pt,shorten <=4pt, ->](0,2*1.732*.3333) -- (0,0);

        \draw[shorten >=7pt,shorten <=4pt, ->](-1,1.732) -- (-2,0);
        \draw[shorten >=7pt,shorten <=4pt, ->](-1,1.732) -- (0,2*1.732);

        \draw[shorten >=7pt,shorten <=4pt, ->](1,1.732) -- (2,0);
        \draw[shorten >=7pt,shorten <=4pt, ->](1,1.732) -- (0,2*1.732);

        \draw[shorten >=7pt,shorten <=4pt, ->](0,0) -- (-2,0);
        \draw[shorten >=7pt,shorten <=4pt, ->](0,0) -- (2,0);
        
        \filldraw [fill=lightgray,thick] (3,0) -- (5,1.732*2) -- (7,0) -- (3,0);
        \draw[thick] (4,1.732)-- (5,2*1.732*.3333);
        \draw[thick] (6,1.732)-- (5,2*1.732*.3333);
        \draw[thick] (5,0)-- (5,2*1.732*.3333);
        \draw[thick] (3,0)-- (5,2*1.732*.3333);
        \draw[thick] (7,0)-- (5,2*1.732*.3333);
        \draw[thick] (5,2*1.732)-- (5,2*1.732*.3333);
        \end{tikzpicture}
        
        \caption{From left to right, a simplex, the scwol constructed by reverse inclusion, and the geometric realization of this scwol.}
        \label{fig:scwolifcation}
\end{figure}
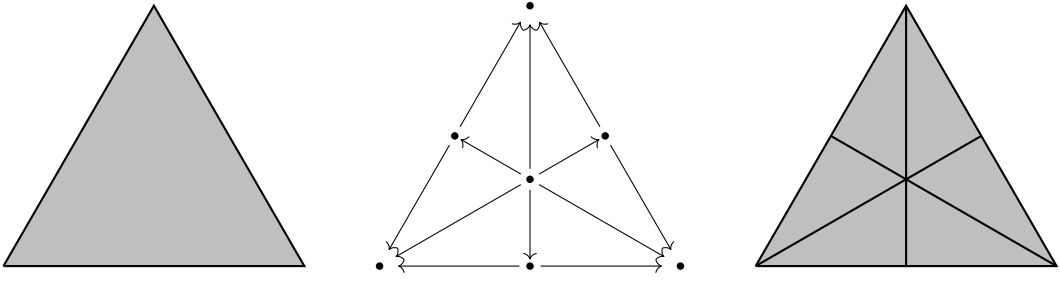

\begin{definition}[\cite{BH}](Complex of Groups)
    A \emph{complex of groups} over a scwol $\mathcal{Y}$, denoted $G(\mathcal{Y})= (G_\sigma, \psi,g_{a,b})$, consists of the following data. 
    \begin{enumerate}
        \item For each $\sigma \in V(\mathcal{Y})$, a \emph{local group} $G_\sigma$,
        \item For each $a \in E(\mathcal{Y})$, an injective homomorphism $\psi_a:G_{i(a)}\longrightarrow G_{t(a)}$, 
        \item For each pair of composable arrows $a,b$, a \emph{twisting element} $g_{a,b} \in G_{t(a)}$ satisfying
        \begin{enumerate}
            \item $Ad(g_{a,b})\psi_{ab} = \psi_a\psi_b$, where $Ad(g_{a,b})$ is conjugation by $g_{a,b}$,
            \item for any composable edges $a,b,c$, we have $\psi_a(g_{b,c})g_{a,bc} = g_{a,b}g_{ab,c}$. 
        \end{enumerate}
    \end{enumerate}
\end{definition}

\begin{definition}\cite[III.$\mathcal{C}$.3.7]{BH}[Fundamental Group of a Complex of Groups]\thlabel{Universal Cover of Complex of Groups}
Let $G(\mathcal{Y}) = (G_\sigma,\varphi_a,g_{a,b})$ be a complex of groups with $T$ a maximal tree in $\mathcal{Y}$, and let $E^\pm(\mathcal{Y})$ be the symbols $\{a^+, \, a^- \, | \, a \in E(\mathcal{Y})\}$. The \emph{fundamental group} $\pi_1(G(\mathcal{Y},T)$ has a presentation with the following generators 
\[\Big(\bigsqcup_{\sigma \in V(\mathcal{Y})} G_\sigma\Big) \bigsqcup E^\pm(\mathcal{Y})\]
and the following relations 
\[\begin{Bmatrix}\text{the relations in the groups } G_\sigma,\\
(a^+)^{-1} = a^- \text{ and } (a^-)^{-1} = a^+, \\
a^+b^+ = g_{a,b}(ab)^+ \text{ for all composable edges } a,b,\\
\varphi_a(g) = a^+ g a^- \text{ for all } a \in E(\mathcal{Y)}, g \in G_{i(a)}\\
a^+ = 1 \text{ for all } a \in T
\end{Bmatrix}\]
\end{definition}
\begin{definition}[Universal Cover of a Complex of Groups]
    Given a complex of groups $G(\mathcal{Y})$ and a maximal tree $T$ in $\YY$, let $G = \pi_1(G(\YY),T)$. The \emph{universal cover} is a scwol $\mathcal{X}$ defined by
    \[V(\mathcal{X}) = \Big\{ (gG_\sigma,\sigma) \; \Big| \; \sigma \in V(\mathcal{Y}), \, gG_\sigma \in G/G_\sigma \Big\}\]
\[E(\mathcal{X}) = \Big\{(gG_{i(a)},a) \; \Big| \; a \in E(\mathcal{Y}), \, gG_{i(a)} \in G/G_{i(a)} \Big\}\]
\[i(gG_{i(a)},a) = (gG_{i(a)},i(a)) \quad \quad \quad t(gG_{i(a)},a) = (ga^-G_{t(a)},t(a))\]
    The group $G$ acts on $\XX$ by left multiplication with quotient $\YY$.
\end{definition}

For the remainder of this section we fix a group $G$ acting on a simply connected simplicial complex $X$ with quotient $Y$, following \thref{stabilizer convention}. Let $\XX,\YY$ be their scwolifications and let $X_b = |\XX|, Y_b = |\YY|$ be their geometric realizations.

We briefly recall how to induce a complex of groups structure on $G$. The interested reader can refer to \cite[III.$\mathcal{C}$]{BH} for more details. Begin by choosing for each object $\tau \in V(\YY)$ a lift $\fitTilde{\tau} \in V(\XX)$. For each $a \in E(\YY)$, this induces a unique choice of lift $\fitTilde{a} \in E(\XX)$ so that $i(\fitTilde{a}) = \fitTilde{i(a)}$, and then we choose a (not unique) $h_a \in G$ so that $h_at(\fitTilde{a}) = \fitTilde{t(a)}$. For $\tau \in V(\YY)$, we set $G_\tau = \mathrm{Stab}_G(\fitTilde{\tau})$ and for each $a \in E(\YY)$, the homomorphism $\psi_a:G_{i(a)} \longrightarrow G_{t(a)}$ is simply conjugation by $h_a$ inside $G$. For composable edges $a,b$, we set $g_{a,b} := h_ah_bh_{ab}^{-1}$. If $X$ is a tree, there are no twisting elements and this is simply the graph of groups decomposition of $G$ as in Bass--Serre theory, only with more categorical language. 

\begin{definition}\thlabel{Complex of spaces compatible with G action}
    A \emph{complex of spaces compatible with the $G$ action on $X$} consists of the following.
    \begin{enumerate}
        \item For each simplex $\sigma \subset X$, a space $X_\sigma$. For each face of $\sigma' \subset \sigma$, a map $\varphi_{\sigma',\sigma}:X_\sigma \longrightarrow X_{\sigma'}$ such that if $\sigma'' \subset \sigma' \subset \sigma$, we have $\varphi_{\sigma'',\sigma'}\varphi_{\sigma',\sigma} = \varphi_{\sigma'',\sigma}$. 
        \item For each $g \in G$ and simplex $\sigma \subset X$, a homeomorphism $g:X_\sigma \longrightarrow X_{g\sigma}$, so that whenever $\sigma' \subset \sigma$ or $h \in G$, the following diagrams commute.
        \begin{center}

        \begin{tikzcd}
X_{\sigma} \arrow[r, "g"] \arrow[d, "{\varphi_{\sigma',\sigma}}"'] & X_{g\sigma} \arrow[d, "{\varphi_{g\sigma',g\sigma}}"] &  & X_\sigma \arrow[rd, "h"'] \arrow[rr, "gh"] &                              & X_{gh\sigma} \\
X_{\sigma'} \arrow[r, "g"]                                         & X_{g\sigma'}                                          &  &                                            & X_{h\sigma} \arrow[ru, "g"'] &             
\end{tikzcd}
        \end{center}
    \end{enumerate}
\end{definition}

The following two definitions are the language needed to state Haefliger's \thref{Haefliger}. We use this theorem to prove \thref{Cor of Haefliger}, which we need to combine cusped spaces in our main theorem.

\begin{definition}\cite[Section 1]{Haefliger1992}
    For $\tau \in V(\YY)$, let $CD_\tau$ be the scwol with $V(CD_\tau) = \{a \in E(\YY) \;|\; t(a) = \tau\}$ and $E(CD_\tau) = \{(a,b) \;| \;a,b \text{ composable edges of } \YY \text{ with } t(b) = i(a), t(a) = \tau\}$. For $(a,b) \in E(CD_\tau),$ we set $i(a,b) = ab$, $t(a,b)= a$. For composition, we set $(a,b)(a',b') = (a,bb')$. There is a functor $j_\tau:CD_\tau \longrightarrow \YY$ which sends $(a,b) \in E(CD_\tau)$ to $b$, hence $j_\tau(a) = i(a)$. We write $D_\tau$ for the geometric realization of $CD_\tau$. The functor $j_\tau$ induces a cellular map $D_\tau \longrightarrow Y_b$, which we also denote $j_\tau$. Given $a \in E(\YY)$, there is a functor $j_a:CD_{i(a)} \longrightarrow CD_{t(a)}$ which sends $(b,c) \in E(CD_{i(a)})$ to $(ab,c)$. This $j_a$ induces a cellular map $D_{i(a)} \longrightarrow D_{t(a)}$, which we also denote $j_a$. 
\end{definition}

\begin{definition}\cite[Section 3.3]{Haefliger1992}
    Let $KY$ be a topological space with a continuous projection $\pi:KY \longrightarrow Y_b$. Each object $\tau \in V(\YY)$ corresponds to a vertex of $Y_b$, and we set $Y_\tau = \pi^{-1}(\tau)$. Let $Y(D_\tau)$ be the subset of $D_\tau \times KY$ of pairs $(x,y)$ with $j_\tau(x) = \pi(y)$. Let $Y(j_\tau),\pi_\tau$ be the projections onto the first and second coordinates. We identify $Y_\tau$ with the fiber $\pi_\tau^{-1}(\tau) \subset Y(D_\tau)$. Any $a \in E(\YY)$ induces a map $Y(j_a):Y(D_{i(a)}) \longrightarrow Y(D_{t(a)})$ sending $(x,y)$ to $(j_a(x),y)$. If $s$ is a section of $\pi$ over the 1--skeleton of $Y_b$, then each fiber $Y_\tau$ has a basepoint $s(\tau)$. This induces a section $s_\sigma$ of $\pi_\sigma$ over the $1$--skeleton of $D_\tau$. With this notation, $KY$ is a \emph{complex of spaces associated to $G(\YY)$} if the following hold.
    \begin{enumerate}
        \item For each $\tau \in V(\YY)$, $Y_\tau$ is connected and there is a retraction $r_\tau:Y(D_\tau) \longrightarrow Y_\tau$ which is homotopic to the identity relative to $Y(\tau)$ so that $r_\tau s_\tau(x) = s_\tau(\tau)$ for $x \in D_\tau^{(1)}$. 
        \item $\pi_1(Y_\tau,s(\tau)) = G_\tau$ and for any $a \in E(\YY)$, restricting $Y(j_a)$ to $Y_{i(a)} \subset Y(D_{i(a)})$ gives a basepoint preserving map
        \[r_{t(a)}Y(j_a):Y_{i(a)} \longrightarrow Y_{t(a)}\]
        which induces $\psi_a$ on fundamental groups.
        \item For two composable edges $a,b \in E(\YY)$ with $\tau = t(a)$, the edge $(a,b)$ of $D_{\tau}$ maps to a loop in $Y_{\tau}$ under $ r_\tau s_\tau$ representing the homotopy class of $g_{a,b}^{-1} \in \pi_1(Y_\tau,s(\tau)) = G_\tau$. 
    \end{enumerate}
    Further, if $KY$ is a cell complex, $\pi$ is a cellular map, and each $Y_\tau$ is a $K(G_\tau,1)$, then we call $KY$ an \emph{aspherical cellular realization of $G(\YY)$}.
\end{definition}

\begin{theorem}\cite[Theorem 3.4.1]{Haefliger1992}\thlabel{Haefliger}
    For each $\tau \in V(\YY)$, let $Y_\tau$ be a fixed choice of $K(G_\tau,1)$ with a basepoint $s(\tau)$. For each $a \in E(\YY)$, let $\varphi_a: Y_{i(a)} \longrightarrow Y_{t(a)}$ be any map realizing $\psi_a$ on fundamental groups. Then there is an aspherical cellular realization $\pi:KY \longrightarrow Y$ where $\pi^{-1}(\tau)$ is the given complex $Y_\tau$ for each $\tau \in V(\YY)$. 
\end{theorem}

We collect some results from Haefliger's proof of the above that we will need. 

\begin{corollary}\thlabel{retraction is a qisom}
    Suppose that for each $\tau \in V(\YY)$, $Y_\tau$ is a cell-complex and for each $a \in E(\YY)$, $\varphi_a$ is a cellular map. Let $\pi:KY \longrightarrow Y$ be an aspherical cellular realization obtained from \thref{Haefliger} using these spaces and maps as input.  
    \begin{enumerate}
        \item Each $Y(D_\tau)$ can be given a cell structure using the cell structure of the $Y_{\tau'}$ for $\tau' \in V(\YY)$, together with some cubes of dimension at most the dimension of $Y$. 
        \item With this cell structure, the retraction $r_\tau$ is a cellular map for each $\tau \in V(\YY)$. Further, if $\fitTilde{r_\tau}$ is a lift of $r_\tau$ to universal covers, then $\fitTilde{r_\tau}$ restricts to a quasi-isometry of $k$--skeletons for all $k \geq 0$.
    \end{enumerate}
\end{corollary}
\begin{proof}
    We briefly sketch Haefliger's construction of $KY$. It proceeds inductively by building a space $KY^k$ for $k = 0, 1, \ldots$. For $k=0$, $KY^0$ is the disjoint union of the chosen spaces $Y_\tau$ for each $\tau \in V(\YY)$. For each $a \in E(\YY)$, choose a map $\varphi_a:Y_{i(a)} \longrightarrow Y_{t(a)}$ realizing $\psi_a$ on fundamental groups. To construct $KY^1$, take $KY^0$ together with a mapping cylinder for each $\varphi_a$, and identify the end of each mapping cylinder with the corresponding space in $KY^0$. For a general $k$, $KY^k$ is a quotient of $KY^{k-1}$ and spaces $Y_{i(a_1)} \times [0,1]^k$ for each $k$--tuple of composable arrows $(a_k,\ldots, a_1)$ in $E(\YY)$. The construction of $KY^{k-1}$ tells us where to glue the faces of $Y_{i(a)} \times \partial [0,1]^k$, and the extension properties of $K(G,1)$ spaces allows us to extend across the interior of $Y_{i(a)} \times [0,1]^k$. For $\tau \in V(\YY)$, the retraction $r_\tau:Y(D_\tau) \longrightarrow Y_\tau$ is a contraction of these added cubes. For example if $Y$ is a graph and $\tau$ is a vertex, then $Y(D_\tau)$ is $Y_\tau$ with a mapping cylinder attached for each edge incident to $\tau$, and the retraction is just the standard deformation retraction onto $Y_\tau$.

    If $Y$ has dimension $n$, then $\YY$ does not have any $n+k$--tuples of composable edges for any $k >0$, meaning this iterated mapping cylinder construction ends after $n$ steps. Further, if each $\varphi_a$ for $a \in E(\YY)$ is cellular, then each mapping cylinder can be given a cell structure so that each $r_\tau$ is the contraction of some cells, meaning $r_\tau$ is a cellular map. This also allows us to understand the skeletons of $KY$; the $k$--skeleton $(KY)_k$ is the $k$--skeleton of each $Y_\tau$ for $\tau \in V(\YY)$ together with cubes of dimension at most $k$. This explains $1$.

    Fix some $\tau \in V(\YY)$ and consider the space $Y(D_\tau)$. Because $r_\tau$ is a homotopy equivalence and a retraction onto $Y_\tau$, it follows that there is a unique lift of the inclusion $\fitTilde{Y_\tau} \longrightarrow \fitTilde{Y(D_\tau)}$ to universal covers, so we can identify $\fitTilde{Y_\tau}$ with a subset of $\fitTilde{Y(D_\tau)}$. Further, $r_\tau$ is homotopic to the identity, so if $h_t:Y(D_\tau) \longrightarrow Y(D_\tau)$ is the homotopy from the identity to $r_\tau$, then $h_\tau$ lifts to a deformation retraction of $\fitTilde{Y(D_\tau)}$ onto this unique lift of $\fitTilde{Y_\tau}$, say $\fitTilde{h_\tau}$. From Haefliger's construction, $h_\tau$ is the contraction of the finitely many cubes added between the $Y_\tau$. If $a:\tau'\longrightarrow \tau$, we may view $Y_{\tau'}$ as a subset of $Y(D_\tau)$, and the restriction of $r_\tau$ is simply $\varphi_a$. Thus $r_\tau$ inherits the cellular properties of the $\varphi_a$, and in particular it can be restricted to any $k$--skeleton. Further, $h_\tau$ translates points a finite distance, so the lift of $r_\tau$ also translates points a finite distance. Thus $\fitTilde{Y_\tau}$ is cobounded in $\fitTilde{Y(D_\tau)}$ and it follows that $r_\tau$ is a quasi-isometry, and it can be restricted for a quasi-isometry between any $k$-skeletons. 
\end{proof}

\begin{proposition}\thlabel{Cor of Haefliger}
    Let $\pi:KY \longrightarrow Y$ be a complex of spaces associated to $G(\YY)$. Then the universal cover $\fitTilde{KY}$ induces a complex of spaces compatible with the $G$ action on $X$. For $\sigma \in V(\XX)$, the space $X_\sigma$ is the image of a lift $\pi_\tau:Y(D_\tau)\longrightarrow KY$ to universal covers. For simplices $\sigma' \subset \sigma$ in $X$, the map $\varphi_{\sigma',\sigma}:X_\sigma \longrightarrow X_{\sigma'}$ is a lift of the map $Y(j_a):Y(D_{i(a)})\longrightarrow Y(D_{t(a)})$ to universal covers, where $a \in E(\YY)$ is the image of the morphism $\sigma\longrightarrow \sigma'$ of $\XX$. If $KY$ is an aspherical cellular realization, the $\varphi_{\sigma',\sigma}$ are cellular maps.  
\end{proposition}
\begin{proof}
    Lift the map $\pi:KY \longrightarrow Y_b$ to a map $p$ between universal covers which is $G = \pi_1(KY)$--equivariant and makes the following diagram commute. 
    
    \begin{center}
        \begin{tikzcd}
        \fitTilde{KY} \arrow[r, "p"] \arrow[d] & X_b \arrow[d] \\
        KY \arrow[r, "\pi"]                    & Y_b          
    \end{tikzcd}
    \end{center}

    Given a vertex $\sigma$ of $X_b$, $\sigma$ projects to a vertex $\tau$ of $Y_b$. The preimage of $Y_\tau$ in $\fitTilde{KY}$ is the disjoint union of lifts of $Y_\tau \hookrightarrow KY$ to universal covers, hence $p^{-1}(\sigma)$ is exactly the image of some lift. This image is contained in the image of some lift of $\pi_\tau:Y(D_\tau) \longrightarrow KY$ to universal covers, and we set $X_\sigma$ to be the image of this lift. This is exactly $p^{-1}(D_\sigma)$, as illustrated in this diagram. 
    
    \begin{center}
        \begin{tikzcd}
    \fitTilde{Y(D_\tau)} \arrow[d] \arrow[r] & X_\sigma \arrow[r, "p"] \arrow[d] & D_\sigma \arrow[d] \\
    Y(D_\tau) \arrow[r, "\pi_\tau"]          & KY \arrow[r, "\pi"]               & D_\tau        
    \end{tikzcd}
    \end{center}

    Since $G$ acts on $\fitTilde{KY}$ and permutes these lifts, we have maps $g:X_\sigma \longrightarrow X_{g \sigma}$ for each $g \in G$, $\sigma \in V(\XX)$ and the triangular diagram in \thref{Complex of spaces compatible with G action} commutes. 
    
    Given $a \in E(\YY)$, we can lift $Y(j_a)$ to a map $\fitTilde{Y(D_\tau) }\longrightarrow \fitTilde{Y(D_{\tau'})}$. There are many choices for this lift, but each identifies $\fitTilde{Y(D_\tau)}$ with a subspace of $\fitTilde{Y(D_{\tau'})}$. If $\sigma \longrightarrow \sigma'$ is a morphism of $\XX$ covering $a \in E(\YY)$, then we can use some such lift to get an inclusion $\varphi_{\sigma',\sigma}: X_\sigma \longrightarrow X_{\sigma'}$. Again, since $G$ permutes these lifts, the square diagram in \thref{Complex of spaces compatible with G action} commutes. 
    
    To check the first condition of \thref{Complex of spaces compatible with G action}, suppose $a,b$ are composable edges of $\YY$ with $\tau = i(b), \tau' = t(b) =i(a), \tau'' = t(a)$. Then $ab$ is the unique morphism of $\YY$ with source $\tau$ and target $\tau''$ because $\YY$ is the scwolification of $Y$ and hence is a simple scwol. Therefore $j_a j_b = j_{ab}$, which implies $Y(j_a)Y(j_b) = Y(j_{ab})$. It follows that composing lifts of $Y(j_b)$, $Y(j_a)$ gives a lift of $Y(j_{ab})$. If $\sigma '' \subset \sigma' \subset \sigma$ are simplices of $X$, the corresponding morphisms in $\XX$ cover some arrows $a,b \in E(\YY)$ as in the previous paragraph. The maps $\varphi_{\sigma'',\sigma'}, \, \varphi_{\sigma',\sigma}, \, \varphi_{\sigma'',\sigma}$ are lifts of $Y(j_a,),\,Y(j_b), Y(j_{ab})$, hence $\varphi_{\sigma'',\sigma} = \varphi_{\sigma'',\sigma'}\varphi_{\sigma',\sigma}$ as needed.  

    Finally, if $\pi:KY \longrightarrow Y$ is an aspherical cellular realization, then all these maps can be taken to be cellular. 
\end{proof}

The next definition is analogous to Definition 2.2 of \cite{Martin}. 

\begin{definition}\thlabel{martins def}
    A \emph{complex of spaces compatible with $G(\YY)$} consists of the following. 
    \begin{enumerate}
        \item For each simplex $\tau$ of $Y$, a space $Y_\tau$ with a $G_\tau$ action. 
        \item For each arrow $a \in E(\YY)$, an embedding $\varphi_{a}:Y_{i(a)} \longrightarrow Y_{t(a)}$ which is $\psi_a$--equivariant, that is, for each $g \in G_{i(a)}$ and $x \in Y_{i(a)}$, we have 
        \[\varphi_a(g\cdot x) = \psi_a(g) \cdot \varphi_a(x),\]
        and such that for every pair of composable edges $a,b \in E(\YY)$, we have
        \[g_{a,b} \circ \varphi_{ab} = \varphi_a\circ\varphi_b.\]
    \end{enumerate}
\end{definition}

Martin does the following in section 9 of \cite{Martin}. Suppose $G(\YY)$ is a complex of hyperbolic groups over a finite simplicial complex $Y$. Beginning with a finite generating set for local groups of cells with maximal dimension, we inductively define a generating set $S_\tau$ for each $\tau \in V(\YY)$ so that if $a:\tau' \longrightarrow \tau$ is a morphism of $\YY$, then $\psi_a(S_{\tau'}) \subset S_\tau$. Let $Y_\tau$ be the Rips complex $P_n(\Gamma_\tau)$ where $\Gamma_\tau$ is the Cayley graph of $G_\tau$ with respect to $S_\tau$. Because there are finitely many hyperbolic groups here, we can choose $n$ large enough so that each $Y_\tau$ is contractible. Whenever $\sigma \subset \sigma'$ in $Y$, we let $\varphi_{\sigma,\sigma'}$ be the induced map induced on these Rips complexes by $\psi_{\sigma,\sigma'}$. 

With this notation, the $Y_\tau$ do \emph{not} form a complex of spaces compatible with $G(\YY)$ because the last condition fails. Explicitly, if $a,b$ are composable edges of $\YY$, then $\varphi_a \circ \varphi_b$ will map $\Gamma_{i(b)}$ to the image of $\psi_a\psi_b(G_{i(b)})$, which is $Ad(g_{a,b})\psi_{ab}$ by the definition of a complex of groups. On the other hand, $g_{a,b}\varphi_{ab}(\Gamma_{i(b)})$ will be a translation of $\psi_{ab}(\Gamma_{i(b)})$, not a conjugation. In other words, there is a missing $g_{a,b}^{-1}$ on the right of $g_{a,b}\circ \varphi_{ab}$ which makes it different from $\varphi_a\circ \varphi_b$. Thus \cite[Proposition 9.4]{Martin} is subtly flawed. This issue is why we use \thref{Haefliger}. We will not use it, but the following proposition shows how to construct structures as in \thref{martins def}, thereby recovering some of Martin's formalism. 

\begin{proposition}\thlabel{diagram chase}
    A complex of spaces compatible with the $G$ action on $X$ induces a complex of spaces compatible with $G(\YY)$. 
\end{proposition}
\begin{proof}
    This is essentially a diagram chase. To construct the complex of spaces compatible with $G(\YY)$, we set $Y_\tau = X_{\fitTilde{\tau}}$ and $\varphi_a$ to be the composition $\varphi_{\fitTilde{t(a)},h_a\fitTilde{i(a)}} h_a:X_{\fitTilde{i(a)}} \longrightarrow X_{\fitTilde{t(a)}}$. It is immediate from (2) in \thref{Complex of spaces compatible with G action} that $G_\tau$ acts on $Y_\tau$ for each $\tau \in V(\YY)$ and that the maps $\varphi_a$ are $\psi_a$--equivariant. To check that $g_{a,b} \circ \varphi_{ab} = \varphi_a\circ \varphi_b$ for composable edges $a,b \in E(\YY)$, consider the following diagram. 
    \begin{center}
    \begin{tikzcd}
X_{\fitTilde{i(b)}} \arrow[rr, "h_b"] \arrow[rrrd] &  & X_{h_b\fitTilde{i(b)}} \arrow[rr, "h_a"] \arrow[dd]   &                                                                & X_{h_ah_b\fitTilde{i(b)}} \arrow[dd] \\
                                                   &  &                                                       & X_{h_{ab}\fitTilde{i(b)}} \arrow[dddd] \arrow[ru, "{g_{a,b}}"] &                                      \\
                                                   &  & X_{\fitTilde{t(b)}} = X_{\fitTilde{i(a)}}  \arrow[rr] &                                                                & X_{h_a\fitTilde{i(a)}} \arrow[dd]    \\
                                                   &  &                                                       &                                                                &                                      \\
                                                   &  &                                                       &                                                                & X_{\fitTilde{t(a)}}                  \\
                                                   &  &                                                       & X_{t(\fitTilde{ab})} \arrow[ru, "g_{a.b}"]                     &                                     
\end{tikzcd}
    \end{center}

    The vertical arrows are the embeddings from the assumed complex of spaces compatible with the $G$ action on $X$ and the horizontal arrows are the homeomorphisms from elements of $G$. The composition $\varphi_a \circ \varphi_b$ is the staircase path, and the composition $g_{a,b} \circ \varphi_{ab}$ goes along the long diagonal arrow, straight down, then along the short diagonal. The commutativity of each square and triangle comes from \thref{Complex of spaces compatible with G action}, so the two paths are equivalent.
\end{proof}

\section{Constructing $\overline{Z}$}\label{section:Constructing overlineZ}

For clarity, we recall the main theorem.

\MainTheorem*

Fix once and for all a complex $X$, a group $G$, an acylindricity constant $A$, a hyperbolicity constant $\delta_0$, and let $Y = X/G$. Recall from \thref{stabilizer convention} that stabilizers are considered pointwise.

\begin{assumption}\label{scaling assumption}
    After rescaling the metric, we may assume that for every simplex $\sigma \subset X$, the distance from $\sigma$ to the boundary of its closed simplicial neighborhood is at least $1$. Equivalently, $d(\sigma, \overline{N}(\sigma) \setminus N(\sigma)) \geq 1$. As $X$ is an $M_\kappa$--complex, this will scale $\kappa, \delta_0$, but it will not change that $\kappa \leq 0$, or hyperbolicity and $\textrm{CAT}(0)$ properties.
\end{assumption}

Let $\XX, \, \YY$ be the scwolifications of $X$ and $Y$. By definition, the simplices of $X$ correspond to objects in $V(\XX)$. For simplices $\tau' \subset \tau \subset Y$, we will write $[\tau'\tau]$ for the corresponding morphism in $E(\YY)$. Choosing a maximal tree $T$ in $\YY$, a lift of $T$ to $\XX$, and lifts for each element of $V(\YY) \setminus T$ induces a choice of twisting elements, hence a complex of groups $G(\YY) = (G_\sigma, \psi_a,g_{a,b})$. This identifies $G$ with $\pi_1(G(\YY),T)$ and $\XX$ with the universal cover of $G(\YY)$ as described in \thref{Universal Cover of Complex of Groups}. We can recover $X$ from this construction as 
\[X = \bigsqcup_{g \in G,\tau \in V(\YY)} (gG_\tau,\tau)\times \tau \Big/\equiv\]
where $\equiv$ is defined as follows: if $\tau'\subset \tau$ are simplices of $Y$ and $x \in \tau'$, then
\[(gG_\tau,\tau,x) \equiv (g[\tau'\tau]^-G_{\tau'},\tau',x).\]
In words, the disjoint union above is one simplex for each object of $\XX$, and $\equiv$ glues these simplices along faces via the corresponding morphisms in $E(\XX)$. We will say a simplex $\sigma = (gG_\tau,\tau)\times \tau$ is \emph{labeled by} $(gG_\tau,\tau)$.

\begin{lemma}\thlabel{our spaces exist}
    There is a complex of spaces compatible with the $G$ action on $X$ satisfying the following.
    \begin{enumerate}
        \item For each $\sigma \in V(\XX)$, there is a cellular retraction $r_\sigma:X_\sigma \longrightarrow C_\sigma$, where $C_\sigma$ is a cusped space for $G_\sigma$, possibly with duplicated edges. This retraction is a quasi-isometry, making $X_\sigma$ into a $\delta$--hyperbolic space and identifying $\partial X_\sigma$ with the Bowditch boundary $\partial G_\sigma$. We write $\overline{X}_\sigma := X_\sigma \sqcup\partial G_\tau$.
        
        \item For each pair of simplices $\sigma' \subset \sigma$ of $X$, the map $\varphi_{\sigma',\sigma}:X_\sigma \longrightarrow X_{\sigma'}$ is an inclusion. It is also a quasi-isometric embedding and extends to an embedding $\varphi_{\sigma',\sigma}:\overline{X}_\sigma \longrightarrow \overline{X}_{\sigma'}$. This extension identifies $\partial G_\sigma$ with $\Lambda G_\sigma \subset \partial G_{\sigma'}$ and has closed image. Further, $r_\sigma(X_{\sigma'})$ is in a bounded neighborhood of some translate of a Lipschitz map between cusped spaces as in \thref{lipschitz}. 

        \item (Dichotomy Property) If two closed simplices $\sigma_1,\sigma_2$ of $X$ intersect in a simplex $\sigma$, then exactly one of the following holds. 
        \begin{enumerate}
            \item Some simplex $\sigma'$ contains both $\sigma_1\cup \sigma_2$, and $\varphi_{\sigma,\sigma_1}(X_{\sigma_1}) \cap \varphi_{\sigma,\sigma_2}(X_{\sigma_2}) = \varphi_{\sigma,\sigma'}(X_{\sigma'})$.
            \item No simplex contains both $\sigma_1,\sigma_2$, and $\varphi_{\sigma,\sigma_1}(\total{\sigma_1}) \cap \varphi_{\sigma,\sigma_2}(\total{\sigma_2}) \subset \partial X_{\sigma}$. In particular, $\varphi_{\sigma,\sigma_1}(X_{\sigma_1}) \cap \varphi_{\sigma,\sigma_2}(X_{\sigma_2}) = \varnothing$. 
        \end{enumerate}
    \end{enumerate}
\end{lemma}
\begin{proof}
    For each $\tau \in V(\YY)$, $G_\tau$ is relatively hyperbolic, so we may apply \thref{kg1 remark} to construct a suitable $K(G_\tau,1)$, say $Y_\tau$. The universal cover of $Y_\tau$ has $1$-skeleton equal to a cusped space for $G_\tau$ with possibly duplicated edges, so it is $\delta$--hyperbolic for some $\delta$. For the basepoint $s(\tau) \in Y_\tau$, we choose the image of $1$ from the cusped space inside $\fitTilde{Y_\tau}$.
    
    Using $3$ of \thref{kg1 remark} for each $a:\tau' \longrightarrow \tau$ in $E(\YY)$, we can choose a basepoint preserving map $Y_\tau' \longrightarrow Y_{\tau}$ which induces $\psi_a$ on fundamental groups and lifts to map between universal covers which restricts to a Lipschitz map on $1$--skeletons as described in \thref{lipschitz}.

    Applying \thref{Haefliger} to $Y_\tau,s(\tau)$, and these specified maps, we receive an aspherical cellular realization of $G(\YY)$, say $\pi:KY \longrightarrow Y$. Applying \thref{Cor of Haefliger} we receive a complex of spaces compatible with the $G$ action on $X$ where each $X_\sigma$ is the image of a lift of $\pi_\tau:Y(D_\tau) \longrightarrow KY$ to universal covers for some $\tau \in V(\YY)$. The maps $\varphi_{\sigma,\sigma'}$ for $\sigma \subset \sigma'$ simplices of $X$ are cellular, so they can be restricted to $k$--skeletons. Using this, we replace each $X_\sigma$ with the $1$--skeleton of $X_\sigma$ and replace each $\varphi_{\sigma,\sigma'}$ with its restriction to $1$--skeletons. 

    If $\sigma$ is a simplex of $X$ lying over a simplex $\tau$ of $Y$, then the retraction $r_\tau$ lifts to a retraction $\fitTilde{r_\tau}$ of $X_\sigma$ onto a copy of the 1--skeleton of $\fitTilde{Y_\tau}$. By our application of \thref{kg1 remark}, $\fitTilde{Y_\tau}$ is a cusped space for $G_\tau$ with possibly duplicated edges, so it is $\delta$--hyperbolic for some $\delta$, and by \thref{retraction is a qisom}, $\fitTilde{r_\tau}$ is a quasi-isometry from $X_\sigma$ onto this cusped space. This proves $(1)$. 
    
    If $\sigma \subset \sigma'$ are simplices of $X$ lying over simplices $\tau \subset \tau'$ in $Y$, then the map $\varphi_{\sigma,\sigma'}$ coming from our application of \thref{Cor of Haefliger} is a lift of the inclusion $Y(D_{\tau'}) \longrightarrow Y(D_\tau)$ to universal covers. Inside of $\fitTilde{Y(D_\tau)}$, there is a copy $\fitTilde{Y_\tau}$ which is a cusped space for $G_{\sigma'}$, and $\fitTilde{r_{\tau}}$ identifies this copy with some translate of the image of a Lipschitz map as in \thref{lipschitz} by our choices of the map $Y_{\tau'} \longrightarrow Y_{\tau}$. Because $\fitTilde{r_\tau}$ is a quasi-isometry, this establishes $(2)$. 

    For (3), suppose $\sigma_1,\sigma_2, \sigma$ are simplices of $X$ with $\sigma_1 \cap \sigma_2 = \sigma$. The objects in the categories $D_{\sigma_1}, D_{\sigma_2}$ are $\sigma_1,\sigma_2$ together with all higher dimensional simplices which contain $\sigma_1, \sigma_2$ respectively. These categories embed into $D_\sigma$, and the only objects in the intersection are simplices which contain both $\sigma_1$ and $\sigma_2$. For example, $\sigma_1,\sigma_2$ might be edges meeting in a vertex. If $\sigma_1, \sigma_2$ are sides of a square $\sigma'$, then $\sigma'$ is an object in $D_{\sigma_1} \cap D_{\sigma_2}$. On the other hand, if there is no higher dimensional cell containing both $\sigma_1,\sigma_2$, then $D_{\sigma_1} \cap D_{\sigma_2} = \varnothing$.

    In the case where there is a higher dimensional cell containing $\sigma_1 \cup \sigma_2$, then because $X$ is a $\textrm{CAT}(0)$ $M_\kappa$--complex, there is a unique cell of minimal dimension containing $\sigma_1 \cup \sigma_2$, which we call $\sigma'$. The uniqueness and minimal dimension implies that $D_{\sigma_1} \cap D_{\sigma_2} = D_{\sigma'}$ as subcategories of $D_{\sigma}$. Down in $Y$, there are corresponding simplices $\tau = \tau_1\cap \tau_2$ and $\tau_1 \cup \tau_2 \subset \tau'$. The spaces $X_{\sigma_1}, X_{\sigma_2}$ are lifts to the universal cover of the spaces $Y(D_{\tau_1}), Y(D_{\tau_2})$. These two lifts intersect in a lift of $Y(D_{\tau'})$, and $3a$ follows.

    If there is no simplex containing both $\sigma_1$ and $\sigma_2$, then the spaces $X_{\sigma_1},X_{\sigma_2}$ do not intersect. This restricts any intersection of the images of $\overline{X_{\sigma_1}}$ and $\overline{X_{\sigma_2}}$ to $\partial X_{\sigma}$, which is exactly the statement of $3b$. 
\end{proof}

The previous lemma gives us a suitable space for each object of $\XX$ and maps between them. We add the geometry of the simplex itself back in and add a helpful label.

\begin{definition}
    For $\sigma \in V(\XX)$, let $\hatt{\sigma} = \{\sigma\} \times \sigma \times \total{\sigma}/\sim$, where $(\sigma, x,\xi) \sim (\sigma, x', \xi)$ for each $x,x' \in \sigma$ and $\xi \in \partial G_\sigma$.
    Write $\hatto{\sigma}: = \{\sigma\} \times \sigma \times X_{\sigma}$ for the points of $\hatt{\sigma}$ not in the boundary.
    For simplices $\sigma \subset \sigma'$ of $X$, let $\hatt{\sigma'}|_\sigma$ be the points of $\hatt{\sigma'}$ with a second coordinate in $\sigma$, as in $\hatt{\sigma'}|_\sigma = \{(\sigma,x,y) \, | \, x \in \sigma\}$.
    We extend $\varphi_{\sigma,\sigma'}$ to a map $\hatt{\sigma'}|_\sigma \longrightarrow \hatt{\sigma}$ by $\varphi_{\sigma,\sigma'}(\sigma',x,z) = (\sigma,x,\varphi_{\sigma,\sigma'}(z))$. 
\end{definition} 

Since $\sigma$ is compact, $\hatto{\sigma}$ is still a $\delta$--hyperbolic metric space on which $G_\sigma$ acts with $\partial \hatto{\sigma} = \partial X_\sigma = \partial G_\sigma$. 

\begin{definition}\thlabel{defn of Z}
    Let
\[Z \sqcup \partial_{Stab}G =\bigg( \bigsqcup_{\sigma \in V(\XX)} \hatt{\sigma}\bigg)\bigg/\simeq\]
where $Z$ is the image of the $\hatto{\sigma}$ and $\partial_{Stab}G$ is the image of the $\partial G_\sigma$. To define $\simeq$, let $\sigma \subset \sigma'$ be simplices of $X$ and $w \in \hatt{\sigma'}$, and set $w \simeq \varphi_{\sigma,\sigma'}(w)$. More explicitly, if $w = (\sigma',x,z)$, we set $(\sigma',x,z)\simeq (\sigma,x,\varphi_{\sigma,\sigma'}(z))$.
\end{definition}

 With this notation, $G$ acts on $Z \sqcup \partial_{Stab}G$ diagonally -- for $g \in G$, $w = (\sigma,x,z)$, we have $gw = (g\sigma, gx,gz)$. In particular, if $g \in G_\sigma$, then $gw = (\sigma,x,gz)$. 

\begin{definition}\thlabel{defn of pi}
    Let $\pi: \bigsqcup_{\sigma \in V(\XX)} \hatt{\sigma} \longrightarrow Z \sqcup \partial _{Stab}G$ be the quotient map. For each simplex $\sigma \subset X$, we write $\pi_\sigma:\hatt{\sigma}\longrightarrow Z \sqcup \partial _{Stab}G$ for the restriction to $\hatt{\sigma}$, and depending on context we may restrict the domain of $\pi_\sigma$ to $\hatto{\sigma}$ or $\partial \hatt{\sigma}$.
\end{definition}
\begin{lemma}\thlabel{internal spaces embed}
    For any $\sigma$, the projection $\pi_\sigma:\hatto{\sigma}\longrightarrow Z$ is injective.
\end{lemma}
\begin{proof}
    If two points in the disjoint union defining $Z$ are identified by $\simeq$, there must be a chain of relations, say 
    \[w_0 \simeq w_1 \simeq \cdots \simeq w_{n-1} \simeq w_n\]
    where $w_0,w_n$ are the points identified and each $w_i \simeq w_{i+1}$ is an equation of the form defining $\simeq$. Fix a simplex $\sigma$. Inducting on $n$, we show that for any such chain with $w_0,w_n \in \hatto{\sigma}$, we have $w_0 = w_n$.

    The defining equation for $\simeq$ doesn't change the coordinate in the simplex, so for any chain, the second coordinate of each $w_i$ is the same, say $x$. There is no chain of length $n=1$ because then the simplex labels could not match. If $n=2$, the chain goes from $\sigma$ to another simplex $\sigma'$, then back to $\sigma$. If $\sigma \subset \sigma'$, this is
    \[w_0 = (\sigma,x,\varphi_{\sigma,\sigma'}(z)) \simeq (\sigma',x,z) \simeq (\sigma,x,\varphi_{\sigma,\sigma'}(z)) = w_2,\]
    and if $\sigma' \subset \sigma$, this is 
    \[w_0 = (\sigma,x,z) \simeq (\sigma',x,\varphi_{\sigma',\sigma}(z)) \simeq (\sigma,x,z) = w_2.\]
    In the above, we use that $\varphi_{\sigma',\sigma}$ is an embedding, which guarantees that $(\sigma',x,\varphi_{\sigma',\sigma}(z)) \neq (\sigma',x,\varphi_{\sigma',\sigma}(y))$ for $y \neq z$. This proves the base case. 

    Now suppose that $w_0 = w_n$ for every chain of length $n$ and we have a chain of length $n+1$ from $w_0$ to $w_{n+1}$. We shorten the chain by doing the first two steps in one step. Let $\sigma', \sigma''$ be the simplex labels for $w_1,w_2$. If $\sigma'' \subset \sigma'\subset \sigma$, then the chain is
    \[w_0 = (\sigma,x,z) \simeq (\sigma',x,\varphi_{\sigma',\sigma}(z)) \simeq (\sigma'',x,\varphi_{\sigma'',\sigma'}\varphi_{\sigma',\sigma}(z)) \simeq \cdots\]
    But $\varphi_{\sigma'',\sigma'}\varphi_{\sigma',\sigma}(z) = \varphi_{\sigma'',\sigma}(z)$, so we can shorten this chain, using this instead
    \[w_0 = (\sigma,x,z) \simeq  (\sigma'',x,\varphi_{\sigma'',\sigma}(z)) \simeq \cdots\]
    The case $\sigma \subset \sigma' \subset \sigma''$ is similar. If $\sigma \subset \sigma '' \subset \sigma'$, then the chain must start with 
    \[w_0 = (\sigma,x,z) \simeq (\sigma',x,z') \simeq (\sigma'',x,\varphi_{\sigma'',\sigma'}(z')) \simeq \cdots\]
    
    where $z = \varphi_{\sigma, \sigma'}(z')$. But then $\varphi_{\sigma,\sigma''}\varphi_{\sigma'',\sigma'}(z') = \varphi_{\sigma,\sigma'}(z') = z$, so again we can shorten this chain in exactly the same way as before. All the other cases are similar. 
\end{proof}

\begin{proposition}\thlabel{defn of p}
    There is a $G$--equivariant projection $p:Z \longrightarrow X$ which projects to the simplex coordinate, $p(\sigma,x,z) = x$. This map can be extended to $p:Z \sqcup \partial X \longrightarrow \overline{X}$ by declaring $p(\eta) = \eta$ for all $\eta \in \partial X$. With the quotient topology on $Z$ and the disjoint union topology on $Z \sqcup \partial X$, $p$ is continuous. 
\end{proposition} 

That $p$ is well defined, $G$--equivariant, and extends to $\partial X$ is immediate from the definition of $Z$, so the only thing to prove is that $p$ is continuous. We introduce some helpful notation which we will use throughout.

\begin{notation}\thlabel{defn of c}
    For any points $x,y \in \overline{X}$, let $[x,y]$ denote the (possibly infinite) geodesic from $x$ to $y$.
    For $x \in X$, let $\sigma_x$ be the unique simplex of $X$ containing $x$ in its interior. By \thref{internal spaces embed}, each $z \in Z$, has a unique representative of the form $(\sigma_{p(z)},p(z),y)$. We call $y$ the \emph{cusped space coordinate of $z$}. For an arbitrary $z$, we write $c(z)$ for the cusped space coordinate of $z$. 
\end{notation}

\begin{proof}[Proof of \thref{defn of p}]
    Suppose $U$ is open in $\overline{X}$. Then $p^{-1}(U) \cap \partial X = U \cap \partial X$, which is clearly open in $\partial X$. To see $p^{-1}(U)\cap Z$ is open in $Z$, we must show that for each simplex $\sigma$ of $X$, $\pi_\sigma^{-1}p^{-1}(U)$ is open in $\hatto{\sigma}$. For any such $\sigma$, $\pi_\sigma^{-1} p^{-1}(U) = \{\sigma\} \times (\sigma \cap U) \times \internal{\sigma}$. This is a product of open sets in $\hatto{\sigma}$, hence open. 
\end{proof}

The following lets us understand a neighborhood of $z \in Z$ as a ball around $p(z)$ and a neighborhood of $c(z) \in \total{\sigma_{p(z)}}$.

\begin{lemma}\thlabel{nbhd of z in Z}
    Let $z \in Z$ and let $x = p(z)$. Suppose $U$ is an open neighborhood of $c(z) \in X_{\sigma_x}$ and $0<\delta$ is so that $B(x,\delta) \subset st(\sigma_x)$. For each simplex $\sigma$ containing $\sigma_x$, let $W_{\sigma} = \varphi_{\sigma_x,\sigma}^{-1}(U) \subset \total{\sigma}$. Then
    \[W = W_z(U,\delta) = \{z' \in Z \; | \; p(z') \in B(x,\delta), \, c(z) \in W_{\sigma_{p(z')}} \}\]
    is an open neighborhood of $z$ in $Z$.
\end{lemma}
\begin{proof}
    For any simplex $\sigma$ of $X$, either $\sigma$ is not a simplex of $st(\sigma_x)$ and $\pi_\sigma^{-1}(W) = \varnothing$, or $\sigma_x \subseteq \sigma$ and $\pi_\sigma^{-1}(W) = \{\sigma\} \times (B(x,\delta)\cap \sigma) \times W_\sigma$, which is a product of open sets from $\sigma, \total{\sigma}$, hence is open in $\hatto{\sigma}$. Thus $\pi^{-1}_\sigma(W)$ is open in $\hatt{\sigma}$ for all $\sigma \subset X$, which by the definition of the quotient topology, shows $W$ is open in $Z$.
\end{proof}

Recall that $X$ is $\delta_0$--hyperbolic so it has a boundary $\partial X$ and the action of $G$ on $X$ extends to an action on this boundary. 

\begin{definition}\thlabel{defn of overlineZ}
    Let
\[\overline{Z} = Z \sqcup \partial_{Stab}G \partial X  \quad \quad \quad \partial G = \partial_{Stab}G \sqcup \partial X.\]
\end{definition} 
Our goal is to endow $\overline{Z}$ with a topology so that it is compact and metrizable, show $G$ acts on it as a convergence group with limit set $\partial G$, and then use Yaman's \thref{Geometrically Finite Convergence implies RelHyp} to conclude that $G$ is relatively hyperbolic. See Figure \ref{fig:gluings} for a diagram of how the Bowditch boundaries are glued in $\partial_{Stab}G$. Note the inverse relationship between dimension of a cell and the 'size' of its limit set. 

\begin{figure}
        \centering
        \scalebox{.8}{
        \begin{tikzpicture}

        \filldraw [fill=lightgray,thick] (-3,0) node[anchor=east]{$v_1$} -- 
        (-1.5,1.732*1.5) node[anchor=east]{$e_1$} --
        (0,1.732*3)node[anchor=south]{$v_2$} -- 
        (1.5,1.732*1.5) node[anchor=west]{$e_2$} -- 
        (3,0) node[anchor = west]{$v_3$} -- 
        (0,0) node[anchor = north]{$e_3$} -- 
        (-3,0) ;
        \node at (0,1.732){$\sigma$};

        \node[circle,draw, label =right:$\partial G_\sigma$, minimum size = 10pt] at (8,1.732) {};
        \draw[shorten >=7pt,shorten <=7pt, ->](8,1.732) -- (5,0);
        \draw[shorten >=7pt,shorten <=7pt, ->](8,1.732) -- (11,0);
        \draw[shorten >=7pt,shorten <=7pt, ->](8,1.732) -- (8,3*1.732);
        \draw[shorten >=7pt,shorten <=7pt, ->](8,1.732) -- (6.5,1.5*1.732);
        \draw[shorten >=7pt,shorten <=7pt, ->](8,1.732) -- (9.5,1.5*1.732);
        \draw[shorten >=7pt,shorten <=7pt, ->](8,1.732) -- (8,0);

        \node[circle,draw,label=below:$\partial G_{v_1}$,minimum size = 50pt] at (5,0){};
        \node[circle,draw,minimum size = 10pt] at (5,0){};
        \node[circle,draw,minimum size = 30pt] at (5+.15,0-.15){};
        \node[circle,draw,minimum size = 30pt] at (5-.15,0+.15){};

        \node[circle,draw,minimum size = 30pt] at (8+.2,3*1.732){};
        \node[circle,draw,minimum size = 30pt] at (8-.2,3*1.732){};
        \node[circle,draw,minimum size = 30pt] at (11+.15,0+.15){};
        \node[circle,draw,minimum size = 30pt] at (11-.15,0-.15){};
        
        \node[circle,draw,label=above:$\partial G_{v_2}$,minimum size = 50pt] at (8,3*1.732){};
        \node[circle,draw,minimum size = 10pt] at (8,3*1.732){};
        \node[circle,draw,label=below:$\partial G_{v_3}$,minimum size = 50pt] at (11,0){};
        \node[circle,draw,minimum size = 10pt] at (11,0){};

        \node[circle,draw,label=left:$\partial G_{e_1}$,minimum size = 30pt] at (6.5,1.732*1.5){};
        \node[circle,draw,minimum size = 10pt] at (6.5,1.732*1.5){};
        \node[circle,draw,label=right:$\partial G_{e_2}$,minimum size = 30pt] at (9.5,1.5*1.732){};
        \node[circle,draw,minimum size = 10pt] at (9.5,1.732*1.5){};
        \node[circle,draw,label=below:$\partial G_{e_3}$,minimum size = 30pt] at (8,0){};
        \node[circle,draw,minimum size = 10pt] at (8,0){};

        \draw[shorten >=18pt,shorten <=18pt, ->](6.5,1.5*1.732) -- (5,0);
        \draw[shorten >=18pt,shorten <=18pt, ->](6.5,1.5*1.732) -- (8,3*1.732);

        \draw[shorten >=18pt,shorten <=18pt, ->](9.5,1.5*1.732) -- (11,0);
        \draw[shorten >=18pt,shorten <=18pt, ->](9.5,1.5*1.732) -- (8,3*1.732);

        \draw[shorten >=20pt,shorten <=18pt, ->](8,0) -- (5,0);
        \draw[shorten >=20pt,shorten <=18pt, ->](8,0) -- (11,0);
        
        \end{tikzpicture}
        }
        \caption{On the left is a geometric simplex of $X$. On the right, each circle represents the Bowditch boundary of the corresponding cell. The arrows represent the maps $\varphi_{\tau,\tau'}$ between boundaries, and $\partial_{Stab}G$ is constructed by identifying points along these arrows. construct $Z$.}
        \label{fig:gluings}
\end{figure}
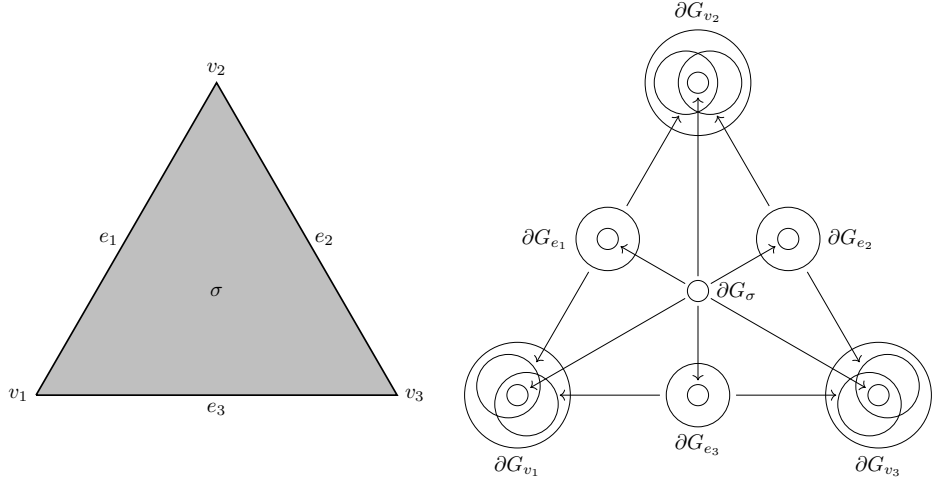

\subsection{Domains and Their Geometry}

The map $p:Z\longrightarrow X$ in \thref{defn of p} projects a point in $Z$ to $X$. The next definition is the analogous concept for $\xi \in \partial_{Stab}G$. 

\begin{definition}\thlabel{domain}
    Let $\xi \in \partial_{Stab}G$. The \emph{domain} of $\xi$, denoted $D(\xi)$, is the subcomplex of $X$ spanned by simplices $\sigma \subset X$ so that $\xi \in \pi_\sigma(\partial G_\sigma)$. We use $V(\xi)$ to denote the vertices of $D(\xi)$. If $\eta \in \partial X$, we set $D(\eta) = \{\eta\}$.
\end{definition} 

\begin{proposition}\thlabel{boundaries embed}
    If $\sigma$ is a simplex of $X$, then the projection $\pi_\sigma: \partial G_\sigma \longrightarrow \partial_{Stab}G$ is injective. 
\end{proposition}

The content of \thref{boundaries embed} is that for a given $\partial G_\sigma$, no points are identified by the relation $\simeq$ in \thref{defn of Z}. If $\sigma$ is a simplex of $X$ with vertex $v$, then $\pi_\sigma = \pi_v \varphi_{v,\sigma}$. Because $\varphi_{v,\sigma}$ is injective, to show $\pi_\sigma$ is injective, it is enough to show $\pi_v$ is injective. As in the proof of \thref{internal spaces embed}, identifications happen along chains and it suffices to consider identifications made along edges of $X$, motivating the following definition.

\begin{definition}\thlabel{xi_path}
    Let $\xi \in \partial_{Stab}G$. A \emph{$\xi$--path} is the data $\{(v_i)_{0 \leq i \leq n},\, (\xi_i)_{0 \leq i \leq n}, \,(x_i)_{1 \leq i \leq n}\}$ where
    \begin{itemize}
        \item the $v_0,\ldots ,\,v_n$ are vertices of $X$ so that $v_i$ is adjacent to $v_{i+1}$, and $e_i = [v_{i-1},v_i]$ is the edge between $v_{i-1},v_i$,
        \item each $\xi_i \in \partial G_{v_i}$ and $\xi_i$ represents the equivalence class of $\xi$ in $\partial_{Stab}G$, so that $\xi = \pi_{v_i}(\xi_i)$,
        \item each $x_i \in \partial G_{[v_{i-1},v_i]}$ with $\varphi_{v_{i-1},e_i}(x_i) = \xi_{i-1}$ and $\varphi_{v_i,e_i}(x_i) = \xi_i$ 
    \end{itemize}
    This data will be denoted $[v_0,\ldots, v_n]_\xi$, and the choices of $x_i,\xi_i$ will be implicit. The \emph{support} of a $\xi$--path is the path in the $1$--skeleton of $X$ given by the $v_i$. If $v_0 = v_n$, we call the $\xi$--path a \emph{$\xi$--loop}.
\end{definition}

\begin{notation}
    If a group $H < G$ stabilizes a simplex $\sigma \subset X$ pointwise, then $H< G_\sigma$ and we can consider the action of $H$ on $\hatt{\sigma}$. Since $G_\sigma$ acts on $\hatt{\sigma}$ as a convergence group, we can consider the limit set of $H$ in this action, denoted $\Lambda_\sigma H \subset \partial G_\sigma$, which is nonempty if $H$ is infinite. If $\sigma'$ is a face of $\sigma$, then $H$ also stabilizes $\sigma'$ and it follows from the definition of the extended $\varphi_{\sigma',\sigma}$ in \thref{our spaces exist} that $\varphi_{\sigma',\sigma}(\Lambda_\sigma H) = \Lambda_{\sigma'}H$.
\end{notation}

\begin{lemma}\thlabel{Stabilizer of a subcomplex limit set}
    Suppose $H < G$ is infinite and stabilizes a connected subcomplex $K \subset X$ pointwise. If $v_0 \in K$ and $\xi \in \Lambda_{v_0}H \subset \partial G_{v_0}$, then $\pi_{v_0}(\xi) \in \pi_\sigma(\partial G_{\sigma})$ for all $\sigma \subset K$, hence $K \subset D(\pi_{v_0}(\xi))$.
\end{lemma}
\begin{proof}
    Because $K$ is connected, $v_0$ can be connected to any $\sigma \subset K$ by a path of simplices $v_0 = \sigma_0,\sigma_1, \ldots , \sigma_n = \sigma$, and we prove the claim by induction on $n$. If $n = 1$, then $v_0$ is a simplex of $\sigma$. By assumption, $\xi \in \Lambda_{v_0}H = \varphi_{v,\sigma}(\Lambda_{\sigma}H)$, so we can choose $\xi' \in \Lambda _{\sigma}H \subset \partial G_\sigma$ mapping to $\xi$. It follows that $\pi_\sigma(\xi') = \pi_{v_0}(\xi)$. This proves the base case.

    Now suppose the claim holds for a path of simplices of length $n$, and $v_0 = \sigma_0, \sigma_1,\ldots ,\sigma_{n+1}$ is a path of simplices in $K$ of length $n+1$. By induction, there is some $\xi' \in \Lambda_{\sigma_n}H$ so that $\pi_{\sigma_n}(\xi') = \pi_{v_0}(\xi)$. If $\sigma_{n+1} \subset \sigma$, then 
    \[\pi_{v_0}(\xi) = \pi_{\sigma_n}(\xi') = \pi_{\sigma_{n+1}}\varphi_{\sigma_{n+1},\sigma_n}(\xi')\]
    and we are finished. If $\sigma_n \subset \sigma_{n+1}$, then as in the base case, $\xi' \in \Lambda_{\sigma_n}H = \varphi_{\sigma_n,\sigma_{n+1}}(\Lambda_{\sigma_{n+1}}H)$, so we can choose some $\xi'' \in \partial G_{\sigma_{n+1}}$ mapping to $\xi'$, and $\pi_{\sigma_{n+1}}(\xi'') = \pi_{\sigma_n}(\xi') = \pi_{v_0}(\xi)$. 
\end{proof}

\begin{lemma}\thlabel{H along xi path is infinite}
    Let $\xi \in \partial_{Stab}G$ and $[v_0,\ldots ,v_n]_\xi$ be a $\xi$--path. Then $H:= \bigcap_i G_{v_i}$ is an infinite, full RQC subgroup of $G_{v_i}$ and the $\xi_i \in \partial G_{v_i}$ from the definition of a $\xi$--path have $\xi_i\in \Lambda_{v_i}H$ for each $i$. 
\end{lemma}
\begin{proof}
    It is clear from the definition of $\xi$--path that the support is contained in $D(\xi)$. Induct on $n$. If $n = 1$, then the $\xi$--path is a single edge $e_1 = [v_0,v_1]$. By definition, $G_{v_0} \cap G_{v_1} = G_{e_1}$, and $G_{e_1}$ is infinite because $x_1 \in \partial G_{e_1} \neq \varnothing$. Further, $G_{e_1}$ is a full RQC subgroup of the $G_{v_i}$ by assumption. By definition of a $\xi$--path, $\xi_i= \varphi_{v_i,e_1}(x_1) \in \Lambda _{v_i}G_{e_1}$. This proves the base case.

    Suppose the lemma holds for a $\xi$--path of length $n$, and $[v_0,\ldots, v_{n+1}]_\xi$ is a $\xi$--path of length $n+1$. By induction, $H = \bigcap_{i=1}^nG_{v_i}$ is an infinite, full RQC subgroup of $G_{v_i}$ for each $i$. Then $H \cap G_{[v_n,v_{n+1}]}$ is the intersection of full RQC subgroups in $G_{v_n}$. By the Limit Set Property \ref{Limit Set Property}, this is also a full RQC subgroup in $G_{v_n}$, and by induction, we have
    
    \[\varphi_{v_n,e_{n+1}}(x_{n+1}) = \xi_n \in \Lambda_{v_n} H \cap \Lambda_{v_n}G_{e_{n+1}} = \Lambda_{v_n}(H \cap G_{e_{n+1}})\]

    Because this limit set is nonempty, $H' = H \cap G_{e_n} = \bigcap_{i=1}^{n+1} G_{v_i}$ is infinite. This also implies $x_{n+1} \in \Lambda_{e_{n+1}}H'$, hence $\xi_{n+1} \in \Lambda_{v_{n+1}}H'$. 
\end{proof}

\begin{proposition}
    For any $\xi \in \partial_{Stab}G$, $D(\xi)$ is convex and $\diam(D(\xi)) \leq A$ where $A$ is the acylindricity constant of $G$ acting on $X$. Further, there is a uniform bound on the number of simplices in $D(\xi)$. 
\end{proposition}
\begin{proof}
    Let $x,x'$ be two points of $D(\xi)$, let $\sigma_x,\sigma_{x'}$ be the unique simplices of $X$ with $x,x'$ in their interior, and let $v,v'$ be vertices of $\sigma_x,\sigma_{x'}$. Since $v,v' \in V(\xi)$, there must be a $\xi$--path joining them, say $[v_0,\ldots, v_n]_\xi$ with $v = v_0,v' = v_n$. After making the $\xi$--path longer, we can assume every vertex of $\sigma_x,\sigma_{x'}$ is in the support of this $\xi$--path. By \thref{H along xi path is infinite}, $H:=\bigcap_i G_{v_i}$ is infinite, fixes $\sigma_x \cup \sigma_{x'}$ pointwise, and $\xi_i \in \Lambda_{v_i}H$ for each $i$. Since $X$ is $\mathrm{CAT}(0)$, $H$ also stabilizes the geodesic $[x,x']$, so $d(x,x') \leq A$ by the definition of acylindricity. Since $x,x'$ were arbitrary, this shows $\diam(D(\xi)) \leq A$.

    Because $H$ fixes $[x,x']$ pointwise, if $[x,x']$ meets the interior of a simplex $\sigma$, then $H\sigma = \sigma$ as a set, but $H$ may not fix $\sigma$ point wise (imagine the plane tiled by unit squares and geodesic between $x =(0,0)$ and $x' =(1,1)$ -- reflecting across the line $x= y$ will fix the endpoints but will not fix the square that the endpoints are corners of). However because $\sigma$ has finitely many vertices, a finite index subgroup of $H$ fixes $\sigma$ pointwise. Since $[x,x']$ meets finitely many simplices by \thref{geod meets finitely many}, this implies a finite index subgroup of $H$ fixes the closed subcomplex spanned by $[x,x']$ pointwise. Let $K$ be this closed subcomplex spanned by $[x,x']$ and apply \thref{Stabilizer of a subcomplex limit set} with this $K$, $v_0$ as the start of our $\xi$--path, $\xi_0 \in \partial G_{v_0}$, and the finite index subgroup of $H$, we have that $K \subset D(\xi)$, hence $D(\xi)$ is convex. 

    Next we prove there is a uniform constant bounding the number of simplices in every domain.  Recall that for every pair of simplices $\sigma \subset \sigma'$ of $X$, $G_{\sigma'}$ is a fully RQC subgroup of $G_\sigma$, and therefore has finite height by \thref{Finite Height}. Because $Y$ is finite, we can choose $N$ so that each such $G_{\sigma'}$ has height at most $N$ in $G_\sigma$. Increasing $N$ if necessary, we can also assume that any simplex $\tau$ of $Y$ has at most $N$ simplices in $st(\tau)$. 
    
    Let $\sigma_0$ be simplex of $D(\xi)$ and choose $\xi_0 \in \partial G_{\sigma_0}$ so that $\pi_{\sigma_0}(\xi_0) = \xi$. There are at most $N$ orbits of simplices in $st(\sigma_0)$, and for any given orbit, at most $N$ simplices can be in $D(\xi)$. Indeed, suppose $g_1 \sigma,g_2 \sigma, \ldots g_m\sigma$ are all simplices of $st(\sigma_0) \cap D(\xi)$ with $g_i \in G_{\sigma_0}$. Then by the Limit Set Property \ref{Limit Set Property}, we have the following in $\partial G_{\sigma_0}$: 
    \[\xi_0 \in \bigcap_{i}\Lambda_{\sigma_0} G_{g_i\sigma} = \bigcap_i \Lambda_{\sigma_0} G^{g_i}_{\sigma} = \Lambda_{\sigma_0} \Big(\bigcap_iG_{\sigma}^{g_i}\Big) \neq \varnothing.\]
    Since the limit set of a finite group is empty, this implies $\bigcap_iG_{\sigma}^{g_i}$ is infinite. Because $G_\sigma$ has height at most $N$ in $G_{\sigma_0}$, this implies $m \leq N$. Thus each $G$ orbit of simplices in $st(\sigma_0)$ contributes at most $N$ simplices to $st(\sigma_0)\cap D(\xi)$, and there are at most $N$ $G$-- orbits, so $st(\sigma_0) \cap D(\xi)$ contains at most $N^2$ simplices. 

    Fix a vertex $v \in V(\xi)$ and inductively define a sequence of finite subcomplexes by $K_1 = v$ and $K_i:=\overline{N}(K_{i-1}) \cap D(\xi)$. That is, $K_i$ is $K_{i-1}$ together with where its closed simplicial neighborhood meets $D(\xi)$. For example $K_2 = \overline{st}(v) \cap D(\xi)$. From the previous paragraph, it is clear that each $K_i$ has finitely many simplices. 

    By \thref{Bounded Length To Bounded Number Of Simplices}, there is some constant $M$ so that any geodesic in $X$ of length at most $A$ meets at most $M$ simplices. Because $D(\xi)$ has diameter at most $A$, for any $x \in D(\xi)$, $[v,x]$ meets at most $M$ simplices of $D(\xi)$. It follows that $[v,x] \subset K_M$, hence $K_M = D(\xi)$ and $D(\xi)$ has finitely many simplices. This bound depends only on $N$ and $M$, which depend only on $G(\YY), \,X$ and $A$, so it is uniform for all domains. 
\end{proof}

\begin{corollary}\thlabel{Finiteness}[Finiteness]
    There exists an integer $d_{\max}$ satisfying the following.
    \begin{enumerate}
        \item For every $\xi \in \partial_{Stab}G$, $D(\xi)$ contains at most $d_{max}$ simplices and any geodesic in the open simplician neighborhood of $D(\xi)$ meets at most $d_{max}$ simplices.
        \item Any closed simplex of $X$ is the union of at most $d_{max}$ open simplices.
    \end{enumerate}
\end{corollary}
\begin{proof}
    The first point follows immediately from the Containment \thref{Containment} and the previous lemma. The second point follows from the finiteness of $Y = X/G$.  
\end{proof}

The integer $d_{max}$ will be an essential tool going forward. The convexity of domains also allows us to upgrade the Dichotomy Property in \thref{our spaces exist}.

\begin{corollary}\thlabel{upgrade to dichotomy}
    If two simplices $\sigma_1,\sigma_2$ of $X$ intersect in a face $\sigma$, then exactly one of the following holds.
    \begin{enumerate}
        \item There is a simplex $\sigma'$ containing $\sigma_1\cup \sigma_2$ so that
            \[\varphi_{\sigma,\sigma_1}(\overline{X_{\sigma_1}}) \cap \varphi_{\sigma,\sigma_2}(\overline{X_{\sigma_2}}) = \varphi_{\sigma,\sigma'}(\overline{X_{\sigma'}}).\]
        \item There is no simplex containing both $\sigma_1,\sigma_2$, and
            \[\varphi_{\sigma,\sigma_1}(\total{\sigma_1}) \cap \varphi_{\sigma,\sigma_2}(\total{\sigma_2}) \subset \partial X_{\sigma}.\]
    \end{enumerate}
\end{corollary}
\begin{proof}
    The only difference between this and the dichotomy property in \thref{our spaces exist} is in the first option, where the property is extended to the intersection in the boundary. Suppose we are in the first case and $\xi \in \varphi_{\sigma,\sigma_1}(\partial G_{\sigma_1}) \cap \varphi_{\sigma,\sigma_2}(\partial G_{\sigma_2})$. Simplices are convex, so if we choose points $x_1 \in \sigma_1$ and $x_2 \in \sigma_2$, then the geodesic $[x_1,x_2]$ is contained in $\sigma'$ and connects points of $D(\xi)$, so $\sigma' \subset D(\xi)$.
\end{proof}

\begin{proof}[Proof of \thref{boundaries embed}]
    Fix a vertex $v$ of $X$ and suppose $\pi_v(\xi) = \pi_v(\xi') = \zeta \in \partial_{Stab}G$. Then there must be a $\xi$--loop $\{(v_i)_{0 \leq i \leq n},\, (\xi_i)_{0 \leq i \leq n}, \,(x_i)_{1 \leq i \leq n}\}$ with $\xi = \xi_0$ and $\xi' = \xi_n$. Without loss of generality, we can assume the support of the $\xi$--path is injective and let $P$ be the $1$--complex given by this loop.

    It is clear that $P  \subset D(\zeta)$. Because $X$ is $\textrm{CAT}(0)$ and domains are convex, we can contract the loop $P$ in $D(\zeta)$ by pulling along geodesics to $v_0$. This gives a map of the disc into $D(\zeta)$, and after a homotopy, we can assume the image lies in the $2$--skeleton of $D(\zeta)$. This yields a finite, contractible $2$--complex with boundary $P$, say $H$. We will call such a subcomplex a \emph{hull} of $P$. The result will follow immediately from the following claim.

    \begin{claim}
        For any such hull, we have $\xi = \xi'$ in $\partial G_v$.
    \end{claim}
    \begin{proof}[Proof of Claim]

    Let $d$ be the number of $2$--cells in $H$ and we prove the claim by induction on $d$. For simplicity, we assume all $2$--cells are triangles, and explain at the end how to modify this to suit squares and more. 

    If $d = 1$, then $P$ is the boundary of a single triangle, say $\sigma$, and $n = 3$. Because the hull is contained in $D(\xi)$, there is some $\xi_\sigma$ so that $\varphi_{v_0,\sigma}(\xi_\sigma) = \xi \in \partial G_{v_0}$. The commutativity of the maps $\varphi$ implies $\xi_0$ completely determines all the data of the $\xi$--path. Explicitly, because $\varphi_{v_0,\sigma} = \varphi_{v_0,e_1}\varphi_{e_1,\sigma}$, it follows that $\varphi_{e_1,\sigma}(\xi_\sigma) = x_1 \in \partial G_{e_1}$, and hence $\varphi_{v_1,\sigma}(\xi_\sigma) = \xi_1$. Continuing around the triangle, we have 
    \[\xi = \varphi_{v_0,\sigma}(\xi_\sigma) = \varphi_{v_0,e_3}\varphi_{e_3,\sigma}(\xi_\sigma) = \varphi_{v_0,e_3}(x_3) = \xi'.\]

    This completes the base case. Now assume any hull with $d$ $2$--cells has $\xi = \xi'$ and let $H$ be a hull with $d+1$ $2$--cells. Suppose $\sigma$ is a $2$--cell of $H$ containing the edge $e_1 = [v_0,v_1]$. We show how to remove $\sigma$ and get a smaller hull. Because $\sigma \subset D(\zeta)$, we can choose $\xi_\sigma \in \partial G_{\sigma}$ so that $\varphi_{e_1,\sigma}(\xi_\sigma) = x_1 \in \partial G_{e_1}$. There are two cases.

    If another side of $\sigma$ is contained in $P$, say $e_2 = [v_1,v_2]$, then the third edge of $\sigma$ is $e :=[v_0,v_2]$. Set $x= \varphi_{e,\sigma}(\xi_\sigma)$. The commutativity of the inclusions yields the following, new $\xi$--loop:
    \[(v_0,v_2,v_3,\ldots,v_n), \, (\xi_0,\xi_2,\xi_3,\ldots,\xi_n), \, (x_1,x,x_3,\ldots, x_n).\]

    The closure of $H \setminus \sigma$ gives a hull of this new loop containing $d$ $2$--cells.

    If no other side of $\sigma$ is contained in $P$, then let $w$ be the remaining vertex of $\sigma$ which is not in $P$. Let $e_1' = [v_0,w]$ and $e_2' = [w,v_1]$, and let $x_1',x_2'$ be the image of $\xi_\sigma$ in $\partial G_{e_1'}, \, \partial G_{e_2'}$ and $\xi_w = \varphi_{w,\sigma}(\xi_{\sigma})$. Again the commmutativity of the diagram yields the following new $\xi$--loop:
    \[(v_0,w,v_1, \ldots ,v_n),\, (\xi_1,\xi_w,\xi_1,\xi_2,\ldots, \xi_n), \, (x_1',x_2',x_3,\ldots ,x_n) .\]
    Again, the closure of $H \setminus \sigma$ gives a hull for the new loop containing at most $d$ triangles, and we are done. 
    
    For cells which are not triangles, the base case $d=1$ is the same, except that $n$ will equal the number of sides of the single cell. For the inductive step, instead of cutting around a triangle, we can cut around an $n$-gon from the hull in a similar way.
\end{proof}
    With the claim proven, we have proven the lemma.
\end{proof}

\section{Geometric Tools}\label{section:Geometric Tools}

\subsection{Nesting and Families}

The following is key to defining open neighborhoods of points in $\partial_{Stab}G$.

\begin{definition} \thlabel{xi family} ($\xi$--family)
    Let $\xi \in \partial_{Stab}G$ and for each vertex $v \in V(\xi)$, let $\xi_v \in \partial G_{v}$ be the representative of $\xi$. A \emph{$\xi$--family} is choice of open neighborhoods $\{U_v \subset \total{v} \, | \,\xi_v \in U_v, \, v \in V(\xi)\}$ so that for every edge $[v,v']$ of $D(\xi)$ and $x \in \total{[v,v']}$, we have 
    \[\varphi_{v,[v,v']}(x) \in U_v \Longleftrightarrow \varphi_{v',[v,v']}(x) \in U_{v'}.\]    
    Equivalently, $\varphi^{-1}_{v,[v,v']}(U_v) = \varphi^{-1}_{v'[v,v']}(U_{v'})$ for each edge $[v,v']$ of $D(\xi)$.
\end{definition}

The following is a basic topological fact which we leave to the reader.

\begin{lemma}\thlabel{final equals subspace}
    Let $B$ be a topological space and let $C_1, \ldots, C_n$ be a finite collection of closed sets in $B$. Let $\sim$ be the relation on the abstract disjoint union $\bigsqcup_i C_i$ with $x \sim y$ whenever $x = y$ in $B$, and let $C = \bigsqcup_i C_i /\sim$. The natural map $p:C \longrightarrow B$ is a homeomorphism onto $\bigcup_i C_i$. In particular, the quotient topology on $C$, the final topology induced by the projections $C_i \longrightarrow C$, and the subspace topology on $\bigcup_i C_i$ are all equal. 
\end{lemma}
\begin{proposition}\thlabel{balloon prop}
    Let $\xi \in \partial_{Stab}G, v_0 \in V(\xi)$, and let $U_{v_0}$ be an open neighborhood of $\xi$ in $\overline{X_{v_0}}$. Then $U_{v_0}$ can be extended to a $\xi$--family.  
\end{proposition}
\begin{proof}
    Suppose $\Sigma$ is some set of closed simplices in $D(\xi)$ and $\{U_\sigma \subset \total{\sigma}, \sigma \in \Sigma\}$ is a collection of open neighborhoods of $\xi$. We will call this collection of neighborhoods \emph{aligned} if for each pair of simplices $\sigma \subset \sigma'$ in $\Sigma$, we have 
    
    \[U_{\sigma} \cap \varphi_{\sigma,\sigma'}(\overline{X_{\sigma'}}) = \varphi_{\sigma,\sigma'}(U_{\sigma'}).\]
    Because $\varphi_{\sigma,\sigma'}$ is an embedding, the equation above is equivalent to $\varphi_{\sigma,\sigma'}^{-1}(U_\sigma) = U_{\sigma'}$. If $\Sigma$ contains all of $D(\xi)$ and the collection of open neighborhoods is aligned, then the sets $\{U_v, v \in V(\xi)\}$ will be a $\xi$--family. 
    
    Consider the set $\{d(\sigma,v_0), \, \sigma \subset D(\xi)\}$. It is finite because $D(\xi)$ has finitely many simplices, and we enumerate it as $\{0= t_0 < t_1 < \cdots < t_m\}$. Let $\Sigma(s) = \{ \sigma \subset D(\xi), d(\sigma,v_0) \leq t_s\}$ where elements of $\Sigma(s)$ are open simplices. Because the distance between sets is an infimum, we have $d(v_0,\sigma) = d(v_0,\overline{\sigma})$ for any simplex $\sigma \subset X$, but this distance may be realized on the boundary of a simplex. If $d(\sigma, v_0) = t_i$  so that $\sigma \in \Sigma(i)$ and $\sigma \subset \sigma'$, then $d(\sigma',v_0) \leq t_i$, hence $st(\sigma) \subset \Sigma(i)$. In particular, $\Sigma(0) = st(v_0)$.  

    For each $\sigma \subset D(\xi)$, we can choose a $y_\sigma \in \overline{\sigma}$ so that $d(y_\sigma,v_0) = d(\sigma,v_0)$. This choice is unique; if there were two choices for $y_\sigma$, say $y$ and $y'$, then the geodesic $[y,y']$ is contained in $\overline{\sigma}$ because simplices are convex, but the CAT(0) inequality implies $[y,y']$ comes strictly closer to $v_0$ than $y$ or $y'$. This would imply $y,y'$ are not actually the closest points in $\sigma$ to $v_0$, a contradiction.
     
     There are finitely many simplices $\omega \in \Sigma(s) \setminus \Sigma(s-1)$ so that $y_\omega$ is in the interior of $\omega$. For such an $\omega$, any simplex $\sigma \subset st(\omega) \cap D(\xi) \setminus \Sigma(s-1)$ must have $y_\sigma = y_\omega$. For such a $\sigma$, $\omega$ must be the unique face of $\overline{\sigma}$ in $\Sigma(s)$, for if $\omega' \subset \sigma$ was another such face, we must have $y_{\omega'} = y_\sigma = y_{\omega}$, hence $\omega = \omega'$. We can enumerate these simplices where $y_\omega$ is in the interior of $\omega$ as $\omega_1,\ldots , \omega_r$, and from this discussion we have
    \[\Sigma(s) \setminus \Sigma(s-1) =D(\xi) \cap  \bigsqcup_j (st(\omega_j) \setminus \Sigma(s-1)).\]

    We imagine inflating a balloon centered at $v_0$ and recording the time and place where the balloon touches each closed simplex for the first time. These are the $t_0,\ldots ,t_m$ and $y_\sigma$. At time $t_s$, the balloon touches some \emph{open} simplices for the first time -- these are like the $\omega$ in the previous paragraph -- and we add all of these open simplices and their disjoint stars to $\Sigma(s-1)$ to form $\Sigma(s)$. We inductively build an aligned collection of open neighborhoods for each $\Sigma(i)$.
    
    For $\sigma \in \Sigma(0) = st(v_0)$, we must set $U_\sigma = \varphi_{v_0,\sigma}^{-1}(U_{v_0})$. To check these choices are aligned, suppose $v_0 \subset \sigma \subset \sigma'$ and $\varphi_{\sigma,\sigma'}(z)\in U_\sigma \cap \varphi_{\sigma,\sigma'}(\overline{X_{\sigma'}})$. Then $\varphi_{v_0,\sigma}\varphi_{\sigma,\sigma'}(z) = \varphi_{v_0,\sigma'}(z) \in U_{v_0}$ by definition of $U_\sigma = \varphi_{v_0,\sigma}^{-1}(U_{v_0})$, hence $z \in \varphi_{v_0,\sigma'}^{-1}(U_{v_0}) = U_{\sigma'}$. This shows the $\subseteq$ containment in the definition of aligned, and the $\supseteq$ is similar.
    
    Suppose we have an aligned collection of sets $\{U_\sigma, \sigma \in \Sigma(s-1)\}$ and we extend this to an aligned collection for $\Sigma(s)$. Let $\omega_1,\ldots \omega_r$ be the simplices of $\Sigma(s) \setminus \Sigma(s-1)$ so that $y_{\omega_i}$ is in the interior of $\omega_{i}$, let $\tau$ be any of the $\omega_i$, and we focus on finding $U_\tau$. Let $st(\tau) \cap \Sigma(i-1) = \{\sigma_1,\ldots ,\sigma_\ell\}$ be the simplices we have already made choices for, and set 
    \[C = \bigcup_i\varphi_{\tau,\sigma_i}(\overline{X_{\sigma_i}}) \quad \quad \quad \quad U  = \bigcup_i \varphi_{\tau,\sigma_i}(U_{\sigma_i}).\]

    \begin{claim}
        For each $i$, $U\cap \varphi_{\tau,\sigma_i}(\overline{X_{\sigma_i}}) = \varphi_{\tau,\sigma_i}(U_{\sigma_i})$.
    \end{claim}

    Given the claim, we show how to finish the proof. Because each $\varphi_{\tau,\sigma_i}$ is injective, the claim implies $\varphi_{\tau,\sigma_i}^{-1}(U) = U_{\sigma_i}$ which is open in $\overline{X_{\sigma_i}}$, hence $U$ is open in the final topology on $C$ induced by the $\varphi_{\tau,\sigma_i}$. By \thref{final equals subspace}, this final topology on $C$ coincides with the subspace topology on $C \subset \overline{X_\tau}$, so this claim allows us to choose $U_\tau$ open in $\overline{X_\tau}$ so that $U_\tau \cap C = U$, hence $U_\tau \cap \varphi_{\tau,\sigma_i}(\overline{X_{\sigma_i}}) = \varphi_{\tau,\sigma_i}(U_{\sigma_i})$ and $U_\tau$ is compatible with our choices for $\Sigma(s-1)$. For the simplices $\sigma \subset st(\tau) \cap D(\xi) \setminus \Sigma(s-1)$ that we have not made choices for yet, we must set $U_\sigma = \varphi_{\tau,\sigma}^{-1}(U_\tau)$. Because the $\Sigma(s) \setminus \Sigma(s-1)$ is a disjoint union of parts of the $st(\omega_i)$, we can repeat this process for each $\omega_i$ without making contradicting choices. To check this collection of sets is aligned, suppose $\sigma \subset \sigma'$ are simplices of $\Sigma(s)$. If $\sigma \in \Sigma(s-1)$, then so is $\sigma'$ and by assumption our choices for $\Sigma(s-1)$ are aligned. If $\sigma \notin \Sigma(s-1)$, then $\sigma \subset st(\omega_i)$ for some $i$ and $U_\sigma = \varphi_{\omega_i,\sigma}^{-1}(U_{\omega_i})$. If $\sigma' \in \Sigma(s-1)$, then $U_{\sigma'} = \varphi_{\omega_i,\sigma'}^{-1}(U_{\omega_i})$ by the claim and the choice of $U_{\omega_i}$, and if $\sigma' \in \Sigma(s) \setminus \Sigma(s-1)$, then $U_{\sigma'} = \varphi_{\omega_i,\sigma'}^{-1}(U_{\omega_i})$ again, this time by definition. From here, we can repeat our reasoning for $\Sigma(0)$ and see that our choices are aligned. It remains to prove the claim.  
    
    \begin{proof}[Proof of Claim]
        Fix $\sigma_i$ and suppose $z \in U \cap \varphi_{\tau,\sigma_i}(\overline{X_{\sigma_i}})$. Since $U = \bigcup_j \varphi_{\tau,\sigma_j}(U_{\sigma_j})$, we can fix $j$ so that $z \in \varphi_{\tau,\sigma_j}(U_{\sigma_j})$. Because $\tau \subset \sigma_i\cap\sigma_j$, we can apply the Dichotomy \thref{upgrade to dichotomy}.

        If there is a simplex $\sigma'$ containing $\sigma_i \cup \sigma_j$, then we must have $\sigma' =\sigma_k$ for some $k$ and $\varphi_{\tau,\sigma_i}(\overline{X_{\sigma_i}}) \cap\varphi_{\tau,\sigma_j}(\overline{X_{\sigma_j}}) = \varphi_{\tau,\sigma_k}( \overline{X_{\sigma_k}})$, hence $z \in \varphi_{\tau,\sigma_k}(\overline{X_{\sigma_k}})$. We can pull $z$ back along $\varphi_{\tau,\sigma_j}$ to some point in $U_{\sigma_j}\cap \varphi_{\sigma_j,\sigma_k}(\overline{X_{\sigma_k}})$ which equals $\varphi_{\sigma_j,\sigma_k}(U_{\sigma_k})$ because the choices for $\Sigma(s-1)$ are aligned. This means 
        \[ z \in \varphi_{\tau,\sigma_k}(U_{\sigma_k}) = \varphi_{\tau,\sigma_i}\varphi_{\sigma_i,\sigma_k}(U_{\sigma_k}) \subset \varphi_{\tau,\sigma_i}(U_{\sigma_i}).\]
        This last containment again uses that our choices were aligned, and is exactly what we need.
        
        If we are in the second case of the dichotomy, there is no simplex containing $\sigma_i,\sigma_j$ but their images only intersect in $\partial G_\tau$, hence $z \in \varphi_{\tau,\sigma_i}(\partial G_{\sigma_i}) \cap \varphi_{\tau,\sigma_j}(U_{\sigma_j})$. 
        Choose points $x_i,x_j$ in the interior of $\sigma_i,\sigma_j$ so that $d(x_i,y_\tau), d(x_i,y_\tau) < 1$ but also $d(x_i,v_0),d(x_i,v_0) < t_{s-1}$. 
        The scaling assumption \ref{scaling assumption} and the definition of $\Sigma(s-1)$ implies $[x_j,x_i] \subset st(\tau) \cap \Sigma(s-1)$. Let $\sigma_j = \tau_0,\tau_1,\ldots ,\tau_q = \sigma_i$ be the path of simplices along $[x_j,x_i]$. Because $[x_j,x_i] \subset \Sigma(i-1)$, $U_{\tau_k}$ is already defined for each $k$, and becuase $[x_j,x_i] \subset st(\tau)$, we can compare the images of these sets in $\overline{X_\tau}$. Further, $z$ represents a point of $\partial_{Stab}G$, $D(z)$ is convex, and $x_j,x_i \in D(z)$, so this path of simplices is contained in $D(z)$. 
        
        We show that $z \in \varphi_{\tau,\tau_k}(U_{\tau_k})$ implies $z \in \varphi_{\tau,\tau_{k+1}}(U_{\tau_{k+1}})$. Since $z \in \varphi_{\tau,\tau_0}(U_{\tau_0})$ by assumption, this will imply $z \in \varphi_{\tau,\tau_{q}}(U_{\tau_q}) = \varphi_{\tau,\sigma_i}(U_{\sigma_i})$, completing the proof of the claim. There are two cases. If $\tau_{k+1} \subset \tau_{k}$, then this is immediate since $z \in \varphi_{\tau,\tau_k}(U_{\tau_k}) \subset \varphi_{\tau,\tau_{k+1}}(U_{\tau_{k+1}})$. If $\tau_k \subset \tau_{k+1}$, then $z \in \varphi_{\tau,\tau_k}(U_{\tau_k})$ by assumption, and $z \in \varphi_{\tau,\tau_{k+1}}(\partial G_{\tau_{k+1}})$ because $\tau_{k+1} \subset D(z)$, so $z$ is the image of a point $z' \in U_{\tau_k} \cap \varphi_{\tau_k,\tau_{k+1}}( \total{\tau_{k+1}}) \subset \total{\tau_k}$. Because our choices for $\Sigma(s-1)$ are aligned, this implies $z' \in \varphi_{\tau_k,\tau_{k+1}}(U_{\tau_{k+1}})$, so there is some $z'' \in U_{\tau_{k+1}}$ so that
        \[z = \varphi_{\tau,\tau_k}(z') = \varphi_{\tau,\tau_k}\varphi_{\tau_k,\tau_{k+1}}(z'') = \varphi_{\tau,\tau_{k+1}}(z'').\]
        Thus $z \in \varphi_{\tau,\tau_{k+1}}(U_{\tau_{k+1}})$, as desired. Now by assumption, $z \in \varphi_{\tau,\tau_0}(U_{\tau_0}) = \varphi_{\tau,\sigma_j}(U_{\sigma_j})$, so repeating this for $1,2,\ldots q$, we have $z \in \varphi_{\tau,\tau_q}(U_{\tau_q}) = \varphi_{\tau,\sigma_i}(U_{\sigma_i})$, completing the proof of the claim.
    \end{proof}
    With the claim proven, the lemma is complete.
\end{proof}

\begin{corollary}\thlabel{xi families exist}
    Let $\xi \in \partial_{Stab}G$ and for each $v \in V(\xi)$, suppose $U_v$ is a neighborhood of $\xi$ in $\total{v}$. Then there exists a $\xi$--family $\VV$ so that $V_v \subset U_v$ for each $v \in V(\xi)$.    
\end{corollary}
\begin{proof}
    If $\VV, \WW$ are $\xi$--families, then the sets $\{U_v \cap W_v, v \in V(\xi)\}$ are also a $\xi$--family. Indeed, if $e = [v,v']$ is an edge of $D(\xi)$, then 
    \[ \varphi^{-1}_{v,e}(U_v \cap W_v) = \varphi^{-1}_{v,e}(U_v) \cap \varphi^{-1}_{v,e}(W_v) = \varphi^{-1}_{v',e}(U_{v'}) \cap \varphi^{-1}_{v',e}(W_{v'}) = \varphi^{-1}_{v',e}(U_{v'}\cap W_{v'}),\]
    where the middle equality follows from $\VV,\WW$ being $\xi$--families. For each $v \in V(\xi)$, we can use the previous+ proposition to get a $\xi$--family extending $U_v$, then take a finite intersection of these $\xi$--families as above to get the desired $\xi$--family.
\end{proof}

\begin{notation}
To lighten notation, we will drop the maps $\varphi_{\sigma, \sigma'}$. Whenever $\sigma \subset \sigma'$, we will view $\total{\sigma'}$ as a subset of $\total{\sigma}$. Because $\varphi_{\sigma,\sigma''} = \varphi_{\sigma,\sigma'}\varphi_{\sigma',\sigma''}$ whenever $\sigma \subset \sigma' \subset \sigma''$, this is unambiguous. We will also identify an element of $\partial_{Stab}G$ with its representatives in each $\total{\sigma}$ for $\sigma \subset D(\xi)$. With this convention, \thref{xi family} can be restated as the following: If $\xi \in \partial_{Stab}G$, then a $\xi$--family is a choice of open neighborhoods $U_v \subset \total{v}$ of $\xi$ for each vertex $v \in V(\xi)$ so that for every edge $[v,v']$ of $D(\xi)$ and $x \in \total{[v,v']}$, we have 
\[x \in U_v \Longleftrightarrow x \in U_{v'}.\]
Equivalently, $U_v \cap \total{[v,v']} = U_{v'} \cap \total{[v,v']}$. 
\end{notation}

\begin{definition}\thlabel{defn:nesting of sets}
    Let $\xi \in \partial_{Stab}G$, $v$ a vertex of $D(\xi)$, and $U$ a neighborhood of $\xi$ in $\total{v}$. We say that a subneighborhood $V\subset U$ is \emph{nested} in $U$ if its closure is contained in $U$ and for every simplex $\sigma$ of $st(v)$ not contained in $D(\xi)$ we have

    \[\total{\sigma} \cap V \neq \varnothing \Longrightarrow \total{\sigma} \subset U.\]
\end{definition}

\begin{lemma}\label{Convergence Property remark}
    Let $\sigma, (\sigma_n)_n$ be simplices of $X$ so that $\sigma \subset \sigma_n$ for each $n$ and each $\sigma_n$ lies over the same simplex of $Y$. Then there is a subsequence, still denoted $(\sigma_n)_n$, and a point $\xi \in \partial G_\sigma$ so that $\total{\sigma_n}$ converges uniformly to $\xi$ in $\total{\sigma}$. 
\end{lemma}
\begin{proof}
    This is the Convergence Property \ref{Convergence Property} combined with the $r_\sigma$ of \thref{our spaces exist}. The $\sigma_n$ all correspond to cosets of the same subgroup of $G_\sigma$ because they all lie over the same simplex of $Y$, so we can write $\sigma_n = g_n\sigma'$ for some simplex $\sigma'\subset st(\sigma)$. Then $r_\sigma(\total{\sigma'})$ is the image of Lipschitz map between cusped spaces as in \thref{lipschitz}, and $r_\sigma(\total{\sigma_n}) = g_nr_\sigma(\total{\sigma'})$. Applying \thref{Convergence Property remark}, we receive a subsequence and a point $\xi \in \partial G_\sigma$ so that $g_nr_\sigma(\total{\sigma'}) \longrightarrow \xi$ uniformly. Since $r_\sigma$ moves points a finite distance in $\internal{\sigma}$, it follows that $g_nX_{\sigma_n} \longrightarrow \xi$ in $\total{\sigma}$.
\end{proof}

The following lemmas will not be used until Section \ref{section:Dynamics}, but we include them here because of their similarity to the previous lemma. 

\begin{lemma}\thlabel{nosubsequence3 remark}
    Suppose $\sigma \subset \sigma_1 \cap \sigma_2$ are simplices of $X$ and $(a_n)_n$ is a sequence in $G_\sigma$. Then there is a sequence $(k_n)_n$ in $G_{\sigma_1}$ and a subsequence of $(a_n)_n$ so that the sets $k_na_n\total{\sigma_2}$ are either constant or converge to a point in $\partial G_\sigma \setminus \partial G_{\sigma_1}$.
\end{lemma}
\begin{proof}
    As in the proof the Convergence Property \ref{Convergence Property remark}, this is simply \thref{NoSubsequence3} rephrased in our context using the maps $r_\sigma$ of \thref{our spaces exist}.
\end{proof}

\begin{lemma}\thlabel{Parabolics are almost cocompact remark}
    Let $v$ be a vertex of $X$, and let $P$ be a finite index subgroup of a maximal parabolic subgroup of $G_v$ fixing a point $\xi \in \partial G_v$. Then there exists a compact subset $K \subset \total{\sigma} \setminus \{\xi\} $ so that
    \begin{enumerate}
        \item $P\big(\partial G_v \cap K) = \partial G_v \setminus \{\xi\}$, that is, $\partial G_{v} \cap K$ is a course fundamental domain for $P$ acting on $\partial G_v \setminus \{\xi\}$, and 
        \item for any simplex $\sigma \subset st(v)$ with $G_\sigma$ finite, there is some $p \in P$ with $p\total{\sigma} \cap K = \varnothing$. 
    \end{enumerate}
\end{lemma}
\begin{proof}
    As with \thref{nosubsequence3 remark}, this follows from \thref{Parabolics are almost cocompact} and the maps $r_\sigma$ of \thref{our spaces exist}. For $\mathbb{F}$ we choose a set of representatives of the $G_v$ orbits of simplices in $st(v)$ with finite stabilizers. 
\end{proof}

\begin{lemma}(Nesting)\thlabel{Nesting}
    Let $\xi \in \partial_{Stab}G$, $v \in V(\xi)$, and $U$ a neighborhood of $\xi$ in $\total{v}$. Then there exists a subneighborhood $\xi \in V \subset U$ so that $V$ is nested in $U$.
\end{lemma}
\begin{proof}
    Because $\total{v}$ is metrizable, we can choose a countable basis $(V_n)_n$ of neighborhoods for $\xi \in \total{v}$. For a contradiction, suppose that no $V_n$ is nested in $U$, that is, for every $n$, there is a simplex $\sigma_n \subset st(v)\setminus D(\xi)$ so that $\total{\sigma} \cap V_n \neq \varnothing$ but also $\total{\sigma} \nsubseteq U$. If there were finitely many $\sigma_n$, after a subsequence we could assume the $\sigma_n = \sigma$ are constant. But then $\xi \notin \total{\sigma}$ and $\total{\sigma}$ is a closed subset of $\total{v}$, so we could choose $V_n$ small enough so that $\total{\sigma} \cap V_n = \varnothing$ which is a contradiction. Thus there must be infinitely many $\sigma_n$. After a subsequence, we can assume the sequence $(\sigma_n)_n$ is injective.  

    Because there are finitely many $G$ orbits of simplices, we can pass to a subsequence so all the $\sigma_n$ lie in the same $G_v$ orbit. The Convergence Property \ref{Convergence Property} then implies the $\total{{\sigma_n}}$ converge uniformly to a point of $\partial G_v$, say $\xi'$. If $\xi' \neq \xi$, then we could choose disjoint neighborhoods $V',V_n$ separating $\xi'$ and $\xi$ and take $n$ large enough so that $\total{{\sigma_n}} \subset V'$. But this contradicts that $\total{{\sigma_n}} \cap V_n \neq \varnothing$, so $\xi' = \xi$. But then we could choose $n$ large enough so that $\total{{\sigma_n}} \subset U$, contradicting the choice of $\sigma_n$. Thus no such sequence can exist.
\end{proof}

\begin{definition}\thlabel{defn:nested}
    Let $\xi \in \partial_{Stab}G$ and let $\UU,\UU'$ be two $\xi$--families. We say $\UU'$ is \emph{nested} in $\UU$ if for every vertex $v \in V(\xi)$, $U'_v$ is nested in $U_v$. Further, we say $\UU$ is \emph{$n$--nested} in $\UU$ if there exist $\xi$--families
    \[\UU' = \UU^0 \subset \UU^1 \subset \cdots \subset \UU^n = \UU\]
    so that $\UU^i$ is nested in $\UU^{i+1}$ for each $i = 0, \ldots n-1$.    
\end{definition}

\begin{lemma}\thlabel{nest as much as you want}
    Given $\xi \in \partial_{Stab}G$, a $\xi$-family $\UU$, and an integer $n > 0$, there exists a $\xi$--family $\UU'$ which is $n$--nested in $\UU$.
\end{lemma}
\begin{proof}
    To get a $\xi$--family which is nested in a given $\UU$, apply the Nesting \thref{Nesting} to each $v \in V(\xi), U_v \subset \total{v}$ to get a family of sets $\{W_v', v \in V(\xi)\}$ so that each $W_v'$ is nested in $U_v$. Apply \thref{xi families exist} to $\{W_v', v \in V(\xi)\}$ to get a $\xi$--family $\mathcal{W} = \{W_v, v \in V(\xi)\}$, and it follows that $\mathcal{W}$ is nested in $\UU$. To get a $\xi$--family which is $n$--nested in $\UU$, repeat this process $n$ times. 
\end{proof}

\subsection{The Crossing Lemma and Refined Families}

\begin{notation}
    For a fixed basepoint $v_0 \in X$ and any point $x \in \overline{X}$, let $\gamma_x$ denote the unit speed parameterization of $[v_0,x]$. Recall that for any $x \in X$, $\sigma_x$ is the unique simplex containing $x$ in its interior. 
\end{notation}

\begin{definition}
    Let $\xi \in \partial_{Stab}G$ and $\varepsilon \in (0,1)$. We use $D^\varepsilon(\xi)$ for the open $\varepsilon$--neighborhood of $D(\xi)$ in $X$, $N(\xi) = N(D(\xi))$ for the open simplicial neighborhood of $D(\xi)$, and 
    \[Lk(\xi) := Lk(D(\xi)) = N(D(\xi)) \setminus D(\xi)\]
    for the simplicial link of $D(\xi)$. By Assumption \ref{scaling assumption}, $D^\e(\xi)\subset N(\xi)$. See Figure \ref{fig:link example} for an example.
\end{definition}

\begin{figure}\label{fig:link example}
    \centering
    \begin{tikzpicture}[
                dot/.style = {circle, fill, minimum size=#1,
              inner sep=0pt, outer sep=0pt}]  
            \draw[ultra thick] (-1,0) node[dot = 7pt] {} -- (1,0) node[dot = 7pt] {};

            \begin{scope}[shift={(-1,0)}]
            \foreach \a in {120,140,...,250}{
            \draw(0,0) -- (\a:2) node[dot = 4pt]{};}
            \end{scope}

            \begin{scope}[shift={(1,0)}]
            \foreach \a in {290,310,...,430}{
            \draw(0,0) -- (\a:2) node[dot = 4pt]{};}
            \end{scope}

            \draw[dashed] (1,-1) arc (-90:90:1);
            \draw[dashed] (-1,1) arc (90:270:1);
            \draw[dashed] (-1,-1) -- (1,-1);
            \draw[dashed] (-1,1) -- (1,1);

            \node[label = above:$D(\xi)$] at (0,0){};
            \node[label = above:$D^{\varepsilon}(\xi)$] at (0,1){};
    \end{tikzpicture}

    \caption{In this example, $D(\xi)$ is the single bold edge, $N(\xi)$ is the bold edge together with the interior of each of the thinner edges, and $Lk(\xi)$ is only the interior of these thinner edges. The neighborhood $D^\e(\xi)$ is everything contained in the dotted region.}
    \label{fig:enter-label}
\end{figure}
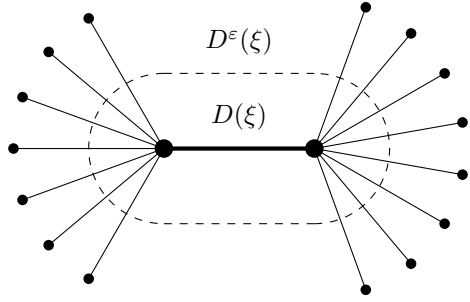

\begin{definition}
    We say a geodesic $\gamma$ in $X$ \emph{enters} $D^\varepsilon(\xi)$ if $\gamma(t_0) \in D^\varepsilon(\xi)$ for some $t_0$. We say $\gamma$ \emph{goes through} $D^\varepsilon(\xi)$ if $\gamma$ enters $D^\e(\xi)$ and there is some $t_1 > t_0$ so that for all $t \geq t_1$, $\gamma(t) \notin D^\varepsilon(\xi)$.
\end{definition}

\begin{definition}[Exit Simplex]
    Fix a basepoint $v_0 \in X$ and let $x \in \overline{X}$, $\xi \in \partial_{Stab}G$, and $\varepsilon \in (0,1)$. If $\gamma_x$ enters $D^\e(\xi)$, then the \emph{exit simplex} of $x$, denoted $\sigma_{\xi,\e}(x)$ is the last simplex met by $\gamma_x$ in $D^\e(\xi)$. More explicitly, if $x \in D^\e(\xi)$, then $\sigma_{\xi,\e}(x) = \sigma_x$, and if $x \notin D^\e(\xi)$, then $\sigma_{\xi,\e}(x)$ is the last simplex of $Lk(\xi)$ met by $\gamma_x$.
\end{definition}

\begin{definition}[Cones]\thlabel{defn:cones}
    Let $\xi \in \partial_{Stab}G$, let $\mathcal{U}$ be a $\xi$--family, and let $\varepsilon\in (0,1)$. Define $\tCone_{\mathcal{U},\varepsilon}(\xi)$ as the set of points $x \in \overline{X}\setminus D(\xi)$ so that 
    \begin{enumerate}
        \item the geodesic $\gamma_x$ from the base point $v_0$ to $x$ enters $D^\e(\xi)$,
        \item for every vertex $v \in \sigma_{\xi,\e}(x) \cap V(\xi)$, $\total{\sigma_{\xi,\e}(x)} \subset U_v$ in $\total{v}$.
    \end{enumerate}
    Further, let $\Cone_{\UU,\e}(\xi)$ be the points of $\tCone_{\xi,\e}(\xi)$ so that $\gamma_x$ \emph{goes through} $D^\e(\xi)$. Equivalently, $\Cone_{\UU,\e}(\xi) = \{x \in \tCone_{\UU,\e}(\xi)  \; | \; d(x,D(\xi)) \geq \e \}$.
\end{definition}

\begin{definition}\thlabel{defn:N_UU}
    For $\xi \in \partial_{Stab}G$ and a $\xi$--family $\UU$, let $N_\UU(\xi)$ be the union of simplices $\sigma \subset Lk(\xi)$ such that for some vertex $v$ of $\sigma \cap D(\xi)$, we have $\total{\sigma} \cap U_v \neq \varnothing$ in $\total{v}$.
\end{definition}

\begin{figure}
        \centering
        \begin{tikzpicture}[
                dot/.style = {circle, fill, minimum size=#1,
              inner sep=0pt, outer sep=0pt}]  
        \filldraw[fill=lightgray, draw=black] (1,1) -- (-1,1) -- (-1,-1) -- (1,-1) -- cycle;
        \draw[dashed] (2,2) -- (-2,2) -- (-2,-2) -- (-2,-2) -- (2,-2) -- cycle;
        \draw (3,1) -- (1,1) -- (1,-1) -- (3,-1) -- cycle;
        
        \node[dot = 3pt, label =below:$v_0$] at (-4,0) {};
        \node[dot = 3pt, label =below:$x$] at (4,0) {};
        \node[dot = 3pt, label =south east:$v_1$] at (1,-1) {};
        \node[dot = 3pt, label =south east:$v_2$] at (1,1) {};
        \node[dot = 3pt] at (2,0) {};
        \node[dot = 3pt] at (0,1.75) {};
        \node[dot = 3pt] at (2.5,1.6) {};
        \node[dot = 3pt] at (-2.75,1.75) {};
        
        \node at (0,-.5) {$D(\xi)$};
        \node at (2,-.5) {$\sigma_{\xi,\e}(x)$};

        \draw (-4,0) -- (-2.75,1.75);
        \draw (-4,0) -- (0,1.75);
        \draw (-4,0) -- (2.5,1.6);
        \draw (-4,0) -- (4,0);
        
        \end{tikzpicture}
        \caption{This figure shows various ways geodesics from a basepoint $v_0$ can interact with $D^\e(\xi)$ for some $\xi \in \partial_{Stab}G$. For example they can not meet it at all, they can enter $D^\e(\xi)$ but avoid $D(\xi)$, meet $D(\xi)$ and go through, etc. For the point $x$, we mark where $\gamma_x$ exits $D^\e(\xi)$, and the simplex containing this point is $\sigma_{\xi,\e}(x)$.}
        \label{fig:Cones}
\end{figure}
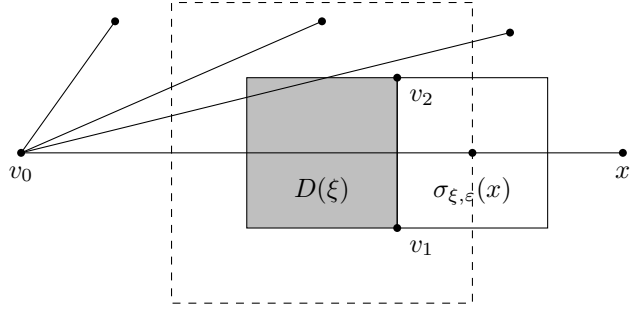

\begin{proposition}[Consistency]\thlabel{Consistency}
    Suppose $\xi \in \partial_{Stab}G$, $\UU$ is a $\xi$--family, $\sigma \subset Lk(\xi)$ is a simplex. If there is some vertex $v \in \sigma \cap V(\xi)$ so that $\total{\sigma} \subset U_v$, then for every vertex $w \in \sigma \cap V(\xi)$, we have $\total{\sigma} \subset U_w$.  
\end{proposition}
\begin{proof}
    Suppose $v \in \sigma \cap V(\xi)$ is as in the assumption and $w \in \sigma \cap V(\xi)$ is some other vertex. Because closed simplices are convex, the geodesic $[v,w]$ is contained in $\overline{\sigma}$, hence there is some path in the 1--skeleton of $\sigma \cap D(\xi)$ from $v$ to $w$, say $v= v_0, v_1 ,\ldots ,v_n = w$. We induct on $n$.
    
    If $n =1$, $v$ and $w$ are connected by an edge $e = [v,w]$. In $\total{v}$, the assumption is that $\total{\sigma} \subset \total{e} \cap U_v$, and the definition of $\xi$--family implies $\total{v} \subset \total{e} \cap U_w$ in $\total{w}$. If $v,w$ are joined by a path of length $n+1$, say $v = v_0, \ldots v_n, v_{n+1} = w$. Letting $e = [v_n,v_{n+1}]$ be the final edge of this path, we repeat the argument of the base case --- $\total{\sigma} \subset \total{e} \cap U_{v_n}$ in $\total{v_n}$ by induction, so the definition of a $\xi$--family implies $\total{\sigma}\subset \total{e}\cap U_{v_{n+1}}$ as desired.  
\end{proof}

\begin{corollary}
    Let $\xi \in \partial_{Stab}G$, $\UU$ be a $\xi$--family, and $\e \in (0,1)$. If $x\in \overline{X}$, $\gamma_x$ enters $D^\e(\xi)$, and for some vertex $v \in \sigma_{\xi,\e}(x)$, $\total{\sigma_{\xi,\e}(x)} \subset U_v$ in $\total{v}$, then $x \in \tCone_{\UU,\e}(\xi)$. In other words, the 'for all' in \thref{defn:cones} (2) can be replaced with `for some' without changing the definition. 
\end{corollary}
\begin{proof}
    If $x$ is as in the assumption, then \thref{Consistency} implies $\total{\sigma_{\xi,\e}(x)} \subset U_w$ for \emph{every} vertex $w \in \sigma_{\xi,\e}(x) \cap V(\xi)$, which is exactly the definition of $x \in \tCone_{\UU,\e}(\xi)$. 
\end{proof}

Fix a vertex $v_0$ of $X$ as a base point for the remainder of this section.
\begin{lemma}[Crossing Lemma]
\thlabel{Crossing}
Let $\xi \in \partial_{Stab}G,$ and let $\mathcal{U},\mathcal{U}'$ be two $\xi$--families, with $\mathcal{U}'$ $n$--nested in $\mathcal{U}$. Let $\sigma_1,\ldots,\sigma_n$ be a path of simplices in $Lk(\xi)$. Suppose $\sigma_1 \subset  N_{\mathcal{U}'}(\xi)$. Then for any $k$ and any vertex $v \in \sigma_k \cap D(\xi)$, we have $\total{\sigma_k} \subset U_v$ in $\total{v}$.
\end{lemma}
\begin{proof}
    
We induct on $n$. For $n=1$, this is a combination of \thref{defn:nesting of sets} and \thref{defn:N_UU}. Explicitly, if $\sigma_1 \subset N_{\UU'}(\xi)$, then by \thref{defn:N_UU}, $\total{\sigma_1} \cap U'_v \neq \varnothing$ for some vertex $v$ of $\sigma_1 \cap V(\xi)$. Because $U'_v$ is nested in $U_v$ in the sense of \thref{defn:nesting of sets}, this implies $\total{\sigma_1} \subset U_v$. By \thref{Consistency}, we have $\total{\sigma} \subset U_w$ for any other vertex $w$ of $\sigma \cap D(\xi)$.

Suppose the result is true for $k$. Let $\sigma_1,\ldots, \sigma_{k+1}$ be a path of simplices in $Lk(\xi)$ and $\UU' = \UU^0\subset \cdots \subset \UU^{k+1} = \UU$. By induction, the result is true for the path $\sigma_1,\ldots ,\sigma_k$ and the nesting families $\UU^0 \subset \cdots \subset \UU^k$, so the only inclusion to prove is the last one. 

If $\sigma_k \subset \sigma_{k+1}$, then every vertex $v$ of $\sigma_k \cap D(\xi)$ is also a vertex of $\sigma_{n+1}$, and in $\total{v}$ we have
\[\total{\sigma_{k+1}} \subset \total{\sigma_k} \subset U^k_v \subset U^{k+1}_v.\]
Since this holds for the vertices of $\sigma_k \cap D(\xi)$, it holds for all vertices of $\sigma_{k+1}\cap D(\xi)$ by \thref{Consistency}.

Now suppose $\sigma_{k+1} \subset \sigma_k$. If $v$ is a vertex of $\sigma_{k+1}\cap D(\xi)$ then it is also a vertex of $\sigma_k$, and by induction we have $\total{\sigma_k} \subset U_v^k$, so
\[\total{\sigma_k}\subset \total{\sigma_{k+1}} \cap U_v^k \neq \varnothing.\]
The definition of nesting implies $\total{\sigma_{k+1}} \subset U^{k+1}_v$.
\end{proof}

Recall the definition of $d_{max}$ from \thref{Finiteness}. 

\begin{definition}[Refined Families]
    Fix $n \geq 0$. Given two $\xi$--families $\UU, \, \UU'$, we say $\UU'$ is \emph{$n$--refined} in $\UU$ if it is $F(\max(n,d_{max}))$--nested in $\UU$, where $F$ is the function described in \thref{Short Paths of Simplices}. 
\end{definition}

Given any $\xi \in \partial_{Stab}G$, $\xi$--family $\UU$, and $n \in \mathbb{N}$, there always exists a $\xi$-- family $\UU'$ which is $n$--refined in $\UU$ by \thref{nest as much as you want}. When applying \thref{Short Paths of Simplices}, we count open simplices. For example, a single triangle is the union of $3$ vertices, $3$ edges, and one $2$--cell, for a total of $7$ simplices. With this convention and point (2) of \thref{Finiteness}, when we apply \thref{Short Paths of Simplices} with $K = D(\xi)$ for some $\xi \in \partial_{Stab}G$ and $K'$ a single closed simplex, we get a path of simplices of length at most $F(d_{max})$.

We chose a vertex as a basepoint for convenience, but all the proofs in this subsection would work with any basepoint $x_0 \in X$ if we replace $v_0$ with $\sigma_{x_0}$ when applying \thref{Short Paths of Simplices}.

\begin{definition}
    For $\xi \in \partial_{Stab}G$, $\UU$ a $\xi$--family, define the \emph{pseudocone} $\pCone_{\UU}(\xi)$ to be the set of points $x \in \overline{X} \setminus D(\xi)$ so that there is some simplex $\sigma \subset N_\UU(\xi) \subset Lk(\xi)$ which meets $\Geod(x,D(\xi))$. We call any such $\sigma$ a \emph{witness} for $x$.
\end{definition}

\begin{lemma}\thlabel{Shadow1}
    For any $\xi \in \partial_{Stab}G$ and basepoint $v_0$, there is a $\xi$--family $\VV$ so that $v_0 \notin \pCone_\VV(\xi)$. Further, if $\xi \notin \partial G_{v_0}$, then $st(v_0) \cap \pCone_\VV(\xi) = \varnothing$.
\end{lemma}
\begin{proof}
    If $v_0 \in D(\xi)$, then $v_0$ cannot be in $\pCone_{\UU}(\xi)$ for any $\xi$--family $\UU$ because $\pCone_{\UU}(\xi) \subset X \setminus D(\xi)$ by definition and $\xi \in \partial G_{v_0}$, so we don't need to consider the further statement. 
    Therefore we can assume $v_0 \notin D(\xi)$ so that $\Geod(v_0,D(\xi)) \cap Lk(\xi) \neq \varnothing$. Choose some $\sigma \subset Lk(\xi)$ which meets $\Geod(v_0,D(\xi))$ and fix a vertex $v \in \sigma \cap V(\xi)$. Because $\total{\sigma}$ is a closed subset of $\total{v}$ missing $\xi$, we can choose a neighborhood $U_v \subset \total{v}$ of $\xi$ so that $U_v \cap \total{\sigma} = \varnothing$. For every other vertex $w \in V(\xi)$, let $U_w = \total{w}$. Apply \thref{xi families exist} to the sets $\{U_w, w \in V(\xi)\}$ to get a $\xi$--family $\VV'$, and let $\VV$ be a $\xi$--family $d_{max}$--refined in $\VV'$. 

    For a contradiction, suppose $\sigma'$ is a witness for $v_0 \in \pCone_{\VV}(\xi)$. Applying \thref{Short Paths of Simplices} to the connected complex $v_0$, the convex complex $D(\xi)$, and the simplices $\sigma', \sigma$, we receive a path of simplices of length at most $F(d_{max})$ from $\sigma'$ to $\sigma$. Because $\sigma' \subset N_\VV(\xi)$ and $\VV$ is $d_{max}$--refined, hence $F(d_{max})$--nested, in the family of sets $\{U_w,w \in V(\xi)\}$, the Crossing \thref{Crossing} applied to this path implies $\total{\sigma} \subset U_v$. But $U_v$ was chosen to avoid $\total{\sigma}$, so this is impossible.

    For the further statement, notice that $\xi \notin \partial G_{v_0}$ implies $D(\xi) \cap st(v_0) = \varnothing$ (here it matters that we are using the \emph{open} star). Suppose $\tau \subset st(v_0)$ has $y \in \tau \cap \pCone_{\VV}(\xi)$ with $\sigma'$ as a witness. We apply \thref{Short Paths of Simplices} to the connected complex $\tau$ with $v_0,y$ as the points, the convex complex $D(\xi)$, and the simplices $\sigma', \sigma$. Since $\tau$ has at most $d_{max}$ simplices too, the path is still of length at most $F(d_{max})$, and the Crossing \thref{Crossing} gives a contradiction as before.
\end{proof}

\begin{lemma}\thlabel{Shadow2}
    Let $\xi \in \partial_{Stab}G$ with $\VV$ a $\xi$--family satisfying the conclusions of \thref{Shadow1}. Let $\UU \subset \UU' \subset \VV$ be $\xi$--families, with $\UU$ $d_{max}$-refined in $\UU'$ and $\UU'$ $[d_{max} + F(d_{max})]$--nested in $\VV$. If $x \in \pCone_{\UU}(\xi)$, then $[v_0,x]$ meets $D(\xi)$ and $x \in \tCone_{\VV,\e}(\xi)$ for all $\e \in (0,1)$. If $\xi \notin \partial G_{v_0}$, then $[y,x]$ meets $D(\xi)$ for each $y \in st(v_0)$. 
\end{lemma}
\begin{proof}
    Suppose $x \in \pCone_{\UU}(\xi)$ and let $\sigma$ be a witness. 
\begin{claim}

    $\Geod(v_0,D(\xi)) \,\cap\, \Geod(x,D(\xi)\big) = D(\xi) $, and if $\xi \notin \partial G_{v_0}$, then $\Geod(v_0,D(\xi)) \cap \Geod(st(v_0),D(\xi)) = D(\xi)$.
\end{claim}

\begin{proof}[Proof of Claim]
Clearly $D(\xi)$ is contained in the intersection, so the claim is that $D(\xi)$ is the entire intersection. For a contradiction, suppose $y \in \Geod(x,D(\xi)) \cap \Geod(v_0,D(\xi)) \setminus D(\xi)$. Then there are points $z', z'' \in D(\xi)$ so that 
    \[y \in [x,z'] \cap [v_0,z''] \quad \quad (\star)\]
and simplices $\sigma', \sigma''$ of $Lk(\xi)$ so that $[x,z']$ meets $\sigma'$ and $[v_0,z'']$ meets $\sigma''$. We apply \thref{Short Paths of Simplices} twice, first to $\sigma_x, D(\xi)$ to get a path of simplices from $\sigma$ to $\sigma'$, then to $\sigma_y, D(\xi)$ to get a path of simplices from $\sigma'$ to $\sigma''$. Each of these paths has length at most $F(d_{max})$. By assumption, $\sigma \subset N_{\UU}(\xi)$, so the Crossing \thref{Crossing} applied to each of these paths implies $\total{\sigma''} \subset V_w$ for any vertex $w$ of $\sigma'' \cap V(\xi)$. But then $\sigma''$ would be a witness for $v_0 \in \pCone_{\VV}(\xi)$, contradicting the choice of $\VV$. Thus no such $y$ can exist.

If $\xi \notin \partial G_{v_0}$ and $\tau \subset st(v_0)$, then we can replace $v_0$ with some $z \in \tau$ in $(\star)$, and find $\sigma', \sigma''$ as before. Again we get a path of simplices of length at most $2F(d_{max})$ from $\sigma$ to $\sigma''$, and again combining this with the Crossing \thref{Crossing} contradicts the definition of $\VV$. This proves the claim.
\end{proof}

Let $K := \Geod(x,D(\xi)) \cup \Geod(v_0,D(\xi))$ and let $[v_0,x]_K$ be the shortest path from $v_0$ to $x$ in $K$. 

\begin{claim}
    $[v_0,x]_K =[v_0,x]$ and $[v_0,x]$ meets $D(\xi)$. If $\xi \notin \partial_{Stab}G$, then $[y,x]_K = [y,x]$ for all $y \in st(v_0)$ and $[y,x]$ meets $D(\xi)$.
\end{claim}
\begin{proof}[Proof of Claim]
    The path $[v_0,x]_K$ connects a point of $\Geod(v_0,D(\xi))$ to a point of $\Geod(x,D(\xi))$, so it must go through the intersection of these two sets. By the previous claim, that intersection is $D(\xi)$, so $[v_0,x]_K$ meets $D(\xi)$. Because $D(\xi)$ is convex, $[v_0,x]_K$ meets $D(\xi)$ in a connected segment, so we can write $[v_0,x]_K$ as a concatenation $[v_0,v_1] \cup [v_1,v_2] \cup [v_2,x]$ for some points $v_1,v_2 \in D(\xi)$. Since $X$ is $\textrm{CAT}(0)$, local geodesics are geodesics (see \cite{BH}[II.1.4]), so if we can show $[v_0,x]_K$ is a local geodesic, we will be done. Since $[v_0,v_2]$ and $[v_1,x]$ are both geodesics in $X$, we only need to deal with the case $v_1 = v_2$.

    If $[v_0,x]_K$ is not a local geodesic at $v_1 = v_2$, then for all $\varepsilon > 0$, there are points $a \in [v_0,v_1]$ and $b \in [v_1,x]$ so that $[a,b]$ is strictly shorter than the concatenation $[a,v_1]\cup [v_1,b]$. Then $[a,b]$ cannot meet $D(\xi)$ -- if it did we could combine paths from $v_0,x$ to this meeting point and get a path shorter than $[v_0,x]_K$. On the other hand, $D(\xi)$ is a convex subset of a $\textrm{CAT}(0)$ space, so $[a,b]$ stays within $\varepsilon$ of $D(\xi)$. The definition of $d_{max}$ then gives us a path of simplices of length at most $d_{max}$ from $\sigma_b$ to $\sigma_a$ along $[a,b]$. 

    Using \thref{Short Paths of Simplices} with $\sigma_x$, $D(\xi)$ and the simplices $\sigma_b$ and the witness $\sigma$, we get a path of simplices of length at most $F(d_{max})$ from $\sigma$ to $\sigma_b$. Combining this with the path from $\sigma_b$ to $\sigma_a$, the Crossing \thref{Crossing} implies that $\total{\sigma_a} \subset U_w \subset V_w$ for any vertex $w \in \sigma_a \cap D(\xi)$. Then $\sigma_a$ is a witness for $v_0 \in \pCone_{\VV}(\xi)$, which again contradicts the choice of $\VV$. This contradiction implies $[v_0,x]_K$ is a local geodesic at $v_1=v_2$, so $[v_0,x]_K = [v_0,x]$.

    To prove the second portion, we can replace $v_0$ with $y$ and argue in exactly the same way, completing the proof of the claim.
\end{proof}

For any $\e \in (0,1)$, we apply \thref{Short Paths of Simplices} to $\sigma_x,D(\xi)$ with the witness $\sigma$ and the exit simplex $\sigma_{\xi,\e}(x)$, which exists since $[v_0,x]$ goes through $D(\xi)$ by the previous claim. We get a path of simplices of length at most $F(d_{max})$ from $\sigma$ to $\sigma_{\xi,\e}(x)$. Since $\sigma$ is a witness and $\UU$ is $d_{max}$--refined in $\UU'$, the Crossing \thref{Crossing} implies $\total{\sigma_{\xi,\e}(x) }\subset U'_v$ for any $v \in \sigma_{\xi,\e}(x) \cap D(\xi)$, which is the definition of $x \in \tCone_{\UU',\e}(\xi) \subset \tCone_{\VV,\e}(\xi)$. This completes the proof of \thref{Shadow2}.
\end{proof}

\begin{assumption}\thlabel{shadow assumption}
    When a basepoint $v_0$ is specified, we will assume that all $\xi$--families are contained a $\xi$--family $\UU_\xi$ satisfying the assumptions of $\UU$ in \thref{Shadow2}. Hence the conclusion of \thref{Shadow2} applies to every $x \in \pCone_{\UU}(\xi)$. We will call such $\xi$--families \emph{admissible with} $v_0$ \emph{as basepoint}.  
\end{assumption}

For a given $\xi \in \partial_{Stab}G$, $\xi$--family $\UU$, and $\e \in (0,1)$, we imagine shining a cone of light from $v_0$ onto $D(\xi)$. Roughly speaking, $\tCone_{\UU,\e}(\xi)$ is the shadow cast by $D(\xi)$ and $\UU$ from this light. The previous assumption allows the following proposition to work, so that this shadow is a genuine shadow in the sense that everything in $\tCone_{\UU,\e}(\xi)$ truly lies behind $D(\xi)$ from the point of view of $v_0$.

\begin{proposition}[Genuine Shadow]\thlabel{Genuine Shadows} Let $\xi \in \partial_{Stab}G$, $\UU$ a $\xi$--family admissible with $v_0$ as a basepoint, and $\e \in (0,1)$. If $x \in \tCone_{\UU,\e}(\xi)$, then $[v_0,x]$ meets $D(\xi)$, and if $\xi \notin \partial_{Stab}G$, then $[y,x]$ meets $D(\xi)$ for every $y \in st(v_0)$.
\end{proposition}
\begin{proof}
    By the definition of cones \ref{defn:cones}, $[v_0,x]$ meets $D^\e(\xi)$, and for any vertex $v \in \sigma_{\xi,\e}(x) \cap V(\xi)$, we have $\total{\sigma_{\xi,\e}}(x) \subset U_v$. This exit simplex may not be a witness for $x$ because it may not meet $\Geod(x,D(\xi))$. However, we can choose $y \in [v_0,x] \cap \sigma_{\xi,\e}(x)$, so that $ \sigma_y = \sigma_{\xi,\e}(x)$. Clearly $\sigma_y$ is a simplex of $Lk(\xi)$ which meets $\Geod(y,D(\xi))$, so $\sigma_y$ is a witness for $y \in \pCone_{\UU}(\xi)$. By Assumption \ref{shadow assumption}, $\UU$ is contained in a $\xi$--family $U_\xi$ satisfying the assumptions of $\UU$ in \thref{Shadow2}. Since $\UU$ is contained in $\UU_\xi$, $\sigma_y$ is also a witness for $y \in \pCone_{\UU_\xi}(\xi)$, so the conclusion of \thref{Shadow2} states that $[v_0,y]$, hence $[v_0,x]$, meets $D(\xi)$. The case where $\xi\notin \partial_{Stab}G$ is identical. 
\end{proof}

\subsection{The Refinement and Star Lemmas}

The next lemma tells us that if we have $\xi$ families $\UU' \subset \UU$ with $\UU'$ appropriately refined in $\UU$, we can move away from a point in the refined cone while staying in the larger cone. 

\begin{lemma}[Refinement Lemma] \thlabel{Refinement Lemma} 
Let $\xi \in \partial_{Stab}G$ and let $\UU' \subset \UU$ be $\xi$--families with $\UU'$ $n$--refined in $\UU$. Let $\sigma_1,\ldots \sigma_n$ be a path of simplices in $X \setminus D(\xi)$. For any $\varepsilon \in (0,1)$, if there is some $x \in \sigma_1$ so that $[v_0,x_1]$ enters $D^\varepsilon(\xi)$ and $\sigma_{\xi,\varepsilon}(x_1) \subset N_{\UU'}(\xi)$, then 
\[\bigcup_{i=1}^n \sigma_i \subset \fitTilde{\Cone}_{\UU,\varepsilon}(\xi).\]
\end{lemma}

\begin{proof}
    Let $y \in \bigcup_{i=1}^n \sigma_i$. Because $\sigma_{\xi,\varepsilon}(x_1) \subset N_{\UU'}(\xi)$ and $\UU'$ is refined in $\UU$, we have $x \in \fitTilde{\Cone}_{\UU,\varepsilon}(\xi)$, thus $[v_0,x_1]$ meets $D(\xi)$ by the Genuine Shadow \thref{Genuine Shadows}. Let $v$ be a vertex of $\sigma_{\xi,\varepsilon}(x_1) \cap D(\xi)$, let $\sigma$ be a simplex of $Lk(\xi)$ met by $[v,y]$, and let $w$ be a vertex of $\sigma \cap D(\xi)$. 
    
    Note that $x_1$ can be connected to $y$ outside of $D(\xi)$ precisely because the path of simplices avoids $D(\xi)$. Therefore we can apply Lemma \ref{Short Paths of Simplices} to the connected subcomplex $\bigcup_{i=1}^n \sigma_i$ and the convex subcomplex $D(\xi)$, using the portion of $[v_0,x_1]$ between $D(\xi)$ and $x_1$ and $[w,y]$. We receive a path of simplices from $\sigma_{\xi,\varepsilon}(x_1)$ to $\sigma$ of length at most $F(\max(n,d_{max}))$. Because $\UU'$ is $n$--refined in $\UU$ and $\sigma_{\xi,\varepsilon}(x_1)\subset N_{\UU'}(\xi)$, the Crossing Lemma \ref{Crossing} implies $\total{\sigma} \subset U_w$. Thus $\sigma$ is a witness for $y \in \pCone_{\UU}(\xi)$. By Assumption \ref{shadow assumption}, $\UU$ is contained in $\xi$--family $\UU_\xi$ satisfying the assumptions of $\UU$ in \thref{Shadow2}. Because $\UU$ is contained in $\UU_{\xi}$, we have $\pCone_{\UU}(\xi) \subset \pCone_{\UU_{\xi}}(\xi)$, and in particular $y \in \pCone_{\UU_{\xi}}(\xi)$. The conclusion of \thref{Shadow2} then applies to $y$, implying $[v_0,y]$ meets $D(\xi)$. This means $\sigma_{\xi,\e}(y)$ exists.
    
    Apply \thref{Short Paths of Simplices} again, this time to the portions of $[v_0,x]$, $[v_0,y]$ in between $D(\xi)$ and $x,y$, and the simplices $\sigma_{\xi,\e}(x),\sigma_{\xi,\e}(y)$. We receive a path of simplices of length at most $F(\max(d_{max},n))$, and the Crossing \thref{Crossing} applied to this path implies the result.
\end{proof}

If $x,y,z$ are points of a $\mathrm{CAT}(0)$ space and $y,z$ are very close, then the geodesics $[x,y]$ and $[x,z]$ are very close. The following lemma uses this property to make the exit simplex of $y$ close to the exit simplex of $z$, where closeness is measured by being in the open star. 

\begin{lemma}[Star Lemma]\thlabel{Star}  Let $\xi \in \partial_{Stab}G$, $\varepsilon \in (0,1)$, and let $x \in X \setminus D^{\varepsilon}(\xi)$ be such that $[v_0,x]$ goes through $D^{\varepsilon}(\xi)$. Then there exists a $\delta$ such that for every $y \in B(x,\delta) \setminus D^\varepsilon(\xi)$, $[v_0,y]$ goes through $D^\varepsilon(\xi)$ and
\[\sigma_{\xi,\varepsilon}(y) \subset st(\sigma_{\xi,\varepsilon}(x)).\]
\end{lemma}
\begin{proof}
    Let $T = d(v_0,x)$ and let $\gamma_x:[0,T]\longrightarrow X$ parameterize the geodesic segment $[v_0,x]$. By assumption, $\gamma_x$ goes through $D^\varepsilon(\xi)$, so there is some $t_0$ with $\gamma_x(t_0) \in \partial D^\varepsilon(\xi)$, as in $d(\gamma_x(t_0),D(\xi)) = \varepsilon$ and $\sigma_{\xi,\varepsilon}(x) = \sigma_{\gamma_x(t_0)}$.
    
    Since $X$ is $\mathrm{CAT}(0)$ and $D(\xi)$ is convex, the function $t \longrightarrow d(\gamma_x(t),D(\xi))$ is convex. Because $\gamma_x$ goes through $D^\varepsilon(\xi)$, this convexity implies there is some interval ending in $t_0$ which $\gamma_x$ takes into $D^\varepsilon(\xi)$. Combining this with the fact that $st(\sigma_{\xi,\varepsilon}(x))$ is an open set, there exists some $r >0$ so that 
    
    \[\gamma_x([t_0-r,t_0)) \subset D^\varepsilon(\xi) \cap st(\sigma_{\xi,\varepsilon}(x)).\]

    Again because stars are open, we can choose $\tau$ so that the $\tau$ neighborhood of $\gamma_x([t_0-r,t_0])$ is contained in $\sigma_{\xi,\varepsilon}(x)$. Since $\gamma_x([t_0-r,t_0])$ is also convex, the $\tau$ neighborhood of it is also convex. To see this, consider two points $y_1,y_2$ in this $\tau$ neighborhood and without loss of generality assume $y_1$ is further from $\gamma_x([t_0-r,t_0])$ than $y_2$ is. Then convexity implies all the points along $[y_1,y_2]$ are closer to $\gamma_x([t_0-r,t_0])$ than $y_1$, hence closer than $\tau$. Let 
    \[k = \varepsilon- d(\gamma_x(t_0-r),D(\xi)) > 0.\]

    We separate into $2$ cases for $d(x,D(\xi)) =\e$ and $d(x,D(\xi)) >\e$.

    First, if $d(x,D(\xi)) = \e$, then using the notation above we have $\gamma(t_0) = x$ and  $T = t_0$. Set $\delta = \frac{1}{10}\min(k,\tau,r)$. To see this $\delta$ works, choose $y \in B(x,\delta)\setminus D^\varepsilon(\xi)$ and let $\gamma_y$ parameterize $[v_0,y]$. By the triangle inequality and $d(y,x) < \delta < r$, we see that $d(v_0,y) \geq t_0-r$, so we can consider $\gamma_y(t_0-r)$. By the $\mathrm{CAT}(0)$ inequality, we know $d(\gamma_x(t_0-r),\gamma_y([t_0-r,t_0])) <\delta$, thus

    \begin{align*}
        d(D(\xi),\gamma_y(t_0-r)) & \leq d(D(\xi),\gamma_x(t_0-r)) + d(\gamma_x(t_0-r),\gamma_y(t_0-r)) \quad \quad (\star)\\
        & \leq \varepsilon \underbrace{-k + \delta}_{<0} \leq \varepsilon
    \end{align*}

    So $\gamma_y(t_0-r) \in D^\varepsilon(\xi)$, and $[v_0,y]$ goes through $D^\varepsilon(\xi)$ since $y \notin D^\varepsilon(\xi)$. Further, $d(y,\gamma_x(t_0)) = d(y,x) < \delta < \tau$, so $\gamma_y(t_0-r)$ and $y$ are in the $\tau$ neighborhood of $\gamma_x([t_0-r,t_0])$. The definition of $\tau$ then implies this terminal portion of $[v_0,y]$ is contained in $st(\sigma_{\xi,\varepsilon}(x))$, and hence so is the simplex $\sigma_{\xi,\varepsilon}(y)$. This completes the proof in the first case.

    Now suppose $d(x,D(\xi)) > \e$. See Figure \ref{fig:Star} for a diagram of what we will set up. Since $T>t_0$, we can choose some $0 < s < \min(\frac{1}{10}(T-t_0),\frac{1}{2}\tau)$ and set 
    
    \[\delta = \frac{1}{10}\min\Big(k,r,s, d(\gamma_x(t_0+s),D^\varepsilon(\xi))\Big) > 0.\] 
    
    To see $\delta$ works, choose $y \in B(x,\delta) \setminus D^\varepsilon(\xi)$. By the triangle inequality and the definition of $s$, we have
    
    \[d(v_0,y) \geq d(v_0,x) - d(x,y) > T - s > t_0 + \frac{9}{10}(T-t_0).\]
    
    In particular $d(v_0,y) > t_0 + s > t_0-r$, so the exact argument in $(\star)$ above implies $[v_0,y]$ goes through $D^\varepsilon(\xi)$ and $d(\gamma_y(t_0-r),\gamma_x(t_0-r)) < \delta < \tau$. By the $\textrm{CAT}(0)$ inequality, $d(\gamma_y(t_0+s),\gamma_y(t_0+s)) < \delta$, but because $\delta < d(\gamma_x(t_0+s),D^\varepsilon(\xi))$, the ball $B(\gamma_x(t_0+s),\delta)$ is disjoint from $D^\varepsilon(\xi)$, so $\gamma_y(t_0+s) \notin D^\varepsilon(\xi)$. Therefore $[v_0,y]$ leaves $D^\e(\xi)$ between times $t_0-r$ and $t_0+s$. We also know
    \[B(\gamma_x(t_0+s),\delta) \subset B(\gamma_x(t_0),\underbrace{s+\delta}_{< \tau}) \subset B(\gamma_x(t_0),\tau). \]
    So the geodesic segment $\gamma_y([t_0-r,t_0+s])$ connects two points in the $\tau$ neighborhood of $\gamma_x([t_0-r,t_0])$. The definition of $\tau$ implies this segment is contained in $st(\sigma_{\xi,\varepsilon}(x))$, and we have $\sigma_{\xi,\varepsilon}(y) \subset st(\sigma_{\xi,\varepsilon}(x))$. This completes the second case.
    \end{proof}

  \begin{figure}
        \centering
\tikzset{every picture/.style={line width=0.75pt}} 

\begin{tikzpicture}[x=0.75pt,y=0.75pt,yscale=-1,xscale=1]

\draw    (60,160) -- (280,160) ;
\draw [shift={(60,160)}, rotate = 0] [color={rgb, 255:red, 0; green, 0; blue, 0 }  ][fill={rgb, 255:red, 0; green, 0; blue, 0 }  ][line width=0.75]      (0, 0) circle [x radius= 3.35, y radius= 3.35]   ;
\draw    (280,160) -- (360,160) ;
\draw [shift={(280,160)}, rotate = 0] [color={rgb, 255:red, 0; green, 0; blue, 0 }  ][fill={rgb, 255:red, 0; green, 0; blue, 0 }  ][line width=0.75]      (0, 0) circle [x radius= 3.35, y radius= 3.35]   ;
\draw    (350,160) -- (430,160) ;
\draw [shift={(350,160)}, rotate = 0] [color={rgb, 255:red, 0; green, 0; blue, 0 }  ][fill={rgb, 255:red, 0; green, 0; blue, 0 }  ][line width=0.75]      (0, 0) circle [x radius= 3.35, y radius= 3.35]   ;
\draw    (410,160) -- (540,160) ;
\draw [shift={(410,160)}, rotate = 0] [color={rgb, 255:red, 0; green, 0; blue, 0 }  ][fill={rgb, 255:red, 0; green, 0; blue, 0 }  ][line width=0.75]      (0, 0) circle [x radius= 3.35, y radius= 3.35]   ;
\draw    (610,160) -- (490,160) ;
\draw [shift={(610,160)}, rotate = 180] [color={rgb, 255:red, 0; green, 0; blue, 0 }  ][fill={rgb, 255:red, 0; green, 0; blue, 0 }  ][line width=0.75]      (0, 0) circle [x radius= 3.35, y radius= 3.35]   ;
\draw  [dash pattern={on 4.5pt off 4.5pt}] (200,118) .. controls (200,102.54) and (212.54,90) .. (228,90) -- (432,90) .. controls (447.46,90) and (460,102.54) .. (460,118) -- (460,202) .. controls (460,217.46) and (447.46,230) .. (432,230) -- (228,230) .. controls (212.54,230) and (200,217.46) .. (200,202) -- cycle ;
\draw   (566.25,160) .. controls (566.25,135.84) and (585.84,116.25) .. (610,116.25) .. controls (634.16,116.25) and (653.75,135.84) .. (653.75,160) .. controls (653.75,184.16) and (634.16,203.75) .. (610,203.75) .. controls (585.84,203.75) and (566.25,184.16) .. (566.25,160) -- cycle ;
\draw   (366.25,160) .. controls (366.25,135.84) and (385.84,116.25) .. (410,116.25) .. controls (434.16,116.25) and (453.75,135.84) .. (453.75,160) .. controls (453.75,184.16) and (434.16,203.75) .. (410,203.75) .. controls (385.84,203.75) and (366.25,184.16) .. (366.25,160) -- cycle ;
\draw  [draw opacity=0][dash pattern={on 4.5pt off 4.5pt}] (234.02,24.33) .. controls (301.21,42.17) and (350,95.71) .. (350,159) .. controls (350,219.98) and (304.7,271.92) .. (241.28,291.59) -- (185,159) -- cycle ; \draw  [dash pattern={on 4.5pt off 4.5pt}] (234.02,24.33) .. controls (301.21,42.17) and (350,95.71) .. (350,159) .. controls (350,219.98) and (304.7,271.92) .. (241.28,291.59) ;  
\draw    (60,160) -- (610,130) ;
\draw [shift={(610,130)}, rotate = 356.88] [color={rgb, 255:red, 0; green, 0; blue, 0 }  ][fill={rgb, 255:red, 0; green, 0; blue, 0 }  ][line width=0.75]      (0, 0) circle [x radius= 3.35, y radius= 3.35]   ;

\draw (41,162) node [anchor=north west][inner sep=0.75pt]   [align=left] {$\displaystyle v_{0}$};
\draw (616,162) node [anchor=north west][inner sep=0.75pt]   [align=left] {$x$};
\draw (350,172) node [anchor=north west][inner sep=0.75pt]   [align=left] {$t_0$};
\draw (391,172) node [anchor=north west][inner sep=0.75pt]   [align=left] {$t_0+s$};
\draw (261,172) node [anchor=north west][inner sep=0.75pt]   [align=left] {$t_{0} -r$};
\draw (276,22) node [anchor=north west][inner sep=0.75pt]   [align=left] {$\displaystyle D^{\varepsilon }( \xi )$};
\draw (180,240) node [anchor=north west][inner sep=0.75pt]   [align=left] {$\tau$ neighborhood of $\gamma _{x}([ t_{0} -r,t_{0}]) \subset st( \sigma_{\xi,\varepsilon}(x))$};
\draw (619,130) node [anchor=north west][inner sep=0.75pt]   [align=left] {$\displaystyle y$};
\draw (591,212) node [anchor=north west][inner sep=0.75pt]   [align=left] {$\displaystyle B( x,\delta )$};

\end{tikzpicture}

        \caption{The second case of the Star \thref{Star}.}
        \label{fig:Star}
\end{figure}
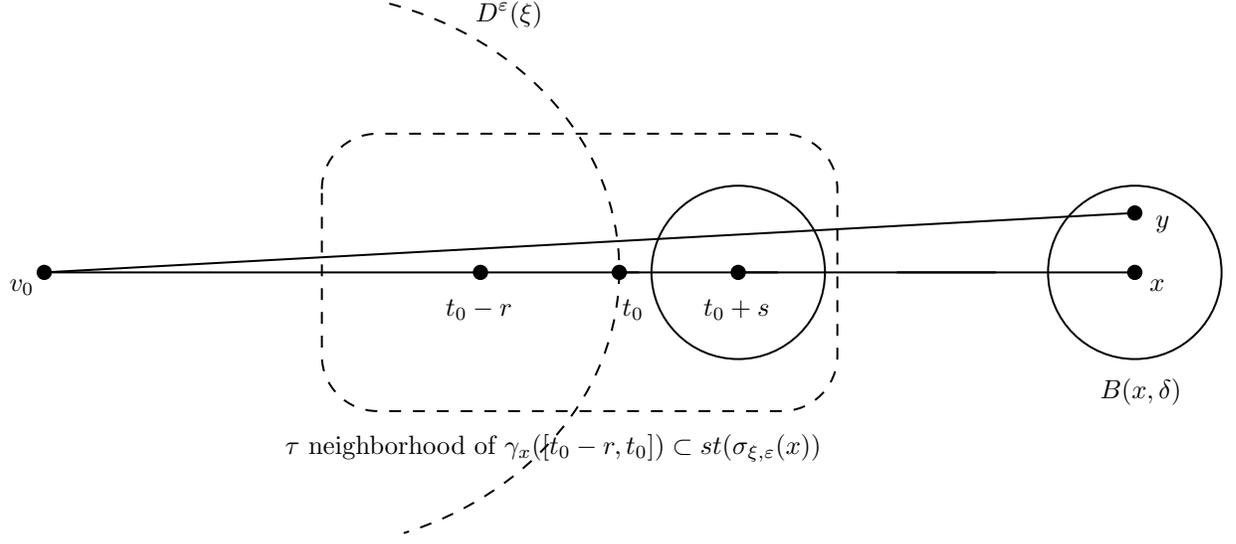

\begin{corollary}\thlabel{ConesAreOpen}
    Let $\xi \in \partial_{Stab}G$, $\UU$ a $\xi$--family, and $\varepsilon \in (0,1)$. Then $\fitTilde{\Cone}_{\UU,\varepsilon}(\xi)$ is open in $\overline{X}$.
\end{corollary}
\begin{proof}
    Take $x \in \fitTilde{\Cone}_{\UU,\varepsilon}(\xi)$. By the Genuine Shadow \thref{Genuine Shadows}, $[v_0,x]$ goes through $D(\xi)$. Suppose for now that $x \in X$.
    
    If $x \in D^\varepsilon(\xi)$, it must be that $\sigma_{\xi, \varepsilon}(\xi) = \sigma_x \subset N_\UU(\xi)$. We can take $\varepsilon'$ so that $B(x,\varepsilon') \subset st(\sigma_x) \cap D^\varepsilon(\xi)$. If $y \in B(x,\varepsilon')$, then $\sigma_x \subset \sigma_y = \sigma_{\xi,\varepsilon}(y)$, hence $\total{\sigma_y} \subseteq \total{\sigma_x} \subset U_v$
    for any vertex $v$ of $\sigma_x \cap D(\xi)$. This is the definition of $y \in \fitTilde{\Cone}_{\UU,\varepsilon}(\xi)$. 

    If $x \notin D^\varepsilon(\xi)$, the Star Lemma \ref{Star} provides us with a suitable $\delta$. If $x \notin \partial D^\varepsilon(\xi)$, equivalently $d(x,D(\xi)) > \varepsilon$, we can assume $\delta$ is small enough so that $D^\varepsilon(\xi) \cap B(x,\delta) = \varnothing$. As in the previous paragraph, the conclusion of the Star Lemma \ref{Star} implies $y \in \Cone_{\UU, \varepsilon}(\xi) \subset \tCone_{\UU,\e}(\xi)$ for all $y \in B(x,\delta)$. 
    
    If $d(x,D(\xi)) = \e$, then $\sigma_{\xi,\varepsilon}(x) = \sigma _x \subset N_\UU(\xi)$. Decreasing $\delta$ if necessary, we can assume $B(x,\delta) \subset st(\sigma_x)$. If $y \in B(x,\delta)$, then either $y \in D^\e(\xi)$ or not. If $y \in D^\e(\xi)$, then $\sigma_{\xi,\e}(y) = \sigma_y \subset st(\sigma_x)$, and as in the first paragraph this implies $y \in \tCone_{\UU,\e}(\xi)$. If $y \notin D^\e(\xi)$, then the conclusion of the Star \thref{Star} implies $\sigma_{\xi,\e}(y) \subset st(\sigma_x)$, and again this implies $y \in \Cone_{\UU,\e}(\xi)$.

    If $x \in \partial X$, then we can choose some $x'$ far along $[v_0,x)$ so that $\sigma_{\xi,\e}(x) = \sigma_{\xi,\e}(x')$. Using this $x'$ and the argument above, we can get a $\delta$ so that $B(x',\delta) \subset \tCone_{\UU,\e}(\xi)$. The set $\{y \in \overline{X} \, | \, \gamma_y \text{ goes through } B(x',\delta)\}$ defines an open neighborhood of $x$ in $\overline{X}$ and is contained in $\tCone_{\UU,\e}(\xi)$.
\end{proof}

Notice that by definition $\Cone_{\UU,\e}(\xi) \cap \overline{D^\e(\xi)} \subset \partial \overline{D^\e(\xi)}$ and the Star \thref{Star} implies that points in $\Cone_{\UU,\e}(\xi) \setminus \overline{D^\e(\xi)}$ have an open neighborhood contained in $\Cone_{\UU,\e}(\xi)$. It is exactly the points in this intersection which prevent $\Cone_{\UU,\e}(\xi)$ from being open, since any ball around them would have points $x \in D^\e(\xi)$ so that $\gamma_x$ does not \emph{leave} $D^\e(\xi)$.

\section{The Topology of $\overline{Z}$}\label{section:The Topology}

\subsection{Definition of the Topology}

A point in $\overline{Z}$ is in either $Z, \partial_{Stab}G$ or $\partial X$. We define a topology on $\overline{Z}$ by defining a neighborhood basis for each type of point. Consider the example in the introduction. We know $\partial \pi_1(\Sigma) = S^1$, and see $\partial_{\langle \gamma \rangle}\pi_1(\Sigma)$ as a tree of circles by contracting the endpoints of lifts of $c$. We encourage the reader to reflect on the what topology the tree of circles has as a quotient of $S^1$. This topology is what we are emulating in the definitions below.

Recall that each $\hatt{\sigma}$ is the product $\{\sigma\} \times \sigma \times \total{\sigma}$, that $\hatto{\sigma} = \{\sigma\} \times \sigma \times X_\sigma$, and that $Z$ is constructed by gluing the $\hatto{\sigma}$ to eachother along faces with the $\varphi_{\sigma,\sigma'}:X_{\sigma'} \longrightarrow X_\sigma$. We call the second coordinate the \emph{simplex coordinate} and the third coordinate the \emph{cusped coordinate}, see \thref{defn of c}. Also recall from \thref{defn of p} that $p:Z \sqcup \partial X \longrightarrow \overline{X}$ is the $G$--equivariant map which projects to the simplex coordinate and is the identity on $\partial X$.

\begin{definition}\label{defn:open sets in Z}
    For $z \in Z$, let $\mathcal{O}_{\overline{Z}}(z)$ be the collection of open sets of $Z$ containing $z$. 
\end{definition} 
 
As a CAT$(0)$ space, we have a natural topology on $\overline{X} = X \cup \partial X$, which we use to define neighborhoods of $\eta \in \partial X \subset \overline{Z}$.
\begin{definition}
    For a base point $v_0 \in X$ and $\eta \in \overline{X}$, let
\[V_{r,\delta}(\eta) = \{x \in \overline{X} \; | \; d(v_0,x) > r \text{ and } \gamma_x(r) \in B(\gamma_\eta(r),\delta)\}.\] 
The $V_{r,\delta}(\eta)$ form a basis of (not necessarily open) neighborhoods for $\eta$ in $\overline{X}$, denoted $\mathcal{O}_{\overline{X}}(\eta)$. For any open set $U \subset \overline{X}$ containing $\eta$, let 
\[V_U(\eta) =p^{-1}(U) \sqcup  \{ \xi \in \partial_{Stab}G \;|\;D(\xi) \subset U\}.\]
To define a neighborhood basis, let $\OO_{\overline{Z}}(\eta) = \{V_U(\eta) \; | \; U \subset \overline{X} \text{ the interior of some } V_{r,\delta}(\eta)\}$.
\end{definition}

\begin{lemma}\cite[6.3]{Martin}\thlabel{RegularityOfBoundary}
    Let $\eta \in \partial X$ with $U$ a neighborhood of $\eta$ in $\overline{X}$, and $k \geq 0$. Then there exists a neighborhood $U' \subset U$  of $\eta$ so that $d(X \setminus U, U') > k$.
\end{lemma}

\begin{definition}
    Let $\xi \in \partial_{Stab}G$ and suppose $\UU$ is a $\xi$--family and $\e \in (0,1)$. Let $W_{\UU,\varepsilon}(\xi)$ be the set of points $z \in Z$ such that $p(z) \in D^\varepsilon(\xi)$ and for every vertex $v \in \sigma_{p(z)} \cap D(\xi)$, the cusped coordinate $c(z)$ is in $U_v$. Let $\partial W_{\UU,\e}(\xi)$ be the points $\xi' \in \partial_{Stab}G$ such that $D(\xi') \setminus D(\xi) \subset \fitTilde{\Cone}_{\mathcal{U},\varepsilon}(\xi)$ and for every vertex $v$ of $D(\xi) \cap D(\xi')$ we have $\xi' \in U_v$. Let
    \[V_{\mathcal{U},\varepsilon}(\xi):= W_{\UU,\varepsilon}(\xi) \cup p^{-1}(\Cone_{\UU,\e}(\xi)) \cup \partial W_{\UU,\e}(\xi) \]
    and let $\OO_{\overline{Z}}(\xi) = \{V_{\UU,\e}(\xi) \; |\; \UU \text{ a } \xi \text{--family}, \e \in(0,1)\}$. 
\end{definition}

\begin{remark}
    Suppose $\xi \in \partial_{Stab}G$ and $V_{\UU,\e}(\xi) \in \OO_{\overline{Z}}(\xi)$. Suppose $z \in Z$ has $p(z) = x\in \tCone_{\UU,\e}(\xi) \setminus \Cone_{\UU,\e}(\xi)$, then $x \in D^\e(\xi)$, $\sigma_{\xi,\e}(x) = \sigma_x$, and for any vertex $v \in \sigma_x \cap V(\xi)$, we have $\total{\sigma_x} \subset U_v$ in $\total{v}$. In particular, $c(z) \in U_v$, so $z \in W_{\UU,\e}(\xi)$. This shows 
    \[p^{-1}(\tCone_{\UU,\e}(\xi)) \subset W_{\UU,\e}(\xi) \cup p^{-1}(\Cone_{\UU,\e}(\xi)).\]
    This means the definition of $V_{\UU,\e}(\xi)$ does not change if we take the preimage of a $\tCone$ instead of a $\Cone$, since we are only adding points which would be in $W_{\UU,\e}(\xi)$ anyway.  
\end{remark}

Note that by Assumption \ref{shadow assumption}, the admissible $\xi$--families in this definition depend on the choice of base point. For unrelated $\xi$--families $\UU,\UU'$ and $\e,\e' \in (0,1)$, there is no obvious relationship between $V_{\UU,\e}(\xi)$ and $V_{\UU',\e'}(\xi)$, but we do have the following. 

\begin{lemma}\thlabel{ContainmentOfV_UU}
Let $\xi \in \partial_{Stab}G$ with $\xi$--families $\UU,\UU'$ and $0 < \e' \leq \e \in (0,1)$. If $\UU'$ is $d_{max}$--nested in $\UU$, then $V_{\UU',\varepsilon'}(\xi) \subset V_{\UU,\varepsilon}(\xi)$.    
\end{lemma}
\begin{proof}
    Let $x \in \fitTilde{\Cone}_{\UU',\varepsilon'}(\xi)$.
    Then $[v_0,x]$ meets $D(\xi)$ by the Genuine Shadow  \thref{Genuine Shadows}, and because a geodesic segment in $N(\xi)$ can meet at most $d_{max}$ simplices, we get a path of simplices of length at most $d_{max}$ from $\sigma_{\xi,\varepsilon'}(x)$ to $\sigma_{\xi,\varepsilon}(x)$. Because $x \in \fitTilde{\Cone}_{\UU',\varepsilon'}(\xi)$, we know $\sigma_{\xi,\varepsilon'}(x) \subset N_{\UU'}(\xi)$, so the Crossing Lemma \ref{Crossing} implies that $\total{\sigma_{\xi,\varepsilon}(x)}\subset U_v$ in $\total{v}$ for any vertex $v$ of $\sigma_{\xi,\varepsilon}(x) \cap V(\xi)$. Thus $x \in \fitTilde{\Cone}_{\UU,\varepsilon}(\xi)$ as well. This shows $\fitTilde{\Cone}_{\UU',\varepsilon'}(\xi) \subset \fitTilde{\Cone}_{\UU,\varepsilon}(\xi)$. If $\eta \in \partial X \cap V_{\UU',\varepsilon'}(\xi)$ or $z \in p^{-1}(\tCone_{\UU',\e'}(\xi)$, this immediately implies $\eta,z \in V_{\UU,\varepsilon}(\xi)$ as well. If $\xi' \in  \partial_{Stab}G \cap V_{\UU',\varepsilon'}(\xi)$, then this implies $D(\xi') \setminus D(\xi) \subset \fitTilde{\Cone}_{\UU,\varepsilon}(\xi)$. For a vertex $v$ of $D(\xi') \cap D(\xi)$, we have $\xi' \in U'_v \subset U_v$, so $\xi' \in V_{\UU,\varepsilon}(\xi)$ too.

    If $z \in W_{\UU',\e'}(\xi)$ then by definition $p(z) = x \in D^{\e'}(\xi)$ and $c(z) \in U'_v$ for every vertex $v \in \sigma_x \cap V(\xi)$. Since $\e' \leq \e$ and $U'_v \subset U_v$, we have $z \in W_{\UU,\e}(\xi)$.
\end{proof}

The following lemma provides conditions where one half of the definition of $\xi' \in \partial W_{\UU,\e}(\xi)$ implies the other half. 

\begin{lemma}\thlabel{Double Refinement}
    Let $\xi,\xi' \in \partial_{Stab}G$, $\varepsilon \in (0,1)$, and let $\UU'' \subset \UU' \subset \UU$ be $\xi$--families, each $d_{max}$--refined and $d_{max}$--nested in the next. 
    \begin{enumerate}[label = (\roman*)]
        \item If $\xi' \in U'_v$ for some vertex $v$ of $D(\xi)\cap D(\xi')$, then $\xi' \in V_{\UU,\varepsilon}(\xi).$
    
        \item If $D(\xi') \cap \fitTilde{\Cone}_{\UU'',\varepsilon}(\xi) \neq \varnothing$, then $\xi' \in V_{\UU,\varepsilon}(\xi)$.
    \end{enumerate}
\end{lemma}
\begin{proof}
    We begin with $(i)$. By the definition of a $\xi$--family and the assumption $\xi' \in U'_v \subset \total{v}$, we have $\xi' \in U'_w$ for any vertex $w$ of $D(\xi)\cap D(\xi')$. If $D(\xi') \subset D(\xi)$, this immediately implies $\xi' \in V_{\UU',\varepsilon}(\xi) \subset V_{\UU,\varepsilon}(\xi)$ since the other condition is vacuously true. If not, let $K$ be a connected component of $D(\xi')\setminus D(\xi)$ and choose a simplex $\sigma \subset K \cap Lk(\xi)$, which exists since $D(\xi')$ is connected. Any other simplex of $K$ can be reached by a path of simplices of length at most $d_{max}$ from $\sigma$ since $D(\xi')$ only has at most $d_{max}$ simplices. For any vertex $w$ of $\sigma \cap D(\xi)$, we have
    \[\xi' \in \partial G_\sigma \cap U'_w \Longrightarrow \sigma \subset N_{\UU'}(\xi)\]
    so applying the Refinement Lemma \ref{Refinement Lemma} implies $K \subset \fitTilde{\Cone}_{\UU,\varepsilon}(\xi)$. Repeating this for every such component tells us $D(\xi')\setminus D(\xi) \subset \fitTilde{\Cone}_{\UU,\varepsilon}(\xi)$, so $\xi' \in V_{\UU,\varepsilon}(\xi)$.

    For $(ii)$, choose $x \in D(\xi') \cap \fitTilde{\Cone}_{\UU'',\varepsilon}(\xi)$. Because cones are defined to be disjoint from domains, $x \notin D(\xi)$. If $D(\xi) \cap D(\xi') = \varnothing$, then $\sigma_x$ can be connected to any other simplex of $D(\xi')$ using a path of simplices of length at most $d_{max}$. The Refinement \thref{Refinement Lemma} then implies $D(\xi') \subset \fitTilde{\Cone}_{\UU',\varepsilon}(\xi)$, hence $\xi \in V_{\UU',\varepsilon}(\xi)$. 
    
    If $D(\xi') \cap D(\xi) \neq \varnothing$, by taking a geodesic from $x$ to a point of $D(\xi')\cap D(\xi)$, we can connect $\sigma_x$ to a simplex $\sigma \subset D(\xi') \cap Lk(\xi)$ with a path of simplices of length at most $d_{max}$. The Refinement \thref{Refinement Lemma} implies $\sigma \subset \fitTilde{\Cone}_{\UU',\varepsilon}(\xi)$, hence $\xi' \in \partial G_\sigma \subset U'_v$ in $\partial G_v$ for any vertex $v$ of $\sigma \cap D(\xi)$. Now we can apply $(i)$ and conclude that $\xi' \in V_{\UU,\varepsilon}(\xi)$.
\end{proof}

\subsection{A basis of neighborhoods}
Let $\OO_{\overline{Z}}$ be the collection of sets $\OO_{\overline{Z}}(z), \OO_{\overline{Z}}(\eta)$ as  $\OO_{\overline{Z}}(\xi)$ for all $z, \eta, \xi \in \overline{Z}$. In this subsection, we prove $\OO_{\overline{Z}}$ forms a basis of neighborhoods for a topology. We begin with the following lemma.

\begin{lemma}[Filtration]\thlabel{Filtration}
    Given $z,z' \in \overline{Z}$, if $U \in \mathcal{O}_{\overline{Z}}(z)$ and $z' \in U$, then there is some $U' \in  \mathcal{O}_{\overline{Z}}(z')$ so that $U' \subset U$. 
\end{lemma}

There are $6$ different cases depending on which part of $\overline{Z}$ that $z,z'$ are in.

\begin{lemma}[Case 1: $z,z' \in \partial X$] \thlabel{Case1Filtration}
If $\eta,\eta' \in \partial X$ and $\eta' \in V_{U}(\eta) \in \OO_{\overline{Z}}(\eta)$, then there is a neighborhood $U'\subset \overline{X}$ of $\eta'$ so that $V_{U'}(\eta') \subset V_U(\eta)$.
\end{lemma}
\begin{proof}
    If $\eta' \in V_U(\eta)$ then by definition $\eta' \in U \cap \partial X$, so $U' :=U$ is an open neighborhood of $\eta$ in $\overline{X}$, and it's clear that $V_{U'}(\eta') = V_U(\eta)$, so in particular we have containment. Any other neighborhood of $\eta'$ contained in $U$ would also work.
\end{proof}

\begin{lemma}[Case 2: $z \in \partial X, z' \in \partial_{Stab}G$] \thlabel{Case2Filtration} Let $\eta \in \partial X$, $U$ be a neighborhood of $\eta$ in $\overline{X}$, and $\xi \in \partial_{Stab}G$ with $\xi \in V_U(\eta)$. Then there is some $V_{\mathcal{U},\varepsilon}(\xi)$ so that $V_{\mathcal{U},\varepsilon}(\xi) \subset V_U(\eta)$.
\end{lemma}
\begin{proof}
    By assumption, $U$ is the interior of some $V_{r,\delta}(\eta)$ and $D(\xi) \subset U$. Domains are finite, hence compact, and $U$ is open, so we can choose $\varepsilon\in (0,1)$ so that $D^\varepsilon(\xi) \subset U$. Let $\mathcal{U}$ be any $\xi$--family. We claim that $V_{\mathcal{U},\varepsilon}(\xi) \subset V_U(\eta)$. 
    
    For any $x \in \fitTilde{\Cone}_{\mathcal{U},\varepsilon}(\xi)$, the geodesic from $v_0$ to $x$ meets $D(\xi)$ by the Genuine Shadow \thref{Genuine Shadows}, say at a point $y$. Because $D(\xi) \subset U \subseteq  V_{r,\delta}(\eta)$, we know $y \in U$ implies $x \in U$, so $D^\varepsilon(\xi) \cup \fitTilde{\Cone}_{\mathcal{U},\varepsilon}(\xi) \subset U$. For any $w \in (Z \cup \partial X) \cap V_{\mathcal{U},\varepsilon}(\xi)$, this implies $p(w) \in U$. For any $\xi' \in \partial_{Stab}G \cap V_{\mathcal{U},\varepsilon}(\xi)$, we have $D(\xi') \subset D^\varepsilon(\xi) \cup \fitTilde{\Cone}_{\mathcal{U},\varepsilon}(\xi) \subset U$. 
\end{proof}

\begin{lemma}[Case 3: $z \in \partial_{Stab}G, z' \in \partial X$]\thlabel{Case3Filtration} Let $\eta \in \partial X, \, \xi \in \partial_{Stab}G,$ and $V_{\mathcal{U},\varepsilon}(\xi) \in \mathcal{O}_{\overline{Z}}(\xi)$. If $\eta \in V_{\mathcal{U},\varepsilon}(\xi)$, then there is some neighborhood $U \subset \overline{X}$ of $\eta$ so that $V_U(\eta)\subset V_{\mathcal{U},\varepsilon}(\xi)$.
\end{lemma}
\begin{proof}
    Since domains are bounded, we can choose $N$ so that $D^\varepsilon(\xi) \subset B(v_0,N)$, and let $x = \gamma_\eta(N+1)$. Because $\eta \in V_{\mathcal{U},\varepsilon}(\xi)$ and $d(x,D(\xi)) > \e$, we have $x \in \Cone_{\mathcal{U},\varepsilon}(\xi)$. By \thref{ConesAreOpen}, there is some $\delta >0$ so that $B(x,\delta) \subset \tCone_{\UU,\varepsilon}(\xi)$. Letting $U = V_{N+1,\delta}(\eta)$, we have $V_{U}(\eta)\subset V_{\mathcal{U},\varepsilon}(\xi)$.
\end{proof}

\begin{lemma}[Case 4: $z,z' \in \partial_{Stab}G$] \thlabel{Case4Filtration} Let $\xi,\xi' \in \partial_{Stab}G$ and suppose $\xi' \in V_{\mathcal{U},\varepsilon}(\xi) \in \OO_{ \overline{Z}}(\xi)$. Then there exists a $\xi'$--family $\VV$ and $\delta \in (0,1)$ so that $V_{\VV,\delta}(\xi') \subset V_{\mathcal{U},\varepsilon}(\xi)$.
\end{lemma}

\begin{proof}
    
We define $\delta$ and a preliminary $\xi'$--family $\UU'$ in two claims. 

\begin{claim}
    There exists $\delta$ so that if $y \in \overline{D^\delta(\xi')} \setminus D^\e(\xi)$, then $y \in \tCone_{\UU,\e}(\xi)$. . 
\end{claim}
\begin{proof}[Proof of Claim]
    If $x \in D(\xi') \cap D^\e(\xi)$, we can choose $\delta_x$ so that $B(x,\delta(x)) \subset D^\e(\xi)$. If $x \in D(\xi') \setminus D^\e(\xi)$, then $x \in \tCone_{\UU,\e}(\xi)$ because $\xi' \in V_{\UU,\e}(\xi)$. Therefore $[v_0,x]$ goes through $D(\xi)$ by the Genuine Shadow \thref{Genuine Shadows} and we can apply the Star \thref{Star} to receive a $\delta_x$ so that if $y \in B(x,\delta_x) \setminus D^\e(\xi)$, then $[v_0,y]$ goes through $D^\e(\xi)$ and $\sigma_{\xi,\e}(y) \subset st(\sigma_{\xi,\e}(x))$. In this case, for any vertex $v \in \sigma_x \cap V(\xi)$, we have $\total{\sigma_{\xi,\e}(y)} \subset \total{\sigma_{\xi,\e}(x)} \subset U_v,$
    hence $y \in \tCone_{\UU,\e}(\xi)$. Now $D(\xi') \setminus D^{\e/2}(\xi)$ is compact, so we can choose finitely many $x_1,\ldots, x_m$ so that \[D(\xi')\setminus D^{\e/2}(\xi) \subset \bigcup_i B(x_i,\tfrac{\delta_{x_i}}{2}).\]
    
    Let $\delta' < \min(\frac{\delta_{x_i}}{2}, \frac{\e}{2})$ and suppose $y \in D^{\delta'}(\xi') \setminus D^\e(\xi)$. There is some $x \in D(\xi')$ so that $d(y,x) < \delta'$, and because $d(y,D(\xi)) \geq \e$, the triangle inequality implies $d(x,D(\xi)) \geq \e - \delta' \geq \frac{\e}{2}$. So $x \in D(\xi') \setminus D^{\e/2}(\xi)$, and our choice of open cover gives us some $x_i$ so that $d(x,x_i) <\frac{\delta_{x_i}}{2}$, hence $d(y,x_i) < \delta +  \delta_{x_i}/2 \leq \delta_{x_i}$. If $x_i \in D^\e(\xi)$, then our choice of $\delta_{x_i}$ would imply $y \in D^\e(\xi)$, which is a contradiction. Therefore $x_i \in D(\xi') \setminus D^\e(\xi)$, and as above this implies $y \in \tCone_{\UU,\e}(\xi)$. Thus $D^{\delta'}(\xi') \setminus D^\e(\xi) \subset \tCone_{\UU,\e}(\xi)$. 

    Letting $\delta = \frac{\delta'}{2}$, we have $\overline{D^\delta(\xi')} \subset D^{\delta'}(\xi')$. 
\end{proof}

\begin{claim}\thlabel{Case4 xi' family}
    There exists a $\xi'$--family $\UU'$ satisfying the following. 
    \begin{enumerate}[label = A\arabic*.]
        \item If $v \in V(\xi) \cap V(\xi')$, then $U'_v \subset U_v$. 

        \item If $y \in D^\delta(\xi') \cap D^\e(\xi)$ and there is some $v' \in \sigma_y \cap V(\xi')$ with $\total{\sigma_y} \subset U'_{v'}$, then $\sigma_y \cap D(\xi) \subset D(\xi')$ and for every $v \in \sigma_y \cap V(\xi)$, we have $\total{\sigma_y} \subset U'_{v} \subset U_{v}$.
        
        \item If $x \in \Cone_{\UU',\delta}(\xi')$ leaves $D^\delta(\xi')$ at a point inside $D^\e(\xi)$, then $x \in \tCone_{\UU,\e}(\xi)$. 
        
    \end{enumerate}
\end{claim}
\begin{proof}[Proof of Claim]
    For $v' \in V(\xi')$, if $v' \notin V(\xi)$, let $V_{v'} = \total{v'}$. Otherwise, $v' \in V(\xi)$ and we let $V_{v'} = U_{v'}$, where $U_{v'}$ comes from the $\xi$--family $\UU$. Note that $U_{v'}$ contains $\xi'$ precisely because $\xi' \in V_{\UU,\e}(\xi)$ and the definition of a $\xi$--family.
    
    For each $v' \in V(\xi')$, let $st_\xi(v') = \{v \in (V(\xi) \cap \overline{st}(v'))\setminus V(\xi') \}$. In words, $st_{\xi}(v')$ is the vertices of the closed star $\overline{st}(v')$ which are in $D(\xi)$ but are \emph{not} in $D(\xi')$. It's possible that $st_\xi(v')$ might be empty. For each $v \in st_\xi(v')$, the geodesic $[v',v]$ goes from $v'$, through a simplex we call $\sigma_{v',v} \subset Lk(v')$, then ends at $v$. For example if $X$ is a graph, $\sigma_{v'v}$ is simply the edge connecting $v,v'$, but if $X$ is a CAT$(0)$ cube--complex, then $v,v'$ might be opposite corners of a square, and $\sigma_{v',v}$ is that square. Note that $\sigma_{v',v}$ is not in $D(\xi')$ because $v \notin D(\xi')$. Therefore $\total{\sigma_{v',v}}$ is a closed subset of $\total{v'}$ not containing $\xi'$, and we can choose an open neighborhood $W_{v',v} = \total{v'} \setminus \total{\sigma_{v',v}}$ of $\xi'$. Set 
     \[W_{v'} = V_{v'} \cap \bigg(\bigcap_{v\in st_\xi(v')} W_{v',v} \bigg).\]
    Repeating this process for each vertex $v' \in V(\xi')$, we let $\UU'$ be a $\xi'$--family $d_{max}$--nested in the collection of sets $\{W_{v'}, v' \in V(\xi')\}$. Explicitly, for each vertex $v' \in V(\xi')$, we have open sets
     \[U'_{v'} = U^0_{v'} \subset U^1_{v'} \subset \cdots \subset U^{d_{max}-1}_{v'} \subset U^{d_{max}}_{v'} \subset W_{v'}\]
     where each containment is a nesting. This is our $\UU'$. Clearly A1 is immediate by the choice of $V_{v'}$. 

    Note that because a simplex $\sigma \subset X$ is isometric to the convex hull in some model space of some number of points in \emph{general} position and domains are convex, any simplex of $X$ which meets $D(\xi)$ does so in some maximum dimensional face. The dimension of that face is $| \sigma \cap V(\xi)|-1$. For example if $\sigma$ is an edge of $D(\xi)$, then $\sigma$ meets $D(\xi)$ in $\sigma$, which contains $2$ vertices, and an edge is a $2-1$--simplex.  

    Let $y,v'$ be as in the assumption of A2. Because $\e,\delta < 1$, we know $\sigma_y \subset N(\xi') \cap N(\xi)$, so $\sigma_y$ meets each of $D(\xi'),D(\xi)$. Let $v$ be any vertex of $\sigma_y \cap V(\xi)$. For a contradiction, suppose $v \notin V(\xi')$, then $v \in st_\xi(v')$ and $\sigma_y$ (or one of its faces) is $\sigma_{v',v}$, hence $\total{\sigma_y} \subseteq \total{\sigma_{v',v}}$ in $\total{v'}$. Putting this together with the assumption that $\total{\sigma_y}\subset U'_{v'}$, we have
    \[\total{\sigma_y} \subseteq U'_{v'} \cap \total{\sigma_{v',v}},\]
    but this is impossible because $U'_{v'}$ and $\total{\sigma_{v',v}}$ are disjoint by construction. This contradiction implies $v \in V(\xi')$. By the definition of a $\xi'$--family and the assumption $\total{\sigma_y}\subset U'_{v'}$, we know $\total{\sigma_y} ]\subset U'_v$. Since $v$ was an arbitrary vertex of $\sigma_y \cap D(\xi)$, this proves A2.

    Suppose $x \in \Cone_{\UU',\delta}(\xi')$ leaves $D^\delta(\xi')$ at a point in $D^\e(\xi)$, and choose $y \in [v_0,x] \cap D^\e(\xi)$ just before $[v_0,x]$ leaves $D^\delta(\xi')$ so that $\sigma_y = \sigma_{\xi',\delta}(\xi)$. Then $[y,x]$ gives a path of simplices $\sigma_{\xi',\delta}(x) = \sigma_1, \sigma_2 \ldots, \sigma_n = \sigma_{\xi,\e}(x)$ in $N(\xi)$ with $n \leq d_{max}$. For $1 \leq k \leq n$, we prove the following.
    
    \begin{enumerate}
        \item $\sigma_k \subset Lk(\xi')$ and further, $\sigma_k \cap D(\xi) \subset D(\xi')$.
        \item For any vertex $v' \in \sigma_k \cap V(\xi')$, $\total{\sigma_k} \subset U^k_{v'}$.
    \end{enumerate}
   
   Induct on $k$. For $k=1$, $y \in D^\delta(\xi')$ by assumption so $\sigma_1 = \sigma_y = \sigma_{\xi',\delta}(x) \subset Lk(\xi')$ already. Because $x \in \Cone_{\UU',\delta}(\xi')$, we know A2 applies to $y$, which is exactly 1 and 2 above. Assuming the result for $k$, we prove 1 and 2 for $k+1$ by breaking into two cases. 

    Suppose $\sigma_k \subset \sigma_{k+1}$. By induction, $\sigma_k \subset Lk(\xi')$, so $\sigma_{k+1} \subset Lk(\xi')$ too. Also by induction, $\total{\sigma_{k+1}} \subset \total{\sigma_k} \subset U^k_{v'} \subset U^{k+1}_{v'}$ for any vertex $v' \in \sigma_k \cap V(\xi')$, so the definition of $\xi'$--family implies 2 for $k+1$. Fix a vertex $v' \in \sigma_{k+1} \cap V(\xi')$ and let $v$ be any vertex of $\sigma_{k+1} \cap V(\xi)$. As in the proof of A2, if $v \notin V(\xi')$, then $\sigma_{k+1}$ (or one of its faces) must be $\sigma_{v',v}$, and this contradicts $\total{v'} \subset U^{k+1}_{v}$. This means $v \in V(\xi')$, and again because $v$ was arbitrary, we have $\sigma_{k+1} \cap D(\xi) \subset D(\xi)$. This completes the inductive step for this case. 

    If $\sigma_{k+1} \subset \sigma_k$, then by induction $\sigma_k \subset Lk(\xi')$ so we can fix a vertex $v' \in \sigma_k \cap V(\xi')$. Because this path of simplices is in $N(\xi)$, we know there is some vertex $v \in \sigma_{k+1} \cap V(\xi)$. Because $\sigma_{k+1} \subset \sigma_k$, this $v$ is also a vertex of $\sigma_k$, so by induction with $1$, we know $v \in V(\xi')$. Because $v$ was arbitrary, this implies $1$ for $\sigma_{k+1}$. For 2, if $v' \in \sigma_{k+1} \cap V(\xi')$, then by induction we have $\total{\sigma_k} \subset \total{\sigma_{k+1}} \cap U^k_{v'}33 \neq \varnothing$, so the definition of nesting implies $\total{\sigma_{k+1}} \subset U^{k+1}_{v'}$. This completes the induction. 

    Consider $k = n$ and recall that $\sigma_n = \sigma_{\xi,\e}(x)$. By 1, $\sigma_{\xi,\e}(x) \cap D(\xi) \subset D(\xi')$, so for any vertex $v \in \sigma_{\xi,\e}(x)\cap V(\xi)$, we have $v \in V(\xi')$ as well. For any such vertex $v$, 2 and the definition of our $\xi'$--family implies 
    \[\total{\sigma_{\xi,\e}(x)} \subset U^n_v \subset V_v = U_v,\]
    hence $x \in \tCone_{\UU,\e}(\xi)$. This completes the proof of A3 and the claim.      
\end{proof}

\begin{claim}
    $\tCone_{\UU',\delta}(\xi') \subset \tCone_{\UU,\e}(\xi)$. 
\end{claim}
\begin{proof}[Proof of Claim]
    Suppose $x \in \tCone_{\UU',\delta}(\xi')$ and choose $y \in [v_0,y] \in D^\delta(\xi')$ just before $[v_0,x]$ leaves $D^\delta(\xi')$ so that $\sigma_y = \sigma_{\xi',\delta}(x)$. If $y \in D^\delta(\xi') \setminus D^\e(\xi)$, then the claim defining $\delta$ implies $y$ and hence $x \in \tCone_{\UU,\e}(\xi)$. If $y \in D^\delta(\xi') \cap D^\e(\xi)$, then $[v_0,x]$ either ends in $\sigma_y$ or leaves $D^\delta(\xi)$ at a point inside $D^\e(\xi)$. Then A2 or A3 imply that $\total{\sigma_y}$ interacts with the vertices of $D(\xi)$ in the right way so that $x \in \tCone_{\UU,\e}(\xi)$. 
\end{proof}

Let $\VV$ be a $\xi'$--family nested in $\UU'$ and note that A1, A2, A3, and the previous claim all apply to $\VV$. We show $V_{\VV,\delta}(\xi') \subset V_{\UU,\e}(\xi)$ by considering the three kinds of points.

If $z \in V_{\VV,\delta}(\xi') \cap Z$, then either $p(z) \in \tCone_{\VV,\delta}(\xi') \subset \tCone_{\UU,\e}(\xi)$ and $z \in V_{\UU,\e}(\xi)$ immediately, or $z \in W_{\VV,\delta}(\xi')$. In this second case, let $y= p(z) \in D^\delta(\xi')$. 

If $y \in D^\delta(\xi') \setminus D^\e(\xi)$, then again the definition of $\delta$ implies $y \in \tCone_{\UU,\e}(\xi)$, hence $z \in V_{\UU,\e}(\xi)$. If $y \in D^\delta(\xi') \cap D^\e(\xi)$, then by definition of $W_{\VV,\delta}(\xi')$, for every vertex $v' \in \sigma_{y} \cap V(\xi')$ we have $c(z) \in V_{v'} \subset U'_v$. But then $c(z) \in \total{\sigma_{y}} \cap V_{v'}$, so the definition of nesting implies $\total{\sigma_{y}} \subset U'_{v'}$ (this is why we had to nest one more time after all the claims). Thus A2 applies to $y$ and tells us that $\sigma_y \cap V(\xi) \subset D(\xi')$, so every vertex of $\sigma_y \cap V(\xi)$ is also a vertex of $V(\xi')$. Using the rest of the conclusion of A2 and A1, we have $\total{\sigma_y} \subset V_{v} \subset U'_{v} \subset U_{v}$, hence $z \in W_{\UU,\e}(\xi)$. 

If $\eta \in V_{\VV,\delta}(\xi') \cap \partial X$, then $\eta \in \tCone_{\VV,\delta}(\xi') \subset \tCone_{\UU,\e}(\xi)$ immediately from the last claim, so $\eta \in V_{\UU,\e}(\xi)$.

If $\zeta \in V_{\VV,\delta}(\xi') \cap \partial_{Stab}G$, we need to show that $D(\zeta) \setminus D(\xi) \subset \tCone_{\UU,\e}(\xi)$ and for all vertices $w \in V(\zeta) \cap V(\xi)$, we have $\zeta \in U_w$. 

For the first condition, notice that a point in $D(\zeta) \setminus D(\xi)$ is either inside $D(\xi')$ or outside $D(\xi')$. The points inside $D(\xi')$ are controlled by $\xi' \in V_{\UU,\e}(\xi)$ and the points outside $D(\xi')$ are controlled by $\zeta \in V_{\VV,\delta}(\xi')$. Explicitly, if $x \in D(\zeta) \setminus D(\xi)$, then either $x \in D(\xi') \setminus D(\xi) \subset \tCone_{\UU,\e}(\xi)$, or $x \in D(\zeta) \setminus D(\xi') \subset \tCone_{\VV,\delta}(\xi') \subset \tCone_{\UU,\e}(\xi)$.

For the second condition, suppose $w \in V(\zeta) \cap V(\xi)$. If $w$ is also in $V(\xi')$, then because $\zeta \in V_{\VV,\delta}(\xi')$, we have $\zeta \in V_w \subset U_w$. If $w \notin V(\xi')$, then again using the definition of $\zeta \in V_{\VV,\delta}(\xi')$ and the previous claim, we have

\[w \in D(\zeta) \setminus D(\xi') \subset \tCone_{\VV,\delta}(\xi') \subset \tCone_{\UU,\e}(\xi),\] 
but cones are disjoint from domains by definition, so this would imply $w \notin V(\xi)$, which is a contradiction. This completes the check on vertices of $V(\zeta) \cap V(\xi')$ finishes the lemma. 

\end{proof}

\begin{lemma}[Case 5: $z \in Z,\,z' \in \partial X$]\thlabel{Case5Filtration} Let $z \in Z,\, \eta \in \partial X$ and let $U$ be an open neighborhood of $\eta$ in $\partial X$. If $z \in V_U(\eta)$, then there is some neighborhood $U' \in \mathcal{O}_{\overline{Z}}(z)$ so that $U' \subset V_U(\eta)$. 
\end{lemma}
\begin{proof}
    By definition, $U$ is an open set containing $p(z)$ in $\overline{X}$, so it contains some ball say $B(p(z),\e)$. Because $p$ is continuous by \thref{defn of p}, $U' =p^{-1}(B(p(z),\e))$ is an open neighborhood of $z$ in $Z$ clearly contained in $V_U(\eta)$.    
\end{proof}

\begin{lemma}[Case 6: $z \in Z,\,z' \in \partial_{Stab}G$]\thlabel{Case6Filtration} Let $z \in Z,\, \xi \in \partial_{Stab}G$ and $V_{\UU,\varepsilon}(\xi) \in \mathcal{O}_{\overline{Z}}(\xi)$. If $z \in V_{\UU,\varepsilon}(\xi)$, then there is some neighborhood $U' \in \mathcal{O}_{\overline{Z}}(z)$ so that $U' \subset V_U(\eta)$. 
\end{lemma}
\begin{proof}
    There are two cases. Let $x = p(z)$.

    If $z \in W_{\UU,\varepsilon}(\xi)$, then by definition $z \in D^\e(\xi)$ and $c(z) \in U_v$ for every vertex $v$ of $\sigma_{x}\cap D(\xi)$. Fixing some choice of $v$, $U_v$ contains a neighborhood $U$ of $c(z)$ in $X_v$, which we can interpret as an open neighborhood of $c(z)$ in $X_{\sigma_x}$ by taking its preimage under $\varphi_{v,\sigma_{x}}$. We can also choose $\delta$ so that $B(x,\delta) \subset st(\sigma_x) \cap D^\e(\xi)$, and applying \thref{nbhd of z in Z}, we find an open neighborhood of $z \in Z$ contained in $W_{\UU,\e}(\xi)$. 

    If $x \in \Cone_{\UU,\e}(\xi) \subset \tCone_{\UU,\e}(\xi)$, then there is some $\delta$ so that $B(x,\delta) \subset \tCone_{\UU,\e}(\xi)$ by \thref{ConesAreOpen}. Then $p^{-1}(B(x,\delta))$ is an open neighborhood of in $Z$ by \thref{defn of p} and is contained in $W_{\UU,\e}(\xi) \sqcup \Cone_{\UU,\e}(\xi)$, hence contained in $V_{\UU,\e}(\xi)$. 
\end{proof}

 \begin{proof}[Proof of Filtration \thref{Filtration}]
    The previous lemmas cover all the cases except $z \in Z, z' \in \partial_{Stab}G$ and $z, z' \in Z$. This first case never happens because neighborhoods of $z \in Z$ are contained in $Z$. In the second case, we are given $z' \in U \in \mathcal{O}_{\overline{Z}}(z)$. By definition, $U$ is open in $Z$, so we can take $U \in \mathcal{O}_{\overline{Z}}(z')$. 
 \end{proof}

\begin{theorem}\thlabel{Oz is a topology} $\mathcal{O}_{\overline{Z}}$ is the basis for a topology on $\overline{Z}$ making it into a second countable space. Further, $Z$ embeds into $\overline{Z}$ as a dense subset. 
\end{theorem}
\begin{proof}
    We must prove that if $U_1,U_2 \in \mathcal{O}_{\overline{Z}}$ and $z \in U_1\cap U_2$, then there is some $W \in \mathcal{O}_{\overline{Z}}(z)$ so that $z \in W \subset U_1 \cap U_2$. There are $3$ cases.

    \begin{enumerate}
        \item If $z \in Z$, then by the Filtration \thref{Filtration} there are open neighborhoods $W_1,W_2 \subset Z$ so that $W_i\subset U_i$. Then $W = W_1\cap W_2$ is open in $Z$ and contained in $U_1 \cap U_2$ as desired.

        \item If $z = \eta \in \partial X$, then by the Filtration \thref{Filtration} there are open neighborhoods $W_1,W_2$ of $\eta$ in $\overline{X}$ so that $V_{W_i}(\eta) \subset U_i$. Then $W = W_1 \cap W_2$ is an open neighborhood of $\eta \in \overline{X}$ and it's clear that 
        \[V_{W}(\eta) \subset V_{W_1}(\eta) \cap V_{W_2}(\eta) \subset U_1 \cap U_2.\]

        \item If $z = \xi \in \partial_{Stab}G$, then by the Filtration \thref{Filtration} there are $V_{\UU_1,\e_1}(\xi),V_{\UU_2,\e_2}(\xi) \in \mathcal{O}_{\overline{Z}}(\xi)$ so that $V_{\UU_i,\e_i}(\xi) \subset U_i$. Let $\UU$ be a $\xi$--family which is $d_{max}$--nested in the family of sets $\{(U_1)_v \cap (U_2)_v, \; v \in D(\xi)\}$ and let $\varepsilon = \min(\varepsilon_1,\varepsilon_2)$. 
        Then Lemma \ref{ContainmentOfV_UU} implies
    
    \[V_{\UU,\varepsilon}(\xi) \subset V_{\UU_1,\varepsilon_1}(\xi) \cap V_{\UU_2,\varepsilon_2}(\xi) \subset U_1 \cap U_2.\]
        
    \end{enumerate}
    
    To see $\partial G$ is second countable, we need to provide a countable basis for the topology. Since $Z$ is the quotient of countably many spaces $\hatt{\sigma}$ which are themselves second countable metric spaces, $Z$ is second countable. We enumerate a basis $\{U_n, n\geq 0\}$.
    
    Since $X$ is a simplicial complex with countably many cells, it is a separable space, hence so is the set $\Lambda$ of points on a geodesic from $v_0$ to some $\eta \in \partial X$ (here $\Lambda$ may not equal $X$ since a given geodesic segment may not extend to a ray). Let $\Lambda'$ be a dense countable subset of $\Lambda$. Then the family of open sets $V_{N,\varepsilon}(\eta)$ with $\gamma_{\eta}(N) \in \Lambda', \varepsilon \in \mathbb{Q}$ forms a countable basis for the topology on $\partial X$. Enumerate them as $\{V_n, n\geq 0\}$. 

    A neighborhood of $\xi \in \partial_{Stab}G$ is defined by choosing a constant $\varepsilon \in (0,1)$, a finite subcomplex of $X$, namely $D(\xi)$, and for every vertex $v$ of that finite subcomplex, a neighborhood $U_v$ of $\xi$ in $\total{v}$. There are countably many choices of domains since each domain is finite, and each $\total{v}$ has a countable basis since it is a compact metric space. By only allowing open subsets from these countable bases, we receive countably many possible $\xi$--families $\UU$ as $\xi$ ranges over $\partial_{Stab}G$. By further only allowing $\varepsilon \in \mathbb{Q}$, we receive a countable collection of sets $V_{\UU,\varepsilon}(\xi)$. Enumerate these as $\{W_n, n\geq 0\}$. It's clear that these form a neighborhood basis for every point $\xi \in \partial_{Stab}G$, and thus the collection $\{U_n,V_n,W_n, \, n \geq 0 \}$ forms a countable basis of neighborhoods for $\overline{Z}$.

    Finally, every element of our basis meets $Z$ by construction, so $Z$ is dense in $\overline{Z}$.
 \end{proof}

\section{Properties of the Topology}\label{section: Properties of the Topology}

\subsection{Independence of basepoint and induced topologies}
In the previous section we endowed $\overline{Z}$ with a topology. We begin by showing this topology is independent of our chosen basepoint and induces the topology we expect on some of the spaces we already understand. The following is another basic topological fact that we leave to the reader.

\begin{lemma}
    Suppose $X$ is a set with two topologies $\tau_1,\tau_2$ and for each $x \in X$, let $N_1(x),N_2(x)$ be a neighborhood basis for $x$. If for each $x \in X$ and $U \in N_1(x)$, there is some $U' \in N_2(x)$ so that $U' \subset U$, then any open set in $\tau_1$ is also open in $\tau_2$. 
\end{lemma}

\begin{lemma}\label{Independent of Basepoint} The topology of $\overline{Z}$ does not depend on the choice of basepoint.
\end{lemma}
\begin{proof}
    Choose two points $x_1,x_2 \in X$, not necessarily vertices. For $z \in \overline{Z}$ and $i =1,2$, let $\mathcal{O}_{\overline{Z}}^i(z)$ be the neighborhood basis of $z$ defined above using $x_i$ as the basepoint.  

    \begin{claim}
        For any $z \in \overline{Z}$, if $U \in \mathcal{O}_{\overline{Z}}^1(z)$, then there is some $U' \in \mathcal{O}_{\overline{Z}}^2(z)$ so that $U' \subset U$. Similarly, we indicate basic open sets and cones with superscripts to indicate which basepoint they are using. 
        
    \end{claim}
    \begin{proof}[Proof of Claim]
        There are $3$ cases depending on if $z \in Z,\partial X, $ or $\partial_{Stab}G$. 

        If $z \in Z$, then $\mathcal{O}_{\overline{Z}}^1(z) = \mathcal{O}_{\overline{Z}}^2(z)$ because these neighborhoods do not reference the basepoint, so we can take $U' = U$. 

        If $z \in \partial X$, then $U = V_{W}(\eta)$ where $W = V_{r,\delta}^1(\eta)$. From \cite{BH}, the topology on $\overline{X}$ does not depend on the choice of basepoint and the $V_{r,\delta}(\eta)$ form a neighborhood basis with any basepoint. Therefore there must be $r',\delta'$ so that $V_{r',\delta'}^2(\eta) \subset V_{r,\delta}^1(\eta)$, where $V_{r',\delta'}^2(\eta)$ uses $x_2$ as a basepoint. Letting $W'$ be the interior of $V_{r',\delta'}^2(\eta)$, it follows that $V^2_{W'}(\eta) \subset V_W^1(\eta)$, so $U' = V_{W'}^2(\eta)$ is our desired set. 

        If $\xi \in \partial_{Stab}G$, then $U = V_{\UU_1,\e}^1(\xi)$ for some $\xi$--family $\UU_1$ which is admissible using $x_1$ as a basepoint. Let $\UU_2''$ be any $\xi$--family which is admissible with $x_2$ as a basepoint, let $\UU_2'$ be a $\xi$--family $d_{max} + F(d_{max})$--nested in the family of sets $\{(U_1)_v \cap (U''_2)_v \, | \, v \in V(\xi)\}$, and let $\UU_2$ be $d_{max}$--refined in $\UU_2'$. We claim $V^2_{\UU_2,\e}(\xi) \subset V_{\UU_1,\e}^1(\xi)$. 

        By definition, $\UU_2$ is contained in the $\xi$--family $\UU_1$ which is admissible with $x_1$ as basepoint, so $\UU_2$ is also admissible with $x_1$ as basepoint. If $z \in W_{\UU_2,\e}(\xi)$, then for every vertex $v \in \sigma_{p(z)} \cap V(\xi)$, we know $c(z) \in (U_2)_v \subset (U_1)_v$, so $z \in W_{\UU_1,\e}(\xi)$. 
        
        If $ y \in \tCone_{\UU_2,\e}^2(\xi)$, then $\sigma_{\xi,\e}^2(y)$ is a witness for $y \in \pCone_{\UU_2}(\xi) \subset \pCone_{\UU_1}(\xi)$. We can apply \thref{Shadow2} with $\UU = \UU_2$ and $\VV = \UU_1$ to conclude that $y \in \tCone_{\UU_1,\e}(\xi)$ for any $e$. Since $y$ was arbitrary, this implies $\tCone^2_{\UU_2,\e}(\xi) \subset \tCone_{\UU_1,\e}^1(\xi)$.
        
        Finally, if $\xi' \in V^2_{\UU_2,\e}(\xi)$, then we must check conditions on $v \in D(\xi') \cap D(\xi)$ and $D(\xi') \setminus D(\xi)$. If $v \in D(\xi') \cap D(\xi)$, then $\xi' \in (U_2)_v \subset (U_1)_2$, and by the previous paragraph, we have
        \[D(\xi') \setminus D(\xi) \subset \tCone_{\UU_2,\e}^2(\xi) \subset \tCone_{\UU_1,\e}^1(\xi),\]
        which checks all the necessary conditions. Hence $U' = V^2_{\UU_2,\e}(\xi)$ satisfies the claim.
    \end{proof}

    With the claim proven, we can apply the previous lemma to see that any set which is open with $x_1$ as a basepoint is also open with $x_2$ as a basepoint. Reversing the roles of $x_1$ and $x_2$, we see that a subset of $\overline{Z}$ is open with respect to $x_1$ if and only if it is open with respect to $x_2$, so the topologies are the same.  
\end{proof}

\begin{proposition}\thlabel{induced topologies}
    The topology of $\overline{Z}$ induces the natural topology on $Z,\partial X$ and $\hatt{v}$ for every vertex $v$ of $X$. 
\end{proposition}
\begin{proof}
    We must show that if $U$ is an open subset of $Z,\partial X,$ or any $\hatt{v}$, then $U$ can be extended to an open set in $\overline{Z}$, and that if $U \in \OO_{\overline{Z}}$, then $U \cap B$ is open in the topology on $B$. 
    
    Consider $Z$ first. Open sets of $Z$ are part of $\mathcal{O}_{\overline{Z}}$ already, so no extension is necessary. On the other hand, if $U \in \OO_{\overline{Z}}$, then $U$ is a basic open set around some point $z \in \overline{Z}$. If $z \in Z$, then $U$ is already open in $\overline{Z}$, and if $z \in \partial G = \partial X \sqcup \partial_{Stab}G$, it follows from \thref{Case5Filtration} and \thref{Case6Filtration} $U \cap Z$ is open in $Z$. 

    Now consider $\partial X$. If $U \subset \partial X$ is open, then $U$ can be extended to an open set $U' \subset \overline{X}$. For any $\eta \in U$, $V_{U'}(\eta)$ is open in $\overline{Z}$ and $V_{U'}(\eta) \cap \partial X = U$, so $V_{U'}(\eta)$ extends $U$ to an open set in $\overline{Z}$. On the other hand, if $U \in \OO_{\overline{Z}}$, then again there are three cases to consider. If $U \subset Z$ then $U \cap \partial X = \varnothing$ and there is nothing to prove. If $U = V_{U'}(\eta)$ for some $\eta \in \partial X$, then $U \cap \partial X = U' \cap \partial X$, which is open in $\partial X$. If $U = V_{\UU,\e}(\xi)$ for some $\xi \in \partial_{Stab}G$, then it follows from \thref{Case3Filtration} that $U \cap \partial X$ is open in $\partial X$. 

    Now suppose $v$ is a vertex of $X$ and consider an open subset $U \subset \hatt{v}$. Choose $\delta>0$ so that $B(v,\delta) \subset st(v)$, then 
    \[U' = \{ z \in Z, \, p(z) \in B(v,\delta), c(z) \in U\}\]
    is open in $Z$, and clearly $U' \cap \hatt{v} = U \cap Z$. For each $\xi \in U \cap \partial G_v$, we can apply \thref{balloon prop} to $U,\xi$ to receive a $\xi$--family $\UU_\xi$, so that $(U_\xi)_v \subset U$. Then
    \[\bigg(U' \cup \bigcup_{\xi \in U \cap \partial G_v} V_{\UU_\xi,\frac{1}{2}}(\xi)\bigg) \cap \hatt{v} = U.\]
    The large set on the left is a union of open sets in $\overline{Z}$ so it is open, and is the desired extension of $U$ to $\overline{Z}$.

    We fix some $U \in \OO_{\overline{Z}}$ and we show $U \cap \hatt{v}$ is open. There are 3 cases, depending on what kind of basic open set $U$ is. 
    
    If $U$ is an open subset of $Z$, then $U \cap \hatt{\sigma}$ is open for every simplex $\sigma$ of $X$ because that's the definition of the quotient topology on $Z$. In particular $U \cap \hatt{v}$ is open in $\hatt{v}$. 
    
    Suppose $U = V_{\UU,\e}(\xi)$ and fix some $z \in V_{\UU,\e}(\xi) \cap \hatt{v}$. We need to find an open neighborhood $U'$ of $z$ in $\hatt{v}$ so that $U' \subset U \cap \hatt{v}$, and consider the 3 possibilities for $z$.
    
    If $z \in Z$, then $z \in U \cap \hatto{v}$ and we need find a neighborhood $U'$ of $z$ in $\hatt{v}$ so that $U' \subset U \cap \hatt{v}$. If $v \in V(\xi)$, then $c(z) \in U_v$ by definition of $W_{\UU,\e}(\xi)$, so $U'=U_v \cap \hatto{v}$ is the desired neighborhood of $z$. If $ v\notin V(\xi)$ then we must have $v \in \Cone_{\UU,\e}(\xi)$. In this case, $p(\hatto{v}) = v \subset \Cone_{\UU,\e}(\xi)$, so $U' = \hatto{v}$ is the desired neighborhood of $z$. 
    
    Suppose $z = \xi' \in V_{\UU,\e}(\xi) \cap \partial G_v$. We know $v \in V(\xi')$, but $v$ may or may not be in $V(\xi)$. If $v \in V(\xi)$, then set $U' = U_v$, and if $v \notin V(\xi)$, set $U' = \total{v}$. Either way, $U'$ is an open neighborhood of $\xi'$ in $\total{v}$. Using \thref{balloon prop}, extend $U'$ to a $\xi'$--family $\UU'$. Using the Filtration \thref{Filtration}, there is a $\xi'$--family $\VV$ and $\e' \in (0,1)$ so that $V_{\VV,\e'}(\xi') \subset V_{\UU,\e}(\xi)$. Let $\VV'$ be a $\xi'$--family $d_{max}$--refined and nested in $\VV$. Recall that the intersection of $\xi'$--families is again $\xi'$--family (see \thref{xi families exist}), so we can combine $\UU'$ and $\VV'$ into a $\xi'$--family $\mathcal{W =}\{U'_v \cap V'_v \, | \, v \in V(\xi')\}$. Then $W_v$ is an open neighborhood of $\xi' \in \hatt{v}$, and we claim $W_v \subset U \cap \hatt{v}$ so that $W_v$ is the desired neighborhood of $\xi'$. Indeed, if $x \in W_v$, then either $x \in \hatto{v}$ or $x \in \partial G_v$. In the first case, we know $c(x) \in W_v \subset U_v$, so $x \in W_{\mathcal{
    W}, \e'}(\xi') \subset V_{\UU,\e}(\xi)$. In the second case, $x \in W_v \subset V'_v$, and because $\VV'$ is $d_{max}$--refined and nested in $\VV$, \thref{Double Refinement} applied to $x$ tells us $x \in V_{\VV,\e'}(\xi') \subset V_{\UU,\e}(\xi)$. This completes the proof that if $U = V_{\UU,\e}(\xi)$, then $U \cap \hatt{v}$ is open.

    Suppose $U = V_{U'}(\eta)$ for some $\eta \in \partial X, U' \subset \overline{X}$. Either $v \notin U'$ so $U \cap \hatt{v} = \varnothing$ and there is nothing to prove, or $v \in U'$ and 
    \[U \cap \hatt{v} = \hatto{v} \sqcup \{ \xi \in \partial_{Stab}G, \, v \in D(\xi) \subset U'\}. \]
    The first set on the right is clearly open in $\hatt{v}$. If $\xi$ is in the second set, then by the Filtration \thref{Filtration}, there is some $V_{\UU,\e}(\xi) \subset U$. Let $\UU'$ be a $\xi$--family $d_{max}$--refined and nested in $\UU$. By the previous case, $W =  V_{\UU',\e}(\xi) \cap \hatt{v}$ is open in $\hatt{v}$, and we claim $W \subset V_{U'}(\eta)$. For $x \in W \cap \hatto{v}$, there is nothing to check since $\hatto{v} \subset V_{U'}(\eta)$. For $\xi' \in W \cap \partial G_v$, we know that $\xi' \in U'_v$, so by the choice of $\UU'$ and \thref{Double Refinement}, $\xi' \in V_{\UU,\e}(\xi) \subset V_{U'}(\eta)$. This shows the second set on the right is open, hence $U \cap \hatt{v}$ is open. 
\end{proof}

The reader might wonder why we limit ourselves to vertices in the previous proposition and don't consider the induced topology on $\hatt{\sigma}$ for an arbitrary simplex $\sigma$ of $X$. The answer is because the result would be false. For example, if $\sigma$ was a higher dimensional simplex, we could choose an open neighborhood $U = \{\sigma\} \times U_1 \times U_2 \subset \hatt{\sigma}$, where $U_1 \subsetneq \sigma$ and $U_2 \subset \total{\sigma}$ are open. If $\xi \in U_2 \cap \partial G_{\sigma}$, then any neighborhood of $\xi$ in our topology on $\overline{Z}$ would contain some $W_{\UU,\e}(\xi)$, which would contain points $z$ with $p(z) \in \sigma \setminus  U_1$. Thus it will be impossible to find an open neighborhood of $\xi$ which meets $\hatt{\sigma}$ in exactly $U$. 

\subsection{$\overline{Z}$ is $T_0$}
By Urysohn's metrization theorem \cite[4.4]{Munkres}, to show $\overline{Z}$ is metrizable it is enough to show $\overline{Z}$ is Hausdorff, separable, and regular. We have already seen it is separable in \thref{Oz is a topology}, and we will prove it is regular and $T_0$, a combination which implies Hausdorff.

\begin{proposition}
    The space $\overline{Z}$ satisfies the $T_0$ separation condition, that is, for every pair of distinct points $z,z' \in \overline{Z}$ there is a neighborhood of $z$ not containing $z'$.
\end{proposition}

The proof is broken into many cases. We continue to use $v_0$ as our notation for the basepoint of the topology.

\begin{lemma}[Case 1: $z,z' \in \partial X$]
    If $\eta,\eta' \in \partial X$, then there is an open neighborhood $U$ of $\eta$ in $\overline{X}$ so that $\eta' \notin V_U(\eta)$.
\end{lemma}
\begin{proof}
    The space $\overline{X}$ is metrizable, hence Hausdorff and we have a nice neighborhood basis for points in $\partial X$. We can choose $r,\delta,r',\delta'$ so that $V_{r,\delta}(\eta)$ and $V_{r',\delta'}(\eta')$ are disjoint and let $U,U'$ be their respective interiors. Then $V_U(\eta)$ and $V_{U'}(\eta')$ are disjoint. 
\end{proof}

\begin{lemma}[Case 2: $z \in \partial X,z' \in \partial_{Stab}G$]
    If $\eta \in \partial X$ and $\xi \in \partial_{Stab}G$, then there is an open neighborhood $U$ of $\eta$ in $\overline{X}$ so that $\xi \notin V_U(\eta)$.
\end{lemma}
\begin{proof}
    Pick $R$ so that $D(\xi) \subset B(v_0,R)$ and let $U$ be the interior of $V_{R+1, 1}(\eta)$. Then  $\xi \notin V_U(\eta)$.
\end{proof}

\begin{lemma}\thlabel{cone can avoid any point}
    Suppose $\xi \in \partial_{Stab}G$ and $x \in \overline{X} \setminus D(\xi)$. Then there is some $\xi$--family $\UU$ so that $x \notin \tCone_{\UU,\e}(\xi)$ for any $\e$.
\end{lemma}
\begin{proof}
    If $[v_0,x]$ does not meet $D(\xi)$, then $x \notin \tCone_{\UU,\e}(\xi)$ for \emph{any} $\xi$--family $\UU$ by the contrapositive of the Genuine Shadow \thref{Genuine Shadows}.

    If $[v_0,x]$ does meet $D(\xi)$, let $\sigma$ be the first simplex met by $[v_0,x]$ after leaving $D(\xi)$. Choose a vertex $v$ of $\sigma \cap D(\xi)$ and a neighborhood $U_v$ of $\xi$ in $\total{v}$ avoiding $\total{\sigma}$. Extend $U_v$ to a $\xi$--family $\UU$ using \thref{balloon prop}, and let $\UU'$ be a $\xi$--family $d_{max}+1$--nested in $\UU$. Observe that the portion of $[v_0,x]$ in $D^\e(\xi)$ gives a path of simplices of length at most $d_{max}$ from $\sigma_{\xi,\varepsilon}(x)$ to $\sigma$. 
    
    For a contradiction, suppose $\e \in (0,1)$ and $x \in V_{\UU',\e}(\xi)$. Then $\sigma_{\xi,\varepsilon}(x) \subset N_{\UU'}(\xi)$, and the Crossing \thref{Crossing} applied to the path of simplices from $\sigma_{\xi,\e}(x)$ to $\sigma$ implies $\total{\sigma} \subset U_v$ in $\total{v}$. But $U_v$ was chosen to avoid $\total{\sigma}$, so this is impossible.
\end{proof}

\begin{lemma}[Case 3: $z \in \partial_{Stab}G, z' \in \partial X$]
    If $\xi \in \partial_{Stab}G$ and $\eta \in \partial X$, then there is an open neighborhood $V_{\UU,\varepsilon}(\xi) \in \mathcal{O}_{\overline{Z}}(\xi)$ so that $\eta \notin V_{\UU,\varepsilon}(\xi)$.
\end{lemma}
\begin{proof}
    Apply the previous lemma with $x = \eta$ to receive a $\xi$--family $\UU$ with $\eta \notin \tCone_{\UU,\e}(\xi)$ for any $\e$. Then $\eta \notin V_{\UU,\e}(\xi)$ for any $\e$.
\end{proof}

\begin{lemma}[Case 4: $z,z' \in \partial_{Stab}G$]
    If $\xi,\xi' \in \partial_{Stab}G$, then there is an open neighborhood $V_{\UU,\varepsilon}(\xi) \in \mathcal{O}_{\overline{Z}}(\xi)$ so that $\xi' \notin V_{\UU,\varepsilon}(\xi)$.
\end{lemma}
\begin{proof}
    If $D(\xi) \cap D(\xi') \neq \varnothing$, we can choose a vertex $v \in V(\xi) \cap V(\xi')$. Let $U_v$ be a neighborhood of $\xi$ in $\total{v}$ not containing $\xi'$ and extend $U$ to a $\xi$--family $\UU$. Then $\xi' \notin V_{\UU,\varepsilon}(\xi)$ for any $\varepsilon$ since it fails the first condition for $\xi' \in \partial W_{\UU,\e}(\xi)$ in $\total{v}$.

    Otherwise, we can suppose $D(\xi) \cap D(\xi') = \varnothing$. Apply \thref{cone can avoid any point} to any $x \in D(\xi')$ to find a $\xi$--family $\UU$ so that $x \notin \fitTilde{\Cone}_{\UU,\varepsilon}(\xi)$ for any $\e$. Then $\xi'$ fails the second condition for $\xi' \in V_{\UU,\e}(\xi)$.
\end{proof}

\begin{lemma}[Case 5: $z \in Z, \, z' \in \partial G$] If $z \in Z$ and $z'\in \partial G$, there is a neighborhood of $z$ not containing $z'$.
\end{lemma}
\begin{proof}
    This is immediate since by definition sets from $\mathcal{O}_{\overline{Z}}(z)$ do not meet $\partial G$. 
\end{proof}

\begin{lemma}[Case 6: $z \in \partial G, z' \in Z$] 
If $z \in \partial G$ and $z' \in Z$, then there is a set $V \in \mathcal{O}_{\overline{Z}}(z)$ so that $z \notin Z$.     
\end{lemma}
\begin{proof}
    Let $x = p(z')$. If $z = \eta \in \partial X$, then any neighborhood $U$ of $\eta$ in $\overline{X}$ which doesn't contain $x$ will have $z' \notin V_U(\eta)$. If $z = \xi \in \partial_{Stab}G$, then $[v_0,x]$ either meets $D(\xi)$ or doesn't. If it doesn't, then $x \notin \tCone_{\UU,\e}(\xi) \cup D(\xi)$ for any $\xi$--family $\UU$ or $\e \in (0,1)$ by the contrapositive of the Genuine Shadow \thref{Genuine Shadows}. If $[v_0,x]$ does meet $D(\xi)$, then it either goes through or $x \in D(\xi)$. If it goes through, we can apply \thref{cone can avoid any point} to $x$ to find a suitable $\xi$--family. If $x \in D(\xi)$, then $\sigma_x \subset D(\xi)$. Choose a vertex $v$ of $\sigma_x$ and a neighborhood $U_v$ of $\xi$ in $\total{v}$ so that $c(z) \notin U_v$. Extend $U_v$ to a $\xi$--family $\UU$ using \thref{balloon prop}. Then $z \notin W_{\UU,\e}(\xi)$ and $z \notin \tCone_{\UU,\e}(\xi)$, so $z \notin V_{\UU,\e}(\xi)$.
\end{proof}

\subsection{$\overline{Z}$ is regular}

Recall that a topological space $B$ is \emph{regular} if for every open set $U\subset B$ and $x \in U$, there exists another open neighborhood $U'$ of $x$ so that $\overline{U'} \subset U$, or equivalently, every point of $B\setminus U$ admits a neighborhood which doesn't meet $U'$. If $B$ is regular and has the $T_0$ seperation property, it's easy to see $B$ is Hausdorff; given distinct points $x,y \in B$, the $T_0$ property implies we can find an open set $U$ containing $x$ but not $y$, and regularity implies we can find another open set $x \in U' \subset U$ so that $\overline{U'} \subset U$. Then $U'$ and $B \setminus \overline{U'}$ are disjoint neighborhoods of $x,y$. 

The goal of this subsection is the following proposition.

\begin{proposition}
    The space $\overline{Z}$ is regular. 
\end{proposition}

Since open sets are unions of basic open sets, it is enough to prove the regularity condition for basic open sets, leaving us with three cases to check.

\begin{lemma}[Case 1: $\eta \in \partial X$]
    Let $\eta \in \partial X$ and $V_U(\eta) \in \mathcal{O}_{\overline{Z}}(\eta)$. Then there exists an open neighborhood $U'$ of $\eta$ so that every point of $\overline{Z} \setminus V_U(\eta)$ admits a neighborhood missing $V_{U'}(\eta)$.
\end{lemma}
\begin{proof}
    Because there are finitely many orbits of simplices, there is some constant $D$ so that all cells of $X$ have diameter at most $D$. Recall that $A$ is the acylindricity constant which bounds the diameter of domains.

    By definition, $U$ is the interior of $V_{r,\delta}(\eta)$ for some $r,\delta$. Using \thref{RegularityOfBoundary}, choose an open neighborhood $W_1 \subset U \subset \overline{X}$ of $\eta$ so that $d(W_1, X \setminus U) \geq A+D+1$. Since $\overline{X}$ is metrizable, it is regular and we can choose a neighborhood $W_2 \subset W_1$ so that $\overline{W_2} \subset W_1$. Because the $V_{r,\delta}(\eta)$ form a neighborhood basis for $\eta$ in $\overline{X}$, we can choose $r',\delta'$ so that $V_{r',\delta'}(\eta) \subset W_2$, and let $U'$ be the interior of $V_{r',\delta'}(\eta)$. Let $x = \gamma_\eta(r') \in U'$, and decreasing $\delta'$ if necessary, we can assume $B(x,\delta') \subset st(x)$. We show $U'$ satisfies the conclusion.
    
    If $\eta' \in \partial X \setminus V_U(\eta)$ then $\eta' \notin \overline{U'}$, so we can choose an open neighborhood $W$ of $\eta'$ so that $W \cap U' = \varnothing$, for example $W :=\overline{X}\setminus \overline{ U'}$. Then $V_{W}(\eta')\cap V_{U'}(\eta)$ contains no points of $\partial X$, since those points would have to be in both $W$ and $U'$. It cannot contain any point $\xi \in \partial_{Stab}G$, since any such $\xi$ would have $D(\xi) \subset W\cap U' = \varnothing$. It cannot contain any point $z \in Z$, since such a $z$ would have $p(z) \in W \cap U' = \varnothing$. Thus $V_{W}(\eta')\cap V_{U'}(\eta) = \varnothing$ as needed.

    If $z \in Z \setminus V_U(\eta)$, then $p(z) \notin U$. Since $\overline{U'} \subset U$, there is an open neighborhood $W$ of $p(z)$ so that $W\cap U' = \varnothing$, for example $W := X \setminus \overline{U'}$. Then $p^{-1}(W)$ is an open neighborhood of $z$ which doesn't meet $V_{U'}(\eta)$. 

    If $\xi \in \partial_{Stab}G \setminus  V_{U}(\eta)$, then by definition $D(\xi) \nsubseteq U$, which means $D(\xi) \cap X \setminus U \neq \varnothing$. Because $\textrm{diam}(N(\xi)) \leq A+D$ and $W_1$ is at least $A+D+1$ from $\overline{X} \setminus W_1$, this implies $N(\xi) \subset X \setminus W_1$. 
    
    \begin{claim}\thlabel{xi family which avoids a point}
        There is a $\xi$--family $\VV$ so that 
        \[U' \cap \big(D(\xi) \cup \tCone_{\UU,\frac{1}{2}}(\xi)\big) = \varnothing.\] 
    \end{claim}
    \begin{proof}[Proof of Claim]
        By \thref{geod meets finitely many}, $\Geod(x,D(\xi))$ spans a finite subcomplex, say $K$. For every $v \in V(\xi)$ and simplex $\sigma$ of $st(v) \cap K \setminus D(\xi)$, let $U_{v,\sigma}$ be an open neighborhood of $\xi$ in $\total{v}$ avoiding $\total{\sigma}$. Then
        \[ V_v:= \bigcap_{\sigma \subset st(v) \cap K \setminus D(\xi)} U_{v,\sigma}\] 
        is a finite intersection of open sets in $\total{v}$, each containing $\xi$. If $v \in V(\xi)$ has $st(v) \cap K = \varnothing$, then this intersection above is empty and we set $V_v = \total{v}$. Let $\VV'$ be a $\xi$--family contained in the sets $\{V_v, v \in V(\xi)\}$, and let $\VV$ be a $\xi$--family $1$--refined in $\VV$. This is our desired $\VV$.

        For a contradiction, suppose $y \in U'\cap \tCone_{\VV,\frac{1}{2}}(\xi)$. Since $y \in U' \subset V_{r',\delta'}(\eta)$, there is some $y' \in [v_0,y] \cap B(x,\delta') \setminus B(v_0,r')$. Because $B(x,\delta') \subset st(x)$, we know $\sigma_x \subseteq \sigma_{y'}$, so there is a path of simplices of length at most $1$ from $\sigma_{y'}$ to $\sigma_x$ which doesn't meet $D(\xi)$. Because $y \in \tCone_{\VV,\frac{1}{2}}(\xi)$, $[v_0,y]$ meets $D^\frac{1}{2}(\xi)$, so $\sigma_{\xi,\frac{1}{2}}(y)$ exists. Because $y' \in U' \subset W_1$ and $N(\xi) \subset X \setminus W_1$, we know $[v_0,y]$ leaves $D^\frac{1}{2}(\xi)$ before reaching $y'$, hence $\sigma_{\xi,\frac{1}{2}}(y') = \sigma_{\xi,\frac{1}{2}}(y)$. Since $y \in \tCone_{\VV,\frac{1}{2}}(\xi)$ and being in the $\tCone$ is a statement about exit simplices, we have $y' \in \tCone_{\VV,\frac{1}{2}}(\xi)$. Combining this with the path of simplices from $\sigma_{y'}$ to $\sigma_x$ and the Refinement \thref{Refinement Lemma}, we see that $x \in \tCone_{\VV',\frac{1}{2}}(\xi)$, which means that $\total{\sigma_{\xi,\frac{1}{2}}(x)} \subset V'_v$ in $\total{v}$ for every vertex $v \in \sigma_{\xi,\frac{1}{2}}(x) \cap V(\xi)$. But $\sigma_{\xi,\frac{1}{2}}(x)$ is a simplex of $K$ and $V'_v$ was chosen specifically to avoid simplices of $K$, so this is a contradiction.
    \end{proof}

    With this claim, it's easy to see that $V_{\VV,\frac{1}{2}}(\xi) \cap V_{U'}(\eta) = \varnothing$. Every point in $V_{\VV,\frac{1}{2}}(\xi)$ has an associated point or set in $\overline{X}$, either $p(z), \eta,$ or $D(\xi')$. This associated point or set is contained in $D(\xi) \cup \tCone_{\VV,\frac{1}{2}}(\xi)$, hence it cannot meet $U'$, so the point cannot be in $V_{U'}(\eta)$.
\end{proof}

\begin{lemma}[Case 2: $\xi \in \partial_{Stab}G$] Let $\xi \in \partial_{Stab}G$ and $V_{\UU^0,\varepsilon}(\xi) \in \mathcal{
O}_{\overline{Z}}(\xi)$. Then there exists a $\xi$--family $\UU$ and $\e' \in (0,1)$ so that $V_{\UU,\e'}(\xi) \subset V_{\UU^0,\e}(\xi)$ and every point not in $V_{\UU^0,\varepsilon}(\xi)$ admits a neighborhood avoiding $V_{\UU,\e'}(\xi)$. Further, if $z' =\xi' \in \partial_{Stab}G$, then this neighborhood has the form $V_{\VV,\e}(\xi')$ where $\tCone_{\UU^0,\e}(\xi) \cap \tCone_{\VV,\e}(\xi') = \varnothing.$
\end{lemma}
\begin{proof}
    For each $v \in V(\xi)$, $\total{v}$ is metrizable, hence regular, so we can choose an open neighborhood $V^1_v $ of $\xi$ so that $\overline{V^1_v} \subset U^0_v$, and then choose a $\xi$--family $\UU^1$ which is $d_{max}$--refined and nested in the sets $\{V^1_v, v \in V(\xi)\}$. Repeat this process to construct $\xi$--families $ \UU^4 \subset \UU^3 \subset \UU^2 \subset \UU^1 \subset \UU^0$, so that for $i = 1,2,3,4$, $\UU^i$ is $d_{max}$--refined and nested in $\UU^{i-1}$, and for every vertex $v \in V(\xi)$, we have $\overline{U^{i}_v} \subset U^{i-1}_v$. We show $V_{\UU^4,\e/2}(\xi)$ is the desired neighborhood of $\xi$.

    \begin{claim}
        If $\eta \in \partial X \setminus V_{\UU^0,\e}(\xi)$, then there is some open neighborhood $W \subset \overline{X}$ so that $V_W(\eta)$ avoids $V_{\UU^3,\e}(\xi) = \varnothing$.
    \end{claim}
    \begin{proof}[Proof of Claim]
        It is enough to choose $R,\delta$ so that
        \[V_{R,\delta}(\eta) \cap \bigg(D(\xi) \cup \tCone_{\UU^3,\e}(\xi)\bigg) = \varnothing.\] 
        Since $N(\xi)$ is bounded, we can choose $R$ so that $N(\xi) \subset B(v_0,R)$. Then $V_{R,\delta}(\eta) \cap D(\xi) = \varnothing$ for any $\delta$, since any point in $V_{R,\delta}(\eta)$ is at least $R$ from $v_0$. Let $x = \gamma_\eta(R)$. There are two cases.

        If $[v_0,\eta)$ does not meet $D(\xi)$, then we can choose $0 < \delta < d([v_0,\eta) , D(\xi))$. If $y \in V_{R,\delta}(\eta)$, then the CAT(0) inequality implies $[v_0,y]$ stays within $\delta$ of $[v_0,\eta]$ for time $R$, hence $[v_0,y]$ avoids $D(\xi)$. Therefore $y \notin \tCone_{\UU^3,\e}(\xi)$ by the contrapositive of the Genuine Shadow \thref{Genuine Shadows}.

        If $[v_0,\eta)$ does meet $D(\xi)$, then because $x \notin \tCone_{\UU^0,\e}(\xi)$ there is a vertex $v$ of $\sigma_{\xi,\varepsilon}(x)\cap D(\xi)$ so that $\total{\sigma_{\xi,\varepsilon}(x)} \nsubseteq U^0_v$ in $\total{v}$. The Star Lemma \ref{Star} gives a constant $\delta > 0$ so that for every $y \in B(x,\delta) \setminus D^\e(\xi)$, $[v_0,y]$ goes through $D^\varepsilon(\xi)$ and $\sigma_{\xi,\varepsilon}(y)\subset st(\sigma_{\xi,\varepsilon}(x))$. Because $x \notin N(\xi)$, we can decrease $\delta$ if necessary to assume $B(x,\delta) \cap D^\e(\xi) = \varnothing$. For a contradiction, suppose $y \in V_{R,\delta}(\eta) \cap \fitTilde{\Cone}_{\UU^3,\varepsilon}(\xi)$. The star condition and the definition of the cone imply that
    
        \[\total{\sigma_{\xi,\varepsilon}(y)} \subset (\total{\sigma_{\xi,\varepsilon}(\eta)} \cap U^3_v)\]
        for any vertex $v \in \sigma_{\xi,\e}(x) \cap V(\xi)$. But $\UU^3$ is nested in $\UU^0$, so this implies $\total{\sigma_{\xi,\e}(\eta)} \subset U^0_v$, hence $\eta \in V_{\UU^0,\e}(\xi)$. This is a contradiction, so no such $y$ can exist.
    \end{proof}

    \begin{claim}
        If $\xi' \in \partial_{Stab}G \setminus V_{\UU^0,\e}(\xi)$, then there is some $\xi'$--family $\VV$ so that $V_{\VV,\e}(\xi')$ avoids $V_{\UU^3,\e}(\xi)$, and $\tCone_{\UU^3,\e}(\xi) \cap \tCone_{\VV,\e}(\xi') = \varnothing.$
    \end{claim}
    \begin{proof}[Proof of Claim]
        We construct a $\xi'$--family very carefully. Consider a vertex $v \in V(\xi)$. If $v \notin V(\xi)$, let $V'_v = \total{v}$. Otherwise, if $v \in V(\xi) \cap V(\xi')$, then $ \xi' \notin U^1_v$ -- if it was, then $\xi'$ would be in $V_{\UU^0,\e}(\xi)$ by the choice of $\UU^1$ and \thref{Double Refinement}, but it's not. Further, we chose $U^2_v$ so that $\overline{U^2_v} \subset U^1_v$, so we can choose a neighborhood $V'_v \subset \total{v}$ of $\xi'$ so that $V'_v \cap U^2_v = \varnothing$. 
        
        By \thref{geod meets finitely many}, $\Geod(v_0,D(\xi))$ is contained in a finite subcomplex $K$. For every vertex $v$ of $D(\xi')$ and simplex $\sigma \subset (st(v) \cap K)\setminus D(\xi')$, let $V_{v,\sigma}'$ be a neighborhood of $\xi'$ in $\total{v}$ disjoint from $\total{\sigma}$. Let

        \[W_v = V'_v \cap \bigg(\bigcap_{\sigma \subset (st(v)\cap K)\setminus D(\xi')} V'_{v,\sigma}\bigg)\]
        and let $\VV$ be a $\xi'$--family $d_{max}$--nested in a $\xi'$--family built from the sets $\{W_v,v \in V(\xi')\}$. We show $V_{\UU^3,\varepsilon}(\xi) \cap V_{\VV,\varepsilon}(\xi') = \varnothing$, and we begin by showing the cones are disjoint. 
        
        For a contradiction, suppose $x \in \tCone_{\UU^3,\varepsilon}(\xi) \cap \tCone_{\VV,\varepsilon}(\xi')$. By the Genuine Shadow \thref{Genuine Shadows}, $[v_0,x]$ goes through both $D(\xi)$ and $D(\xi')$. Let $\sigma,\sigma'$ be the last simplex met by $[v_0,x]$ in $D(\xi),D(\xi')$. There are $3$ cases, depending on which of $\sigma, \sigma'$ is met first and if they are distinct.

        If $\sigma = \sigma'$, then $[v_0,x]$ leaves both domains at the same time. Let $\tau \subset Lk(\xi) \cap Lk(\xi')$ be the next simplex met by $[v_0,x]$ after $\sigma = \sigma'$. Going backwards along $[v_0,x]$ from the point it leaves $D^\e(\xi)$ to $\tau$ gives a path of simplices from $\sigma_{\xi,\e}(x)$ to $\tau$ of length at most $d_{max}$, and the Crossing \thref{Crossing} combined with the assumption that $x \in \tCone_{\UU^3,\e}(\xi)$ implies $\total{\tau}\subset U^2_v$ for any vertex $v$ of $\sigma = \sigma'$. Applying the same strategy to $\sigma_{\xi',\e}(x)$, we have $\total{\tau} \subset W_v$. But then  $\total{\tau}\subset U^2_v \cap W_v$, which is a contradiction because $U^2_v \cap W_v = \varnothing$. Hence $\sigma \neq \sigma'$.
        
        If $[v_0,x]$ meets $\sigma$ first and then $\sigma'$, let $x' \in \sigma' \cap [v_0,x]$. If $x,x'$ are both far from $D(\xi)$ then $\sigma_{\xi,\e}(x) = \sigma_{\xi,\e}(x')$, but if one of $x,x'$ is close to $D(\xi)$, then these exit simplices might be different. Either way, $[v_0,x]$ gives a path of simplices in $Lk(\xi)$ from $\sigma_{\xi,\varepsilon}(x)$ to $\sigma_{\xi,\varepsilon}(x')$ of length at most $d_{max}$. Applying the Crossing \thref{Crossing} to this path with $x \in \tCone_{\UU^3,\varepsilon}(\xi)$ and $\UU^3$ $d_{max}$--nested in $\UU^2$, we get $\total{\sigma_{\xi,\varepsilon}(x')} \subset U^2_v$ for any vertex $v$ of $\sigma_{\xi,\varepsilon}(x') \cap D(\xi)$. This means $x' \in D(\xi') \cap \fitTilde{\Cone}_{\UU^2,\varepsilon}(\xi)$, so Lemma \ref{Double Refinement} implies $\xi' \in V_{\UU^0,\varepsilon}(\xi)$, a contradiction.

        Thus $[v_0,x]$ must meet $\sigma'$ first and $\sigma$ second. Then the first simplex after $\sigma'$ along $[v_0,x]$, say $\tau$, is contained in $K$. There is a path of simplices along $[v_0,x]$ from $\sigma_{\xi',\varepsilon}(x)$ to $\tau$ of length at most $d_{max}$. Applying the Crossing \thref{Crossing} to this path with $x \in \tCone_{\VV,\varepsilon}(\xi')$ and $\VV$ $d_{max}$--nested in the sets $\{W_v,v \in V(\xi')\}$ implies $\total{\tau} \subset W_v$ for any vertex $v$ of $\tau \cap D(\xi')$. But $W_v$ was explicitly chosen to avoid $\total{\tau}$ because $\tau \subset K$ so this is again a contradiction. 
        
        This shows no such $x$ can exist and the cones are disjoint. This immediately implies that $V_{\UU^3,\e}(\xi)$ and $V_{\VV,\e}(\xi')$ cannot intersect in $\partial X$ and we show they cannot intersect in $Z$ or $\partial_{Stab}G$ either. 

        Suppose $z \in Z \cap  V_{\UU^3,\e}(\xi)\cap V_{\VV,\e}(\xi')$. Because the cones are disjoint, we must have $z \in W_{\UU^3,\e}(\xi) \cap W_{\VV,\e}(\xi')$, hence $x:= p(z) \in D^\e(\xi) \cap D^\e(\xi')$. If there is a vertex $v \in \sigma_x \cap V(\xi) \cap V(\xi')$, then in $\total{v}$, we have $c(z) \in V_v \cap U^3_v$, which is impossible because these sets are disjoint. If there is no such vertex, then $\sigma_x$ is not entirely contained in either domain and we can choose vertices $v \in \sigma_x \cap V(\xi)$ and $v' \in \sigma_x \cap V(\xi')$ so that $[v,v']$ is not contained in either domain. Thus $\sigma_x$ or one of its faces is contained in $K$. and this again leads to a contradiction.

        Suppose $\xi'' \in \partial_{Stab}G\cap  V_{\UU^3,\e}(\xi)\cap V_{\VV,\e}(\xi')$. If there is a vertex $v \in V(\xi') \cap V(\xi'') \setminus V(\xi)$, then $v \in \tCone_{\UU^3,\e}(\xi) \cap D(\xi')$, and \thref{Double Refinement} implies $\xi' \in V_{\UU^1,\e}(\xi)$, a contradiction. If there is a vertex $v \in V(\xi) \cap V(\xi'') \setminus V(\xi')$, then $v \in \tCone_{\VV,\e}(\xi')$, hence $[v_0,v]$ meets $D(\xi')$ by the Genuine Shadows \thref{Genuine Shadows}. Hence $\sigma_{\xi',\e}(v) \subset K$ and we find a contradiction because $\tCone_{\VV,\e}(\xi')$ avoids the simplices of $K$. Thus $D(\xi'') \subset D(\xi) \cap D(\xi'')$. But this is also impossible, for if $v \in V(\xi) \cap V(\xi) \cap V(\xi'')$, then we must have $\xi'' \in U^3_v \cap V_v$, but this intersection is empty. Thus no such $\xi''$ can exist. 
    \end{proof}

    \begin{claim}
        If $z \in Z \setminus V_{\UU^0,\e}(\xi)$, then there is an open neighborhood $W \subset Z$ of $z$ which avoids $V_{\UU^2,\e/2}(\xi)$. 
    \end{claim}
    \begin{proof}
        Let $x = p(z)$ and consider the two cases $d(x,D(\xi)) >\e/2$ and $d(x,D(\xi)) \leq \e/2$ separately. 
        
        If $d(x,D(\xi)) > \e/2$, then we can choose $\delta_1$ so that $B(x,\delta_1) \cap D^{\e/2}(\xi) = \varnothing$. If $[v_0,x]0$ does not meet $D(\xi)$, then using the CAT$(0)$ inequality and the convexity of $D(\xi)$, we can choose a $\delta_2$ so that $[v_0,y]$ does not meet $D(\xi)$ for any $y \in B(x,\delta_2)$, hence $B(x,\delta_2) \cap \tCone_{\mathcal{W},\e'}(\xi) = \varnothing$ for any $\xi$--family $\mathcal{W}$ and $\e'$ by the contrapositive of \thref{Genuine Shadows}. Taking $\delta= \min(\delta_1,\delta_2)$, $p^{-1}(B(x,\delta))$ is an open neighborhood of $z$ avoiding $V_{\UU^0,\e}(\xi)$. If $[v_0,x]$ \emph{does} meet $D(\xi)$, we can apply the Star \thref{Star} to find a $\delta_2$ so that for all $y \in B(x,\delta_2)$, $[v_0,y]$ goes through $D^{\e/2}(\xi)$ and $\sigma_{\xi,\e/2}(y) \subset st(\sigma_{\xi,\e/2}(x)$. Because $x \notin \tCone_{\UU^3,\e/2}(\xi)$, we can let $\delta = \min(\delta_1,\delta_2)$ again and see that 
        \[B(x,\delta) \cap \big(D^{\e/2}(\xi) \cup \tCone_{\UU^3,\e/2}(\xi)\big) = \varnothing.\]
        Thus $p^{-1}(B(x,\delta))$ is the desired neighborhood of $z$. 

        If $d(x,D(\xi)) \leq \e/2$, then we can choose $\delta$ so that $B(x,\delta) \subset D^\e(\xi) \cap st(\sigma_x)$. Because $z \notin V_{\UU^0,\e}(\xi)$, $c(z) \notin U^0_v$ for any vertex $v \in \sigma_x \cap V(\xi)$. By the choice of $U^1_v$, we can choose a neighborhood $U$ of $c(z)$ in $\total{v}$ so that $U \cap U^1_v = \varnothing$. Applying \thref{nbhd of z in Z} to $\delta, U$, we get the desired neighborhood of $z$. 
    \end{proof}

    By \thref{ContainmentOfV_UU}, $V_{\UU^4,\e/2}(\xi)$ is contained in each of the neighborhoods of $\xi$ used in the preceding claims. The preceding claims show that any point $z \in \overline{Z} \setminus V_{\UU^0,\e}(\xi)$ has a neighborhood avoiding $V_{\UU^4,\e/2}(\xi)$, proving the lemma.
\end{proof}

\begin{corollary}\thlabel{DisjointCones}
        Let $\xi,\xi' \in \partial_{Stab}G$ be distinct points and let $\varepsilon \in (0,1)$. There exists a $\xi$--family $\UU$ and $\xi'$--family $\UU'$ so that $V_{\UU,\varepsilon}(\xi) \cap V_{\UU',\varepsilon}(\xi') = \varnothing$ and $\fitTilde{\Cone}_{\UU,\varepsilon}(\xi) \cap \fitTilde{\Cone}_{\UU',\varepsilon}(\xi') = \varnothing$ for any $\e \in (0,1)$.
\end{corollary}
\begin{proof}
    Beginning with any neighborhood of $\xi$, we can apply the previous lemma to receive suitable neighborhoods of $\xi,\xi'$ so that the cones are disjoint.
\end{proof}

\begin{lemma}[Case 3: $z \in Z$]
    If $z \in Z$ and $U \in \mathcal{O}_{\overline{Z}}(z)$, then there is some neighborhood of $W \subset U$ of $z$ so that if $z' \in \overline{Z}\setminus U$, $z'$ admits a neighborhood avoiding $W$. 
\end{lemma}
\begin{proof}
    Let $x = p(z)$. Because $U$ is open in $Z$, $U$ meets $\hatto{\sigma_x}$ in an open set, and because basic open sets are a basis for the product topology on $\hatto{\sigma_x}$, there are open sets $W_1 \subset \sigma_x$ and $W_2 \subset X_{\sigma_x}$ so that  
    \[\{\sigma_x\} \times W_1\times W_2 \subset U \cap \hatto{\sigma_x}\]
    and $x \in W_1 \subset \sigma_x$ and $c(z) \in W_2$. Choose $\delta >0$ so that $B(x,\delta) \subset st(\sigma_x)$ and $\overline{B(x,\delta)} \cap \sigma_x$ is contained in the interior of $\sigma_x$. Decreasing $\delta$ further if necessary, we can assume $\overline{B(x,\delta)} \cap \sigma_x\subset W_1$. Because $X_{\sigma_x}$ is metrizable and hence regular, we can choose a $W_2'$ so that $\overline{W_2'} \subset W_2$. Using \thref{nbhd of z in Z} with $\delta$ and $W_2'$, set 
    \[W = U \cap W_z(B(x,\delta),W_2').\]
    Note that for any $z' \in W$, $p(z') \in B(x,\delta)$ and $c(z')$ can be interpreted as a point of $W_2'$. To show this $W$ satisfies the lemma, fix some $z' \in \overline{Z} \setminus U$ and consider three cases.

    If $z' \in Z$, set $x' = p(z')$. If $d(x,x') > \delta$, then there is a $\delta'$ so that $B(x',\delta') \cap B(x,\delta) = \varnothing$. Then $p^{-1}(B(x',\delta'))$ is an open neighborhood of $z'$ which doesn't meet $W$. If $d(x,x') \leq \delta$, then because $\overline{B(x,\delta)} \cap \sigma_x \subset W_1$ we have $x' \in p(U)$. If $c(z')$ was in $W_2$, then we would have $z' \in U$, which is not true, so $c(z') \notin W_2$. Since $\overline{W_2'} \subset W_2$, we can choose a neighborhood $W_{z'} \subset X_{\sigma_x}$ of $c(z')$ which doesn't meet $W_2'$. We can also choose $\delta'$ so that $B(x',\delta') \subset st(\sigma_x)$. Applying \thref{nbhd of z in Z} to $\delta', W_{z'}$, we get a neighborhood of $z'$ disjoint from $W$. 

    If $z' = \eta\in \partial X$, then we can choose any $R,\delta$ so that $B(x,\delta) \cap V_{R,\delta}(\eta) = \varnothing$. If $U'$ is the interior of $V_{R,\delta}(\eta)$, then clearly $V_U(\eta) \cap W = \varnothing$. 

    If $z' = \xi \in \partial_{Stab}G$, then consider the cases $x \notin D(\xi)$ and $x \in D(\xi)$ separately. If $x \notin D(\xi)$, then apply \thref{cone can avoid any point} to $x$ and receive a $\xi$--family $\UU$ so that $x \notin \tCone_{\UU,\e}(\xi)$ for any $\e$. Choose $\e$ so that $D^\e(\xi) \cap B(x,\delta) = \varnothing$ and let $\UU'$ be a $\xi$--family $1$--refined in $\UU$. For a contradiction, suppose there is a point $x' \in B(x,\delta) \cap \tCone_{\UU',\e}(\xi)$. Then $\sigma_x \subseteq \sigma_{x'}$, so there is a path of simplices of length $1$ from $\sigma_{x'}$ to $\sigma_x$. The Refinement \thref{Refinement Lemma} then implies $x \in \tCone_{\UU,\e}(\xi)$, contradicting the choice of $\UU$. Thus 
    \[B(x,\delta) \cap \big(D^\e(\xi) \cup \tCone_{\UU',\e}(\xi)\big) = \varnothing.\]
    It follows that $W \cap V_{\UU,\e}(\xi) = \varnothing$. 

    If $x \in D(\xi)$, then because $\total{\sigma_x}$ is metrizable, we can choose a neighborhood $U_{\sigma_x} \subset \total{\sigma_x}$ of $\xi$ avoiding $W_2$. Using \thref{balloon prop}, we can extend $U_{\sigma_x}$ to a $\xi$--family $\UU$. For every $z' \in W$, because the cusped space coordinate $c(z') \in W_2'$ so $c(z') \notin U_{\sigma_x}$, $z'$ cannot be in either $W_{\UU,\delta}(\xi)$ or $\tCone_{\UU,\delta}(\xi)$, thus $W \cap V_{\UU,\delta}(\xi) = \varnothing$.
\end{proof}

\begin{corollary}
    $\overline{Z}$ is metrizable.
\end{corollary}
\begin{proof}
    We have shown $\overline{Z}$ has the $T_0$ separation property and is regular, so $\overline{Z}$ is Hausdorff. Since $\overline{Z}$ is also second countable by \thref{Oz is a topology}, Urysohn's metrization theorem \cite[4.4]{Munkres} implies $\overline{Z}$ is metrizable. 
\end{proof}

\subsection{$\overline{Z}$ is compact}

In this section we prove the following.
\begin{theorem}[Compactness]\thlabel{Compactness}
    $\overline{Z}$ is compact.
\end{theorem}


\begin{lemma}\thlabel{Martin's}
    Let $(x_n)_n$ be a sequence of points in $X$ and let $\sigma^n_k$ be the $k^{th}$ simplex met by the geodesic $[v_0,x_n]$. Let $L(k) = \{\sigma_k^n, \, n\geq 0\}$.
    \begin{enumerate}
        \item If some subsequence of $(x_n)_n$ is bounded and all $L(k)$ are finite, then a further subsequence converges to a point $x \in X$. 
        \item If $x_n \longrightarrow \infty$ and all $L(k)$ are finite, then some subsequence converges to a point $\eta \in \partial X$. 
        \item If some $L(k)$ is infinite, let $k$ be minimal so that $L(k)$ is infinite. Then there is some $\sigma \in L(k-1)$, $\xi \in \partial G_{\sigma}$, and subsequence of $(x_n)_n$ so that for any $\xi$--family $\UU$ and $\varepsilon \in (0,1)$, we have $x_n \in \tCone_{\UU,\varepsilon}(\xi)$ for $n$ large enough.
    \end{enumerate}
\end{lemma}

\begin{proof}
    Suppose all $L(k)$ are finite. Then we can pass to a subsequence so that $\sigma^n_k$ is eventually constant. For $r \in \mathbb{N}$, let $\pi_r$ be the projection of $X$ onto $\overline{B(v_0,r)}$. For $r=1$, our choice of subsequence implies $\pi_r(x_n)$ is eventually a sequence in a single closed simplex, which is compact, so we can pass to a subsequence so that $\pi_r(x_n)$ converges. Iterating this for $r=2,3,\ldots$ and taking a diagonal sequence, we get a subsequence so that $\pi_r(x_n)$ converges to a point in $\overline{B(v_0,r)}$ for each $r \in \mathbb{N}$. Because $\overline{X}$ can be topologized as a projective limit of these balls, this implies the $(x_n)_n$ converges to a point of $\overline{X}$. If this sequence is bounded, that point must be in $X$, and if the sequence is unbounded that point must be in $\partial X$. This proves $1$ and $2$. 

    Now suppose $L(k)$ is infinite for some minimal $k$. For any $k' > k$, $L(k')$ must also be infinite, for if $L(k')$ was finite, then $\Geod(v_0,L(k'))$ would meet finitely many simplices by \thref{geod meets finitely many}, but $L(k) \subset \Geod(v_0,L(k'))$, which is a contradiction. After a subsequence, we can assume that
    
    \begin{enumerate}
        \item the geodesics $[v_0,x_n]$ all cross the same first $k-1$ simplices, as in there are open simplices $\sigma_{1},\ldots \sigma_{k-1}$ so that $\sigma_1^n = \sigma_1,\ldots \sigma_{k-1}^n = \sigma_{k-1}$ for all $n \geq 0$,
        \item $\sigma_k^n$ contains $\sigma_{k-1}$ for all $n$, since $\sigma_{k-1}$ has finitely many faces,
        \item the sequence $(\sigma_k^n)_{n}$ is injective,
        \item no $\sigma_k^n$ is contained in $D(\xi)$, since $D(\xi)$ contains finitely many simplices,
        \item all $\sigma_k^n$ lie above the same simplex of $Y$, since there are finitely many $G$ orbits of simplices in $X$. 
    \end{enumerate}

    The properties $3,4$ allow us to apply the Convergence Property \ref{Convergence Property remark} and receive a $\xi \in \partial G_{\sigma_{k-1}}$ so that $\total{\sigma_k^n}$ converge to $\xi$ uniformly in $\total{\sigma_{k-1}}$. THis is our candidate $\xi$ and we fix a vertex $v \in \sigma_{k-1} \subset D(\xi)$. 
    
    Given some $V_{\UU,\e}(\xi) \in \OO_{\overline{Z}}(\xi)$ and let $\UU'$ be a $\xi$--family $d_{max}$--nested in $\UU$. Each $[v_0,x_n]$ meets $D(\xi)$, passes through $\sigma_{k-1}$, and then meets $\sigma_k^n \nsubseteq D(\xi)$. Because domains are convex, this means $x_n \notin D(\xi)$ for each $n$ and $\sigma_k^n$ is the first simplex of $Lk(\xi)$ met by $[v_0,x_n]$. Following along $[v_0,x_n]$, we get a path of simplices of length at most $d_{max}$ from $\sigma_k^n$ to $\sigma_{\xi,\e}(x_n)$. For large enough $n$, we have $\total{\sigma_k^n}\subset U'_v$ in $\total{v}$, and for such $n$ the Crossing \thref{Crossing} applied to this path of simplices implies $\total{\sigma_{\xi,\e}(x_n)}\subset U_{v'}$ for any vertex $v' \in \sigma_{\xi,\e}(x_n) \cap V(\xi)$. Hence $x_n \in \tCone_{\UU,\e}(\xi)$, and we're done. 
\end{proof}

\begin{corollary}\thlabel{OffToInfinity}
    Let $v$ be a vertex of $X$ and let $(g_n)_n$ be a sequence in $G$ such that $d(v_0,g_nv) \longrightarrow \infty$. Let $\sigma_k^n$ be the $k^{th}$ simplex met by $[v_0,g_nv]$ and let $L(k) = \{\sigma_k^n, \, n \geq 0\}$. After a subsequence, exactly one of the following holds. 
    \begin{enumerate}
        \item If all $L(k)$ are finite, then there is some $\eta \in \partial X$ so that $g_n \hatt{v} \longrightarrow \eta$ uniformly.

        \item If $k$ is minimal so that $L(k)$ is infinite, then there is some $\sigma \in L(k-1)$ and $\xi_0 \in \partial G_\sigma$ so that $g_n \hatt{v} \longrightarrow \xi_0$ uniformly. 
    \end{enumerate}
\end{corollary}
\begin{proof}
    If all $L(k)$ are finite, then \thref{Martin's} implies there is some subsequence and some $\eta \in \partial X$ so that $g_nv \longrightarrow \eta$. Given $V_U(\eta) \in \OO_{\overline{Z}}(\eta)$, then $U$ is an open neighborhood of $\eta$ in $\overline{X}$ and we can find a neighborhood $U' \subset U$ of $\eta$ so that $d(U', X \setminus U) > A+1$ by \thref{RegularityOfBoundary}, and then choose $R,\delta$ so that $V_{R,\delta}(\eta) \subset U'$. Let $W$ be the interior of $V_{R,\delta}(\eta)$. Since $g_n v \longrightarrow \eta$, $g_n v \in W$ for $n$ large enough. For such $n$, any $z \in \hatt{v}$ has $p(z) = v \in W$, hence $z \in V_{W}(\eta)$. Further, if $g_nv \in W$, then because domains have diameter at most $A$,
    \[ g_n v \in D(g_n\xi) \subset B(g_nv,A+1) \subset U\]
    for any $\xi \in \partial G_{v}$, so $g_n \xi \in V_U(\eta)$. This shows that for any $V_{U}(\eta)$, $g_n \hatt{v} \subset V_U(\eta)$ for $n$ large enough, proving case 1.

    If $k$ is minimal so that $L(k)$ is infinite, then \thref{Martin's} implies that after a subsequence, there is some $\sigma \in L(k-1)$ and $\xi_0 \in \partial G_{\sigma}$ so that for any $\xi_0$--family $\UU$ and $\e \in (0,1)$, $g_nv \in \tCone_{\UU,\e}(\xi)$ for $n$ large enough. Because $d(D(\xi_0),g_nv) \longrightarrow \infty$ and domains have bounded diameter, we can choose $N$ so that $d(D(\xi_0),g_nv) \geq A+1$ for all $n \geq N$. This implies that if $\xi \in \partial G_v$ and $n \geq N$, we have $D(g_n\xi) \cap D(\xi_0) = \varnothing$. 
    
    If $V_{\UU,\e}(\xi_0) \in \OO_{\overline{Z}}(\xi_0)$, let $\UU'$ be a $\xi_0$--family $d_{max}$--refined in $\UU$. Using the conclusion of \thref{Martin's} and increasing $N$ if necessary, we can assume $g_nv \in \tCone_{\UU',\e}(\xi_0)$ for all $n \geq N$, which immediately implies $g_nz \in V_{\UU,\e}(\xi_0)$ for all $z \in \hatt{v}$. If $\xi \in \partial G_v$ and $n \geq N$, then because $D(g_n\xi) \cap D(\xi_0) = \varnothing$, every point of $D(g_n\xi)$ can be connected to $g_nv$ by a path of simplices avoiding $D(\xi_0)$. This path of simplices has length at most $d_{max}$ because $D(\xi)$ has at most $d_{max}$ simplices. Because $g_nv \in \tCone_{\UU',\e}(\xi_0)$, the Refinement \thref{Refinement Lemma} implies $D(g_n\xi) \subset \tCone_{\UU,\e}(\xi_0)$, hence $g_n \xi \in V_{\UU,\e}(\xi_0)$ as needed. 
\end{proof}

\begin{lemma}\thlabel{sequence in Z has convergent subseq}
    If $(z_n)_n$ is a sequence in $Z$, then $(z_n)_n$ has a subsequence converging to a point in $\overline{Z}$.
\end{lemma}
\begin{proof}
    Let $x_n = p(z_n)$ and apply \thref{Martin's}. 
    
    If there is a bounded subsequence of $x_n$ and all $L(k)$ are finite, then some subsequence converges to a point $x \in X$. For $n$ large enough so that $x_n \in st(\sigma_x)$, so we can interpret $c(z_n)$ as points in $\internal{\sigma_x}$. Because $\total{\sigma_x}$ is compact, we can pass to a subsequence so that $c(z_n)$ converges to some point $y\in \total{\sigma_x}$. If $y \in \internal{\sigma_x}$, then it's clear that $z_n \longrightarrow (\sigma_x,x,y) \in \hatto{\sigma_x}$, and if $y \in \partial G_{\sigma_x}$, it's clear that $z_n \longrightarrow y$.
    
    If there is an unbounded subsequence of $(x_n)_n$ with all $L(k)$ finite, then $x_n \longrightarrow \eta \in \partial X$ and it is clear that $z_n \longrightarrow \eta$ in $\overline{Z}$ as well. 
    
    Similarly, if some minimal $k$ has $L(k)$ infinite, then we pass to a subsequence and receive some $\xi \in \partial_{Stab}G$ so that for any $\xi$--family $\UU$ and $\e \in (0,1)$, we have $x_n \in \fitTilde{\Cone}_{\UU,\varepsilon}(\xi)$ for $n$ large enough. It follows that $z_n \longrightarrow \xi$. 
\end{proof}

\begin{proof}[Proof of \thref{Compactness}]
    Because $\overline{Z}$ is metrizable, it is enough to show $\overline{Z}$ is sequentially compact. By \thref{Oz is a topology}, $Z$ is dense in $\overline{Z}$ and by \thref{sequence in Z has convergent subseq}, every sequence in $Z$ has a convergent subsequence. It follows that $\overline{Z}$ is sequentially compact.
\end{proof}

\section{Dynamics of the Action}\label{section:Dynamics}

\subsection{$G$ is a Convergence Group}
In this section, we show $G$ acts as a convergence group on $\overline{Z}$ with limit set $\partial G$. Given an arbitrary sequence $(g_n)_n$, we must produce a $\xi_+,\xi_-$ so that $(g_n,\xi_+,\xi_-)$ is an ART (see \thref{defn:Convergence group}). In the course of these proofs we will translate $(g_n)_n$ on the left and right or use the sequence $(g_n^{-1})_n$ instead. The following lemma shows these changes can be undone at the end of the proof to make the original sequence a convergence sequence. 

\begin{lemma}
    Let $H$ be a group acting on a compact metrizable space $M$, let $(g_n)_n$ be an infinite sequence in $H$, and let $\xi_+,\xi_- \in M$. The following are equivalent.
    \begin{enumerate}
        \item $(g_n,\xi_+,\xi_-)$ is an ART.
        \item $(g_n^{-1},\xi_-,\xi_+)$ is an ART. 
        \item $(hg_nk,h\xi_+,k^{-1}\xi_-)$ is an ART, for any $h,k \in H$.
    \end{enumerate}
\end{lemma}
\begin{proof}
    Notice that $(2) \Longrightarrow (1)$ follows from $(1) \Longrightarrow (2)$ using $(g_n^{-1},\xi_-,\xi_+)$. Similarly $(3) \Longrightarrow (1)$ follows from $(1) \Longrightarrow(3)$, so it suffices to show only two implications starting with $(1)$. Suppose $(g_n,\xi_+,\xi_-)$ is an ART.

    Let $U_+, U_-$ be arbitrary neighborhoods of $\xi_+,\xi_-$. For $n$ large enough, we have $g_n(M \setminus U_-) \subset U_+$. For any such $n$, $g_n$ can only use $U_-$ to cover all of $M \setminus U_+$, as in $M \setminus U_+ \subset g_nU_-$. Hence $g_n^{-1}(M\setminus U_+) \subset U_-$. Thus the $g_n^{-1}$ take complements of open neighborhoods of $\xi_+$ uniformly to $\xi_-$, and $(g_n^{-1},\xi_-,\xi_+)$ is an ART. This proves $(1) \Longrightarrow(2)$.

    Let $h,k \in H$ and let $U,V$ be arbitrary neighborhoods of $h\xi_+, k^{-1}\xi_-$. Then $h^{-1}U, kV$ are neighborhoods of $\xi_+,\xi_-$, and for large enough $n$ we have $g_n(M\setminus kV) \subset h^{-1}U$, hence $hg_nk(M \setminus V) \subset U$. This shows $(hg_nk,g\xi_+,k^{-1}\xi_-)$ is an ART and proves $(1) \Longrightarrow (3)$.
\end{proof}

\begin{lemma}\thlabel{Small translation finite}
    Let $(g_n)_n$ be an injective sequence in $G$. Suppose there are vertices $v_0,\,v_1$ of $X$ so that $g_nv_0 = v_1$ for all $n$. Then $(g_n)_n$ is a convergence sequence.
\end{lemma}
\begin{proof}
    By the previous lemma, it is enough to show $(g_1^{-1}g_n)_n$ is a convergence sequence, so we can assume $g_nv_0 = v_0$ for all $n$, hence $(g_n)_n$ is a sequence in $G_{v_0}$. Since $G_{v_0}$ acts on $\total{v_0}$ as a convergence group with limit set $\partial G_{v_0}$, we can pass to a subsequence and choose $\xi_+,\xi_- \in \partial G_{v_0}$ so that $(g_n,\xi_+,\xi_-)$ is an ART for $G_{v_0}$ acting on $\total{v_0}$. This means for all compact subsets $K \subset \total{v_0} \setminus \{\xi_-\}$, $g_nK$ converges to $\xi_+$ uniformly. Our goal is to upgrade this to compact subsets $K \subset \overline{Z} \setminus \{\xi_-\}$ so that $(g_n)_n$ is a convergence sequence for $G$ acting on $\overline{Z}$. We use $v_0$ as the base point for the topology.

    \begin{claim}\label{ClaimF}
        After translating $(g_n)_n$, passing to a subsequence, and relabeling $\xi_+,\xi_-$, there is a finite subcomplex $F \subset D(\xi_-) \cap D(\xi_+)$ so that 
    \begin{enumerate}[label = F\arabic*.]
        \item all $g_n$ fix $F$ pointwise,
        \item for all vertices $v$ of $F$, $(g_n,\xi_+,\xi_-)$ is an ART for $G_v$ acting on $\total{v}$, and
        \item for all $\sigma \subset Lk(F)$, we have $g_n\total{\sigma} \longrightarrow \xi_+$ uniformly in $\total{v}$ for any vertex $v$ of $F \cap \sigma$.
    \end{enumerate}
    \end{claim}
    \begin{proof}[Proof of Claim]

    First, assume that for every simplex $\sigma \subset st(v_0)$, the set $\{g_n \sigma, \, n \geq 0\}$ is infinite. We claim that after a subsequence, $F = v_0$ works. Indeed, F1 and F2 are immediate, and if $v_0 \in \sigma$ and $\xi_- \notin \total{\sigma}$, then $\total{\sigma}$ is a compact subset of $\total{v_0} \setminus \{\xi_-\}$, so $g_n \total{\sigma} \longrightarrow \xi_+$ uniformly already. This proves F3 for all but the simplices $\sigma \subset Lk(v_0) \cap D(\xi_-)$. Enumerate the simplices of $D(\xi_-) \cap Lk(v_0)$ as $\sigma_1, \sigma_2, \ldots, \sigma_m$, with $m \leq d_{max}$. By assumption, $\{g_n\sigma_1 , \, n \geq 0\}$ is infinite, so using the Convergence Property \ref{Convergence Property} we can pass to a subsequence so that $g_n \total{\sigma_1}$ converges uniformly to some $\xi_0 \in \partial G_{v_0}$. Picking a point $z \in \internal{\sigma_1} \subset \total{v_0}$, we have $z \neq \xi_-$, so $g_nz \longrightarrow\xi_+$ because $(g_n,\xi_+,\xi_-)$ is an ART, but also $g_nz \longrightarrow \xi_0$, so $\xi_+ = \xi_0$. Repeating this subsequence maneuver for $\sigma_2, \sigma_3, \ldots, \sigma_m$, we get that $g_n \total{\sigma_i} \longrightarrow \xi_+$ uniformly for each $i$. Since $g_nv_0 = v_0$ already, this shows $F = v_0$ would work.

    If the assumption doesn't hold, there is some simplex $\sigma \subset Lk(v_0)$ so that $\{g_n\sigma, \, n\geq 0\}$ is finite. After a subsequence we can assume $g_n\sigma$ is constant. Then $g_1^{-1}g_n \sigma = \sigma$ for all $n$ and we can replace the ART $(g_n,\xi_+,\xi_-)$ for $G_{v_0}$ acting on $\total{v_0}$ with $(g_1^{-1}g_n,g_1^{-1}\xi_+,\xi_-)$. After this replacement, all the $g_n$ fix $\sigma$, so they fix $\total{\sigma}$ as a closed subset of $\total{v_0}$. Fix some $z \in \internal{\sigma} \subset \total{v}$ and notice that $z \neq \xi_-, \xi_+$ because $\xi_-,\xi_+ \in \partial G_{v_0}$. For all $n$, $g_n z \in \total{\sigma}$ and $g_nz \longrightarrow \xi_+$, so $\xi_+ \in \partial G_\sigma$. Using $g_n^{-1}$ in the same way, we have $\xi_- \in \partial G_\sigma$, hence $\sigma \subset D(\xi_-) \cap D(\xi_+)$. 

    Let $v$ be another vertex of $\sigma$. Viewing $(g_n)_n$ as a sequence in $G_\sigma \subset G_v$ and applying the definition of a convergence group, we can pass to a subsequence and find points $\xi_+',\xi_-' \in \partial G_v$ so that $(g_n,\xi_+',\xi_-')$ is an ART for $G_v$ acting on $\total{v}$. Using the same $z \in \total{\sigma}$ as before, we observe that $(g_nz)_n$ is a sequence in $\internal{\sigma} \subset \total{v}$ which converges to both $\xi_+'$ and $\xi_+$, hence $\xi_+ ' = \xi_+$. Using $g_n^{-1}z$ in the same way, we have $\xi_-' = \xi_-$. Repeating this for all the vertices of $\sigma$ makes $(g_n)_n$ a convergence sequence for each vertex with attractive and repelling points $\xi_+$ and $\xi_-$. This extends the properties of $v_0$ to $\sigma$ and all its vertices.

    At this point, either $\{g_n\sigma' , \, n\geq 0\}$ is infinite for all simplices $\sigma' \subset Lk(\sigma)$ or there is a simplex $\sigma' \subset Lk(\sigma)$ so that $\{g_n\sigma' , \, n\geq 0\}$ is finite. In the first case, we can argue exactly as in the first paragraph, replacing $v_0$ with $\sigma$ and see that $F = \sigma$ satisfies the claim. In the second case, we can argue as above with $\sigma'$ in place of $\sigma$; After possibly translating the sequence by $g_1^{-1}$ and taking subsequences, each $g_n$ fixes $\sigma \cup \sigma' \subset D(\xi_-) \cap D(\xi_+)$ and for each vertex $v$ of $\sigma \cup \sigma'$, $(g_n)_n$ is a convergence sequence in $G_v$ with attractive and repelling points $\xi_+,\xi_- \in \partial G_{v}$. Notice that because $g_1$ fixes $\sigma$, translating the entire sequence by $g_1^{-1}$ will give a new sequence which also fixes $\sigma$, so after translating and relabeling, the new sequence does indeed fix $\sigma \cup \sigma'$.

    We iterate the previous paragraph with $\sigma \cup \sigma'$ instead of $\sigma$. Either $F= \sigma \cup \sigma'$ satisfies the claim, or we repeat the argument and add a third simplex. Because our growing subcomplex is contained in $D(\xi_-) \cap D(\xi_+)$, which has at most $d_{max}$ simplices, we must repeat this argument at most $d_{max}$ times before getting a suitable $F$. This proves the claim.
    \end{proof}
    To show $(g_n)$ is a convergence sequence for $G$ acting on $\overline{Z}$, it is enough to fix arbitrary neighborhoods $V_{\UU,\varepsilon}(\xi_-), V_{\VV,\varepsilon}(\xi_+)$ and find some $N$ so that $n \geq N$ implies
    \[g_n(\overline{Z} \setminus V_{\UU,\varepsilon}(\xi_-)) \subset V_{\VV,\varepsilon}(\xi_+).\]
    Let $\UU^3 \subset \UU^2 \subset \UU^1 \subset \UU$ be a sequence of $\xi_-$--families each $d_{max}$--refined and nested in the next. Let $\VV'$ be a $\xi_+$--family $d_{max}$--refined and nested in $\VV$. If $v \in \sigma \subset D(\xi_+)$, we make $V'_v$ smaller if necessary so that $\total{\sigma} \nsubseteq V'_v$. We will use the contrapositive of this later; if $\sigma \subset N(\xi)$ is a simplex and $v \in \sigma \cap  V(\xi_+)$ so that $\total{\sigma} \subset V'_v$, then $\sigma$ is \emph{not} a simplex of $D(\xi_+)$.

\begin{claim}\thlabel{Claim2ofCase1}
    With $(g_n)_n$ and $F$ satisfying the conclusion of Claim \ref{ClaimF}, there is some $N$ so that $n \geq N$ implies the following.
\begin{enumerate}
        \item For all vertices $v$ of $F$, $g_n (\total{v} \setminus U^3_v) \subset V'_v$.
        \item $g_n\big(\overline{X} \setminus (\fitTilde{\Cone}_{\UU^2,\varepsilon}(\xi_-) \cup D(\xi_-))\big) \subset \fitTilde{\Cone}_{\VV,\varepsilon}(\xi_+)$.
        \item $g_n(D(\xi_-) \setminus F) \subset \fitTilde{\Cone}_{\VV,\varepsilon}(\xi_+)$.
\end{enumerate}
\end{claim}
\begin{proof}[Proof of Claim]
    
    Applying F$2$ of Claim \ref{ClaimF} to each of the finitely many vertices of $F$, we find $N$ large enough to satisfy $(1)$. There are also finitely many simplices $\sigma \subset Lk(F) \cap D(\xi_-)$, and applying F3 to each of these and possibly increasing $N$, we have $n \geq N$ implies $g_n \total{\sigma} \subset V'_v$ for each vertex of $v$ of $\sigma \cap F$. This is our $N$.

    To prove $(2)$, consider a point $x \in \overline{X} \setminus (\fitTilde{\Cone}_{\UU^2,\varepsilon}(\xi_-) \cup D(\xi_-))$. Then $x \notin F$ because $F \subset D(\xi_-)$, so let $\sigma$ be the first simplex met by $[v_0,x]$ after leaving $F$ (recall $v_0$ is the basepoint of our topology). Because the $g_n$ fix $F$ pointwise, the geodesics $g_n[v_0,x] = [v_0,g_nx]$ share the initial segment in $F$ and only differ starting in $g_n\sigma$. For each $n$, the portion of $[v_0,g_nx]$ outside of $F$ gives a path of simplices of length at most $d_{max}$ through $N(\xi_+)$ from $g_n\sigma$ to $\sigma_{\xi_+,\varepsilon}(g_nx)$.

    If $\sigma \subset Lk(F) \cap D(\xi_-)$, then we have chosen $N$ large enough so that $n \geq N$ implies $g_n\total{\sigma} \subset V'_v$, and by the choice of $\VV'$ just before the claim, this implies $g_n\sigma$ is not a simplex of $D(\xi_+)$, hence $g_n\sigma$ is also the first simplex met by $[v_0,g_nx]$ outside of $D(\xi_+)$. The Crossing Lemma \ref{Crossing} applied to the path of simplices from $g_n\sigma$ to $\sigma_{\xi_+,\varepsilon}(g_nx)$ now implies $g_nx \in \fitTilde{\Cone}_{\VV,\varepsilon}(\xi_+)$.

    If $\sigma \nsubseteq D(\xi_-)$, then $[v_0,x]$ gives a path of simplices of length at most $d_{max}$ from $\sigma$ to $\sigma_{\xi_-,\varepsilon}(x)$. Let $v$ be a vertex of $\sigma \cap F$. If $\total{\sigma} \cap U^3_v \neq \varnothing$, then the Crossing Lemma \ref{Crossing} applied to this path of simplices would imply $\total{\sigma_{\xi_-,\varepsilon}}(x) \subset U^2_v$ for any vertex $v$ of $\sigma_{\xi_-,\varepsilon}(x) \cap D(\xi_-)$. But that is the definition of $x \in \fitTilde{\Cone}_{\UU^2,\varepsilon}(\xi_-)$, which is a contradiction. Thus $\total{\sigma} \subset (\total{v} \setminus U^3_v)$, and for $n \geq N$, we have
    \[g_n \total{\sigma} \subset g_n (\total{v} \setminus U^3_v) \subset V'_v.\]
    Exactly as in the previous paragraph, this implies $g_nx \in \fitTilde{\Cone}_{\VV,\varepsilon}(\xi_+)$. This proves $(2)$. 

    For $(3)$, let $\sigma' \subset D(\xi_-) \setminus F$. Then a geodesic from any point of $\sigma'$ to $v_0$ gives a path of simplices of length at most $d_{max}$ from $\sigma'$ to some $\sigma \subset Lk(F) \cap D(\xi_-)$, since $v_0 \in D(\xi_-)$ and domains are convex. For $n \geq N$, we have $g_n\total{\sigma} \subset V_v'$ for any vertex $v$ of $\sigma \cap F$. The Refinement Lemma \ref{Refinement Lemma} then implies $g_n\sigma' \subset \fitTilde{\Cone}_{\VV,\varepsilon}(\xi_+)$. This proves $(3)$.
\end{proof}

\begin{claim}
    With $N$ from the previous claim, $n \geq N$ implies 
    \[g_n(\overline{Z} \setminus V_{\UU,\varepsilon}(\xi_-)) \subset V_{\VV,\varepsilon}(\xi_+).\]
\end{claim}
\begin{proof}[Proof of Claim]
    Fix $z \in \overline{Z} \setminus V_{\UU,\e}(\xi_-)$ and consider three cases.
    
    If $z = \eta \in \partial X$, then by definition $\eta \notin \fitTilde{\Cone}_{\UU,\varepsilon}(\xi_-)$. In particular $\eta \notin \fitTilde{\Cone}_{\UU^2,\varepsilon}(\xi_-)$, so the claim implies that for $n\geq N$, we have $g_n \eta \in \fitTilde{\Cone}_{\VV,\varepsilon}(\xi_+)$, hence $g_n \eta \in V_{\VV,\varepsilon}(\xi_+)$.
    
    If $z = \xi \in \partial_{Stab}G$, then $D(\xi) \cap \fitTilde{\Cone}_{\UU^2,\varepsilon}(\xi_-) = \varnothing$ and for any vertex $v$ of $D(\xi) \cap D(\xi_-)$, we have $\xi \notin U^1_v$ in $\partial G_v$, since either of these conditions would imply $\xi \in V_{\UU,\varepsilon}(\xi_-)$ by Lemma \ref{Double Refinement}. To show $g_n\xi \in V_{\VV,\varepsilon}(\xi_+)$, we have to show various conditions on how $D(\xi)$ interacts with $D(\xi_+)$. To check those conditions, we break $D(\xi)$ into three parts:
    \[D(\xi) = \big(D(\xi) \cap F\big) \sqcup \big(D(\xi) \cap D(\xi_-) \setminus F\big) \sqcup \big(D(\xi) \setminus D(\xi_-)\big).\]

    If $x$ is a point of either the second or third parts, then either $(3)$ or $(2)$ of Claim \ref{Claim2ofCase1} implies $g_nx \in \fitTilde{\Cone}_{\VV,\varepsilon}(\xi_+)$. Note that either because Cones are disjoint from domains by definition, or from following the proof \thref{Claim2ofCase1}, we know $g_n x\notin D(\xi_+)$. This means that $n \geq N$ implies $D(g_n\xi) \setminus F \subset \tCone_{\VV,\e}(\xi)$, and $D(g_n\xi) \cap D(\xi_+) \subset F$, so only in $F$ do we need to check the condition on vertices of $D(g_n\xi) \cap D(\xi)$. If $v = g_nv$ is a vertex of $D(g_n\xi) \cap F$, $(1)$ of Claim \ref{Claim2ofCase1} implies that $g_n\xi \in g_n(\total{v} \setminus U^2_v) \subset V_v'$ as needed. 

    Finally, if $z \in Z$, then let $x = p(z)$. Suppose $x \in F$, so that $g_nx = x$ for all $n$. By definition of $z \notin V_{\UU,\varepsilon}(\xi_-)$, we know that for any vertex $v$ of $\sigma_x$, $c(z) \notin U_v$, so by $(1)$ of the \thref{Claim2ofCase1}, we have $g_nc(z) \in V'_v$ for all such $v$, hence $g_nz \in V_{\VV,\varepsilon}(\xi_+)$. If $x \notin F$, then either $x \in D(\xi_-)$ or $x \in (X \setminus \fitTilde{\Cone}_{\UU^2,\varepsilon}(\xi_-) \cup D(\xi_-))$, and these cases are covered by $(3)$ and $(2)$ of \thref{Claim2ofCase1}.
\end{proof}

As explained just before \thref{Claim2ofCase1}, because $V_{\UU,\e}(\xi_-), V_{\VV,\e}(\xi_+)$ were arbitrary neighborhoods of $\xi_+,\xi_-$, the previous claim shows $(g_n,\xi_+,\xi_-)$ is an ART, proving the lemma.
\end{proof}

\begin{lemma}\thlabel{Small translation infinite}
    Let $(g_n)_n$ be an injective sequence of elements. Suppose that for some (hence any) vertex $v$ the sequence $(g_nv)_n$ is bounded but there are no vertices $u,w$ so that $g_nu = w$ for infinitely many $n$. Then $(g_n)_n$ is a convergence sequence. 
\end{lemma}
\begin{proof}

    We begin with $2$ claims which explore the assumption on how the $(g_n)_n$ act. 

    \begin{claim}\thlabel{ClaimR}
    \
    \begin{enumerate}[label = (R\arabic*)]
        \item If $K, K'$ are finite subcomplexes of $X$, then for all $n$ large enough, $g_nK' \cap K = \varnothing$.

        \item For any vertex $w$ of $X$ and $\xi, \xi' \in \partial _{Stab}G$, the complex spanned by $g_n\Geod(w,D(\xi))$ does not meet $D(\xi')$ for $n$ large enough. 

        \item For any vertex $w$ of $X$, $x \in \overline{X}$, and $\xi \in \partial_{Stab}G$, the geodesic $g_n[w,x]$ does not meet $D(\xi)$ for $n$ large enough.
    \end{enumerate}
    \end{claim}
    \begin{proof}[Proof of Claim]
        
    For a contradiction, suppose $g_n K' \cap K \neq \varnothing$ for infinitely many $n$. Because $K$ is finite, some vertex of $K$ appears infinitely often in the sets $g_nK' \cap K$, say $w$. There are finitely many possibilities for $g_n^{-1}w \in K'$, so some vertex of $K'$ appears infinitely many times, say $u$. But then $g_n u = w$ for infinitely many $n$, contradicting the assumption on $(g_n)_n$. This proves R1.
    
    Both R2 and R3 are special cases of R1. For R2, take $K = D(\xi')$ and $K'$ the subcomplex spanned by $\Geod(w,D(\xi))$. Then $K$ has at most $d_{max}$ simplices and $K'$ is finite by \thref{geod meets finitely many}. 

    For R3, the set $\{g_n w, \, n \geq 0\}$ is bounded by assumption and $D(\xi)$ has diameter at most $A$, so there is some constant $M$ so that $D(\xi) \subset B(g_nw,M)$ for all $n$, hence any intersection between $g_n[w,x]$ and $D(\xi)$ happens in $B(g_nw,M)$. By \thref{Bounded Length To Bounded Number Of Simplices}, the initial segment of $[w,x]$ of length $M$ is contained in a finite subcomplex, say $K'$. Then R3 follows from R1 with $K = D(\xi)$. 
    \end{proof}
    
    \begin{claim}\thlabel{ClaimSuited}

        Let $v_0$ be a vertex of $X$, $\varepsilon \in (0,1)$, $W \subset \overline{X}$, and $\xi_1 \in \partial_{Stab}G$. Fix a vertex $v_1$ of $D(\xi_1)$ as the basepoint to consider cones and let $\UU' \subset \UU$ be two $\xi_1$--families with $\UU'$ $3d_{max}+1$--nested in $\UU$. Suppose that for all $n \geq N$,
    \begin{enumerate}
        \item $g_nv_0 \in \fitTilde{\Cone}_{\UU',\varepsilon}(\xi_1)$ and 
        \item for all $x \in W$, $g_n[v_0,x]$ does not meet $D(\xi_1)$.
    \end{enumerate}
    Then $g_nW \subset \fitTilde{\Cone}_{\UU,\varepsilon}(\xi_1)$ for all $n \geq N$. 
    \end{claim}

    \begin{proof}[Proof of Claim]
        
    Fix $x \in W$ and assume for now that $x \in X$. We will use Lemma \ref{Short Paths of Simplices} three times to get a path of simplices from $\sigma_{\xi_1,\varepsilon}(g_nv_0)$ to $\sigma_{\xi_1,\varepsilon}(g_nx)$ in $N(\xi_1)$, then apply the Crossing Lemma \ref{Crossing} to transfer assumption $(1)$ from $g_nv_0$ to $g_nx$.

    Fix $n \geq N$. Because $D(\xi_1)$ is compact and disjoint from $g_n[v_0,x]$, we can choose $y \in D(\xi_1)$ achieving $d(g_n[v_0,x],D(\xi_1))>0$. Let $\tau$ be a simplex of $Lk(\xi_1)$ whose interior is met by $[y,g_nv_0]$ at a point $u$. Similarly, let $\tau'$ be a simplex of $Lk(\xi_1)$ whose interior is met by $[y,g_nx]$ at a point $u'$. By choosing $u,u'$ closer to $D(\xi_1)$, we may assume $g_nv_0 \neq u, g_nx \neq u$. See Figure \ref{fig:Claim2 diagram}.
    
    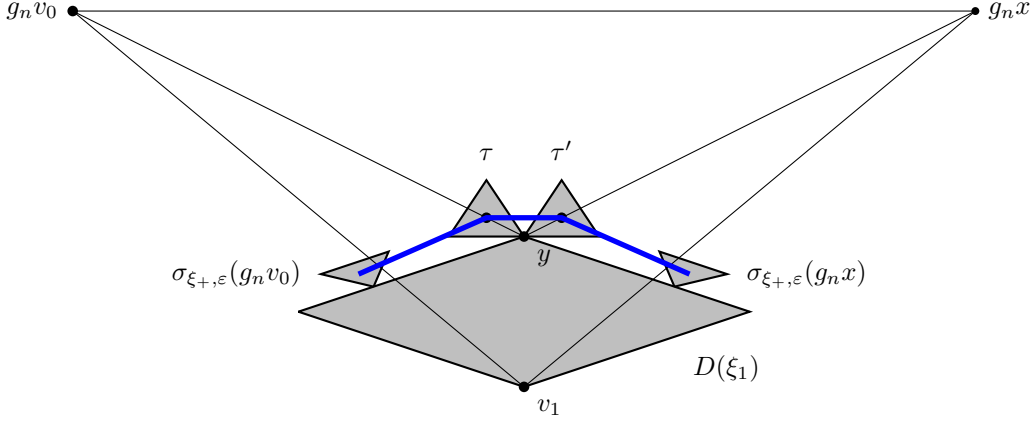
\begin{figure}
        \centering
        \begin{tikzpicture}[
                dot/.style = {circle, fill, minimum size=#1,
              inner sep=0pt, outer sep=0pt},
                    ]  

        \filldraw [fill=lightgray,thick] (-3,0) -- (0,1) -- (3,0) -- (0,-1) --  (-3,0);
        \filldraw [fill = lightgray, thick] (2,.333) -- (2.7,.5) node[label = right:$\sigma_{\xi_+,\varepsilon}(g_nx)$] {} -- (1.8,.8) -- (2,.333); 
        \filldraw [fill = lightgray, thick] (-2,.333) -- (-2.7,.5) node[label = left:$\sigma_{\xi_+,\varepsilon}(g_nv_0)$] {} -- (-1.8,.8) -- (-2,.333); 
        \filldraw [fill = lightgray, thick] (0,1) -- (-.5,1.75)  node[label = above:$\tau$] {} -- (-1,1) -- (0,1); 
        \filldraw [fill = lightgray, thick] (0,1) -- (.5,1.75)  node[label = above:$\tau'$] {} -- (1,1) -- (0,1); 

        \node[dot = 4pt,label =below right:$v_1$] at (0,-1) {};
        \node[dot = 4pt,label =below right:$y$] at (0,1) {};
        \node[label = below right:$D(\xi_1)$] at (2,-.333) {};
        \draw (0,-1) -- (-6,4);
        \draw (0,-1) -- (6,4);
        \draw (0,1) -- (-6,4);
        \draw (0,1) -- (6,4);

        \draw (-6,4) node[dot = 4pt, label = left:$g_nv_0$]{} -- (6,4) node[dot = 3pt, label = right:$g_nx$] {};
        \draw (-.5,1.25) node[dot = 4pt]{} -- (.5,1.25) node[dot = 4pt] {};

        \draw[blue,line width = 2pt] (-2.2, .5) -- (-.5,1.25) -- (.5,1.25) -- (2.2,.5);
    \end{tikzpicture}
        \caption{The situation of Claim \ref{ClaimSuited}. The blue line represents the path of simplices constructed. The left and right hand segments are the first two applications of Lemma \ref{Short Paths of Simplices} and the middle segment is the path of simplices along $[u,u']$.}
        \label{fig:Claim2 diagram}
    \end{figure}
    
    The first two applications of Lemma \ref{Short Paths of Simplices} are easy; Using $v_1,y \in D(\xi_1)$ as points in a convex subcomplex $D(\xi_1)$ and $g_nv_0$ as our finite subcomplex, we have a path of simplices of length at most $d_{max}$ from $\sigma_{\xi_1,\varepsilon}(g_nv_0)$ to $\tau$. Using $g_n\sigma_x$ instead of $g_nv_0$, we have a path of simplices of length at most $d_{max}$ from $\tau'$ to $\sigma_{\xi_1,\varepsilon}(g_nx)$.$^\dagger$ 
    
    Now we need one from $\tau$ to $\tau'$. Let $w$ be the closest point of $D(\xi_1)$ to $g_nv_0$ and consider the geodesic triangle with corners $g_nv_0,w,y$. The leg $[w,y]$ is contained in $D(\xi_1)$ by convexity and $u \in [g_nv_0,y]$ by definition. The CAT$(0)$ inequality implies $d(u,D(\xi_1)) < d(g_nv_0,D(\xi_1))$. Using $g_nx$ instead of $g_nv_0$, we get $d(u',D(\xi_1)) < d(g_nx,D(\xi))$. In any CAT$(0)$ metric space, the distance to a convex set along a geodesic is a convex function, see \cite[II.2]{BH}. Letting $\gamma$ parameterize the geodesic $[u,u']$ and using the geodesic $g_n[v_0,x]$ as our convex set, we get the first of the following inequalities:

    \begin{align*} 
    d(\gamma(t),g_n[v_0,x]) &\leq \max\big(d(u,g_n[v_0,x]),d(u',g_n[v_0,x])\big) \\ 
     & < d(y,g_n[v_0,x])\\
     & = d(D(\xi_1),g_n[v_0,x]).
     \end{align*}
    This shows $[u,u']$ stays strictly closer to $g_n[v_0,x]$ than any point of $D(\xi_1)$, so $[u,u']$ does not meet $D(\xi_1)$. By the convexity of $D(\xi_1)$, we also have that $[u,u'] \subset D^\varepsilon(\xi)$ since $u,u' \in D^\varepsilon(\xi)$. Therefore $[u,u']$ gives a path of simplices in $Lk(\xi_1)$ from $\tau$ to $\tau'$, which has length at most $d_{max}$ since $[u,u']$ is a geodesic in $N(\xi_1)$. Combining these, we have a path of simplices 
    \[\sigma_{\xi_1,\varepsilon}(g_nv_0) \longrightarrow \tau \longrightarrow \tau' \longrightarrow \sigma_{\xi_1,\varepsilon}(g_nx)\]
    Each subpath has length at most $d_{max}$, so the concatenation has length at most $3d_{max}$. Since $g_nv_0 \in \fitTilde{\Cone}_{\UU',\varepsilon}(\xi_1)$ and $\UU'$ is $3d_{max}+1$ nested in $\UU$, the Crossing Lemma \ref{Crossing} applied to this path implies $g_nx \in \fitTilde{\Cone}_{\UU,\varepsilon}(\xi_1)$ as claimed. 
    
    If $x = \eta \in \partial X$, the only change is in the sentence marked $^\dagger$. Since $\eta$ doesn't have a finite subcomplex to apply Lemma \ref{Short Paths of Simplices} to, we use the $\textrm{CAT}(0)$ inequality and choose $x'$ far along $g_n[v_0,\eta)$ so that $\sigma_{\xi_1,\varepsilon}(g_n\eta) \subset st(\sigma_{\xi_1,\varepsilon}(x'))$, and use $x'$ as we used $x$ above. We get a path of simplices of length $3d_{max}$ from $\sigma_{\xi_1,\varepsilon}(g_nv_0)$ to $\sigma_{\xi_1,\varepsilon}(x')$, and the containment $\sigma_{\xi_1,\varepsilon}(g_n\eta) \subset \sigma_{\xi_1,\varepsilon}(x')$ adds possibly one step to the path of simplices, hence the $+1$ in $3d_{max}+1$. This proves the claim.
    \end{proof}
    
    When assumptions $(1)$ and $(2)$ of Claim \ref{ClaimSuited} hold for some $N$, we say $N$ is \emph{suited to} $v_0, v_1, W, \varepsilon, \xi_1$ and $\UU'\subset \UU$, or suited to $W$ if the rest is clear.

    With these 2 claims, we can start on the real proof. If $\partial_{Stab}G = \varnothing$, $G$ is hyperbolic because it acts geometrically on the $\delta$--hyperbolic space $X$ and our main theorem is trivial. Otherwise, choose $\xi_0 \in \partial_{Stab}G$. Because $\overline{Z}$ is compact and metrizable, after a subsequence we can find points $\xi_+,\xi_- \in \partial G$ so that $g_n\xi_0 \longrightarrow \xi_+$ and $g_n^{-1}\xi_0 \longrightarrow \xi_-$. Because the $g_nD(\xi_0)$ lie in a bounded set, $\xi_+,\xi_- \notin \partial X$. Fix a vertex $v_+ \in D(\xi_+)$ as a basepoint.

    \begin{claim}\thlabel{Claim3ofCase2}
        For any $\xi \in \partial_{Stab}G$, $g_n\xi \longrightarrow \xi_+$ and $g_n^{-1}\xi \longrightarrow \xi_-$.
    \end{claim}

    \begin{proof}[Proof of Claim]
    
    With the terminology above, this is simple to argue. Fix an arbitrary open neighborhood of $\xi_+$, say $V_{\UU,\varepsilon}(\xi_+)$, and let $\UU' \subset \UU$ be a $\xi_+$--family $3d_{max}+1$--nested in $\UU$. Because $g_n\xi_0 \longrightarrow \xi_+$, we can choose $N$ so that $n \geq N$ implies $g_n\xi_0 \in V_{\UU',\varepsilon}(\xi_+)$. Using R1 with $K = D(\xi_0)$ and $K' =D(\xi_+)$ and increasing $N$ if necessary, we can assume $n \geq N$ implies $g_nD(\xi_0) \cap D(\xi_+) = \varnothing$, which implies $g_nv_0 \in \fitTilde{\Cone}_{\UU',\varepsilon}(\xi_+)$. This satisfies $(1)$ of \thref{ClaimSuited}. Using R2 and possibly increasing $N$, we have $g_n\Geod(v_0,D(\xi)) \cap D(\xi_+) = \varnothing$ for all $n \geq N$, so assumption $(2)$ holds for all $x \in D(\xi)$. Then $N$ is suited to $D(\xi)$, hence $g_nD(\xi) \subset \fitTilde{\Cone}_{\UU,\varepsilon}(\xi_+)$, hence $g_n\xi \in V_{\UU,\varepsilon}(\xi_+)$ for all $n \geq N$. Since $V_{\UU,\varepsilon}(\xi_+)$ was arbitrary, this proves $g_n\xi \longrightarrow\xi_+$.

    Using $g_n^{-1},\xi_-$ instead of $g_n,\xi_+$ in the paragraph above gives the second half of the claim.
    \end{proof}

    \begin{claim}\thlabel{Claim4ofCase2}
    For any $z \in \overline{Z}$ with $z \neq \xi_-$ and open neighborhood $V_{\UU,\varepsilon}(\xi_+)$ of $\xi_+$, there is a neighborhood $U$ of $z$ and $N$ so that $g_nU \subset V_{\UU,\varepsilon}(\xi_+)$ for all $n \geq N$.
    \end{claim}
    \begin{proof}[Proof of Claim]
        
    Fix $V_{\UU,\varepsilon}(\xi_+)$ and let $\UU'$ be $3d_{max}+1$--nested in $\UU$. Using \thref{Claim3ofCase2} with $\xi_+$, we have $g_n\xi_+ \longrightarrow \xi_+$ and using R1 with $K = K' =D(\xi_+)$, we can find $N_0$ large enough that $g_nv_+ \in \fitTilde{\Cone}_{\UU',\varepsilon}(\xi_+)$ for all $n \geq N_0$. We assume all $N$ below are larger than $N_0$ so that assumption $(1)$ of \thref{ClaimSuited} is always satisfied. There are $3$ cases.

    Suppose $z = \eta \in \partial X$. We find a neighborhood $W \subset \overline{X}$ of $\eta$ and $N$ is suited to $W$ so that $g_nW \subset \fitTilde{\Cone}_{\UU,\varepsilon}(\xi_+)$ for all $n \geq N$, hence $g_nV_W(\eta) \subset V_{\UU,\varepsilon}(\xi_+)$ as desired. 

    Because $(g_nv_+)_n$ and $N(\xi_+)$ are bounded, the number 
    \[M:= \sup_{y \in N(\xi_+), \, n \geq 0}d(g_nv_+,y)\]
    is finite. We claim there are finitely many $n$ so that $g_n[v_+,\eta)$ meets $N(\xi_+)$. If there were infinitely many such $n$, then we could pass to a subsequence and find $x_n \in [v_+,\eta)$ so that $g_nx_n \in N(\xi)$ for all $n$. Each $x_n$ has $d(v_+,x_n) \leq M$, and the initial segment of $[v_+,\eta)$ of length $M$ meets finitely many simplices by \thref{Bounded Length To Bounded Number Of Simplices}, so after a subsequence we can assume all $x_n$ are in the same simplex, say $\sigma = \sigma_{x_n}$ for all $n$. Then $g_n\sigma \cap D(\xi) \neq \varnothing$ for all $n$, and because $\sigma,D(\xi)$ each have finitely many vertices, this leads to a contradiction of the assumption on $(g_n)_n$. This shows that indeed there are only finitely many $n$ so that $g_n[v_+,\eta)$ meets $N(\xi_+)$. Because $D^\e(\xi) \subset N(\xi)$ for all $\e \in (0,1)$, we can fix some $\e \in (0,1)$ and $N$ so that $g_n[v_+,\eta)$ does not meet $D^\e(\xi)$ for all $n \geq N$. 
    
    We claim $N$ is suited to $W = V_{M,\varepsilon}(\eta)$. Let $x \in V_{M,\varepsilon}(\eta)$ and for a contradiction suppose $g_n[v_+,x]$ meets $D(\xi_+)$ at a point $y$. By definition of $M$, $d(g_nv_+,y) \leq M$, so $g_n[v_+,x]$ meets $D(\xi_+)$ in time at most $M$. But during this time, $g_n[v_+,x]$ stays within $\varepsilon$ of $g_n[v_+,\eta)$ because $x \in V_{M,\varepsilon}(\eta)$, hence avoids $D(\xi_+)$, so this contradicts the definition of $\varepsilon$ and no such $y$ can exist. Thus $V_W(\eta), N$ satisfy the claim.

    Suppose $z = \xi \in \partial_{Stab}G$. We construct a neighborhood $V_{\VV',\varepsilon'}(\xi)$ and $N$ suited to $D(\xi) \cup \fitTilde{\Cone}_{\VV,\varepsilon'}(\xi)$. It follows that $g_nV_{\VV,\varepsilon'}(\xi) \subset V_{\UU,\varepsilon}(\xi_+)$ for all $n \geq N$, satisfying the claim.

    Because $\xi \neq \xi_-$, we can choose neighborhoods $V_{\VV,\varepsilon'}(\xi)$ and $V_{\WW,\varepsilon''}(\xi_-)$ with disjoint cones by Lemma \ref{DisjointCones}. Let $\VV'$ be $d_{max}$--nested in $\VV$. Because $g_n^{-1}\xi_+ \longrightarrow \xi_-$ by Claim \ref{Claim3ofCase2} and using R1 with $K = D(\xi_-), \, K' = D(\xi_+)$, we can choose $N$ so that $g_n^{-1}D(\xi_+) \subset \fitTilde{\Cone}_{\WW,\varepsilon''}(\xi_-)$ for all $n \geq N$. Using R2 and possibly increasing $N$, we can assume $g_n\Geod(v_+,D(\xi)) \cap D(\xi_+) = \varnothing$ for all $n \geq N$. 

    We claim that this $N$ is suited to $D(\xi) \cup \fitTilde{\Cone}_{\VV',\varepsilon'}(\xi)$. It is clearly suited to $D(\xi)$ by the last sentence of the previous paragraph. Let $x \in \fitTilde{\Cone}_{\VV',\varepsilon}(\xi)$ and for a contradiction suppose $g_n[v_+,x]$ meets $D(\xi_+)$ for some $n \geq N$, say $y \in g_n[v_+,x] \cap D(\xi_+)$. See Figure \ref{fig:Claim4 diagram}. Now $[v_+,x]$ goes through $D(\xi)$ by the Genuine Shadow \thref{Genuine Shadows}, so let $x'$ be the last point of $D(\xi)$ along $[v_+,x]$ and decompose $g_n[v_+,x] = g_n[v_+,x'] \cup g_n[x',x]$. By the last sentence of the previous paragraph, $y$ cannot be in the first half of this decomposition, so it must be in the second. Translating by $g_n^{-1}$ and recalling that $x'$ is the \emph{last} point of $[v_+,x]$ in $D(\xi)$, we see that $[v_+,g_n^{-1}y]$ goes through $D(\xi)$.

    \begin{figure}
        \centering
        \begin{tikzpicture}[
                dot/.style = {circle, fill, minimum size=#1,
              inner sep=0pt, outer sep=0pt}
                    ]
            
            \filldraw[fill=lightgray, thick](0,0) circle (1.5);
            \filldraw[fill=lightgray, thick, label = below:$D(\xi_+)$](2.75,0) circle (.75);
            \filldraw[fill=lightgray, thick, label = below:$D(\xi_+)$](-2.5,-1.5) circle (.75);
            \draw (0,0) node[dot = 4pt, label = below:$v_+$]{} -- (3.5,0) node[dot = 4pt, label = above right:$x'$]{} --(5,0) node[dot = 4pt, label = below:$g_n^{-1}y$]{} -- (7,0) node[dot = 4pt, label = below:$x$]{};
            \draw (-4,-3) node[dot = 4pt, label = below:$g_nv_+$]{} -- (0,1) node[dot = 4pt, label = below right:$y$]{} -- (1,2) node[dot = 4pt, label = above:$g_nx$]{};
            \node[label = below:$D(\xi_+)$] at (0,-1.5){};
            \node[label = below:$D(\xi)$] at (2.75,-.75){};
            \node[label = below:$g_nD(\xi)$] at (-2.5,-2.25){};
            \node[dot = 4pt] at (-2.5 +.53,-1.5+.53){};

    \end{tikzpicture}
        \caption{A diagram of the situation in Claim \ref{Claim4ofCase2} with $z = \xi$.}
        \label{fig:Claim4 diagram}
    \end{figure}
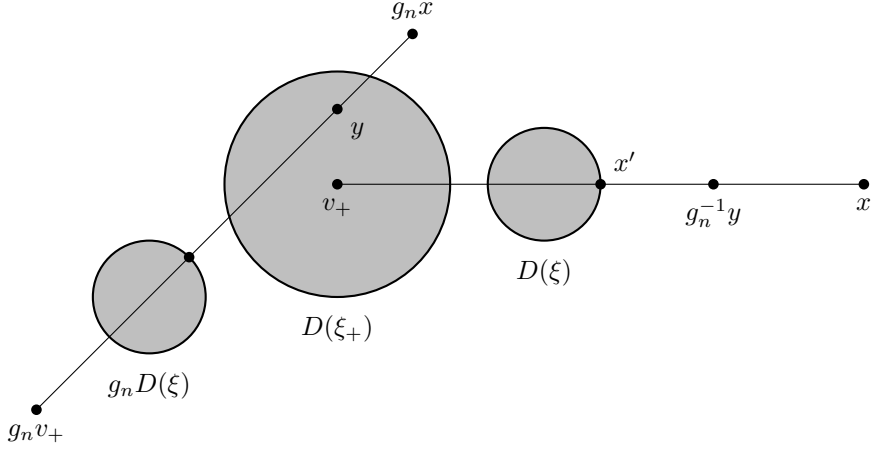
    
    If $g_n^{-1}y$ is outside of $N(\xi)$, then $\sigma_{\xi,\varepsilon'}(g_n^{-1}y) = \sigma_{\xi,\varepsilon'}(x)$, but maybe $g_n^{-1}y$ is close to $D(\xi)$ and these exit simplices are different. Either way, the portion of $[x',x]$ in $D^{\varepsilon'}(\xi)$ gives a path of simplices of length at most $d_{max}$ in $Lk(\xi)$ between these exit simplices. Because $x \in \fitTilde{\Cone}_{\VV',\varepsilon}(\xi)$ and $\VV'$ is $d_{max}$--nested in $\VV$, the Crossing Lemma \ref{Crossing} implies $g_n^{-1}y \in \fitTilde{\Cone}_{\VV,\varepsilon}(\xi)$. On the other hand, $g_n^{-1}y \in g_n^{-1}D(\xi_+) \subset \fitTilde{\Cone}_{\WW,\varepsilon''}(\xi_-)$ by the choice of $N$. Therefore

    \[ g_n^{-1}y \in \fitTilde{\Cone}_{\WW,\varepsilon''}(\xi_-) \cap \fitTilde{\Cone}_{\VV,\varepsilon'}(\xi).\]
    But this set is empty because we chose the cones to be disjoint, so this is a contradiction. Hence $N$ is indeed suited to $D(\xi) \cup \tCone_{\VV,\e'}(\xi)$, and this case is done.

    Finally, suppose $z \in Z$ and let $x = p(z)$. Choose $\delta$ so that $B(x,\delta) \subset st(\sigma_x)$, and let $U = p^{-1}(B(x,\delta))$. By \thref{geod meets finitely many}, there is some constant $M$ so that $[v_+,x]$ meets $M$ simplices ($M$ may be very large if $x$ is far away, but it won't matter), and we let $\UU''$ be a $\xi_+$--family $M+1$--refined in $\UU$. Because $g_n\xi_+\longrightarrow \xi_+$, we can choose $N$ large enough so that $g_nv_+ \in \fitTilde{\Cone}_{\UU'',\varepsilon}(\xi_+)$ for all $n\geq N$. Using R3 and possibly increasing $N$, we can assume $g_n[v_+,x] \cap D(\xi_+) = \varnothing$ so that $[g_nv_+,x]$ gives a path of simplices in $X \setminus D(\xi_+)$ of length $M$ from $g_nv_+$ to $g_nx$.

    If $z' \in U$, let $x' = p(z')$. By the choice of $U$, $\sigma_x \subset \sigma_{x'}$, so we have a path of simplices of length at most $M+1$ from $g_nv_+$ to $g_n\sigma_{x'}$. The Refinement Lemma \ref{Refinement Lemma} applied to this path of simplices implies $g_n\sigma_{x'} \subset \fitTilde{\Cone}_{\UU,\varepsilon}(\xi_+)$, hence $g_nz' \in V_{\UU,\varepsilon}(\xi_+)$. Since $z'$ was arbitrary, this implies $g_nU \subset V_{\UU,\varepsilon}(\xi_+)$ and $U, N$ satisfy the claim.
    \end{proof}
    
    Now we show $(g_n,\xi_+,\xi_-)$ is an ART. Fix a compact set $K \subset \overline{Z} \setminus \{\xi_-\}$ and neighborhood $V_{\UU,\varepsilon}(\xi_+)$ of $\xi_+$. By \thref{Claim4ofCase2}, every $z \in K$ has a neighborhood $U_z$ and $N_z$ so that $g_nU_z \subset V_{\UU,\varepsilon}(\xi_+)$ for $n \geq N_z$. By taking a finite cover of $K$ and a maximum, we can find $N$ so that $g_nK \subset V_{\UU,\varepsilon}(\xi_+)$.
\end{proof}

\begin{lemma}\thlabel{Large translation}
    Suppose $(g_n)_n$ is an injective sequence of elements and suppose for some vertex $v_0$ so that $d(v_0,g_nv_0) \longrightarrow \infty$. Then $(g_n)_n$ is a convergence sequence. 
\end{lemma}

\begin{proof}
    As in the previous lemma, we may assume $\partial_{Stab}G \neq \varnothing$, fix $\xi_0 \in \partial_{Stab}G$ and $v_0 \in D(\xi_0)$, then use the compactness of $\overline{Z}$ to find a subsequence and points $\xi_+, \xi_- \in \partial G$ so that $g_n\xi_0 \longrightarrow \xi_+$ and $g_n^{-1}\xi_0 \longrightarrow \xi_-$. To show $(g_n,\xi_+,\xi_-)$ is an ART it suffices to fix arbitrary neighborhoods $W_-$ of $\xi_-$ and $W_+$ of $\xi_+$ and show that for $n$ large enough, $g_n(\overline{Z} \setminus W_-) \subset W_+$. Without loss of generality, $W_-$ is a basic open neighborhood and we set $K = \overline{Z} \setminus W_-$. We will argue the cases of $\xi_- \in \partial X$ and $\xi_- \in \partial_{Stab}G$ in parallel. 

    If $\xi_- \in \partial X$, let $W_- = V_U(\xi_-)$ for some neighborhood $U \subset \overline{X}$ of $\xi_-$. Since $g_n^{-1}\xi_0 \longrightarrow \xi_-$ in $\overline{Z}$, we know $g_n^{-1}v_0 \longrightarrow \xi_-$ in $\overline{X}$. Because domains have diameter at most $A$ and there are finitely many isometry types of simplices, we can choose a constant $\alpha$ so that $\diam(N(\xi)) < \alpha$ for all $\xi \in \partial_{Stab}G$. By Lemma \ref{RegularityOfBoundary}, we can find a subneighborhood, say $\xi_- \in U' \subset U$ so that $d(X \setminus U, U') > \alpha$. With these choices, if $\xi$ is any point of $\partial_{Stab}G$, then either $D(\xi) \cap U' = \varnothing$, or $D(\xi) \cap U' \neq \varnothing$ hence $D(\xi) \subset U$. 
    
    If $\xi_- \in \partial_{Stab}G$, let $W_- = V_{\UU,\varepsilon}(\xi_-)$ for some $\xi_-$--family $\UU$ and $\varepsilon \in (0,1)$.  Recall that $\delta_0$ is the hyperbolicity constant of $X$, and apply \thref{Bounded Length To Bounded Number Of Simplices} to find a constant $d$ so that any geodesic of length $\delta_0$ meets at most $d$ simplices. Choose $\xi_-$--families $\UU^3 \subset \UU^2 \subset \UU^1 \subset \UU$, each $max(d_{max},d)$--refined in the next.
    
    For $x,y,z \in \overline{X}$, we use $(x,y)_z$ for the Gromov product of $x,y$ based at $z$, and we let $W_k(x) = \{y \in \overline{X} \, | \, (y,x)_{v_0} \geq k\}$. 

    \begin{claim}\thlabel{Claim1ofCase3}
        For any $k \geq 0$, 
    \begin{enumerate}
        \item if $\xi_- \in \partial X$, then $g_n(\overline{X} \setminus U') \subset W_k(g_nv_0)$ for all $n$ large enough, and
        \item if $\xi_- \in \partial_{Stab}G$, then $g_n(\overline{X} \setminus \fitTilde{\Cone}_{\UU^2,\varepsilon}(\xi_-)) \subset W_k(g_nv_0)$ for $n$ large enough.
    \end{enumerate}
    \end{claim}

    \begin{proof}[Proof of Claim]
        Fix $k$ and consider the case $\xi_- \in \partial X$ first. We claim there is a constant $C$ so that for any $x \in X\setminus U'$, $(g_n^{-1}v_0,x)_{v_0} \leq C$ for all $n$. If not, we could take a sequence of points $y_n \in X \setminus U'$ so that $(g_n^{-1}v_0, y_n)_{v_0} \longrightarrow \infty$. But then by the definition of $\partial X$, the sequence $(y_n)_n$ represents $\xi_-$, which contradicts that $y_n \in X\setminus U'$ and $\xi_- \in U'$, so $C$ exists. Two Gromov products add up to the side of a triangle, so for any $x \in X \setminus U'$ we have

    \[d(v_0,g_nv_0) = (v_0,g_nx)_{g_nv_0} + (g_nv_0, g_nx)_{v_0}\]
    The left side grows arbitrarily large by assumption and $(v_0,g_nx)_{g_nv_0} = (g_n^{-1}v_0,x)_{v_0} \leq C$ by the previous paragraph. Therefore we can choose $N$ large enough that $d(v_0,g_nv_0) \geq k+C$ for all $n \geq N$, hence $(g_nv_0,g_nx)_{v_0} \geq k$, proving $(1)$ for $x \in X \setminus U'$. 
    
    If $x \in \partial X \setminus U'$, $x$ can be represented by a sequence $(x_m)_m$ in $X \setminus U'$. Our choice of $N$ above implies $(g_nv_0,g_nx_m)_{v_0} \geq k$ for all $m$ and $n \geq N$. In particular $\liminf_m(g_nv_0,g_nx_m) \geq k$ for all such $n$. Thus
    
    \[(g_nv_0,g_nx)_{v_0}:= \sup_{y_m\rightarrow g_nx} \liminf_m (g_nv_0,y_m)_{v_0} \geq  \liminf_m(g_nv_0,g_nx_m) \geq k \]
    Thus the same choice of $N$ implies $g_n(\overline{X} \setminus U') \subset W_k(g_nv_0)$.

    For $(2)$, consider the case $\xi_- \in \partial_{Stab}G$. As above, we claim there is a constant $C$ so that if $x \notin \fitTilde{\Cone}_{\UU^2,\varepsilon}(\xi_-)$, then $(g_n^{-1}v_0,x)_{v_0} \leq C$ for all $n$.
    
    If not, we could take a sequence $y_n \notin \fitTilde{\Cone}_{\UU^2,\varepsilon}(\xi_-) $ so that $(g_n^{-1}v_0,y_n)_{v_0} \longrightarrow \infty$. The tripod points of the triangles $v_0, y_n, g_n^{-1}v_0$ give points $a_n \in [v_0,y_n]$ and $b_n \in [v_0,g_n^{-1}v_0]$ so that $d(a_n,b_n) \leq \delta_0$ and $a_n,b_n \longrightarrow \infty$. Because $g_n^{-1}\xi_0 \longrightarrow \xi_-$ , $v_0 \in V(\xi_0)$, and $g_n^{-1}v_0 \longrightarrow \infty$, we can choose some $N$ so that $n \geq N$ implies $g_n^{-1}v_0 \in \tCone_{\UU^3,\e}(\xi_-)$. Because $b_n \in [v_0,g_n^{-1}v_0]$ and $b_n \longrightarrow \infty$, we can increase $N$ and assume that $b_n$ occurs after $[v_0,g_n^{-1}v_0]$ leaves $D^\e(\xi_-)$ for all $n \geq N$. Then $\sigma_{\xi_-,\e}(g_n^{-1}v_0) = \sigma_{\xi_-,\e}(b_n)$, and since $g_n^{-1}v_0 \in \tCone_{\UU^3,\e}(\xi_-)$, we have that $b_n \in \tCone_{\UU^3,\e}(\xi_-)$ as well. Increasing $N$ even more, we can assume that $d(b_n,N(\xi_-)) \geq 1+\delta_0$ so that $[a_n,b_n]$ does not meet $D^\e(\xi_-)$. Since $[a_n,b_n]$ has length at most $\delta_0$, it gives a path of simplices of length at most $d$ outside $D(\xi_-)$ from $b_n$ to $a_n$. Because $b_n \in \fitTilde{\Cone}_{\UU^3,\varepsilon}(\xi_-)$ and $\UU^3$ is at least $d$--refined in $\UU^2$, the Refinement \thref{Refinement Lemma} implies $a_n \in \tCone_{\UU^2,\varepsilon}(\xi_-)$, hence $y_n \in \fitTilde{\Cone}_{\UU^2,\varepsilon}(\xi_-)$. But this contradicts the definition of the $y_n$, so $C$ must exist.

    With $C$ in hand, we can proceed exactly as in the previous case with $\tCone_{\UU^2,\e}(\xi_-)$ instead of $U'$ to prove $(2)$.
    \end{proof}

    Recall there is a projection $p:Z \sqcup \partial X \longrightarrow \overline{X}$ from \thref{defn of p}. It is convenient to extend the definition of domains to points $\eta \in \partial X$ by declaring $D(\eta) = \{\eta\}$. For a subset $B\subset \overline{Z}$, we extend the definition of $p$ by setting 
    \[p(B) = \{p(z), \, z \in Z \cap B\} \cup \bigcup_{z \in \partial G \cap B} D(z).\]

    \begin{claim}\thlabel{Claim2ofCase3}
        For any $k \geq 0$, we have $g_n p(K) \subset W_k(g_nv_0)$ for all $n$ large enough. 
    \end{claim}
    \begin{proof}[Proof of Claim]
    Consider $\xi_- \in \partial X$ first. We claim that $p(K) \subset \overline{X} \setminus U'$, so that \thref{Claim1ofCase3} implies \thref{Claim2ofCase3} immediately. Indeed, if $z \in K \cap Z$ or $\eta \in K\cap \partial X$, then $p(z),p(\eta) \notin U$ by definition of $K = \overline{Z} \setminus V_{U}(\xi_-)$. If $\xi \in K \cap \partial_{Stab}G$, then $p(\xi) = D(\xi)$ and if $D(\xi)$ met $U'$, then the choice of $U'$ would imply $D(\xi) \subset U$, contradicting $\xi \in K = \overline{Z} \setminus V_{U}(\xi_-)$. Hence $D(\xi) \subset X \setminus U'$, and we're done with this case.
    
    Consider $\xi_- \in \partial_{Stab}G$ and fix $k \geq 0$. Note that $p(K) \subset N(\xi_-) \cup (\overline{X}\setminus \tCone_{\UU^2,\e}(\xi_-))$. Indeed, if $z \in Z \setminus V_{\UU,\e}(\xi_-)$, then either $p(z) \notin \Cone_{\UU,\e}(\xi_-)$ or $p(z) \in D^\e(\xi_-)$ but $c(z)$ prevents $z$ from being in $V_{\UU,\e}(\xi_-)$. If $\eta \in \partial X \setminus  V_{\UU,\e}(\xi_-)$, then by definition $p(\eta) = \eta \notin \Cone_{\UU,\e}(\xi)$. If $\xi \notin V_{\UU,\e}(\xi_-)$, then $D(\xi) \cap \tCone_{\UU^2,\e}(\xi-) = \varnothing$ by \thref{Double Refinement}, but $D(\xi)$ may meet $D(\xi_-)$. 

    Because domains have bounded diameter, we can choose $M \geq 0$ so that $B(v_0,M)$ contains $N(\xi_-)$. Choose $N$ so that $n \geq N$ implies $d(v_0,g_nv_0) >2k + 2M$. If $y \in B(v_0,M)$ and $n \geq N$, then 
    \[(g_nv_0,g_ny)_{v_0} = \tfrac{1}{2}\big(\underbrace{d(g_nv_0,v_0)}_{\geq 2k+2M} + \underbrace{d(g_ny,v_0)}_{\geq 0} - \underbrace{d(g_nv_0,g_ny)}_{\leq 2M}\big)\geq \tfrac{1}{2}(2k+2M - 2M) = k.\]
    In other words, $n \geq N$ implies $g_nB(v_0,M) \subset W_k(g_nv_0)$. Using (2) of \thref{Claim1ofCase3} and possibly increasing $N$, we can assume that $n \geq N$ also implies $g_n\tCone_{\UU^2,\e}(\xi_-) \subset W_k(g_nv_0)$. Putting all this together, $n \geq N$ implies
     \[g_nK \subset g_n\bigg(N(\xi_-) \cup (\overline{X} \setminus \tCone_{\UU^2,\e}(\xi_-)\bigg) \subset W_k(g_nv_0).\]
    \end{proof}

    \begin{claim}
         For $n$ large enough, $g_nK \subset W_+$. 
    \end{claim}
    \begin{proof}[Proof of Claim]
    First suppose $\xi_+ \in \partial X$ and let $W_+ = V_{U_+}(\xi_+)$. The $W_k(\xi_+)$ form a neighborhood basis for $\xi_+$ in $\overline{X}$, so we can fix some large $k$ with $W_k(\xi_+) \subset U_+$. It is enough to show that for $n$ large enough, $g_np(K) \subset W_k(\xi_+) \subset U_+$.

    Using \thref{Claim2ofCase3}, choose $N$ so that $n \geq N$ implies $g_np(K) \subset W_{k+\delta_0}(g_nv_0)$. Since the $g_n\xi_0 \longrightarrow \xi_+$ in $\overline{Z}$, the $g_nv_0$ must converge to $\xi_+$ in $\overline{X}$, so the Gromov products $(g_nv_0,g_mv_0)_{v_0} \longrightarrow \infty$. Increasing $N$ if necessary, we can assume $n,m \geq N$ implies $(g_nv_0,g_mv_0)_{v_0} > k+\delta_0$. If $x \in p(K)$ and $n,m \geq N$, then the four point condition of hyperbolicity implies 

    \[ (g_nx,g_mv_0)_{v_0} \geq \min\{(g_nx,g_nv_0)_{v_0}, (g_nv_0,g_mv_0)_{v_0}\} - \delta_0 \geq k.\]
    In particular, for any $n \geq N$, we have $\liminf_{m\rightarrow \infty}(g_nx,g_mv_0)_{v_0} \geq k$. Applying this to $(g_nx, \xi_+)_{v_0}$, $n \geq N$ implies

    \[(g_nx,\xi_+) :=\sup_{y_m \longrightarrow \xi_+} \liminf_m(g_nx , y_m)_{v_0} \geq \liminf_m(g_nx,g_mv_0) \geq k, \]
    hence $g_nx \in W_k(\xi_+)$ as needed. 

    Now suppose $\xi_+ \in \partial_{Stab}G$ and let $W_+ = V_{\VV,\varepsilon'}(\xi_+)$. Recall any geodesic of length at most $\delta_0$ meets at most $d$ simplices, and let $\VV'$ be a $\xi_+$--family $d$--refined in $\VV$. Choose $k$ large enough so that the $\delta_0$--neighborhood of $N(\xi_+)$ is contained in $B(v_0,k)$. Applying \thref{Claim2ofCase3} with this $k$, we find $N$ so that $n \geq N$ implies $g_np(K) \subset W_k(g_nv_0)$. Since $g_n\xi_0 \longrightarrow \xi_+$ and $g_nv_0 \longrightarrow \infty$, we can increase $N$ if necessary so that $n \geq N$ also implies $g_nv_0 \in \tCone_{\VV',\e}(\xi_+)$. 

    With this $N$, we claim $g_np(K) \subset \tCone_{\VV,\e'}(\xi_+)$. Indeed, if $x \in p(K)$, then we can consider the tripod points of the triangle $g_nx,g_nv_0,v_0$ to get $a \in [v_0,g_nv_0], b \in [v_0,g_nx]$ so that $d(a,b) \leq \delta_0$ and, because $g_np(K) \subset W_k(g_nv_0)$, $d(a,v_0) = d(b,v_0) \geq k$. Because $N(\xi_+) \subset B(v_0,k)$, we know $\sigma_{\xi_+,\e}(g_nv_0) = \sigma_{\xi_+,\e}(a)$, hence $a \in \tCone_{\VV',\e'}(\xi_+)$. Further, the geodesic $[a,b]$ gives a path of simplices of length at most $d$ from $\sigma_a$ to $\sigma_b$, and the choice of $k$ implies this path of simplices doesn't meet $N(\xi)$. The Refinement \thref{Refinement Lemma} then implies $b \in \tCone_{\VV,\e'}(\xi_+)$. Just as with $a,g_nv_0$, we know that $\sigma_{\xi_+,\e'}(g_nx) = \sigma_{\xi_+,\e'}(b)$, hence $g_nx \in \tCone_{\VV,\e'}(\xi_+)$. This proves that $g_np(K) \subset \tCone_{\VV,\e'}(\xi_+)$, and it follows that $g_nK \subset V_{\VV,\e'}(\xi_+)$
    \end{proof}
    Since $K = \overline{Z}\setminus W_-$ and $W_+$ were arbitrary, the previous claim completes the proof of the lemma.
\end{proof}

\begin{theorem}\thlabel{Convergence Group}
    $G$ acts on $\overline{Z}$ as a convergence group. 
\end{theorem}
\begin{proof}
    If $(g_n)_n$ is an infinite sequence in $G$, then one of \thref{Small translation finite}, \thref{Small translation infinite}, or \thref{Large translation} implies $(g_n)_n$ is a convergence sequence, hence $G$ acts as a convergence group on $\overline{Z}$.
\end{proof}

\subsection{$G$ is geometrically finite}
In the previous section, we saw that $G$ acts as a convergence group on the compact metrizable space $\overline{Z}$ with limit set $\partial G$. To apply Yaman's \thref{Geometrically Finite Convergence implies RelHyp}, we must show $G$ is geometrically finite, which means showing every point of $\partial G$ is either a conical limit point or a bounded parabolic point. 

\begin{lemma}\thlabel{boundaries are limit sets}
    For each simplex $\sigma \subset X$, the limit set $\Lambda G_{\sigma}$ is exactly $\partial G_{\sigma}$.
\end{lemma}
\begin{proof}
    If $G_\sigma$ is finite, then $\partial G_{\sigma} = \Lambda G_{\sigma} = \varnothing$. If $G_\sigma$ is infinite, then $\partial G_{\sigma}$ is nonempty, and $G_{\sigma}$--invariant. Further $\partial G_\sigma$ is compact by \thref{induced topologies}, hence closed because $\overline{Z}$ is Hausdorff.
\end{proof}

\begin{lemma}\thlabel{Conical in Vertices}
    If $\sigma$ is a simplex of $X$ and $\xi \in \partial G_{\sigma}$ is a conical limit point for $G_\sigma$ acting on $\hatt{\sigma}$, then $\xi$ is a conical limit point for $G$ acting on $\overline{Z}$.
\end{lemma}
\begin{proof}
    If $\xi \in \partial G_{\sigma}$ is a conical limit point, then there is a sequence $(g_n)_n$ in $G_\sigma$ and distinct points $\xi_+,\xi_- \in \partial G_{\sigma}$ so that $g_n \xi \longrightarrow \xi_-$ and $g_n \xi'\longrightarrow \xi_+$ for all $\xi' \in \partial G_{\sigma} \setminus \{\xi\}$. It follows from \thref{induced topologies} that $g_n\xi \longrightarrow \xi_-$ as a sequence in $\overline{Z}$, so to show $\xi$ is a conical limit point for $G$ acting on $\overline{Z}$, we just need to upgrade the last part of the previous sentence to `for all $\xi' \in \partial G \setminus \{\xi\}$'. 
    
    Because $G$ acts as a convergence group on $\overline{Z}$, we can pass to a subsequence and find $\xi_+',\xi_-' \in \partial G$ so that $(g_n,\xi_+',\xi_-')$ is an ART. 
    Fix some $z \in \hatto{\sigma}$. Viewing $(g_nz)_n$ as a sequence in $\hatt{\sigma}$ and knowing that $z \neq \xi$, we know that $g_n z \longrightarrow \xi_+$. On the other hand, $z \neq \xi_-$ so $g_nz \longrightarrow \xi_+'$ too, so $\xi_+ = \xi_+'$. For a contradiction, suppose $\xi \neq \xi_-'$. Because $(g_n,\xi_+',\xi_-')$ is an ART, we know $g_n \xi \longrightarrow \xi_+' = \xi_+$. But we also know that $g_n\xi \longrightarrow \xi_-$ and that $\xi_- \neq \xi_+$, so this is a contradiction. Thus $(g_n,\xi_+,\xi)$ is an ART, which completes the upgrade explained at the end of the previous paragraph.  
\end{proof}

\begin{lemma}[Parabolic Points]\thlabel{Parabolic Points}
    Suppose $v_0$ is a vertex of $X$ and $\xi \in \partial G_{v_0}$ is a bounded parabolic point for $G_{v_0}$ acting on $\partial G_{v_0}$. Then $\xi$ is a bounded parabolic point for $G$ acting on $\partial_{Stab}G$ and the stabilizer of $\xi$ in $G$ is finitely generated.
\end{lemma}
\begin{proof}
    Let $P'$ be the stabilizer of $\xi$ in $G$. Then $P'$ fixes $D(\xi)$ set--wise and because domains have finitely many simplices, there is a finite index subgroup $P$ which fixes $D(\xi)$ point--wise. If $v \in V(\xi)$, then $P < G_v$ and $P$ is infinite, so $\xi$ is also a parabolic point for $G_v$ acting on $\partial G_v$. Let $P_v$ be the stabilizer of $\xi$ in $G_v$ and recall that $P_v$ is finitely generated by our definition of relatively hyperbolic groups. We have $P < P_v < P'$ and $P$ finite index in $P'$, so $P_v$ is finite index in $P'$, and $P'$ must be finitely generated because it has a finitely generated finite index subgroup.
    
    Applying \thref{Parabolics are almost cocompact remark} for each $v \in V(\xi)$, we receive a collection of compact sets $\{K_v \subset \total{v} \setminus \{\xi\}, v \in V(\xi)\}$. Each $K_v$ is closed, so we can choose a $\xi$--family $\UU$ with $K_v \cap U_v = \varnothing$ in each $\total{v}$. Using $v_0$ as our basepoint, $K := \partial_{Stab}G \setminus V_{\UU,\frac{1}{2}}(\xi)$ is compact because it is closed in the compact space $\partial_{Stab}G$.
    
    \begin{claim}\thlabel{K is almost fund domain}
        For any $x \in \overline{X} \setminus D(\xi)$, there is some $p \in P$ so that $px \notin \tCone_{\UU,\frac{1}{2}}(\xi)$. 
    \end{claim}
    \begin{proof}[Proof of \thref{K is almost fund domain}]
        Because $x \notin D(\xi)$ and $v_0 \in D(\xi)$ is our basepoint, $[v_0,x]$ leaves $D(\xi)$. Let $y$ be the last point of $[v_0,x]$ in $D^\frac{1}{2}(\xi)$ so that $\sigma_y = \sigma_{\xi,\frac{1}{2}}(x)$ (it's possible that $y = x$ if $x$ is close to $D(\xi)$). Because $P$ fixes $D(\xi)$ point wise and acts by isometries, it follows that for any $p \in P$, $py$ is the last point of $p[v_0,x] = [v_0,px]$ in $D^\frac{1}{2}(\xi)$, hence
        
        \[p\sigma_y = p\sigma_{\xi,\frac{1}{2}}(x) = \sigma_{\xi,\frac{1}{2}}(px).\]
        Fix $v \in \sigma_y \cap D(\xi)$ and consider $G_{\sigma_y}$. If $G_{\sigma_y}$ is infinite, then $\partial G_{\sigma_y} \neq \varnothing$ in $\partial G_v$, and because $K_v$ is a fundamental domain for $P$ acting on $\partial G_v \setminus \{\xi\}$, there is some $p \in P$ so that $p\partial G_{\sigma_y} \cap K_v \neq \varnothing$. Since $K_v \cap U_v = \varnothing$, this means $p\partial G_{\sigma_y} = \partial G_{\sigma_{\UU,\frac{1}{2}}(px)} \nsubseteq U_v$, hence $px \notin \tCone_{\UU,\frac{1}{2}}(\xi)$, as desired. On the other hand, if $G_{\sigma_y}$ is finite, then by the choice of $K_v$, there is some $p \in P$ so that $p\total{\sigma_y} \cap K_v \neq \varnothing$ in $\total{v}$. Just as before, this implies $px \notin \tCone_{\UU,\frac{1}{2}}(\xi)$.
    \end{proof}

    We show that for any $z \in \partial_{Stab}G \setminus \{\xi\}$, there is some $p \in P$ so that $pz \in K$, hence $K \cap \partial G$ is a compact coarse fundamental domain for the action of $P$ on $\partial_{Stab}G \setminus \{\xi\}$.
    
    If $z = \eta \in \partial X$, \thref{K is almost fund domain} immediately implies there is some $p \in P$ so that $p \eta \notin \Cone_{\UU,\frac{1}{2}}(\xi)$, hence $p \eta \in K$.

    If $z = \xi' \in \partial_{Stab}G \setminus \{\xi\}$, then we need to find $p\in P$ so that $p\xi'$ fails one of the two conditions necessary for $p\xi' \in V_{\UU,\frac{1}{2}}(\xi)$. If $D(\xi') \cap D(\xi)$ is empty, then we can choose any $x \in D(\xi')$, apply \thref{K is almost fund domain} to find $p \in P$ so that $px \notin \tCone_{\UU,\frac{1}{2}}(\xi)$, and conclude that $p\xi' \notin V_{\UU,\frac{1}{2}}(\xi)$. If $D(\xi') \cap D(\xi)$ nonempty, then we can choose a vertex $v \in D(\xi')\cap D(\xi)$. Because $K_v$ is a fundamental domain for $P$ acting on $K_v$, there is some $p \in P$ so that $p\xi' \in K_v$, hence $p\xi'\notin V_{\UU,\frac{1}{2}}(\xi)$.
\end{proof}

\begin{lemma}\thlabel{Conical on boundary}
    Every point $\eta \in \partial X$ is a conical limit point.     
\end{lemma}
\begin{proof}
    Fix a point $\eta \in \partial X$ and a vertex $v_0$ of $X$ with $\partial G_{v_0} \neq \varnothing$. The geodesic $[v_0,\eta)$ meets infinitely many simplices and there are finitely many $G$ orbits, so it meets some $G$ orbit infinitely many times. This gives a simplex $\sigma$ and a sequence $(g_n)_n$ so that $[v_0,\eta)$ meets the interior of $g_n^{-1}\sigma$ for each $n$ (we take $g_n^{-1}$ instead of $g_n$ for notational convenience). Then $g_n[v_0,\eta)$ meets the interior of $\sigma$ for each $n$. Fix a vertex $v$ of $\sigma$.

\begin{claim}
    There is a subsequence of $(g_n)_n$, elements $h_n \in G_v$, and a point $\xi_+ \in \partial G \setminus \partial G_v$ so that $h_ng_n\total{v} \longrightarrow \xi_+$ uniformly. 
\end{claim}
\begin{proof}[Proof of Claim]
    We use $v$ as the basepoint of our topology. Let $\sigma^1_n$ be the first open simplex met by $[v,g_nv]$ after leaving $v$. After a subsequence, all $\sigma^1_n$ lie above the same simplex of $Y$, so they correspond to cosets of the same subgroup of $G_v$, and we can choose $h_n \in G_v$ so that all $h_n\sigma^1_n = \sigma_1$ for some simplex $\sigma \subset st(v)$. Then each geodesic $[v,h_ng_nv]$ meets the simplices $v,\sigma_1,\ldots$, and after a further subsequence, each $[v,h_ng_nv]$ leaves $\sigma_1$ along the same face, say $\tau_1$. Let $\sigma_n^2\subset st(\tau_1)$ be the next simplex met by $[v,h_ng_nv]$ after $\tau_1$. We begin an iterative argument.

    \begin{enumerate}
        \item Take a subsequence so that all $\sigma_n^2$ lie over the same simplex of $Y$, which means there is a simplex $\sigma_2 \subset st(\tau_1)$ and a sequence $(a_n)_n$ in $G_{\tau_1}$ so that $\sigma^2_n = a_n\sigma_2$. Applying \thref{nosubsequence3 remark} to $G_{\sigma_1}, G_{\sigma_2}$ in $G_{\tau_1}$, we get a sequence $k_n \in G_{\sigma_1} < G_v$ so that the sets $k_na_n\total{\sigma_2}$ are either constant or converge to a point $\xi_+ \in \partial G_{\tau_1} \setminus \partial G_{\sigma_1}$.

        Suppose we are in this second case where the $k_na_n\total{\sigma_2}$ converge to some $\xi_+ \in \partial G_{\tau_1} \setminus \partial G_{\sigma_1}$. Then the sequence $k_na_n\sigma_2$ must be infinite, so we can apply the second case of \thref{OffToInfinity} to the sequence $(k_nh_ng_nv)_n$. We receive a subsequence and a $\xi_0 \in \partial G_{\tau_1}$ so that $\total{v} \longrightarrow \xi_0$ uniformly. Since the $k_na_n\total{\sigma_2} = k_n\total{\sigma^2_n}$ already converge to $\xi_+$, we must have $\xi_0 = \xi_+$. If $\xi_+ \in \partial G_v$, then the initial portion of each $[v,k_nh_ng_nv]$ connects two simplices of $D(\xi_+)$, namely $v, \tau_1$. Because domains are convex, this initial portion must be contained in $D(\xi_+)$, but it goes through $\sigma_1$ and $\xi_+ \notin \partial G_{\sigma_1}$. This contradiction implies $\xi_+ = \xi_0 \notin \partial G_v$, and we have proven the claim. 

        So either we are finished as in the previous paragraph, or the simplices $k_na_n\sigma_2 = k_n\sigma_n^2$ are all the same, and we relabel this simplex as $\sigma_2$ and all the geodesics $[v_,k_nh_ng_n v]$ meet the simplices $v,\sigma_1,\tau_1,\sigma_2,\ldots$. Let $g^1_n = k_nh_ng_n$, and proceed to 2 or 3 depending on whether or not $G_{\sigma_1} \cap G_{\sigma_2}$ is finite or infinite.

        \item If $G_{\sigma_1}\cap G_{\sigma_2}$ is finite, we apply \thref{OffToInfinity} to $(g_n^1v)_n$. If $g_n^1\total{v}$ converges uniformly to some $\eta' \in \partial X$, then we can take $\xi_+ = \eta'$ and prove the claim. Otherwise, as in the proof of \thref{OffToInfinity}, the path of simplices $\sigma_1, \tau_1, \sigma_2$ extends to a path of simplices $\sigma_1,\tau_1,\sigma_2,\ldots,\sigma_m$ crossed by every geodesic $[v,g_n^1v]$, and $g_n^1\total{v}\longrightarrow \xi_+ \in \partial G_{\sigma_m}$. If $\xi_+ \in \partial G_v$, then the convexity of domains implies this path of simplices is contained in $D(\xi_+)$. We can construct a $\xi_+$--path going through each vertex of this path of simplices, and then \thref{H along xi path is infinite} tells us the path of simplices is stabilized by an infinite subgroup $H$. But $H < G_{\sigma_1} \cap G_{\sigma_2}$, so this is a contradiction. Therefore $\xi_+ \notin \partial G_v$, and the claim is proven. 
        
        \item If $G_{\sigma_1} \cap G_{\sigma_2}$ is infinite, then we pass to a subsequence so that each $[v,g^1_nv]$ leaves $\sigma_2$ in the same face $\tau_2$ and let $\sigma_n^3$ be the next simplex met by $[v,g_n^1v]$ after leaving $\tau_2$. We know $G_{\sigma_1} \cap G_{\sigma_2}$ is full RQC in $G_{\sigma_2}$ by \thref{Limit Set Property} and $G_{\sigma_2}$ is full RQC in $G_{\tau_2}$ so $G_{\sigma_1} \cap G_{\sigma_2}$ is full RQC in $G_{\tau_2}$ as well. This means we can repeat the argument in step 1 replacing $G_{\sigma_1}$ with $G_{\sigma_1} \cap G_{\sigma_2}$ and $\sigma_n^2$ with $\sigma_n^3$. We either prove the claim, or find a simplex $\sigma_3$ and a sequence $k_n' \in G_{\sigma_1} \cap G_{\sigma_2} < G_v$ so that $[v,k_n'g_n^1v]$ all cross the simplices $v,\sigma_1,\tau_1,\sigma_2,\tau_2, \sigma_3,\ldots$. Letting $g_n^2 = k_n'g_n^1$, we again proceed to either step 2 or 3 depending on the cardinality of $G_{\sigma_1} \cap G_{\sigma_2} \cap G_{\sigma_3}$. Hopefully it is clear how to iterate this argument.
    \end{enumerate}    

    Each time we repeat step 3, we add a pair of simplices to a path of simplices $v, \sigma _1, \tau_1, \sigma_2, \ldots$. By \thref{Bounded Length To Bounded Number Of Simplices}, there is a constant $k$ so that any geodesic of length $A+1$ meets at most $k$ simplices. Contrapositively, a geodesic which meets more than $k$ simplices must be longer than $A+1$, hence have finite stabilizer by the definition of $A$ as the acylindricity constant. This means we can repeat step 3 at most $k$ times before getting to step 2, which finishes the proof of the claim. 
\end{proof}

    We pass to a subsequence of $(g_n)_n$ and choose $h_n \in G_v, \xi_+ \in \partial G \setminus \partial G_{v}$ as in the claim and let $b_n = h_ng_n$. Recall that $[v_0,\eta]$ meets the interior of $g_n^{-1}\sigma$ for each $n$, so $b_ng_n\sigma = h_n\sigma \subset st(v)$ and $b_n[v_0,\eta)$ goes through the interior of some simplex of $st(v)$ for each $n$. 
    
    Because $G$ is a convergence group, we can pass to a subsequence and find points $\xi_+',\xi_-' \in \partial G$ so that $(b_n,\xi_+',\xi_-')$ is an ART. Fix some $z \in \internal{v}$ and note that $p(z) = v$. Because $z \neq \xi_+',\xi_-'$, the sequences $b_nz, b_n^{-1}z$ converge to $\xi_+',\xi_-'$ respectively. From the claim, $b_nz \longrightarrow \xi_+$, hence $\xi_+' = \xi_+$. If $U \subset \overline{X}$ is any neighborhood of $\eta$, then for large enough $n$, we have $b_n^{-1}v = g_n^{-1}h_n^{-1}v = g_n^{-1}v \in U$. This shows $b_n^{-1}z \longrightarrow \eta$, hence $\eta = \xi_-'$ and this ART is $(b_n,\xi_+,\eta)$. From the definition of an ART, we know that every point $\xi \in \partial G \setminus \{ \eta\}$ has $b_n \xi \longrightarrow \xi_+$, and after a subsequence we can assume $b_n\eta$ converges to some point of $\partial G$, say $\xi_-$. To show $(b_n)_n$ makes $\eta$ into a conical limit point, it remains to show that $\xi_- \neq \xi_+$. For a contradiction, suppose that $b_n\eta \longrightarrow \xi_+$.

    If $\xi_+ \in \partial X$, then $b_nv_0 \longrightarrow \xi_+$. Each geodesic ray $b_n[v_0,\eta)$ meets the bounded set $st(v)$, hence the Gromov products $(b_nv_0,b_n\eta)_v$ are bounded. But if both $b_nv_0, b_n\eta$ converge to $\xi_+$, then these Gromov products should go to infinity. This contradiction implies $\xi_+ \notin \partial X$, hence $\xi_+ \in \partial_{Stab}G$. 

    Let $\UU$ be any $\xi_+$--family which is admissible with $v$ as a basepoint and recall the last part of the Genuine Shadow \thref{Genuine Shadows}, namely that because $\xi_+ \notin \partial G_v$, any geodesic from a point of $st(v)$ and a point of $\tCone_{\UU,\frac{1}{2}}(\xi_+)$ meets $D(\xi_+)$. For each $n$, the ray $b_n[v_0,\eta)$ can be written as a concatenation $[b_nv_0,x] \cup [x,b_n\eta)$ with $x \in st(v)$. For large enough $n$, $b_nv_0, b_n\eta\in \tCone_{\UU,\frac{1}{2}}(\xi_+)$, hence both pieces of the concatenation meet $D(\xi_+)$ and $x$ is in a geodesic connecting two points of $D(\xi_+)$. Because domains are convex, this means $v \in \sigma_x \subset D(\xi_+)$, but this contradicts $\xi_+ \notin \partial G_v$. 

    Thus $b_n\eta$ cannot converge to $\xi_+$, so $\xi_+$ and $\xi_-$ are distinct and $(b_n)_n$ makes $\eta$ into a conical limit point.
\end{proof}

\begin{proposition}\thlabel{geometrically finite}
    $G$ acts as a convergence group on $\overline{Z}$ with limit set $\partial G$. Further, $G$ acts as a geometrically finite convergence group on $\partial G$, and the parabolic subgroups of $G$ are finitely generated. 
\end{proposition}
\begin{proof}
    $G$ acts as a convergence group on $\overline{Z}$ by \thref{Convergence Group}. To understand its limit set, note that $\partial G$ is clearly $G$--invariant, and is closed because every point of $Z$ admits a neighborhood avoiding $\partial G$. Because $\Lambda G$ is the smallest closed $G$--invariant subset of $\overline{Z}$, this implies $\Lambda G \subset \partial G$. On the other hand, every point of $\partial G$ is either a conical limit point or a bounded parabolic point by \thref{Conical in Vertices}, \thref{Parabolic Points}, and \thref{Conical on boundary}. Conical limit points are in $\Lambda G$ by \thref{limit points are limit points}. Because $\Lambda G$ can also be characterized as the points of $\overline{Z}$ on which $G$ does \emph{not} act properly discontinuously and parabolic points have infinite stabilizer, parabolic points are also in $\Lambda G$. Thus $\partial G \subset \Lambda G$ and we have equality. 

    Since every point of $\partial G$ is either a conical limit point or a bounded parabolic point, $G$ satisfies \thref{Geometrically Finite} and is a geometrically finite convergence group on $\partial G$.

    Every point in $\partial X$ is a conical limit point by \thref{Conical on boundary}, hence not a parabolic point by \thref{limit points are limit points}. Thus all parabolic points are points of $\partial_{Stab}G$, and these have finitely generated stabilizer by \thref{Parabolic Points}.
\end{proof}

We are finally ready to apply \thref{Geometrically Finite Convergence implies RelHyp}. We restate the main theorem and finish its proof. 

\MainTheorem*

\begin{theorem}
    $G$ is relatively hyperbolic. The maximal parabolic subgroups of $G$ are virtually maximal parabolic subgroups of vertex stabilizers and each simplex stabilizer is a full RQC subgroup of $G$.
\end{theorem}
\begin{proof}
    $G$ acts on the compact metrizable space $\partial G$ as a geometrically finite convergence group with finitely generated parabolic subgroups by \thref{Compactness} and \thref{geometrically finite}. By Theorem 2S of \cite{Tukia1994}, either $|\partial G| \leq 2$ or $\partial G$ is an infinite perfect set. If $|\partial G|=0,1$ or 2, then $G$ is either finite, hyperbolic relative to itself, or virtually cyclic, so our main theorem is trivial. Therefore we can assume $\partial G$ is infinite and perfect and \thref{Geometrically Finite Convergence implies RelHyp} implies $G$ is relatively hyperbolic.
    
    Towards understanding the maximal parabolic subgroups, all points of $\partial X$ are conical limit points by \thref{Conical on boundary}, so if $\xi \in \partial_{Stab}G$ is a parabolic point with $P = Stab_G(\xi)$, then $P$ permutes the simplices of $D(\xi)$. Because domains are finite, we can pass to a finite index subgroup, say $P'$, which fixes $D(\xi)$ pointwise, hence $P' < G_\sigma$ for every simplex $\sigma \subset D(\xi)$, and $P'$ makes $\xi$ a parabolic point for $G_v$ acting on $\partial G_v$ for any $v \in V(\xi)$.

    It follows from \thref{boundaries are limit sets} that simplex stabilizers are RQC subgroups of $G$. To see they are \emph{full} RQC subgroups, take a simplex $\sigma \subset X$ and a maximal parabolic subgroup $P = Stab_G(\xi)$ with $\xi \in \partial_{Stab}G$ so that $G_\sigma \cap P$ is infinite. 
    
    We claim $\sigma \subset D(\xi)$. Take an infinite sequence $g_n \in G_\sigma \cap P$, then pass to a subsequence and find points $\xi_+,\xi_- \in \partial G$ so that $(g_n,\xi_+,\xi_-)$ is an ART. Fixing any $z \in \hatto{\sigma}$, we know that $z \neq \xi_+,\xi_-$ so $g_nz \longrightarrow \xi_+$ and $g_n^{-1}z \longrightarrow \xi_-$. But $\hatt{\sigma}$ is closed in $\overline{Z}$, so it contains these limit points and $\xi_+,\xi_- \in \partial G_\sigma$. Now either $\xi = \xi_- \in \partial G_\sigma$, or $\xi \neq \xi_-$ so $g_n\xi = \xi \longrightarrow \xi_+$, hence $\xi = \xi_+ \in \partial G_\sigma$. Either way, $\xi \in \partial G_\sigma$, which means $\sigma \subset D(\xi)$ and our claim is proven. As above, we use that domains are finite to pass to a finite index subgroup $P' <P$ which fixes $D(\xi)$ pointwise, hence fixes $\sigma$. Thus $P' < G_{\sigma} \cap P < P$ and $G_{\sigma} \cap P$ is finite index in $P$ as needed.
\end{proof}


\end{document}